\documentclass[11pt, letterpaper, reqno]{amsart}
\usepackage[utf8]{inputenc} 	
\usepackage{microtype} 			
\usepackage{geometry} 			
\usepackage{fancyhdr} 			
\usepackage{amsmath}			
\usepackage{amsthm}		 		
\usepackage{amssymb}	 		
\usepackage{bm}					
\usepackage{mathrsfs}			
\usepackage{thmtools}
\usepackage{xcolor}				
\definecolor{darkcandyapp}{rgb}{0.66, 0.1, 0.1}
\definecolor{darkblue}{rgb}{0.1, 0.1, 0.66}
\definecolor{darkgreen}{rgb}{0.1, 0.45,.1}
\definecolor{mygreen}{rgb}{0.0, 0.66,0}
\definecolor{darkpurple}{rgb}{.33,0,.33}
\definecolor{whitesmokedark}{RGB}{235,235,235}
\definecolor{gainsboro}{RGB}{220,220,220} 
\usepackage{tikz}				
\usetikzlibrary{decorations.markings}
\usepackage{pgfplots}
\pgfplotsset{compat=1.13}
\usetikzlibrary{patterns}
\usepackage[all]{xy}			
\usepackage{graphicx}			
\usepackage{caption}
\usepackage{subcaption}
\usepackage{floatrow}
\usepackage{enumitem} 			
\usepackage{booktabs}
\usepackage{array}
\usepackage{lmodern}			
\usepackage[T1]{fontenc}		
\usepackage[colorlinks=true, citecolor=cyan, urlcolor=magenta]{hyperref}
\usepackage{cleveref}
\usepackage[backend=bibtex, backref=true, style=alphabetic, firstinits=true, maxbibnames=6, isbn=false, date=year]{biblatex}
\makeatletter
\renewbibmacro{in:}{}
\DeclareFieldFormat{pages}{#1}
\renewcommand*{\bibnamedash}{%
	\leavevmode\raise +0.6ex\hbox to 5.5ex{\hrulefill}.\space\space}

\InitializeBibliographyStyle{\global\undef\bbx@lasthash}

\newbibmacro*{bbx:savehash}{\savefield{fullhash}{\bbx@lasthash}}

\renewbibmacro*{author}{%
	\ifboolexpr{
		test \ifuseauthor
		and
		not test {\ifnameundef{author}}
	}
	{%
		\iffieldequals{fullhash}{\bbx@lasthash}
		{\bibnamedash\addcomma\space}
		{\printnames{author}}%
		\usebibmacro{bbx:savehash}%
		\iffieldundef{authortype}
		{}
		{%
			\setunit{\addcomma\space}%
			\usebibmacro{authorstrg}%
		}%
	}
	{\global\undef\bbx@lasthash}%
}
\makeatother
\newtheorem{propositionx}{Proposition}[section]
\newenvironment{proposition}
{\pushQED{\qed}\propositionx}
{\popQED\endpropositionx}
\newenvironment{propositionp}
{\pushQED{\qed}\propositionx}
{\popQED\endpropositionx}
\newtheorem{theoremx}{Theorem}
\newenvironment{theorem}
{\pushQED{\qed}\theoremx}
{\popQED\endtheoremx}
\newtheorem{corollaryx}[propositionx]{Corollary}
\newenvironment{corollary}
{\pushQED{\qed}\corollaryx}
{\popQED\endcorollaryx}
\newtheorem{lemmax}[propositionx]{Lemma}
\newenvironment{lemma}
{\pushQED{\qed}\lemmax}
{\popQED\endlemmax}

\theoremstyle{remark}
\newtheorem*{remarkxx}{Remark}
\newtheorem{remarkx}[propositionx]{Remark}
\newenvironment{remark*}
{\pushQED{\qed}\remarkxx}
{\popQED\endremarkxx}
\newenvironment{remark}
{\pushQED{\qed}\remarkx}
{\popQED\endremarkx}
\newtheorem{examplex}[propositionx]{Example}
\newenvironment{example}
{\pushQED{\qed}\examplex}
{\popQED\endexamplex}
\newtheorem*{warning*}{Warning}
\newcommand{\dd}{\,\mathrm{d}}

\newcommand{\bbB}{\mathbb{B}}
\newcommand{\bbC}{\mathbb{C}}

\newcommand{\bbH}{\mathbb{H}}

\newcommand{\bbM}{\mathbb{M}}
\newcommand{\bbN}{\mathbb{N}}
\newcommand{\bbO}{\mathbb{O}}
\newcommand{\bbP}{\mathbb{P}}

\newcommand{\bbR}{\mathbb{R}}
\newcommand{\bbS}{\mathbb{S}}

\newcommand{\bbX}{\mathbb{X}}

\newcommand{\calA}{\mathcal{A}}

\newcommand{\calC}{\mathcal{C}}
\newcommand{\calD}{\mathcal{D}}
\newcommand{\calE}{\mathcal{E}}
\newcommand{\calF}{\mathcal{F}}

\newcommand{\calI}{\mathcal{I}}

\newcommand{\calK}{\mathcal{K}}

\newcommand{\calN}{\mathcal{N}}

\newcommand{\calR}{\mathcal{R}}
\newcommand{\calS}{\mathcal{S}}

\newcommand{\calV}{\mathcal{V}}

\newcommand{\scrI}{\mathscr{I}}

\newcommand{\frakM}{\mathfrak{M}}
\newcommand{\frakN}{\mathfrak{N}}

\newcommand{\bfk}{\mathbf{k}}

\newcommand{\bfv}{\mathbf{v}}

\newcommand{\bfx}{\mathbf{x}}
\newcommand{\bfy}{\mathbf{y}}

\newcommand{\bmtheta}{\bm{\theta}}

\newcommand{\bmxi}{\bm{\xi}}

\title{Massive wave propagation near null infinity}
\author{Ethan Sussman}
\date{December 16th, 2024 (Last update: exposition added, many typos and a handful of minor mistakes fixed), May 1st, 2023 (Preprint)}
\email{ethanws@mit.edu}
\address{Department of Mathematics, Massachusetts Institute of Technology, Massachusetts 02139-4307, USA}

\makeatletter
\@namedef{subjclassname@2020}{\textup{2020} Mathematics Subject Classification}
\makeatother
\subjclass[2020]{Primary 35L05; 35B40, Secondary 35P25, 35C20, 58J40, 58J47}

\newlength\octheight
\newlength\octheightr
\newlength\arrowscale
\begin{document}
	
\begin{abstract}
	We study, fully microlocally, the propagation of massive waves on the \emph{octagonal compactification} 
	\[\bbO=[\overline{\bbR^{1,d}};\scrI;1/2]\] 
	of asymptotically Minkowski spacetime, which allows a detailed analysis both at timelike and spacelike infinity (as previously investigated using Parenti--Shubin--Melrose's sc-calculus) and, more novelly, at null infinity, denoted $\scrI$. The analysis is closely related to Hintz--Vasy's recent analysis of massless wave propagation at null infinity using the ``e,b-calculus'' on $\bbO$.
	We prove several elementary corollaries regarding the Klein--Gordon IVP.
	Our main technical tool is a fully symbolic pseudodifferential calculus, $\Psi_{\mathrm{de,sc}}(\bbO)$, the ``de,sc-calculus'' on $\bbO$.
	The `de' refers to the structure (``double edge'') of the calculus at null infinity, and the `sc' refers to the structure (``scattering'') at the other boundary faces.  
	We relate this structure to the hyperbolic coordinates used in other studies of the Klein--Gordon equation. Unlike hyperbolic coordinates, the de,sc- boundary fibration structure is Poincar\'e invariant.
\end{abstract}

\maketitle
\tableofcontents

\section{Introduction}

The subject of this paper is the propagation of massive waves on Minkowski-like spacetimes $(\bbR^{1,d}_{t,\bfx},g)$. 
We will be more precise later regarding the meaning of ``Minkowski-like'' in the previous sentence; see \S\ref{sec:proofs} for the precise conditions, which define the class of what we call \emph{admissible} metrics. For now, we just note that such spacetimes are 
\begin{itemize}
	\item non-trapping, in the sense that null geodesics asymptote in the usual way, 
	\item globally hyperbolic, with $t=x_0$ a smooth time function and with each $\Sigma_T = \{(T,\bfx):\bfx\in \bbR^d\}$, $T\in \bbR$, a Cauchy hypersurface -- so that the Cauchy problem with data specified on $\Sigma_0$ is well-posed -- and 
	\item asymptotically flat, in both the spacelike and timelike directions, so that the metric asymptotes to the Minkowski metric at large distances \emph{and} large times. Our main theorem allows only short-range perturbations of the Minkowski metric (see \Cref{prop:allowed_decay}). Through much of \S\ref{sec:propagation}, the discussion applies to longer-range perturbations, though we will not attempt to be sharp.
\end{itemize}
The main microlocal estimates below depend only on the asymptotic structure of the metric, but global hyperbolicity and the non-trapping assumption are required for the various applications given.

Of course, the exact Minkowski spacetime $(\bbR^{1,d},g_\bbM)$, with 
\begin{equation} 
	g_\bbM = - \mathrm{d}t^2+ \sum_{j=1}^d \mathrm{d}x_j^2,
\end{equation} 
counts as admissible. Our sign convention for Lorentzian metrics is the mostly positive one. 
\begin{remark}[Our vs.\ usual notion of asymptotic flatness]
	We do not allow stationary spacetimes besides the exact Minkowski spacetime itself, as other stationary spacetimes do not asymptote to the Minkowski spacetime at large \textit{times}. The usual notion of asymptotic flatness (see e.g.\ \cite[Chp.\ 11]{Wald}) only restricts large distance behavior. Consequently, our analysis here excludes many physically important spacetimes, indicating that there is much work left to be done.
\end{remark}

Fix $\mathsf{m}>0,d\in \bbN^+$. Given an admissible Lorentzian metric $g$ on $\bbR^{1,d}$, let  
\begin{align}
	\begin{split}  
	\square_g &= -\frac{1}{|g|^{1/2}} \sum_{i,j=0}^d \frac{\partial}{\partial x_i} \Big[ |g|^{1/2} g^{ij} \frac{\partial}{\partial x_j} \Big] \\ 
	&=  -\frac{1}{|g|^{1/2}} \Big( \sum_{i=0}^d \frac{\partial}{\partial x_i} \Big[ |g|^{1/2} g^{i0} \frac{\partial}{\partial t} \Big] + \sum_{j=1}^d \frac{\partial}{\partial t} \Big[ |g|^{1/2} g^{0j} \frac{\partial}{\partial x_j} \Big] \Big) + \triangle_g 
	\end{split} 
	\label{eq:square}
\end{align} 
denote the associated d'Alembertian, with the sign convention being such that the (likely time-dependent) Laplace--Beltrami portion $\triangle_g$ is positive semidefinite.

Consider the Klein--Gordon equation
\begin{equation}
	\square_g u + Qu + \mathsf{m}^2 u  = f,  
	\label{eq:KG} 
\end{equation} 
where $u$ is the unknown, $f$ is the forcing, and $Q$ is drawn from a subspace of appropriate first-order differential operators whose coefficients decay at infinity (more precisely, are short range; the precise condition is in \S\ref{sec:proofs}). For instance, $Q$ can be any Schwartz function, considered as a multiplication operator, so included in this setup is 
\begin{equation} 
	\square+\mathsf{m}^2 + V = \frac{\partial^2}{\partial t^2} - \sum_{i=1}^d \frac{\partial^2}{\partial x_i^2} +\mathsf{m}^2 + V
	\label{eq:exactKG}
\end{equation} 
for $V\in \mathcal{S}(\bbR^{1,d})$, 
which governs the evolution of massive waves on the exact Minkowski spacetime in the presence of the ``potential'' $V$. 
Here, $\square = \square_{g_{\bbM}} = \partial_t^2 -( \partial_{x_1}^2 + \cdots + \partial_{x_d}^2)$ is the exact flat spacetime d'Alembertian.

The behavior of solutions of the associated initial value problem (IVP)
\begin{equation}
	\begin{cases}
		\square_g u +Q u + \mathsf{m}^2 u = f\in \calS(\bbR^{1+d}), \\ 
		u|_{t=0} = u^{(0)} \in \calS(\bbR^d), \\
		\partial_t u|_{t=0} = u^{(1)}\in \calS(\bbR^d),
	\end{cases}
	\label{eq:IVP}
\end{equation}
is a rather classical topic. 

\begin{remark}
	In this introduction, and in \S\ref{sec:proofs} (in which the results stated in this introduction are proven), we restrict attention to the case when the forcing $f$ and initial data $u^{(0)},u^{(1)}$ are Schwartz. Our main estimates, proven in \S\ref{sec:propagation}, \S\ref{sec:radialpoint}, are much more general.
\end{remark}

Nevertheless, it has apparently remained open to establish (beyond the exact Minkowski case) that the solution $u$ admits a full asymptotic expansion at infinity. At infinity means as $t\to\infty$ (allowing $x_1,\dots,x_d\to\infty$ as well). 
Alternatively, at infinity roughly means at the boundary of the radial compactification
\begin{equation} 
	\bbM = \overline{\bbR^{1,d}} = \overline{\bbR^{1+d}} = \bbR^{1+d} \cup \{\infty \bbS^{d}\}
\end{equation} 
of the spacetime.

\begin{figure}
	\begin{tikzpicture}
		\draw[dashed] (-2.55,-2.6) rectangle (4,2.6);
		\draw[fill=lightgray!20] (0,0) circle (2);
		\coordinate (a) at (1.97,0.347);
		\coordinate (b) at (-1.97,-0.347); 
		\coordinate (c) at (1.414,1.414);
		\coordinate (d) at (-1.414,-1.414);
		\coordinate (e) at (0.347,1.97);
		\coordinate (f) at (-0.347,-1.97); 
		\draw[gray] (a) -- (b);
		\draw[gray] (b) to[out=20, in=180] (a);
		\draw[gray] (b) to[out=0, in=200] (a);
		\draw[gray] (c) -- (d);
		\draw[gray] (d) to[out=55, in=215] (c);
		\draw[gray] (d) to[out=35, in=235] (c);
		\draw[gray] (e) -- (f);
		\draw[gray] (f) to[out=90, in=250] (e);
		\draw[gray] (f) to[out=70, in=270] (e);
		\draw[darkcandyapp, ->] (2.1,0) to[out=90, in=-70] (1.95,.8) node[right] {$\operatorname{tan}^{-1}(t/r)$ };
		\node () at (-1,1) {$\bbM$};
	\end{tikzpicture}
	\begin{tikzpicture}
	\draw[dashed] (-2.6,-2.6) rectangle (2.6,2.6);
	\draw[fill=lightgray!20] (0,0) circle (2);
	\draw[orange, dashed] (1.414,1.414) -- (-1.414,-1.414) ;
	\draw[orange, dashed] (-1.414,1.414) -- (1.414,-1.414) ;
	\draw[orange, dashed] (0,.714) ellipse (20pt and 3pt);
	\draw[orange, dashed] (0,-.714) ellipse (20pt and 3pt);
	\draw[orange, dashed] (0,.314) ellipse (8pt and 2pt);
	\draw[orange, dashed] (0,-.314) ellipse (8pt and 2pt);
	\draw[orange, dashed] (0,1.1) ellipse (30.5pt and 4pt);
	\draw[orange, dashed] (0,-1.1) ellipse (30.5pt and 4pt);
	\draw[black, dashed] (0,1.414) ellipse (39pt and 4pt);
	\draw[black, dashed] (0,-1.414) ellipse (39pt and 4pt);
	\node () at (-1,0) {$\bbM$};
	\fill[black] (1.414,1.414) circle (.05);
	\fill[black] (1.414,-1.414) circle (.05);
	\fill[black] (-1.414,1.414) circle (.05);
	\fill[black] (-1.414,-1.414) circle (.05);
	\node[below right] () at (1.414,-1.414) {$\scrI^-$};
	\node[above right] () at (1.414,1.414) {$\scrI^+$};
	\node[right] () at (2,0) {$i^0$};
	\node[above] () at (0,1.95) {$C_+$};
	\node[below] () at (0,-1.95) {$C_-$};
	\end{tikzpicture}
	\caption{The radial compactification $\bbM\hookleftarrow \bbR^{1,d}$ of Minkowski spacetime is topologically a ball, $\bbM\cong \bbB^{1+d}$. On the left, we show three families of straight parallel lines $\{(t,x) : \bfx=\bfx_0+\bfv t\}$ in $\bbR^{1,d}$ (each family consisting of three parallel lines, so three different $x_0\in \bbR^{1,d}$), as seen from the compactified perspective -- one family of timelike lines, one of null lines, and one of spacelike lines. These hit timelike infinity $C=C_+\cup C_-$, null infinity $\scrI=\scrI^+\cup \scrI^-$, and spacelike infinity $i^0$, respectively. These subsets of $\partial \bbM\cong \bbS^d$ are shown in the figure on the right, in which the light cone is also shown in orange.}
	\label{fig:M}
\end{figure}

Such a result appears below. However, its formulation requires working on a more complicated compactification than $\bbM$. For now, we state a more elementary formulation that can be phrased using only $\bbM$:
\begin{theorem}
	Given the setup above:
	\begin{enumerate}[label=(\alph*)]
		\item in the region $\bbM\backslash \operatorname{cl}_\bbM \{|t|\geq r\}$, the solution $u$ of the IVP is Schwartz, 
		\item for any $\bfv\in \bbR^d$ with $\lVert \bfv \rVert = 1$ and $\bfx_0\in \bbR^d$, the solution $u$, restricted to the line $
			\gamma_{\bfv,\bfx_0}=\{(t,\bfx)\in \bbR^{1,d} : \bfx = \bfx_0+ \bfv  t\}$
		is Schwartz as a function of $t$. Moreover, the same is true for all derivatives of $u$. 
		\item 
		Within $\{|t|>r+1\}$, we can write 
		\begin{equation}
			u = |t|^{-d/2} e^{-i\mathsf{m}\sqrt{t^2-r^2}} u_{-} + |t|^{-d/2} e^{+i\mathsf{m}\sqrt{t^2-r^2}} u_+ 
			\label{eq:u_decomp_initial}
		\end{equation}
		for $u_\pm \in C^\infty(\bbM\backslash \scrI)$. Moreover, if we let $\overline{C}_\pm$ denote the
		(closed) past and future caps 
		\begin{equation} 
			\overline{C}_\pm = \mathrm{cl}_\bbM\{\pm t\geq r\} \cap \partial \bbM,
		\end{equation} 
		then, for each $\varsigma \in \{-,+\}$, each term in the asymptotic expansion of $u_\varsigma$ at $C_\pm$ is a Schwartz function on $\overline{C}_\pm$, i.e.\ decays rapidly when approaching the boundary $\overline{C}_\pm \backslash C_\pm$.
	\end{enumerate}
	In summary, $u$ decays rapidly at spacelike infinity and null infinity, and at timelike infinity it has a full asymptotic expansion, the terms in which decay rapidly at null infinity.
	
	We can combine (and slightly strengthen) the three parts of this theorem as follows: let $\bbX_0 = \mathrm{cl}_{\bbM}\{|t|\geq r\}\backslash \{0\}$ denote the set of all points in $\bbM$ timelike or lightlike with respect to the origin (and excluding the origin itself). Then, there exist $u_\pm \in \calS(\bbX_0)$ and $u_0\in \calS(\bbR^{1,d})$ such that the support of $u_\pm$ excludes the origin and  
	\begin{equation}
		u = |t|^{-d/2} e^{-i\mathsf{m}\sqrt{t^2-r^2}} u_{-} + |t|^{-d/2} e^{+i\mathsf{m}\sqrt{t^2-r^2}} u_+ +u_0 
	\end{equation}
	holds globally. (Here, we are considering $u_\pm$ as functions on $\bbR^{1,d}$ supported within $\bbX_0$.)
	\label{thm:easy}
\end{theorem}
Here, $r= \lVert \bfx \rVert$ is the \emph{spatial} Euclidean radial coordinate, and $\operatorname{cl}_\bullet$ is used to denote closure in $\bullet$.

For the exact Klein--Gordon operator, a proof of essentially this result can be found in \cite[\S7.2]{HormanderNL}. The proof there utilizes the global Fourier transform to produce the solution of the IVP in terms of oscillatory integrals  whose asymptotics can be extracted via the method of stationary phase. Hence, it does not generalize to the case when the PDE has variable coefficients. As is by now well-known, the Parenti--Shubin--Melrose sc-calculus \cite{MelroseSC}\cite{VasyGrenoble} straightforwardly allows us to estimate the solution to the IVP in weighted $L^2$-based Sobolev spaces, including the variable coefficient case. The basic estimates are discussed in \cite{VasyGrenoble, VasyKG}, and standard modifications using module regularity \cite{HassellMelroseVasy, HassellETAL, HassellNL} allow one to establish asymptotic expansions (cf.\ \cite{MelroseSC}) away from \emph{null infinity}, 
\begin{equation}
	\scrI = \mathrm{cl}_\bbM\{|t|=r\} \cap \partial \bbM. 
\end{equation}
The upshot is that, if $u$ solves the IVP, then
\begin{itemize}
	\item $u$ is Schwartz away from the timelike caps $\overline{C}_\pm$ 
	\item within $\{|t|>r\}$, \cref{eq:u_decomp_initial} holds.
\end{itemize}
Thus, the sc- tools yield all of \Cref{thm:easy} \emph{except} the parts having to do with decay at null infinity. Our main contribution is to analyze the situation near null infinity.

Before stating our main theorem, it is worth explaining \textit{why} the sc-calculus is not well-suited to proving decay at null infinity. This will serve as motivation for our choice of compactification $\bbO\hookleftarrow \bbR^{1,d}$. 
We will assume throughout this paper that the reader is familiar with the sc-calculus. See \cite{VasyGrenoble} for an exposition of this theory. However, the statement of our main theorem does not require any microlocal analysis whatsoever. Consequently, the reader unfamiliar with the sc- terminology may wish to skip the next subsection, proceeding directly to \S\ref{subsec:main}.

\subsection{Limitations of the sc-calculus}
\label{subsec:limitations}

Recall that the compactified phase space relevant to the sc-calculus is the (radially compactified) sc-cotangent bundle 
\begin{equation} 
	{}^{\mathrm{sc}}\overline{T}^* \bbM = \underbrace{\bbM}_{\text{base}} \times \underbrace{\bbB^{1+d}}_{\text{fiber}} \hookleftarrow T^* \bbR^{1,d}.
\end{equation}
Over the interior of $\bbM$, this is the usual (radially compactified) cotangent bundle $\overline{T}^* \bbR^{1,d}$. On the other hand, ${}^{\mathrm{sc}} T^* \bbM =\bbM\times \bbR^{1+d}_{\tau,\xi}$ is the vector bundle whose smooth sections are of the form 
\begin{equation}
	\tau(t,\bfx)\dd t + \xi(t,\bfx)\cdot \dd \bfx,\quad \tau\in C^\infty(\bbM;\bbR), \xi \in C^\infty(\bbM;\bbR^d).
\end{equation}
The ball bundle ${}^{\mathrm{sc}}\overline{T}^* \bbM$ just results from radially compactifying the factor $\smash{\bbR^{1+d}_{\tau,\xi}}$. Topologically, $	{}^{\mathrm{sc}}\overline{T}^* \bbM$ is just a product of two $(1+d)$-balls. It has two boundary hypersurfaces, fiber infinity and base infinity.

Exactly at null infinity, the notion of module regularity needed to extract asymptotics for the Klein--Gordon equation becomes problematic. The reason for this is that the sources/sinks (a.k.a.\ radial sets) of the sc-Hamiltonian flow hit fiber infinity there (\Cref{fig:sc}). Relatedly, the phases in \cref{eq:u_decomp_initial} become singular at the light cone:
\begin{equation}
	 \dd \sqrt{t^2- r^2} = \frac{t\dd t- r\dd r}{\sqrt{t^2-r^2}},
	 \label{eq:012}
\end{equation}
and it is these sc- 1-forms that parametrize the radial sets over $\overline{C}_\pm$. Dually, consider the first-order differential operator that one inverts near points in $C_\pm = \smash{\overline{C}_\pm^\circ}$ to produce asymptotic expansions (see \S\ref{sec:asymptotics}):
\begin{equation}
	\frac{1}{\sqrt{t^2-r^2}} \Big(t \frac{\partial}{\partial t}  +r \frac{\partial}{\partial r}\Big) \mp i \mathsf{m},
	\label{eq:misc_010}
\end{equation}
which is related to the hyperbolic symmetries of the PDE (see \cref{eq:misc_ptf}). This too
becomes singular at the light cone.  Multiplying by $(t^2-r^2)^{1/2}$ cures this but causes other problems in the extraction of asymptotic expansions (it degenerates in the relevant sense at the light cone). Passing to hyperbolic coordinates facilitates the extraction process but complicates the regularity theory and breaks Poincar\'e invariance.
At first glance, these issues seem like they should be merely technical. 
This does not appear to be the case. Even if it is, the fact remains that the situation at null infinity requires clarification.

\begin{figure}[t]
	\begin{center}
		\begin{tikzpicture}[scale=1.1, decoration={
				markings,
				mark=at position 0.53 with {\arrow[scale=1.5,>=latex]{>}}}]
			\draw[dashed, gray] (-1,1) -- (-2.9,1) node[left] {${}^{\mathrm{sc}}\Sigma_{\mathsf{m},+} \cap  \mathrm{df}$ };
			\node[gray] () at (-4.0,.6) {(left moving, hidden)};
			\draw[fill=gray!10] (0,0) circle (1\octheight);
			\draw[fill=gray!30] (0,0) circle (.6\octheight);
			\coordinate (1) at (-.43\octheight,-.43\octheight);
			\coordinate (2) at (.705\octheight,-.705\octheight); 
			\coordinate (3) at (.43\octheight,.43\octheight);
			\coordinate (4) at (-.705\octheight,.705\octheight);
			\draw[gray, postaction={decorate}] (1) -- (3);
			\draw[gray, postaction={decorate}] (1) -- (3);
			\draw[gray, postaction={decorate}] plot [smooth,tension=1] coordinates { (1) (-.1\octheight,.1\octheight) (3) };
			\draw[gray, postaction={decorate}] plot [smooth,tension=1] coordinates { (1) (-.2\octheight,.2\octheight) (3) };
			\draw[gray, postaction={decorate}] plot [smooth,tension=1] coordinates { (1) (-.35\octheight,.35\octheight) (3) };
			\draw[gray, postaction={decorate}] plot [smooth,tension=1] coordinates { (1) (.1\octheight,-.1\octheight) (3) };
			\draw[gray, postaction={decorate}] plot [smooth,tension=1] coordinates { (1) (.2\octheight,-.2\octheight) (3) };
			\draw[gray, postaction={decorate}] plot [smooth,tension=1] coordinates { (1) (.35\octheight,-.35\octheight) (3) };
			\draw[gray, postaction={decorate}] (-.5\octheight,-.5\octheight) arc (-135:45:.7\octheight);
			\draw[gray, postaction={decorate}] (-.5\octheight,-.5\octheight) arc (225:45:.7\octheight);
			\draw[gray, postaction={decorate}] (-.625\octheight,-.625\octheight) arc (-135:45:.9\octheight);
			\draw[gray, postaction={decorate}] (-.625\octheight,-.625\octheight) arc (225:45:.9\octheight);
			\node[circle,draw=darkcandyapp, fill=darkcandyapp, inner sep=0pt,minimum size=4pt] (1c) at (1) {};
			\node[circle,draw=darkcandyapp, fill=darkcandyapp, inner sep=0pt,minimum size=4pt] (2c) at  (2) {};
			\draw[darkcandyapp] plot [smooth,tension=1] coordinates { (1) (-.3\octheight,-.75\octheight) (2)};
			\node[circle,draw=darkcandyapp, fill=darkcandyapp, inner sep=0pt,minimum size=4pt] (3c) at (.43\octheight,.43\octheight) {};
			\node[circle,draw=darkcandyapp, fill=darkcandyapp, inner sep=0pt,minimum size=4pt] (4c) at (-.705\octheight,.705\octheight) {};
			\draw[darkcandyapp] plot [smooth,tension=1] coordinates { (3) (.3\octheight,.75\octheight) (4)};
			\draw[dashed, gray] (2,0) -- (2.6,0) node[right] {${}^{\mathrm{sc}}\Sigma_{\mathsf{m},+}\cap \mathrm{bf}$ };
			\draw[dashed, gray] (-1,0) -- (-2.6,0) node[left] {${}^{\mathrm{sc}}\Sigma_{\mathsf{m},+}\cap \mathrm{df}$ };
			\node[gray] () at (-3.6,-.4) {(right moving)};
			\draw[dashed, ->] (2.2,.05) to[out=90, in=-65] (1.98,1) node[right] {$\operatorname{arctan}(t/x)$};
		\end{tikzpicture}
	\end{center}
	\caption{The sc-Hamiltonian flow within one sheet ${}^{\mathrm{sc}}\Sigma_{\mathsf{m},+}$ of the sc-characteristic set, depicted in the case $d=1$. The central disk (dark gray) represents one component of ${}^{\mathrm{sc}}\Sigma_{\mathsf{m},+} \cap {}^{\mathrm{sc}} \bbS^* \bbM$ (which is disconnected if $d=1$), i.e.\ one half of the portion of ${}^{\mathrm{sc}}\Sigma_{\mathsf{m},+}$ at fiber infinity, labeled df. The other half is hidden from view; it attaches to the outer circle. (When $d=1$, the characteristic set of the wave operator over a point in $\bbR^{1,d}$ consists of four points in the cosphere bundle. Two of those lie in ${}^{\mathrm{sc}}\Sigma_{\mathsf{m},-}$.) As in \Cref{fig:M}, time is oriented upwards, and the spatial coordinate $x$ is horizontal. The lighter gray annulus depicts the portion of ${}^{\mathrm{sc}}\Sigma_{\mathsf{m},+}$ over base infinity (labeled bf); this is one sheet of a hyperboloid fibered over base infinity. (The other sheet would be in ${}^{\mathrm{sc}}\Sigma_{\mathsf{m},-}$.) The radial sets, at which the properly scaled Hamiltonian vector field vanishes, are colored red. The important point is that the radial sets hit fiber infinity over null infinity (located at $45^\circ,135^\circ,225^\circ,315^\circ$ in the figure). }
	\label{fig:sc}
\end{figure}

That the sc- radial set hits fiber infinity correctly suggests that the solution to the Klein--Gordon initial value problem with generic Schwartz initial data and forcing has sc-wavefront set in the corner 
\begin{equation}
	{}^{\mathrm{sc}}\bbS^*_{\partial \bbM}\bbM = \partial\,{}^{\mathrm{sc}}\overline{T}^*_{\partial \bbM} \bbM \subseteq {}^{\mathrm{sc}}\overline{T}^* \bbM
\end{equation}
of the radially compactified sc-cotangent bundle. 
This is simply a consequence of the fact that wavefront sets (like singular supports) are closed; if the endpoints of the radial set (again, see \Cref{fig:sc}) were not included in the sc-wavefront set, then a neighborhood thereof would \emph{also} be absent. This would imply that $u$ is Schwartz in a neighborhood of $\scrI\subset \bbM$, which is certainly not generic, as can be seen easily in the constant-coefficient case. 

So, a typical solution of the IVP possesses sc-wavefront set at the corner of the compactified sc-cotangent bundle. Even to the user of the sc-calculus, the interpretation of such wavefront set might not be as familiar as the interpretation of sc-wavefront set in the interiors of the fibers, or over the interior of the base. Over the interior of the base, sc-wavefront set is just ordinary wavefront set, so captures failures of smoothness. In the interiors of the fibers over $\partial \bbM$, sc-wavefront set is just the ordinary wavefront set of the Fourier transform --- it detects oscillatory terms of the form 
\begin{equation}
	A(t,\bfx) e^{ i \omega t + i \bfk\cdot \bfx} , \quad \omega\in \bbR,\,\bfk\in \bbR^d,\, A\in C^\infty(\bbM)\backslash \calS.
\end{equation}
Put differently, this sc-wavefront set detects failures of decay at finite frequency. Instead, sc-wavefront set at the corner of the compactified sc-cotangent bundle captures what can roughly be thought of as failures of decay at \emph{infinite} frequency. 
\begin{example}[$e^{ix^2}$]
	A simple example of a function $u\in \calS'(\bbR_x)$ in 1D whose sc-wavefront set is entirely at the corner of the sc-cotangent bundle is $u(x)=\exp(ix^2)$. This is smooth, so free of ordinary wavefront set, but so is its Fourier transform. This implies that the sc-wavefront set is a subset of the corner. But, since $u$ is not Schwartz, it must have \emph{some} sc-wavefront set, which then must be at the corner. Indeed, as $x\to\infty$, $\exp(i x^2)$ oscillates faster than any finite frequency term $e^{i\sigma x},\sigma \in \bbR$, hence lies at ``infinite frequency.''
\end{example}
This example does not shed much light on the oscillations $e^{\pm i\mathsf{m}\sqrt{t^2-r^2}}$ which \cref{eq:u_decomp_initial} says are present in the long-time asymptotics of solutions of the Klein--Gordon IVP. Indeed, if we fix $v=|t|-r$ and then follow $u$ along a level set of $v$, 
\begin{equation}
	e^{\pm i \mathsf{m} \sqrt{t^2-r^2}} =e^{\pm i \mathsf{m} v \sqrt{|t|+r} }
	\label{eq:misc_016} 
\end{equation}
is oscillating, as $|t|\to\infty $, \textit{slower} than any finite frequency term $e^{i \sigma |t|}$. So, it may not be clear why such oscillations should be associated with infinite frequency. The following example may help:
\begin{example}[$e^{ixy}$]
	On $\bbR^{2}_{x,y}$, the function $u=e^{ixy}$ has sc-wavefront set at the corner. Indeed, it is smooth, and a simple argument using microlocalized elliptic estimates for $\partial_x-iy$ and $\partial_y - ix$ shows that there is no wavefront set in the interiors of the fibers. However, $u$ is certainly not Schwartz, so there must be some sc-wavefront set at the corner. 
	
	In fact, since $u$ is not Schwartz in any conic region, there must be some sc-wavefront set at the corner over every point of $\partial \overline{\bbR^2}$.
	In particular, this holds for the point $(\infty,0)\in \partial \overline{\bbR^2}$ where the positive $x$-axis hits the boundary of the compactification.
	This may be confusing at first, since, if we fix $y_0\in \bbR$ and send $x\to\infty$, it might appear that $u(x,y_0)=\exp(ixy_0)$ is oscillating with finite frequency $y_0$. This might lead to the expectation of a whole line's worth of sc-wavefront set over $(\infty,0)$, passing through the fiber, one point on the line for each value of $y_0$. (And then points at fiber infinity would be included because $y_0$ can be arbitrarily large.) But, we already know that the sc-wavefront set of $u$ is entirely at fiber infinity. To resolve the paradox, we must remember that the fibers of the sc-cotangent bundle are two-dimensional. We can talk not just about the radial frequency, but also a tangential frequency. Indeed, if $X\gg 1$, then, as $y$ varies, $u(X,y) = e^{i Xy}$ is oscillating with frequency $X$. Since $X$ can be arbitrarily large, this suggests infinite frequency in the $y$ direction.
\end{example}

So, $e^{\pm i \mathsf{m} \sqrt{t^2-r^2}}$ is infinite frequency at the light cone not in the direction along the light cone but in the direction \emph{across} the light cone. This is essentially what the singularity of \cref{eq:012} at the light cone means.

Regardless of its interpretation, sc-wavefront set is an obstruction to decay.
But, as is well-known at least in the exact Minkowski case \cite{Winicour}\cite{Klainerman}\cite{HormanderNL}, massive waves (unlike massless waves) do not have an associated ``radiation field'': the solution $u$ to the IVP is rapidly decaying at null infinity, even though there exists sc-wavefront set over it. 
Indeed, \Cref{thm:easy} tells us that $u$ is, in an appropriate sense, rapidly decaying at null infinity.
\textit{What, then, is the sc-wavefront set detecting} in this case? 
One answer is that it is detecting that any neighborhood of null infinity in $\bbM$ contains points at timelike infinity, at which solutions to the IVP do not decay rapidly. Thus, the wavefront set at the corner is a technical artifact of the fact that sc-wavefront sets are closed.

This suggests the following key idea:
\emph{in order to study massive wave propagation along null geodesics, we should work with a compactification of $\smash{\mathbb{R}^{1,d}}$ that separates individual null geodesics from timelike infinity.} The compactification $\bbM$ does not suffice.

\subsection{A better compactification}
\label{subsec:main}

One compactification that does the trick is the usual Penrose diagram $\smash{\bbP\hookleftarrow \bbR^{1,d}}$ of Minkowski space. 
But, the Penrose diagram does not offer adequate resolution at timelike infinity, where solutions to Klein--Gordon display their oscillatory asymptotic tails. Rather, as in \cite{BasinVasyWunsch}\cite{HintzVasyScriEB}, we use a third compactification $\bbO\hookleftarrow \bbR^{1,d}$ that refines both the radial and Penrose compactifications in the sense that one has compatible blowdown maps $\bbO\to \bbP,\bbM$. 
The space $\bbO$ can thus be constructed in two equivalent ways: by performing a polar blowup of $\scrI\subseteq \bbM$, in which case we write
\begin{equation} 
	\bbO=[\bbM;\scrI;1/2],
\end{equation}
or by blowing up spacelike and timelike infinity in $\bbP$ in an appropriate way. 
The space $\bbO$ is a manifold-with-corners (mwc) with corners of codimension two.

\begin{remark}[The `$1/2$']
	After blowing up $\scrI\subseteq \bbM$, it is convenient to modify the smooth structure at the front faces of the blowup so that the original boundary-defining-functions (bdfs) $\varrho_{\mathrm{Nf}}$ of the front faces of $\bbO_0=[\bbM;\scrI]$ become the squares 
	\begin{equation} 
		\varrho_{\mathrm{Nf}}=\varrho_{\mathrm{nf}}^2
	\end{equation} 
	of the new bdfs $\varrho_{\mathrm{nf}}$. This is the `$1/2$' in ``$\bbO=[\bbM;\scrI;1/2]$.'' This does not change the essential features of the compactification, so the reader can ignore it for now.
\end{remark}

\begin{figure}[h!]
	\begin{tikzpicture}
		\draw[dashed] (-2,-2) rectangle (2,2);
		\coordinate (i) at (.75,.75); 
		\fill[lightgray!20] (-1.8,-1.8) -- (1.8,-1.8) arc(0:90:3.6) -- cycle;
		\draw[gray] (-1.2,-1.8) -- (i);
		\draw[gray] (-.6,-1.8) -- (i);
		\draw[gray] (-1.8,-1.2) -- (i);
		\draw[gray] (-1.8,-.6) -- (i);
		\draw (1.8,-1.8) arc(0:90:3.6);
		\fill[black] (i) circle (.07);
		\node[above right] () at (i) {$\scrI^+$};
		\node () at (1,-1.2) {$\bbM$};
		\draw[darkcandyapp, ->] (1,.2) to[out=125, in=-35] (.2,1) node[left] {$\frac{t-r}{t+r}$};
	\end{tikzpicture}
	\begin{tikzpicture}
		\draw[dashed] (-2,-2) rectangle (2,2);
		\coordinate (i) at (.75,.75); 
		\fill[lightgray!20] (-1.8,-1.8) -- (1.8,-1.8) arc(0:90:3.6) -- cycle;
		\draw (1.8,-1.8) arc(0:90:3.6);
		\draw[gray] (-1.2,-1.8) -- (1.3,1);
		\draw[gray] (-.6,-1.8) -- (1.6,1.1);
		\draw[gray] (-1.8,-1.2) -- (1,1.3);
		\draw[gray] (-1.8,-.6) -- (1.1,1.6);
		\fill[white] (i) circle (1);
		\begin{scope}
			\clip (-1.81,-1.81) -- (1.81,-1.81) arc(0:90:3.62) -- cycle;
			\filldraw[fill=white] (i) circle (1);
		\end{scope}
		\node () at (1,-1.2) {$\bbO_0$};
		\draw[darkcandyapp, ->] (.6,-.07) to[out=165, in=-75] (-.07,.6) node[right] {$t-r$}; 
		\draw[darkcandyapp, ->] (0,0) -- (-.4,-.4) node[below left] {$\frac{1}{t}$};
	\end{tikzpicture}
	\begin{tikzpicture}
	\draw[dashed] (-2,-2) rectangle (2,2);
	\coordinate (i) at (.75,.75); 
	\fill[lightgray!20] (-1.8,-1.8) -- (1.8,-1.8) arc(0:90:3.6) -- cycle;
	\draw (1.8,-1.8) arc(0:90:3.6);
	\draw[gray] (-1.2,-1.8) -- (1.3,1);
	\draw[gray] (-.6,-1.8) -- (1.6,1.1);
	\draw[gray] (-1.8,-1.2) -- (1,1.3);
	\draw[gray] (-1.8,-.6) -- (1.1,1.6);
	\fill[white] (i) circle (1);
	\begin{scope}
		\clip (-1.81,-1.81) -- (1.81,-1.81) arc(0:90:3.62) -- cycle;
		\filldraw[fill=white] (i) circle (1);
	\end{scope}
	\node () at (1,-1.2) {$\bbO$};
	\draw[darkcandyapp, ->] (.6,-.07) to[out=165, in=-75] (-.07,.6) node[right] {$t-r$}; 
	\draw[darkcandyapp, ->] (0,0) -- (-.4,-.4) node[below left] {$\frac{1}{\sqrt{t}}$};
	\end{tikzpicture}
	\caption{
		The polar blowup procedure to construct $\bbO$, shown near $\scrI^+$. Parallel null lines $\gamma_{\bfv,\bullet}$ which, in $\bbM$, hit $\scrI^+$ instead asymptote to different points at the front face of the blowup, $\mathrm{nFf}$.
	}
	\label{fig:blowup}
\end{figure}
 
The manifold-with-corners $\bbO$ is depicted in \Cref{fig:o}, where we have labeled its faces Pf for past timelike infinity, nPf for past null infinity, Sf for spacelike infinity, nFf for future null infinity, and Ff for future timelike infinity. 
We will refer to $\bbO$  as the \emph{octagonal} compactification of Minkowski spacetime, as in the $d=1$ case it is literally an octagon, and the faces nPf, Sf, and nFf are disconnected, each consisting of two components. In this case, it is a slight abuse of terminology to refer to nPf, Sf, and nFf as faces (rather they are a union of faces), but it is a harmless one.

\begin{figure}[t]
	\begin{tikzpicture}[scale=.75]
		\coordinate (one) at (.413\octheight,\octheight) {};
		\coordinate (two) at (\octheight,.413\octheight) {}; 
		\coordinate (three) at (\octheight,-.413\octheight) {};
		\coordinate (four) at (.413\octheight,-\octheight) {};
		\coordinate (five) at (-.413\octheight,-\octheight) {};
		\coordinate (six) at (-\octheight,-.413\octheight) {};
		\coordinate (seven) at (-\octheight,.413\octheight) {}; 
		\coordinate (eight) at (-.413\octheight,\octheight) {};
		\draw[fill=gray!10] (one) -- (two) -- (three) -- (four) -- (five) -- (six) -- (seven) -- (eight) -- cycle;
		\node (Ff) at (0,1.2\octheight) {$\mathrm{Ff}$};
		\node (nFf) at (-.96\octheight,.81\octheight) {$\mathrm{nFf}$};
		\node (Sf) at (-1.20\octheight,0) {$\mathrm{Sf}$};
		\node (Pf) at (0,-1.2\octheight) {$\mathrm{Pf}$};
		\node (nPf) at (-.96\octheight,-.81\octheight) {$\mathrm{nPf}$};
		\draw[dashed] (0,-.75\octheight) ellipse (39pt and 4pt);
		\begin{scope}
			\clip (-\octheight,-.5\octheight) rectangle (+\octheight,-.413\octheight);
			\draw (0,-.413\octheight) ellipse (59.5pt and 4pt);
		\end{scope}
		\draw[dashed] (0,-.413\octheight) ellipse (59pt and 4pt);
		\draw[dashed] (0,0) ellipse (60pt and 4pt);
		\begin{scope}
			\clip (-\octheight,+.3\octheight) rectangle (+\octheight,+.413\octheight);
			\draw (0,+.413\octheight) ellipse (59.5pt and 4pt);
		\end{scope}
		\draw[dashed] (0,+.413\octheight) ellipse (59pt and 4pt);
		\draw[dashed] (0,+.75\octheight) ellipse (39pt and 4pt);
		\draw[fill=gray!10] (0,\octheight) ellipse (24.5pt and 4pt);
		\draw[fill=gray!20] (0,-\octheight) ellipse (24.5pt and 4pt);
	\end{tikzpicture}
	\qquad
	\begin{tikzpicture}[scale=.75] 
		\coordinate (one) at (.413\octheight,\octheight) {};
		\coordinate (two) at (\octheight,.413\octheight) {}; 
		\coordinate (three) at (\octheight,-.413\octheight) {};
		\coordinate (four) at (.413\octheight,-\octheight) {};
		\coordinate (five) at (-.413\octheight,-\octheight) {};
		\coordinate (six) at (-\octheight,-.413\octheight) {};
		\coordinate (seven) at (-\octheight,.413\octheight) {}; 
		\coordinate (eight) at (-.413\octheight,\octheight) {};
		\coordinate (seven_eight_half) at (-.7\octheight,.7\octheight) {};
		\coordinate (one_two_half) at (.7\octheight,.7\octheight) {};
		\coordinate (seven_eight_third) at (-.6\octheight,.8\octheight) {};
		\coordinate (one_two_third) at (.6\octheight,.8\octheight) {};
		\coordinate (seven_eight_half_ref) at (-.7\octheight,-.7\octheight) {};
		\coordinate (one_two_half_ref) at (.7\octheight,-.7\octheight) {};
		\coordinate (seven_eight_third_ref) at (-.6\octheight,-.8\octheight) {};
		\coordinate (one_two_third_ref) at (.6\octheight,-.8\octheight) {};
		\draw[fill=gray!10] (one) -- (two) -- (three) -- (four) -- (five) -- (six) -- (seven) -- (eight) -- cycle;
		\fill[fill=gray!20] (eight) --  (-.75\octheight,.65\octheight) -- (0,0) -- (.75\octheight,.65\octheight) -- (one) -- cycle;
		\fill[fill=gray!20] (five) --  (-.75\octheight,-.65\octheight) -- (0,0) -- (.75\octheight,-.65\octheight) -- (four) -- cycle;
		\fill[pattern=crosshatch, pattern color=darkcandyapp, opacity=.3] (one) -- (one_two_half) .. controls (-.1\octheight,.5\octheight) and (.1\octheight,.5\octheight) .. (seven_eight_half) -- (eight) -- cycle;
		\fill[fill=red!50, opacity=.7] (one) -- (one_two_third) .. controls (-.1\octheight,.65\octheight) and (.1\octheight,.65\octheight) .. (seven_eight_third) -- (eight) -- cycle;
		\fill[pattern=crosshatch, pattern color=darkcandyapp, opacity=.3] (four) -- (one_two_half_ref) .. controls (-.1\octheight,-.5\octheight) and (.1\octheight,-.5\octheight) .. (seven_eight_half_ref) -- (five) -- cycle;
		\fill[fill=red!50, opacity=.7] (four) -- (one_two_third_ref) .. controls (-.1\octheight,-.65\octheight) and (.1\octheight,-.65\octheight) .. (seven_eight_third_ref) -- (five) -- cycle;
		\node (Ff) at (0,1.2\octheight) {$\mathrm{Ff}$};
		\node (nFf) at (.96\octheight,.81\octheight) {$\mathrm{nFf}$};
		\node[color=darkcandyapp] (??) at (-1.4\octheight,.7\octheight) {$\operatorname{supp} \chi$};
		\node[color=red!50] (???) at (-1.1\octheight,1.1\octheight) {$\chi^{-1}(\{1\})$};
		\node (Sf) at (1.20\octheight,0) {$\mathrm{Sf}$};
		\node (Pf) at (0,-1.2\octheight) {$\mathrm{Pf}$};
		\node (nPf) at (.96\octheight,-.81\octheight) {$\mathrm{nPf}$};
	\end{tikzpicture}
	\caption{A mwc diffeomorphic to $\bbO$ when $d=2$, with labeled faces (\textit{left}). The union $\mathrm{nPf}\cup\mathrm{nFf}$ is the lift of $\scrI$ to $\bbO$. The lift of future infinity is $\mathrm{Ff}$, the lift of past infinity is $\mathrm{Pf}$, and the lift of spacelike infinity is $\mathrm{Sf}$. The support conditions on $\chi$ in \Cref{thm:main} (\textit{right}); $\operatorname{supp} \chi$ is denoted in red crosshatch, and $\chi=1$ identically on the solid red region. The set $\mathrm{cl}_\bbO \{t^2>r^2\}^\circ$ is a slightly darker gray.}
	\label{fig:o}
\end{figure}

Now, a key point is that $|t|-r$ is a smooth coordinate along and near the interiors of $\mathrm{nPf},\mathrm{nFf}$. One way to see why this should be the case is that it implies that, fixing a unit vector $\bfv\in \bbR^d$ and looking at the lines
\begin{equation} 
\gamma_{\bfv,\bfx_0}=\{(t,\bfx)\in \bbR^{1,d} : \bfx = \bfx_0+ \bfv t\},
\end{equation} 
the endpoints of $\gamma_{\bfv,\bfx_0}$ in $\mathrm{nPf}\cup \mathrm{nFf}$ are different for different values of $\bfx_0\cdot \bfv$. This is in contrast to the situation in $\bbM$, where \Cref{fig:M} shows that these all asymptote to the same point in $\scrI$. However, as depicted in that figure, different values of $ \bfx_0 \cdot \bfv$ lead to $\gamma_{\bfv,\bfx_0}$ hitting $\scrI$ at different \emph{angles}. It is precisely the different angles that are resolved by performing a polar blowup. This is depicted in \Cref{fig:blowup}.

For each face $\mathrm{f}$ of $\bbO$, let $\varrho_{\mathrm{f}} \in C^\infty(\bbO;\bbR^{\geq 0})$ denote a bdf of $\mathrm{f}$. The statements below will mostly not depend on the particular choices of bdfs.

We can now state our main theorem, in its cleanest formulation:
\begin{theorem}
	Given the setup above, and given any $\chi \in C^\infty(\bbO)$ supported in $\mathrm{cl}_{\bbO}\{t^2 \geq r^2\}^\circ =  \mathrm{cl}_{\bbO}\{t^2 \geq r^2\} \setminus \mathrm{cl}_{\bbO}\{t^2 = r^2\}$ and identically equal to $1$ in some neighborhood of $\mathrm{Pf}\cup \mathrm{Ff}$, $u$ has the form 
	\begin{equation}
		u = u_0 + \chi \varrho_{\mathrm{Pf}}^{d/2}\varrho_{\mathrm{Ff}}^{d/2} e^{-i \mathsf{m} \sqrt{t^2-r^2}} u_- + \chi \varrho_{\mathrm{Pf}}^{d/2}\varrho_{\mathrm{Ff}}^{d/2} e^{+i \mathsf{m} \sqrt{t^2-r^2}} u_+
	\end{equation}
	for some $u_0 \in \calS(\bbR^{1,d})$ and some  $u_\pm \in \varrho_{\mathrm{nPf}}^\infty\varrho_{\mathrm{Sf}}^\infty\varrho_{\mathrm{nFf}}^\infty C^\infty(\bbO) = \bigcap_{k\in \bbN} \varrho_{\mathrm{nPf}}^k\varrho_{\mathrm{Sf}}^k\varrho_{\mathrm{nFf}}^k C^\infty(\bbO)$.
	\label{thm:main}
\end{theorem}

From this, \Cref{thm:easy} follows immediately. (And, in fact, the two theorems are equivalent.)

The support of $\chi$ is chosen such that $(t^2-r^2)^{1/2}$ is a (one-step) polyhomogeneous function on a neighborhood of $\operatorname{supp} \chi$. 
\Cref{thm:main} therefore shows that $u$ is of exponential-polyhomogeneous type on $\bbO$, which is a precise way of saying that the five boundary hypersurfaces of $\bbO$ give a complete set of asymptotic regimes.

The proof of the theorem is in \S\ref{sec:proofs}, using the results of \S\ref{sec:asymptotics}, \S\ref{sec:propagation}, \S\ref{sec:radialpoint}. 

One globally-defined choice of $\varrho_{\mathrm{nPf}},\varrho_{\mathrm{nFf}}$ is 
\begin{equation}
	\varrho_\pm = \Big( \Big( \frac{t}{\sqrt{1+t^2+r^2}} \mp \frac{1}{\sqrt{2}}  \Big)^2 + \frac{1}{1+t^2+r^2} \Big)^{1/4}, 
\end{equation}
where $\varrho_-=\varrho_{\mathrm{nPf}}$ and $\varrho_+ = \varrho_{\mathrm{nFf}}$. These can then be used to construct globally-defined bdfs of the other three faces. Indeed, if $\varrho=1/(1+t^2+r^2)^{1/2} \in C^\infty(\bbM)$ denotes a bdf of $\partial \bbM$, then $\varrho/(\varrho_-^2\varrho_+^2)$ is a bdf of $\mathrm{Pf}\cup\mathrm{Sf}\cup \mathrm{Ff}$, and for each $\mathrm{f} \in \{\mathrm{Pf},\mathrm{Sf},\mathrm{Ff}\}$, we can modify this function near the other two faces to yield $\varrho_{\mathrm{f}}$. 
Alternatively, the bdfs $\varrho_{\mathrm{Pf}},\varrho_{\mathrm{nPf}},\varrho_{\mathrm{nFf}},\varrho_{\mathrm{Ff}}$ can be chosen such that, near $\mathrm{Ff}$ and away from $\mathrm{cl}_\bbO\{r=0\}$, 
\begin{equation}
	\varrho_{\mathrm{nFf}}  = \sqrt{\frac{t-r}{t+r}}, \qquad
	\varrho_{\mathrm{Ff}} = \frac{1}{t-r}, 
	\label{eq:misc_019}
\end{equation}
and similarly near $\mathrm{Pf}$, with $t$ replaced by $-t$. The same applies near $\mathrm{Sf}$, except $r$ and $t$ should be switched:
\begin{equation}
	\varrho_{\mathrm{nFf}}  = \sqrt{\frac{r-t}{t+r}}, \qquad
	\varrho_{\mathrm{Sf}} = \frac{1}{r-t} 
\end{equation}
near $\mathrm{nFf}\cap \mathrm{Sf}$. This is summarized in \Cref{fig:blowup_extended}.

For the reader who has not seen the octagonal blowup before, we include a proof of the claims above. It is a straightforward and elementary computation regarding polar coordinates:
\begin{proof}[Proof.]
	Working near one of the corners of $\bbO$, we can form $\bbO$ by working in $\mathrm{cl}_\bbM \{|t|\geq r\}$ or $\mathrm{cl}_\bbM \{ |t|\leq r\}$ and then blowing up $\scrI$. (The results then need to be stitched together at the light cone.) We illustrate this in  $\mathrm{cl}_\bbM \{t\geq r\}$, the other cases being similar. Near $\scrI$, 
	\begin{equation}
		\mathrm{cl}_\bbM \{t\geq r\} \cong [0,1)_{1/(t+r)}\times [0,1)_{(t-r)/(t+r)}\times  \bbS^{d-1}_\theta, 
	\end{equation}
	where $\theta=\bfx/r$. The function $1/(t+r)$ is a local bdf of $\overline{C}_+$ and $(t-r)/(t+r)$ is a local bdf of the light cone. It follows that, performing a polar blowup of $\scrI$, and staying away from the light cone, $(t-r)/(t+r)$ serves as a local bdf of the front face $\mathrm{nFf}$ of the blowup and the ratio $(1/(t+r))/((t-r)/(t+r))=1/(t-r)$ serves as a local bdf of $\mathrm{Ff}$. This is the exact same computation as one regarding polar coordinates on $\bbR^2$: when performing a polar blowup 
	\begin{equation}
		[0,\infty)^2_{x,y} \rightsquigarrow [0,\infty)_{r=\sqrt{x^2+y^2}}\times [0,\pi/2]_{\varphi=\operatorname{tan}^{-1}(y/x)}
		\label{eq:polar_example}
	\end{equation}
	of the origin, on the blown up space $y =r\sin \varphi$ serves as a local bdf of the front face $\{r=0\}$ away from $[0,\infty)_r \times \{0\}_\varphi$ and $x/y=\cot\varphi$ serves as a local bdf of $[0,\infty)_r\times \{\pi/2\}_\varphi$ in the same set. Just replace $y$ with $(t-r)/(t+r)$ and $x$ with $1/(t+r)$.  
	
	Finally, recalling that $\bbO = [\bbM;\scrI;1/2]$ and not $[\bbM;\scrI]$, to get $\varrho_{\mathrm{nFf}}$ we take the square root of $\varrho_{\mathrm{NFf}} = (t-r)/(t+r)$ to get \cref{eq:misc_019}. 
	
	As for the claim above that $t-r$ is smooth near any point in the interior of $\mathrm{nFf}$, this is proven using a slightly different computation: near $\scrI^+$, 
	\begin{equation}
		\bbM \cong  [0,1)_{1/(t+r)}\times (-1,1)_{(t-r)/(t+r)}\times  \bbS^{d-1}_\theta.
	\end{equation}
	So, the claim is equivalent to the observation that, in the polar blowup 
	\begin{equation}
		[0,\infty)_x\times \bbR_y \rightsquigarrow [0,\infty)_{r=\sqrt{x^2+y^2}}\times [-\pi/2,\pi/2]_{\varphi=\operatorname{tan}^{-1}(y/x)},
	\end{equation}
	the ratio $y/x=\tan\varphi$ is smooth away from $\{\varphi=\pm \pi/2\}$. Just replace $y$ with $(t-r)/(t+r)$ and $x$ with $1/(t+r)$ as before.
\end{proof}
So, to say that some function is smooth at the corner $\mathrm{Ff}\cap \mathrm{nFf}$ means that it admits a joint Taylor series in the coordinates in \cref{eq:misc_019}. Similar statements apply regarding the other corners of $\bbO$.  
Near any point in the interior of $\mathrm{nPf}$ or $\mathrm{nFf}$, $1/(r+|t|)^{1/2}$ can be taken as a local bdf. Thus, smoothness at $\mathrm{nPf}^\circ\cup \mathrm{nFf}^\circ$ is closely related to the existence of asymptotic expansions with respect to light cone coordinates.

We discuss $\bbO$ further in \S\ref{sec:geometry}.

\subsection{Irregularity obstructs decay at null infinity}

If the initial data is not rapidly decaying, then the solution to the IVP is not necessarily rapidly decaying at $\mathrm{Sf}$, nor at $\mathrm{nPf}\cup \mathrm{nFf}$. Conversely, if the initial data $(u^{(0)}, u^{(1)})$ and forcing $f$ satisfy 
\begin{align}
	\begin{split}
	f &\in H_{\mathrm{sc}}^{m-1,s+1}(\bbR^{1+d}) = (1+r^2+t^2)^{-(s+1)/2} H^{m-1}(\bbR^{1+d}) \\
	(u^{(0)}, u^{(1)}) &\in H_{\mathrm{sc}}^{m,s+1}(\bbR^d) \times H_{\mathrm{sc}}^{m-1,s+1}(\bbR^d)
	\end{split}
\end{align}
for $m\in \bbN$ and $s\in \bbR$, where $H_{\mathrm{sc}}^{m,s}(\bbR^d) = \langle r \rangle^{-s} H^m(\bbR^d)$,  
then one expects $u\in H_{\mathrm{sc}}^{m,s}(\mathbb{R}^{1+d})$ near the interior of $\mathrm{Sf}$.
So the amount of decay of $u$ in the spacelike region is controlled by the amount of decay of the initial data and the forcing. 

At null infinity, a lack of \emph{regularity} (i.e.\ smoothness) also obstructs decay.
For example, using the vector-field method, Klainerman \cite[Theorems 2 \& 3]{Klainerman} shows that, in the exact Minkowski case with zero forcing (and under an assumption about supports), there exists some $c=c_{\mathsf{m},d}>0$ such that 
\begin{equation}
	|u(t,\bfx)| \leq 
	c t^{-d/2}
	\begin{cases}
		 (t+r)^{-k/2} I^0_{\bbP,k+\lceil d/2 \rceil}(u,\Sigma_0)  &( t> 0, r\geq t), \\ 
		 (t-r+1)^{k/2} (t+r)^{-k/2} \log (t-r+1) I^0_{\bbP,k+\lceil d/2 \rceil}(u,H_1) & (t> 0, t>r), 
	\end{cases}
\end{equation}
where 
\begin{equation} 
	I^0_{\bbP,k+\lceil d/2 \rceil}(u,\Sigma_0) = O ( \lVert u^{(0)} \rVert_{ H^{k+\lceil d/2 \rceil,k+\lceil d/2 \rceil}_{\mathrm{sc}}(\bbR^d) } + \lVert u^{(1)} \rVert_{H^{k+\lceil d/2 \rceil-1,k+\lceil d/2 \rceil-1}_{\mathrm{sc}}(\bbR^d) } ) 
\end{equation} 
is a quantity depending on the $L^2(\bbR^d)$ norms of $u$ and its derivatives up to order $k+\lceil d/2 \rceil$ on the Cauchy hypersurface $\Sigma_0 = \{(t,\bfx):t=0\}$, and similarly for $I^0_{\bbP,k+\lceil d/2 \rceil}(u,H_1)$ on $H_1 = \{(t,\bfx) : t^2 - r^2 = 1\}$. So, 
\begin{equation}
	|u(t,\bfx)| =  O\bigg(
	\begin{cases}
		\varrho_{\mathrm{Sf}}^{(d+k)/2} \varrho_{\mathrm{nFf}}^{d+k}  & (r\geq t \geq r/2)  \\
		\varrho_{\mathrm{nFf}}^{d+k-} \varrho_{\mathrm{Ff}}^{d/2-}   & (t\geq 1,t>r)
	\end{cases}
	\bigg) . 
\end{equation}
If our initial data only has a finite amount of Sobolev regularity, we can only conclude decay at null infinity to some corresponding finite order, with one extra order of decay for every extra order of regularity.

As an instructive example of what can happen when our forcing is not smooth: 
\begin{example} \label{ex:Green}
Consider the advanced and retarded propagators $D_+,D_-\in \smash{\calS'(\bbR^{1,d})}$ for $\square+\mathsf{m}^2$. Let us recall how these arise from the solution of the forward and reverse problems. These read
\begin{equation}
	\begin{cases}
		\square u (t,\bfx) +\mathsf{m}^2 u(t,\bfx) = \delta(t) f(\bfx) \\
		u(t,\bfx) = 0 \text{ for $\pm t<0$},
	\end{cases}
	\label{eq:forward/reverse}
\end{equation}
for $f\in \calS(\bbR^d_{\bfx})$, where the positive choice of sign gives the forward problem and the negative choice gives the reverse problem. The distributions $D_\pm$ are the Green's functions for these problems in the sense that the unique $u\in \calD'(\bbR^{1,d})$ satisfying \cref{eq:forward/reverse} is $f*D_\pm$, where the convolution is in the spatial variables only. 
Then, $D_\pm$ solves
\begin{equation} 
	(\square +\mathsf{m}^2) D_\pm(t,\bfx)= \delta,
\end{equation} 
where $\delta=\delta(t)\delta^d(\bfx) \in \calS'(\bbR^{1,d})$ is a Dirac $\delta$-function located at the spacetime origin. So, $D_\pm$ solves the Klein--Gordon equation with a non-smooth forcing.

A straightforward but somewhat nontrivial calculation (that can be found in e.g.\ \cite[\S2.3]{ScharfFiniteQED}) reveals that $D_\pm$ is given by 
\begin{equation}
	D_\pm(t,\bfx) =  \pm \frac{1}{2\pi}\Theta(\pm t) \delta(t^2-r^2) \mp \frac{\mathsf{m}}{4\pi} \Theta(\pm t)\Theta(t^2-r^2) \frac{1}{\sqrt{t^2-r^2} } J_1\big(\mathsf{m}\sqrt{t^2-r^2}\big)
	\label{eq:D3d}
\end{equation}
if $d=3$, where $J_1$ denotes the Bessel function of the first kind of order one and $\Theta$ denotes a Heaviside step function. A similar formula holds for other $d\in \bbN^+$. 
The Heaviside step function $\Theta(t)=1_{t\geq 0}$ in \cref{eq:D3d} guarantees 
\begin{equation}
	\operatorname{supp }D_\pm(t,\bfx) \subseteq \{(t,\bfx)\in \bbR^{1,d} : \pm t \geq 0\} .
\end{equation}
We highlight the following features of $D_\pm$ which can be read off \cref{eq:D3d}:
\begin{itemize}
	\item Outside of any neighborhood $U\subseteq \bbM$ of $\mathrm{cl}_\bbM\{|t|\leq r\}$ (and in particular away from null infinity), it follows from the large argument asymptotics of the Bessel function \cite[\S9.2]{SpecialFunctions} that 
	\begin{equation}
		D_\pm =  |t|^{-3/2} e^{-i\mathsf{m}\sqrt{t^2-r^2}} d_{\pm,-} + |t|^{-3/2}e^{+i\mathsf{m}\sqrt{t^2-r^2}} d_{\pm,+}
	\end{equation}
	for some $d_{\pm,-},d_{\pm,+} \in C^\infty(\bbM)$, just as in \cref{eq:u_decomp_initial}. In particular, the nonsmoothness of the forcing does not obstruct decay at timelike infinity, as can already be proven using the sc-calculus.
	\item Though the convolution of $D_\pm$ with any Schwartz function is rapidly decaying at null infinity, $D_\pm$ are themselves not rapidly decaying there: 
	\begin{equation}
		 D_\pm = \mp \sqrt{ \frac{\mathsf{m} \varrho^3}{ 8\pi^3}} \frac{1}{v^{3/4}} \cos \Big( \frac{\mathsf{m} v^{1/2}}{\varrho} - \frac{3\pi}{4} \Big) + o_v(\varrho^{3/2}), 
		 \label{eq:dpmosc}
	\end{equation} 
	 $v = |t|-r$ and $\varrho = (|t|+r)^{-1/2} $, where the decay rate of the $o_v(\varrho^{3/2})$ term is uniformly bounded in $v\gg 0$, with a complete asymptotic expansion in $\varrho$. We therefore have precisely $O(\varrho^{3/2})$ decay at null infinity. 
	 So, it is indeed the case that the irregularity of the forcing leads to a lack of decay at null infinity.

	It is  not at all apparent from \cref{eq:D3d} why solutions to the forward and reverse problems with $f\in \calS(\bbR^d$) should be rapidly decaying at null infinity, though this does hold. 
	
	\item  
	The oscillations in \cref{eq:dpmosc} take the form 
	\begin{equation} 
		\sim \exp(\pm i \mathsf{m}v^{1/2} \varrho^{-1}),
		\label{eq:misc_n3n}
	\end{equation} 
	which implies the presence of a certain sort of wavefront set (at finite frequencies) on the Penrose diagram $\bbP$. 
	Notice that if we use $1/r \sim \varrho^2$ as a boundary-defining-function here, this being what is usually done, then \cref{eq:misc_n3n} is an oscillation at \emph{zero} sc-frequency. However, the form of the oscillations suggests instead using what we called $\varrho \sim 1/r^{1/2}$ as a bdf. Then, \cref{eq:misc_n3n} is at finite sc-frequency \emph{with respect to this choice of smooth structure}. (This is why we work with $\bbO$ instead of $\bbO_0$.)
	
	Note that, as $v\to 0^+$, it is \emph{not} the case that the sc-frequency 
	\begin{equation} 
		\mathrm{d} \Big( \frac{v^{1/2}}{\varrho} \Big)  = \frac{1}{2} \frac{\mathrm{d} v}{ \varrho v^{1/2}} - \frac{v^{1/2}}{\varrho^2} \mathrm{d} \varrho = \zeta \frac{\mathrm{d} v}{\varrho} + \xi \frac{\mathrm{d} \varrho}{\varrho^2}
	\end{equation} 
	in \cref{eq:misc_n3n} converges to the zero section of the sc-cotangent bundle. The opposite is true --- the sc-frequency approaches fiber infinity. As $v\to 0^+$, the component of the frequency dual to $\varrho$, this being $\xi = - v^{1/2}$, converges to zero in the relevant sense, but the component $\zeta = v^{-1/2}/2$ dual to $v$ blows up, so the overall effect is that the frequency gets large. Again, we see that the oscillations present in solutions of the Klein--Gordon equation oscillate slowly along the light cone and rapidly across the light cone.
\end{itemize}

Since $D_\pm$ is identically zero and therefore free of any sort of wavefront set over $\mathrm{cl}_\bbO\{|t|<r-\varepsilon\}$ for each $\varepsilon>0$, this example suggests a form of propagation over null infinity, in which singularities in the interior of the spacetime travel along null geodesics, hit the corner of the appropriate radially compactified cotangent bundle (see below) over null infinity, and then propagate down into the fibers while propagating forwards along null infinity. This is shown in \Cref{fig:Greens}.
\end{example} 

\subsection{Summary of methods}

Consider now the form of $\square$ on $\bbO$: 
\begin{itemize}
	\item Since the interiors of the timelike and spacelike caps of $\bbM$ are canonically diffeomorphic with the interiors of $\mathrm{Pf},\mathrm{Sf},\mathrm{Ff}$, the operator $\square+\mathsf{m}^2$ is a sc-differential operator there.
	\item At the interior of null infinity on the Penrose diagram, $\square$ has the form $\varrho^2 \square_0$ for an (unweighted) edge operator $\square_0$ \cite{HintzVasyScriEB}. The same can not be said for $\square+\mathsf{m}^2$, as though we can write 
	\begin{equation}
		\square + \mathsf{m}^2 = \varrho^2 (\square_0 + \varrho^{-2} \mathsf{m}^2), 
	\end{equation}
	$\varrho^{-2} \mathsf{m}^2$ is too large as $\varrho \to 0$ to be an unweighted edge operator. 
	Nevertheless, $\square+\mathsf{m}^2$ can be regarded as an unweighted ``double edge'' (abbreviated ``de'' for short) operator. 
\end{itemize}
The double edge operators were introduced by Lauter \& Moroianu in \cite{LauterMoroianu}, and we refer to this work for a discussion of the double edge calculus in a setting without corners. 
A key feature is that the de-calculus is, like the sc-calculus, under symbolic control. This means that de-$\Psi$DOs are controlled via a suitable notion of principal symbol modulo compact errors. Standard symbolic constructions from the theory of Kohn--Nirenberg pseudodifferential operators on compact manifolds go through with straightforward modifications.

In the $d=1$ case, the de- structure at null infinity is just the sc- structure at null infinity:
\begin{equation}
	\operatorname{Diff}_{\mathrm{de}}(\bbO\backslash (\mathrm{Pf}\cup \mathrm{Sf}\cup \mathrm{Ff})) = \operatorname{Diff}_{\mathrm{sc}}(\bbO\backslash (\mathrm{Pf}\cup \mathrm{Sf}\cup \mathrm{Ff})) \text{ if }d=1. 
\end{equation} 
(Note that this is not the same thing as the sc- structure on $\bbM$.)
However, 
\begin{equation}
	\operatorname{Diff}_{\mathrm{de}}(\bbO\backslash (\mathrm{Pf}\cup \mathrm{Sf}\cup \mathrm{Ff})) \subsetneq  \operatorname{Diff}_{\mathrm{sc}}(\bbO\backslash (\mathrm{Pf}\cup \mathrm{Sf}\cup \mathrm{Ff})) \text{ if }d\geq 2.
\end{equation}
The reason is that, in the $d\geq 2$ case, the angular derivatives are required to vanish to an extra order versus the sc-differential operators:
\begin{equation}
	\varrho_{\mathrm{nf}}\partial_\theta \in \operatorname{Diff}_{\mathrm{sc}}(\bbO\backslash (\mathrm{Pf}\cup \mathrm{Sf}\cup \mathrm{Ff})), \quad 	\varrho_{\mathrm{nf}}^2\partial_\theta \in \operatorname{Diff}_{\mathrm{de}}(\bbO\backslash (\mathrm{Pf}\cup \mathrm{Sf}\cup \mathrm{Ff})),
\end{equation} 
where $\varrho_{\mathrm{nf}}=\varrho_{\mathrm{nPf}}\varrho_{\mathrm{nFf}}$, 
but
\begin{equation}
	\varrho_{\mathrm{nf}}\partial_\theta \not\in  \operatorname{Diff}_{\mathrm{de}}(\bbO\backslash (\mathrm{Pf}\cup \mathrm{Sf}\cup \mathrm{Ff})).
\end{equation} 
However, it is actually $\varrho_{\mathrm{nf}}^2 \partial_\theta$ that one finds in elements of $\operatorname{Diff}_{\mathrm{sc}}(\bbM)$, like $\square+\mathsf{m}^2$. This is the ultimate reason why one must work with the de-calculus instead of the sc-calculus at null infinity; $\square+\mathsf{m}^2$ is in the sc-calculus there, but the vanishing of the angular derivatives make it degenerate from that perspective. In the example above, the oscillations present in the examined functions $D_\pm$ were present only at zero angular momentum; this effectively reduced us to the $d=1$ case. This is why it sufficed to talk about ``sc-frequencies'' and not ``de-frequencies'' while discussing the oscillations at null infinity. Going forwards, only the latter will be referenced vis-a-vis the situation in the interior of null infinity.

The structure of $\square+\mathsf{m}^2$ suggests that, in order to analyze the Klein--Gordon equation everywhere on $\bbO$, including the corners, we define a pseudodifferential calculus 
\begin{equation}
	\Psi_{\mathrm{de,sc}} = \Psi_{\mathrm{de,sc}}(\bbO) =  \bigcup_{m\in \bbR}\bigcup_{\mathsf{s} \in \bbR^5} \Psi_{\mathrm{de,sc}}^{m,\mathsf{s}}
	\label{eq:misc_t31}
\end{equation}
consisting of pseudodifferential operators ($\Psi$DOs) that are sc-$\Psi$DOs at $\mathrm{Pf},\mathrm{Sf},\mathrm{Ff}$ and de-$\Psi$DOs at $\mathrm{nPf},\mathrm{nFf}$, being in an appropriate sense both simultaneously at the corners of $\bbO$. This calculus will arise by ``quantizing'' a $C^\infty(\bbO)$-module $\calV_{\mathrm{de,sc}}$ of \emph{de,sc-vector fields} on $\bbO$. Roughly, these are smooth vector fields on $\bbR^{1,d}$ which are sc-vector fields at $\mathrm{Pf},\mathrm{Sf},\mathrm{Ff}$ and de-vector fields at $\mathrm{nPf},\mathrm{nFf}$. A precise version of this definition appears later (\cref{eq:Vdef}, \cref{eq:Vdef_local}). It turns out that an equivalent, but less transparent, global definition is 
\begin{equation}
	\calV_{\mathrm{de,sc}} = \operatorname{span}_{C^\infty(\bbO)}\{ \rho_{\mathrm{nPf}}\rho_{\mathrm{nFf}} \partial_t,\rho_{\mathrm{nPf}}\rho_{\mathrm{nFf}} \partial_{x_j}, \rho_{\mathrm{nPf}}^{-1}\rho_{\mathrm{nFf}}^{-1} \chi_0 (\partial_{|t|} + \partial_r), \chi_0 r^{-1} \partial_{\theta_k} :j\leq d,k\leq d-1  \},
	\label{eq:Vdef_global}
\end{equation}
where $\chi_0\in C^\infty(\bbM)$ is some fixed function supported away from $\operatorname{cl}_\bbM\{tr=0\}$ and identically equal to $1$ near null infinity (so that $\chi_0(\partial_{|t|}+\partial_r)$, $\chi_0 r^{-1}\partial_{\theta_k}$ are smooth vector fields, defined using the coordinate system $t,r,\theta_1,\dots,\theta_{d-1}$).

Of course, we have a corresponding algebra $\operatorname{Diff}_{\mathrm{de,sc}}(\bbO)$ of de,sc-differential operators with smooth coefficients.  

In \cref{eq:misc_t31}, 
\begin{equation}
	\Psi^{m,(s_{\mathrm{Pf}},s_{\mathrm{nPf}},s_{\mathrm{Sf}},s_{\mathrm{nFf}},s_{\mathrm{Ff}} )}_{\mathrm{de,sc}} =  \varrho_{\mathrm{Pf}}^{-s_{\mathrm{Pf}}}\varrho_{\mathrm{nPf}}^{-s_{\mathrm{nPf}}}\varrho_{\mathrm{Sf}}^{-s_{\mathrm{Sf}}}\varrho_{\mathrm{nFf}}^{-s_{\mathrm{nFf}}}\varrho_{\mathrm{Ff}}^{-s_{\mathrm{Ff}}} \Psi^{m,\mathsf{0}}_{\mathrm{de,sc}}, 
\end{equation}
so $m$ is the ``differential order'' and $\mathsf{s}\in \bbR^5$ measures decay at the five different faces of $\bbO$. Like the constituent de- and sc- calculi, the de,sc-calculus is under symbolic control. The relevant symbols are precisely conormal functions on a compactification 
\begin{equation}
	{}^{\mathrm{de,sc}} \overline{T}^* \bbO \hookleftarrow T^* \bbR^{1,d} 
\end{equation}
of the cotangent bundle of Minkowski space. This is the entire space of a $\bbB^{1+d}$-bundle ${}^{\mathrm{de,sc}} \pi : {}^{\mathrm{de,sc}} \overline{T}^* \bbO\to \bbO$ over $\bbO$. It is canonically diffeomorphic to ${}^{\mathrm{sc}}\overline{T}^* \bbM$ away from null infinity and ${}^{\mathrm{de}}\overline{T}^* \bbO$ away from timelike and spacelike infinity. See \S\ref{sec:geometry}. 

The connection between the geometric setup here and the hyperbolic coordinates employed in \cite{Klainerman} is discussed in \S\ref{sec:asymptotics}. As far as asymptotic expansions are concerned, the two are not equivalent in general, but for the application to \Cref{thm:main} we need only consider functions decaying rapidly at null infinity, for which the distinction is not important. One selling point of $\bbO$ is that, like the Klein--Gordon equation itself, it is Poincar\'e invariant in the sense that the elements of the Poincar\'e group lift to diffeomorphisms of $\bbO$ (\Cref{prop:invariance}). In contrast, hyperbolic coordinate systems do not interact well with translations (\Cref{rem:parabolic_non-invariant}). The Poincar\'e invariance of the approach here is therefore a feature, though we still use hyperbolic coordinates to extract the asymptotic expansions at $\mathrm{Pf}\cup \mathrm{Ff}$. 

The d'Alembertian $\square=\square_{g_{\bbM}}$ lies in $\operatorname{Diff}_{\mathrm{de,sc}}^{2,\mathsf{0}} \subseteq \Psi_{\mathrm{de,sc}}^{2,\mathsf{0}}$. 
This is a consequence of the fact that the Minkowski metric is a de,sc-metric. Later, we check this claim directly (\Cref{prop:dAlembertian_calculation}). A complication is that 
\begin{equation}
	\partial_t, \partial_{x_i} \in \operatorname{Diff}_{\mathrm{de,sc}}^{1,(0,1,0,1,0)} \setminus \operatorname{Diff}_{\mathrm{de,sc}}^{1,\mathsf{0}},
	\label{eq:misc_035}
\end{equation}
and not $\partial_t, \partial_{x_i} \in \operatorname{Diff}_{\mathrm{de,sc}}^{1,\mathsf{0}}$.
The particular linear combination of derivatives appearing in $\square$ has cancellations at null infinity, and so one gets 
\begin{equation} 
	\square \in \smash{\Psi_{\mathrm{de,sc}}^{2,\mathsf{0}}}
\end{equation} 
and not merely $\square \in \Psi_{\mathrm{de,sc}}^{2,(0,2,0,2,0)}$.  We will perform these computations in \S\ref{sec:geometry}.
The function $p\in C^\infty(T^* \bbR^{1,d})$ defined  by 
\begin{equation}
		p :  \tau \dd t + \sum_{i=1}^d \xi_i \dd x_i \mapsto -\tau^2 + \sum_{i=1}^d \xi_i^2 + \mathsf{m}^2
\end{equation}
defines an element of $\smash{\sigma_{\mathrm{de,sc}}^{2,\mathsf{0}}(\square + \mathsf{m}^2)}$, where $\smash{\sigma_{\mathrm{de,sc}}^{2,\mathsf{0}}}$ denotes a to-be-defined ``de,sc- (joint) principal symbol map.'' This will also be checked later (\Cref{prop:symbol}). 
Of course, $p$ is the full symbol of $\square + \mathsf{m}^2$ in the uniform Kohn--Nirenberg calculus, and thus $p\in \sigma_{\mathrm{sc}}^{2,0}(\square + \mathsf{m}^2)$, but neither of these obviously imply that $p$ is sufficient to represent the principal de,sc-symbol. A priori, it is not even obvious that $p$ is a symbol on the de,sc- phase space. These statements must be checked. 
	
A consequence of $p\in\smash{ \sigma_{\mathrm{de,sc}}^{2,\mathsf{0}}(\square + \mathsf{m}^2)}$ is that commutators of $\square + \mathsf{m}^2$ with de,sc-$\Psi$DOs have de,sc- principal symbols given by Poisson brackets of their symbols with $p$. 
We can therefore prove propagation estimates in the usual way, via the construction of a positive commutator, for which one constructs symbols that are monotone along the (appropriately scaled) de,sc-Hamiltonian flow 
\begin{equation} 
		\mathsf{H}_p = \frac{\varrho_{\mathrm{df}} H_p}{\varrho_{\mathrm{Pf}}\varrho_{\mathrm{nPf}}\varrho_{\mathrm{Sf}}\varrho_{\mathrm{nFf}}\varrho_{\mathrm{Ff}} } =
		\frac{2\varrho_{\mathrm{df}}}{\varrho_{\mathrm{Pf}}\varrho_{\mathrm{nPf}}\varrho_{\mathrm{Sf}}\varrho_{\mathrm{nFf}}\varrho_{\mathrm{Ff}} } \Big[  \tau \frac{\partial}{\partial t} - \sum_{i=1}^d \xi_i \frac{\partial}{\partial x_i}\Big]  \in   \calV_{\mathrm{b}}({}^{\mathrm{de,sc}}\overline{T}^* \bbO), 
		\label{eq:Hp_rescaled}
\end{equation} 
on the de,sc-characteristic set 
\begin{equation} 
	\Sigma_{\mathsf{m}} = \operatorname{Char}_{\mathrm{de,sc}}^{2,\mathsf{0}}(\square+\mathsf{m}^2) = \tilde{p}^{-1}(\{0\}) \cap (\partial\,{}^{\mathrm{de,sc}}\overline{T}^* \bbO). 
\end{equation} 
Here, $\xi$ is the frequency coordinate dual to $x$, and $\tau$ is the frequency coordinate dual to $t$; additionally, $\tilde{p} \in C^\infty({}^{\mathrm{de,sc}}\overline{T}^* \bbO;\bbR )$ is the function $\tilde{p}= \varrho_{\mathrm{df}}^2 p$, where $\varrho_{\mathrm{df}}$ denotes a defining function of fiber infinity  
\begin{equation} 
	\mathrm{df}={}^{\mathrm{de,sc}} \bbS^* \bbO\subset {}^{\mathrm{de,sc}}\overline{T}^* \bbO.
\end{equation}  
We will study the structure of the Hamiltonian flow 
\begin{equation} 
	\Phi_\bullet=\exp(\mathsf{H}_p\bullet) :{}^{\mathrm{de,sc}}\overline{T}^* \bbO \to {}^{\mathrm{de,sc}}\overline{T}^* \bbO 
\end{equation} 
in \S\ref{sec:dynamics}.
In the $d=1$ case, the flow, restricted to one component $\Sigma_{\mathsf{m},+}$ of $\Sigma_{\mathsf{m}}$, is depicted in \Cref{fig:O}. 
More specifically, $\Sigma_{\mathsf{m},\pm} \subseteq \smash{{}^{\mathrm{de,sc}} \overline{T}^* \bbO}$ is the sheet of $\Sigma_{\mathsf{m}}$ on which the temporal frequency $\tau$ satisfies $\pm \tau>0$. 

As seen in the figure, $\mathsf{H}_p$ vanishes at several points on $\Sigma_{\mathsf{m},\pm}$.\footnote{When we speak of b-vector fields vanishing, we always mean vanishing only as smooth vector fields, i.e.\ the usual sense. For example, $x\partial_x \in \calV_{\mathrm{b}}[0,\infty)_x$ is vanishing at $x=0$, even though it is a nonvanishing section of the b-tangent bundle ${}^{\mathrm{b}} T [0,\infty)_x$. So, a vanishing point of $\mathsf{H}_p$ on a boundary component $\mathrm{f}$ of the de,sc-phase space is a vanishing point of the induced flow on $\mathrm{f}$.}  We split the vanishing set of $\mathsf{H}_p$ into several components, 
\begin{equation}
	\calR^{\pm}_{+},\calR^{\pm}_{-}, \calN^{\pm}_{+}, \calN^{\pm}_{-}, \calC^{\pm}_{+}, \calC^{\pm}_{-}, \calK^{\pm}_{+}, \calK^{\pm}_{-}, \calA^\pm_+,\calA^\pm_- \subseteq \Sigma_{\mathsf{m},\pm}
	\label{eq:radial_sets}
\end{equation}
our \emph{radial sets}. We abbreviate $\calR = \calR^+_+\cup \calR^-_-\cup \calR^+_-\cup \calR^-_+$ and likewise for the other radial sets.
The radial sets $\calR^+_+$, $\calR^-_-$, $\calR^+_-$, $\calR^-_+$ depend on $\mathsf{m}$, but we omit this from the notation. The sign in the superscript denotes which sheet of the characteristic set the component lies in, with a positive sign denoting positive $\tau$ component, and the sign in the subscript denotes which half-space $\mathrm{cl}_\bbO\{\pm t>0\}$ the component lies in. 
The interpretation of the different radial sets is as follows: 
\begin{itemize}
	\item The radial sets $\smash{\calR^\pm_+,\calR^\pm_-}$, located over the timelike caps, are where the de,sc-wavefront set associated with the oscillations of the asymptotic tails of solutions of the Klein--Gordon IVP lives,
	\item $\calN^\pm_+,\calN^\pm_-$ are the endpoints of the Hamiltonian flow along which singularities in the interior of the spacetime (including singularities in initial data) propagate, thus are the entryway for singularities in the interior to the fibers over null infinity, 
	\item $\calC^\pm_+,\calC^\pm_-, \calK^\pm_+,\calK^\pm_-$ are the parts of the corners of the de,sc-phase space lying in the portion of the characteristic set $\Sigma_{\mathsf{m}}$ with zero momentum in the directions dual to the angular coordinates, and not already in one of the $\calN$'s; $\calC$ is over the corner with timelike infinity and $\calK$ is over the corner with spacelike infinity,  and
	\item $\calA^\pm_+,\calA^\pm_-$ are additional radial sets that show up only in the (1+$d$)-dimensional case for $d\geq 2$ and are therefore not depicted in \Cref{fig:O} (but see \Cref{fig:flowplot} and \Cref{fig:globalflowplot}). These can be probed via families of null geodesics with \textit{large} angular momentum. 
\end{itemize}
The simplest of the radial sets to define are $\calR^\pm_-,\calR^\pm_+$. Identifying $\smash{{}^{\mathrm{de,sc}}T^*_{\mathrm{Pf}^\circ\cup \mathrm{Tf}^\circ} \bbO}$ with $\smash{{}^{\mathrm{sc}}T^*_{C_-\cup C_+} \bbM}$, 
\begin{align}
	\begin{split} 
	\calR^\pm_- &= \mathrm{cl}_{{}^{\mathrm{de,sc}} T^* \bbO } (\calR^\pm_{0} \cap {}^{\mathrm{sc}} \pi^{-1}(C_-) )  \\ 
	\calR^\pm_+ &= \mathrm{cl}_{{}^{\mathrm{de,sc}} T^* \bbO } (\calR^\pm_{0} \cap {}^{\mathrm{sc}} \pi^{-1}(C_+) ),
	\label{eq:R_def}
	\end{split}
\end{align}
where 
\begin{equation} 
	\calR^\pm_0 = \operatorname{Graph}(\pm \mathsf{m} \dd \sqrt{t^2-r^2}|_{C_-\cup C_+} )
\end{equation} 
are the two (disconnected) radial sets of the usual sc-Hamiltonian flow on ${}^{\mathrm{sc}}\pi:{}^{\mathrm{sc}} \overline{T}^* \bbM \to \bbM$, one in each sheet, depicted in \Cref{fig:sc}. In other words, for each $\sigma \in \{-,+\}$, 
\begin{equation}
	\calR^\pm_\sigma  = \operatorname{cl}_{{}^{\mathrm{de,sc}} T^* \bbO }( \!\!\!\!\!\!\!\!\!\! \underbrace{\operatorname{WF}_{\mathrm{sc}}(e^{\pm i\mathsf{m} \sqrt{t^2-r^2}} ) \cap {}^{\mathrm{sc}} \pi^{-1}(C_\sigma )}_{\text{portion of sc-characteristic set over base infinity}}\!\!\!\!\!\!\!\!\!\!\! ).
\end{equation}
This substantiates the description above: the radial set $\calR$ is where the de,sc-wavefront set associated with the oscillations of the asymptotic tails of solutions of the Klein--Gordon IVP lives.
In contrast to $\calR_0$, the radial set $\calR$ \emph{does not hit fiber infinity}. See \Cref{fig:O}, where this is indicated. Consequently, we have well-behaved notions of module regularity associated with $\calR$. These are discussed below, in \S\ref{subsec:test_modules}, and put to work elsewhere in the paper.

A quick way of seeing that $\calR$ does not hit fiber infinity is the following.
Since $(t^2-r^2)^{1/2} = 1/(\varrho_{\mathrm{nFf}}\varrho_{\mathrm{Ff}})$ near $\mathrm{nFf}\cap\mathrm{Ff}$, 
\begin{equation}
	\pm  \mathsf{m} \dd \sqrt{t^2-r^2} = \mp  \mathsf{m} \Big(  \frac{\dd \varrho_{\mathrm{nFf}}}{\varrho_{\mathrm{nFf}}^2 \varrho_{\mathrm{Ff}}} + \frac{\dd \varrho_{\mathrm{Ff}}}{\varrho_{\mathrm{nFf}} \varrho_{\mathrm{Ff}}^2 } \Big) .
	\label{eq:misc_061}
\end{equation}
The right-hand side is a typical example of a nonvanishing and nonsingular de,sc- one-form (see \S\ref{subsec:phasespace}). Since the sc- radial set $\calR_0$ is the graph of the left-hand side over the timelike caps of $\bbM$, it must be the case that $\calR$ is the graph of the left-hand side of \cref{eq:misc_061} over $\mathrm{Pf}\cup \mathrm{Ff}$: 
\begin{equation}
	\calR_\sigma^\pm = \operatorname{Graph}\Big\{\mp  \mathsf{m} \Big(  \frac{\dd \varrho_{\mathrm{nFf}}}{\varrho_{\mathrm{nFf}}^2 \varrho_{\mathrm{Ff}}} + \frac{\dd \varrho_{\mathrm{Ff}}}{\varrho_{\mathrm{nFf}} \varrho_{\mathrm{Ff}}^2 } \Big)\Big|_{\mathrm{Tf} } \Big\} \subset {}^{\mathrm{de,sc}}\overline{T}^*_{\mathrm{Tf}} \bbO,
	\label{eq:R_alt}
\end{equation}
$\mathrm{Tf}\in \{\mathrm{Pf},\mathrm{Ff}\}$.
The difference is that, since the right-hand side of \cref{eq:misc_061} is nonsingular as a de,sc- one-form, $\calR$ does not hit fiber infinity. 

Note that $\calR$ can be identified with $\calR_0$ over the interior of the timelike caps.

When studying the IVP in \S\ref{subsec:IVP}, control is propagated through the radial sets (starting at a neighborhood of the Cauchy hypersurface $\Sigma_0$ and ending at $\calR$) in the following order: 
\begin{equation}
	\calA, \calN\backslash {}^{\mathrm{de,sc}}\pi^{-1}(\mathrm{Tf}\cap \mathrm{nf} ) , \calK,\calC,\calN,\calR.
	\label{eq:Cauchy_order}
\end{equation}
In \S\ref{subsec:scattering}, we will also study the scattering problem, which consists of specifying incoming data at $t=-\infty$. When doing so, control is propagated through the radial sets (starting at $\calR_-$ and ending at a neighborhood of $\Sigma_0$) in the following order:
\begin{align}
	&\calR^+_-, \calN_-^+ \backslash {}^{\mathrm{de,sc}} \pi^{-1}(\mathrm{Sf}\cap \mathrm{nPf}) , \calC_-^+, \calK_-^+, \calN_-^+, \calA_-^+, 
	\label{eq:scat_order_1}
	\intertext{on the $\tau>0$ sheet (\Cref{fig:O}) and}
	&\calR^-_-, \calN_-^- \backslash {}^{\mathrm{de,sc}} \pi^{-1}(\mathrm{Sf}\cap \mathrm{nPf}) , \calC_-^-, \calK_-^-, \calN_-^-, \calA_-^-, 
	\label{eq:scat_order_2}
\end{align}
on the other. Then, once control is known near $\Sigma_0$, control can be propagated forwards as in the Cauchy problem, in the same order as \cref{eq:Cauchy_order}, ending at $\calR_+$.
The backwards scattering problem, in which outgoing data is specified, is of course similar.

\begin{remark} Note the flow segments, the darker arrows in \Cref{fig:O}, connecting the two endpoints of each component of $\calN$. As a consequence of the existence of this $\calN$-to-$\calN$ path, we are forced to prove two separate radial point estimates at $\calN$: one in which control is propagated into a proper portion (a ``ray,'' beginning over spacelike or timelike infinity, stopping short of the other corner) and another in which the whole is controlled altogether. 
For unsurprising technical reasons, the former is somewhat subtle. We only prove the estimates needed here, though we do not rule out that more can be said.
\end{remark}

\begin{figure}[t]
	\begin{center}
		\begin{tikzpicture}[scale=.8]
			\coordinate (one) at (-.413\octheight,\octheight) {};
			\coordinate (two) at (-\octheight,.413\octheight) {}; 
			\coordinate (three) at (-\octheight,-.413\octheight) {};
			\coordinate (four) at (-.413\octheight,-\octheight) {};
			\coordinate (five) at (.413\octheight,-\octheight) {};
			\coordinate (six) at (\octheight,-.413\octheight) {};
			\coordinate (seven) at (\octheight,.413\octheight) {}; 
			\coordinate (eight) at (.413\octheight,\octheight) {};
			\coordinate (oner) at (-.413\octheightr,\octheightr) {};
			\coordinate (twor) at (-\octheightr,.413\octheightr) {}; 
			\coordinate (threer) at (-\octheightr,-.413\octheightr) {};
			\coordinate (fourr) at (-.413\octheightr,-\octheightr) {};
			\coordinate (fiver) at (.413\octheightr,-\octheightr) {};
			\coordinate (sixr) at (\octheightr,-.413\octheightr) {};
			\coordinate (sevenr) at (\octheightr,.413\octheightr) {}; 
			\coordinate (eightr) at (.413\octheightr,\octheightr) {};
			\coordinate (oneint) at (-1.25*.413\octheight,1.25\octheight) {};
			\coordinate (eightint) at (.413*1.25\octheight,1.25\octheight) {};
			\coordinate (fourint) at (-.413*1.25\octheight,-1.25\octheight) {};
			\coordinate (fiveint) at (.413*1.25\octheight,-1.25\octheight) {};
			\draw[fill=gray!10] (oner) -- (twor) -- (threer) -- (fourr) -- (fiver) -- (sixr) -- (sevenr) -- (eightr) -- cycle;
			\draw[fill=gray!30] (one) -- (two) -- (three) -- (four) -- (five) -- (six) -- (seven) -- (eight) -- cycle;
			\draw (one) -- (oner);
			\draw (two) -- (twor);
			\draw (three) -- (threer);
			\draw (four) -- (fourr);
			\draw (five) -- (fiver);
			\draw (six) -- (sixr);
			\draw (seven) -- (sevenr);		
			\draw (eight) -- (eightr);
			\begin{scope}[decoration={
					markings,
					mark=at position 0.51 with {\arrow[scale=1.5,>=latex, color=gray]{>}}}]
				\draw[gray, postaction={decorate}] (-.706\octheight,-.706\octheight) -- (.706\octheight,.706\octheight); 
				\draw[gray, postaction={decorate}] plot[smooth,tension=.1] coordinates {
					(-.636\octheight,-.776\octheight) (.125\octheight,-.125\octheight) (.776\octheight,.636\octheight)};
				\draw[gray, postaction={decorate}] plot[smooth,tension=.1] coordinates {
					(-.776\octheight,-.636\octheight) (-.125\octheight,+.125\octheight) (.636\octheight, .776\octheight)};
				\draw[gray, postaction={decorate}]  plot[smooth,tension=.3] coordinates{ 
					(-.566\octheight,-.846\octheight) (-.456\octheight,-.756\octheight) (.275\octheight,-.275\octheight) (.756\octheight, .456\octheight)  (.846\octheight,.566\octheight) };
				\draw[gray, postaction={decorate}]  plot[smooth,tension=.3] coordinates{ 
					(-.846\octheight,-.566\octheight) (-.756\octheight,-.456\octheight) (-.275\octheight,.275\octheight) (.456\octheight,.756\octheight)  (.566\octheight,.846\octheight) };
				\draw[gray, postaction={decorate}]  plot[smooth,tension=.4] coordinates{ 
					(-.526\octheight,-.886\octheight) (-.146\octheight,-.786\octheight) (.45\octheight,-.45\octheight) (.786\octheight,.146\octheight) (.886\octheight,.526\octheight)    }; 
				\draw[gray, postaction={decorate}]  plot[smooth,tension=.4] coordinates{ 
					(-.886\octheight,-.526\octheight) (-.786\octheight,-.146\octheight) (-.45\octheight,.45\octheight) (.146\octheight,.786\octheight) (.526\octheight,.886\octheight)    }; 
				\draw[postaction={decorate}] (three) -- (two);
				\draw[postaction={decorate}] (threer) -- (three); 
				\draw[postaction={decorate}] (seven) -- (sevenr); 
				\draw[postaction={decorate}] (threer) -- (twor);
				\draw[postaction={decorate}] (fourr) -- (threer);
				\draw[postaction={decorate}] (two) -- (one);
				\draw[postaction={decorate}] (fiver) -- (fourr);
				\draw[postaction={decorate}] (six) -- (seven);
				\draw[postaction={decorate}] (sixr) -- (sevenr);
				\draw[postaction={decorate}] (sevenr) -- (eightr);
				\draw[postaction={decorate}] (eightr) -- (oner);
				\draw[gray, postaction={decorate}]  plot[smooth,tension=.5] coordinates{ (fourint) (-.513\octheightr,-.85\octheightr)  (-.75\octheightr,-.443\octheightr) 	(-.886\octheight,-.526\octheight) };
				\draw[gray, postaction={decorate}]  plot[smooth,tension=.5] coordinates{ 	(.886\octheight,.526\octheight) (.75\octheightr,.443\octheightr)  (.513\octheightr,.85\octheightr)  (eightint)};
				\draw[gray, postaction={decorate}]  plot[smooth,tension=.5] coordinates{ (fiveint) (.413\octheightr,-.65\octheightr)  (.75\octheightr,-.443\octheightr) 	(.886\octheightr,-.526\octheightr) };
				\draw[gray, postaction={decorate}]  plot[smooth,tension=.5] coordinates{ 	(-.886\octheightr,.526\octheightr) (-.75\octheightr,.443\octheightr) (-.413\octheightr,.65\octheightr)  (oneint)};
				\draw[gray, postaction={decorate}] (threer) -- (two);
				\draw[gray, postaction={decorate}] (six) -- (sevenr);
				\draw[gray, postaction={decorate}] (-.2\octheight,-1.25\octheight) -- (five);
				\draw[gray, postaction={decorate}] (.2\octheight,-1.25\octheight) -- (fourr);
				\draw[gray, postaction={decorate}] (one) -- (.2\octheight,1.25\octheight);
				\draw[gray, postaction={decorate}] (eightr) -- (-.2\octheight,1.25\octheight);
			\end{scope}
		\begin{scope}[decoration={
				markings,
				mark=at position 0.51 with {\arrow[scale=1.5,>=latex, color=black]{>}}}] 
				\draw[postaction={decorate}] (one) -- (eight);
				\draw[postaction={decorate}] (two) -- (one); 
				\draw[postaction={decorate}] (twor) -- (two); 
				\draw[postaction={decorate}] (four) -- (five);
				\draw[postaction={decorate}] (six) -- (sixr); 
				\draw[postaction={decorate}] (five) -- (six);
		\end{scope} 
			\begin{scope}[decoration={
					markings,
					mark=at position 0.75 with {\arrow[scale=1.5,>=latex, color=gray]{>}}}]
				\draw[postaction={decorate}] (fourint) -- (four);
				\draw[postaction={decorate}] (fourint) -- (fourr);
				\draw[postaction={decorate}] (fiveint) -- (five);
				\draw[postaction={decorate}] (fiveint) -- (fiver);
				\draw[postaction={decorate}] (one) -- (oneint);
				\draw[postaction={decorate}] (eight) -- (eightint);
				\draw[postaction={decorate}] (oner) -- (oneint);
				\draw[postaction={decorate}] (eightr) -- (eightint);
				\draw[gray, postaction={decorate}]  plot[smooth,tension=.6] coordinates{ (fourint) (-.63\octheight,-1.1\octheight) (-.526\octheight,-.886\octheight) };
				\draw[gray, postaction={decorate}]  plot[smooth,tension=.6] coordinates{(-.616\octheightr,.796\octheightr) (-.73\octheight,1.19\octheight) (oneint) };
				\draw[gray, postaction={decorate}]  plot[smooth,tension=.6] coordinates{ (fiveint) (.73\octheight,-1.19\octheight) (.616\octheightr,-.796\octheightr) };
				\draw[gray, postaction={decorate}]  plot[smooth,tension=.6] coordinates{(.526\octheight,.886\octheight) (.66\octheight,1.0\octheight) (eightint) };
			\end{scope}
			\draw[thick, darkblue] (oner) -- (twor);
			\filldraw[color=darkblue] (oner) circle (2pt);
			\filldraw[color=darkblue] (twor) circle (2pt);
			\draw[thick, darkblue] (eight) -- (seven);
			\filldraw[color=darkblue] (eight) circle (2pt);
			\filldraw[color=darkblue] (seven) circle (2pt);
			\draw[thick, darkblue] (four) -- (three);
			\filldraw[color=darkblue] (four) circle (2pt);
			\filldraw[color=darkblue] (three) circle (2pt);
			\draw[thick, darkblue] (fiver) -- (sixr);
			\filldraw[color=darkblue] (fiver) circle (2pt);
			\filldraw[color=darkblue] (sixr) circle (2pt);
			\filldraw[color=darkgreen] (one) circle (2pt);
			\filldraw[color=darkgreen] (five) circle (2pt);
			\filldraw[color=mygreen] (two) circle (2pt);
			\filldraw[color=mygreen] (six) circle (2pt);
			\filldraw[color=mygreen] (threer) circle (2pt);
			\filldraw[color=mygreen] (sevenr) circle (2pt);
			\filldraw[color=darkgreen] (fourr) circle (2pt);
			\filldraw[color=darkgreen] (eightr) circle (2pt);
			\filldraw[color=darkcandyapp] (oneint) circle (2pt);
			\filldraw[color=darkcandyapp] (eightint) circle (2pt);
			\draw[thick, darkcandyapp] (oneint) -- (eightint);
			\filldraw[color=darkcandyapp] (fourint) circle (2pt);
			\filldraw[color=darkcandyapp] (fiveint) circle (2pt);
			\draw[thick, darkcandyapp] (fourint) -- (fiveint);
			\node[color=darkblue] (nm) at (.82\octheightr,-.82\octheightr) {$\calN^+_-$};
			\node[color=darkblue] (np) at (-.82\octheightr,.82\octheightr) {$\calN^+_+$};
			\node[color=mygreen] (kp) at (+1.175\octheightr,.42\octheightr) {$\calK^+_+$};
			\node[color=mygreen] (km) at (-1.175\octheightr,-.42\octheightr) {$\calK^+_-$};
			\node[color=darkgreen] (cp) at (+.59\octheightr,+.995\octheightr) {$\calC^+_+$};
			\node[color=darkgreen] (cm) at (-.59\octheightr,-.995\octheightr) {$\calC^+_-$};
			\node[color=darkcandyapp] (bb) at (.22\octheightr,-.93\octheightr) {$\calR^+_-$};
			\node[color=darkcandyapp] (bc) at (-.22\octheightr,.91\octheightr) {$\calR^+_+$};
		\end{tikzpicture}
		\begin{tikzpicture}[scale=.8, xscale=-1]
			\coordinate (one) at (-.413\octheight,\octheight) {};
			\coordinate (two) at (-\octheight,.413\octheight) {}; 
			\coordinate (three) at (-\octheight,-.413\octheight) {};
			\coordinate (four) at (-.413\octheight,-\octheight) {};
			\coordinate (five) at (.413\octheight,-\octheight) {};
			\coordinate (six) at (\octheight,-.413\octheight) {};
			\coordinate (seven) at (\octheight,.413\octheight) {}; 
			\coordinate (eight) at (.413\octheight,\octheight) {};
			\coordinate (oner) at (-.413\octheightr,\octheightr) {};
			\coordinate (twor) at (-\octheightr,.413\octheightr) {}; 
			\coordinate (threer) at (-\octheightr,-.413\octheightr) {};
			\coordinate (fourr) at (-.413\octheightr,-\octheightr) {};
			\coordinate (fiver) at (.413\octheightr,-\octheightr) {};
			\coordinate (sixr) at (\octheightr,-.413\octheightr) {};
			\coordinate (sevenr) at (\octheightr,.413\octheightr) {}; 
			\coordinate (eightr) at (.413\octheightr,\octheightr) {};
			\coordinate (oneint) at (-1.25*.413\octheight,1.25\octheight) {};
			\coordinate (eightint) at (.413*1.25\octheight,1.25\octheight) {};
			\coordinate (fourint) at (-.413*1.25\octheight,-1.25\octheight) {};
			\coordinate (fiveint) at (.413*1.25\octheight,-1.25\octheight) {};
			\draw[fill=gray!10] (oner) -- (twor) -- (threer) -- (fourr) -- (fiver) -- (sixr) -- (sevenr) -- (eightr) -- cycle;
			\draw[fill=gray!30] (one) -- (two) -- (three) -- (four) -- (five) -- (six) -- (seven) -- (eight) -- cycle;
			\draw (one) -- (oner);
			\draw (two) -- (twor);
			\draw (three) -- (threer);
			\draw (four) -- (fourr);
			\draw (five) -- (fiver);
			\draw (six) -- (sixr);
			\draw (seven) -- (sevenr);		
			\draw (eight) -- (eightr);
			\begin{scope}[decoration={
					markings,
					mark=at position 0.51 with {\arrow[scale=1.5,>=latex, color=gray]{>}}}]
				\draw[gray, postaction={decorate}] (-.706\octheight,-.706\octheight) -- (.706\octheight,.706\octheight); 
				\draw[gray, postaction={decorate}] plot[smooth,tension=.1] coordinates {
					(-.636\octheight,-.776\octheight) (.125\octheight,-.125\octheight) (.776\octheight,.636\octheight)};
				\draw[gray, postaction={decorate}] plot[smooth,tension=.1] coordinates {
					(-.776\octheight,-.636\octheight) (-.125\octheight,+.125\octheight) (.636\octheight, .776\octheight)};
				\draw[gray, postaction={decorate}]  plot[smooth,tension=.3] coordinates{ 
					(-.566\octheight,-.846\octheight) (-.456\octheight,-.756\octheight) (.275\octheight,-.275\octheight) (.756\octheight, .456\octheight)  (.846\octheight,.566\octheight) };
				\draw[gray, postaction={decorate}]  plot[smooth,tension=.3] coordinates{ 
					(-.846\octheight,-.566\octheight) (-.756\octheight,-.456\octheight) (-.275\octheight,.275\octheight) (.456\octheight,.756\octheight)  (.566\octheight,.846\octheight) };
				\draw[gray, postaction={decorate}]  plot[smooth,tension=.4] coordinates{ 
					(-.526\octheight,-.886\octheight) (-.146\octheight,-.786\octheight) (.45\octheight,-.45\octheight) (.786\octheight,.146\octheight) (.886\octheight,.526\octheight)    }; 
				\draw[gray, postaction={decorate}]  plot[smooth,tension=.4] coordinates{ 
					(-.886\octheight,-.526\octheight) (-.786\octheight,-.146\octheight) (-.45\octheight,.45\octheight) (.146\octheight,.786\octheight) (.526\octheight,.886\octheight)    }; 
				\draw[postaction={decorate}] (three) -- (two);
				\draw[postaction={decorate}] (threer) -- (three); 
				\draw[postaction={decorate}] (seven) -- (sevenr); 
				\draw[postaction={decorate}] (threer) -- (twor);
				\draw[postaction={decorate}] (fourr) -- (threer);
				\draw[postaction={decorate}] (two) -- (one);
				\draw[postaction={decorate}] (fiver) -- (fourr);
				\draw[postaction={decorate}] (six) -- (seven);
				\draw[postaction={decorate}] (sixr) -- (sevenr);
				\draw[postaction={decorate}] (sevenr) -- (eightr);
				\draw[postaction={decorate}] (eightr) -- (oner);
				\draw[gray, postaction={decorate}]  plot[smooth,tension=.5] coordinates{ (fourint) (-.513\octheightr,-.85\octheightr)  (-.75\octheightr,-.443\octheightr) 	(-.886\octheight,-.526\octheight) };
				\draw[gray, postaction={decorate}]  plot[smooth,tension=.5] coordinates{ 	(.886\octheight,.526\octheight) (.75\octheightr,.443\octheightr)  (.513\octheightr,.85\octheightr)  (eightint)};
				\draw[gray, postaction={decorate}]  plot[smooth,tension=.5] coordinates{ (fiveint) (.413\octheightr,-.65\octheightr)  (.75\octheightr,-.443\octheightr) 	(.886\octheightr,-.526\octheightr) };
				\draw[gray, postaction={decorate}]  plot[smooth,tension=.5] coordinates{ 	(-.886\octheightr,.526\octheightr) (-.75\octheightr,.443\octheightr) (-.413\octheightr,.65\octheightr)  (oneint)};
				\draw[gray, postaction={decorate}] (threer) -- (two);
				\draw[gray, postaction={decorate}] (six) -- (sevenr);
				\draw[gray, postaction={decorate}] (-.2\octheight,-1.25\octheight) -- (five);
				\draw[gray, postaction={decorate}] (.2\octheight,-1.25\octheight) -- (fourr);
				\draw[gray, postaction={decorate}] (one) -- (.2\octheight,1.25\octheight);
				\draw[gray, postaction={decorate}] (eightr) -- (-.2\octheight,1.25\octheight);
			\end{scope}
			\begin{scope}[decoration={
					markings,
					mark=at position 0.51 with {\arrow[scale=1.5,>=latex, color=gray]{>}}}] 
				\draw[postaction={decorate}] (one) -- (eight);
				\draw[postaction={decorate}] (two) -- (one); 
				\draw[postaction={decorate}] (twor) -- (two); 
				\draw[postaction={decorate}] (four) -- (five);
				\draw[postaction={decorate}] (six) -- (sixr); 
				\draw[postaction={decorate}] (five) -- (six);
			\end{scope} 
			\begin{scope}[decoration={
					markings,
					mark=at position 0.75 with {\arrow[scale=1.5,>=latex, color=gray]{>}}}]
				\draw[postaction={decorate}] (fourint) -- (four);
				\draw[postaction={decorate}] (fourint) -- (fourr);
				\draw[postaction={decorate}] (fiveint) -- (five);
				\draw[postaction={decorate}] (fiveint) -- (fiver);
				\draw[postaction={decorate}] (one) -- (oneint);
				\draw[postaction={decorate}] (eight) -- (eightint);
				\draw[postaction={decorate}] (oner) -- (oneint);
				\draw[postaction={decorate}] (eightr) -- (eightint);
				\draw[gray, postaction={decorate}]  plot[smooth,tension=.6] coordinates{ (fourint) (-.63\octheight,-1.1\octheight) (-.526\octheight,-.886\octheight) };
				\draw[gray, postaction={decorate}]  plot[smooth,tension=.6] coordinates{(-.616\octheightr,.796\octheightr) (-.73\octheight,1.19\octheight) (oneint) };
				\draw[gray, postaction={decorate}]  plot[smooth,tension=.6] coordinates{ (fiveint) (.73\octheight,-1.19\octheight) (.616\octheightr,-.796\octheightr) };
				\draw[gray, postaction={decorate}]  plot[smooth,tension=.6] coordinates{(.526\octheight,.886\octheight) (.66\octheight,1.0\octheight) (eightint) };
			\end{scope}
			\draw[thick, darkblue] (oner) -- (twor);
			\filldraw[color=darkblue] (oner) circle (2pt);
			\filldraw[color=darkblue] (twor) circle (2pt);
			\draw[thick, darkblue] (eight) -- (seven);
			\filldraw[color=darkblue] (eight) circle (2pt);
			\filldraw[color=darkblue] (seven) circle (2pt);
			\draw[thick, darkblue] (four) -- (three);
			\filldraw[color=darkblue] (four) circle (2pt);
			\filldraw[color=darkblue] (three) circle (2pt);
			\draw[thick, darkblue] (fiver) -- (sixr);
			\filldraw[color=darkblue] (fiver) circle (2pt);
			\filldraw[color=darkblue] (sixr) circle (2pt);
			\filldraw[color=darkgreen] (one) circle (2pt);
			\filldraw[color=darkgreen] (five) circle (2pt);
			\filldraw[color=mygreen] (two) circle (2pt);
			\filldraw[color=mygreen] (six) circle (2pt);
			\filldraw[color=mygreen] (threer) circle (2pt);
			\filldraw[color=mygreen] (sevenr) circle (2pt);
			\filldraw[color=darkgreen] (fourr) circle (2pt);
			\filldraw[color=darkgreen] (eightr) circle (2pt);
			\filldraw[color=darkcandyapp] (oneint) circle (2pt);
			\filldraw[color=darkcandyapp] (eightint) circle (2pt);
			\draw[thick, darkcandyapp] (oneint) -- (eightint);
			\filldraw[color=darkcandyapp] (fourint) circle (2pt);
			\filldraw[color=darkcandyapp] (fiveint) circle (2pt);
			\draw[thick, darkcandyapp] (fourint) -- (fiveint);
			\node[color=darkblue] (nm) at (.82\octheightr,-.82\octheightr) {$\calN^+_-$};
			\node[color=darkblue] (np) at (-.82\octheightr,.82\octheightr) {$\calN^+_+$};
			\node[color=mygreen] (kp) at (+1.175\octheightr,.42\octheightr) {$\calK^+_+$};
			\node[color=mygreen] (km) at (-1.175\octheightr,-.42\octheightr) {$\calK^+_-$};
			\node[color=darkgreen] (cp) at (+.59\octheightr,+.995\octheightr) {$\calC^+_+$};
			\node[color=darkgreen] (cm) at (-.59\octheightr,-.995\octheightr) {$\calC^+_-$};
			\node[color=darkcandyapp] (bb) at (.22\octheightr,-.93\octheightr) {$\calR^+_-$};
			\node[color=darkcandyapp] (bc) at (-.22\octheightr,.91\octheightr) {$\calR^+_+$};
		\end{tikzpicture}
	\end{center}
	\caption{The de,sc-Hamiltonian flow within one sheet of the de,sc-characteristic set $\Sigma_{\mathsf{m},+}$, when $d=1$ (in which case de,sc- means sc,sc-). When $d=1$, $\Sigma_{\mathsf{m},+}$ consists of two octagons, one of which contains the left-moving points at fiber infinity and the other of which contains the right-moving points, connected by the characteristic set over the boundary. 
	Conventions are mostly as in \Cref{fig:sc};
	due to difficulties with perspective, we have depicted $\Sigma_{\mathsf{m},+}$ from two points of view, with one component of fiber infinity hidden in each perspective. 
	Fiber infinity is depicted in dark gray. Each of the lighter quadrilateral panels depicts the portion of $\Sigma_{\mathsf{m},+}$ over one of the faces $\mathrm{f} \in \{\mathrm{Pf},\mathrm{nPF},\mathrm{Sf},\mathrm{nFf},\mathrm{Ff}\}$. (Since $d=1$, $\mathrm{nPf},\mathrm{Sf},\mathrm{nFf}$ each consist of two connected components.) The radial sets, $\color{darkcandyapp}\calR$, $\color{darkgreen}\calC$, $\color{darkblue}\calN$, $\color{mygreen}\calK$, are depicted in various colors. The source of the flow is $\calR^+_-$, and, correspondingly, the sink is $\calR^+_+$. The set $\calR$ is, away from its endpoints, identifiable with the sc-radial set $\calR_0$ depicted in \Cref{fig:sc}. This is why they are both colored red. The other six radial sets (each of which, because $d=1$, consists of two connected components, only one of which is labeled above) are saddle points. The flow in $\Sigma_{\mathsf{m},-}$ looks similar, with the arrows reversed.}
	\label{fig:O}
\end{figure}

Hintz and Vasy \cite{HintzVasyScriEB} have recently investigated massless wave propagation near null infinity using fully microlocal tools very similar to those used here. 
In contrast to the de,sc-calculus employed below, their e,b-calculus is not symbolic, for the same reason that the e- (``edge'') and b- (``boundary'') calculi are not symbolic. While this is necessary when studying mass\textit{less} wave propagation, for which radiation must be understood, this means that Hintz and Vasy do not study propagation at finite frequencies. For the reasons sketched above, finite frequencies are important in understanding massive wave propagation. 
The purely symbolic de,sc-calculus turns out to be well-suited for this purpose.

The radial sets $\calN,\calC,\calK,\calA$ previously appeared in \cite{HintzVasyScriEB} under different aliases. 
The inclusion $\smash{\bbS^* \bbR^{1,d}} \hookrightarrow \smash{{}^{\mathrm{de,sc}}\bbS^* \bbO}$ extends to a diffeomorphism  
\begin{equation} 
	{}^{\mathrm{e,b}} \bbS^* \bbO \to {}^{\mathrm{de,sc}}\bbS^* \bbO, 
\end{equation}
so fiber infinity of the de,sc-cotangent bundle is canonically identifiable with fiber infinity of the e,b-cotangent bundle. So, besides the fiber radial direction, the situation at fiber infinity is the same here as in \cite{HintzVasyScriEB}.
In \cite{HintzVasyScriEB}, the authors use terminology which, while fitting for the analysis at fiber infinity, is misleading when the flow in the de,sc-fiber radial directions is considered. From their e,b-perspective, the components of $\calN$ are global sources and sinks. From the de,sc-perspective, this role is instead played by $\calR$. More confusingly, de,sc-singularities can propagate from $\calN \cap {}^{\mathrm{de,sc}} \pi^{-1}(\mathrm{Sf})$ through the interior of the fibers of the de,sc-cotangent bundle back up to $\calK$ (as can be seen in \Cref{fig:O}), so $\calN$ is not even a source/sink for the flow between the radial sets already studied in \cite{HintzVasyScriEB} once one considers the fiber radial component of the flow.  So, some terminological change is necessary. 

Finally, we point out recent work \cite{HassellNL} of Gell-Redman, Gomes, and Hassell on the nonlinear Schr\"odinger equation. Their regularity theory bears some similarities to the test modules used below, but it does not appear possible to straightforwardly apply their approach to the Klein--Gordon equation.

\begin{remark}[Comparing $\Psi_{\mathrm{de,sc}}(\bbO)$- and $\Psi_{\mathrm{sc}}(\bbM)$- frequency variables]
	The reader may wonder how de,sc- frequencies compare to sc-frequencies (where by ``sc-frequency'' we mean ordinary frequencies, i.e.\ coordinates in the usual phase space $\smash{{}^{\mathrm{sc}} T^* \bbM}$, not $\smash{{}^{\mathrm{sc,sc}}T^* \bbO}$).
	After all, our main selling point for the de,sc- framework is that the oscillations present in solutions of the Cauchy problem for the Klein--Gordon equation lie at finite de,sc- frequency, whereas we saw above and in \Cref{fig:sc} that the same is not true in the original sc-framework. Thus, some finite de,sc-frequencies correspond to infinite sc-frequencies. On the other hand, \cref{eq:misc_035} indicates that a typical sc-frequency lies at infinite de,sc-frequency. Another manifestation of this is that, in the de,sc-phase space, all bicharacteristics of the Hamiltonian flow over spacelike infinity end at (de,sc-)fiber infinity (see \Cref{fig:O}), whereas this is not true in the sc-phase space. So, de,sc-frequencies are, in general, \emph{neither larger nor smaller} than sc-frequencies; they are larger in one direction, but smaller in another. 
	
	Let us make this concrete, at least in 1+1D, where de- just means sc- (so de,sc- means sc,sc-). So, we are comparing ${}^{\mathrm{sc}}T^* \bbM$ with ${}^{\mathrm{sc,sc}}T^* \bbO$ and asking how functions whose oscillations correspond to points in these two bundles differ. To avoid confusion, we will speak of $\Psi_{\mathrm{sc}}$- and $\Psi_{\mathrm{de,sc}}=\Psi_{\mathrm{sc,sc}}$- frequencies. We will only discuss the situation near points in the interior of null infinity. Recall that functions on $\bbR^{1,1}$ that oscillate with finite $\Psi_{\mathrm{sc}}$-frequency look like plane waves $\exp(i(t \omega+x\Xi ))$, for $\omega,\Xi\in \bbR$.  Written in terms of the coordinates $v=t-x$ and $\varrho=1/\smash{(|t|+x)^{1/2}}$, such a plane wave takes the form 
	\begin{equation}
		\exp \Big[ \frac{i}{2} \Big( (\omega-\Xi) v + \frac{\omega+\Xi}{ \varrho^2} \Big)  \Big] . 
		\label{eq:sc_oscillations_example}
	\end{equation}
	On the other hand, a function whose oscillations lie at finite $\Psi_{\mathrm{de,sc}}$-frequency looks like 
	\begin{equation}
		\exp \Big[ \frac{i}{\varrho} (\zeta v+\xi) \Big] 	
		\label{eq:desc_oscillations_example}
	\end{equation}
	near points in the interior of null infinity, 
	for some $\xi,\zeta \in \bbR$. This is analogous to how the plane wave $\exp(i\xi x+i\zeta y)$  on $\bbR^2_{x,y}$ looks when written in terms of $1/r$ and $\varphi = \operatorname{arctan}(y/x)$: 
	\begin{equation} 
		\exp(i \xi x + i \zeta y ) = \exp(i  r (\xi \cos\varphi+\zeta \sin \varphi )) \sim \exp\Big[ \frac{i}{1/r} (\zeta \varphi + \xi)   \Big]
	\end{equation} 
	for small $\varphi$; just replace $1/r$ with $\varrho$ and $\varphi$ with $v$. Comparing \cref{eq:sc_oscillations_example} and \cref{eq:desc_oscillations_example}, we see that the function with finite $\Psi_{\mathrm{de,sc}}$- frequency is oscillating more slowly as $\varrho\to 0^+$ for $v$ fixed, and is oscillating more rapidly as $v$ varies for $\varrho$ fixed but very small. 
	This comparison is summarized  in \Cref{fig:comp}. 
	At the end of \S\ref{subsec:limitations}, we saw that the oscillations $\exp(\pm i \mathsf{m} \sqrt{t^2-r^2})$ present in solutions of the Klein--Gordon equation are, when compared to plane waves, oscillating slowly along the light cone and oscillating more rapidly in the direction across it. These are exactly the sorts of oscillations that the de,sc-calculus is apparently designed to detect (as we already saw in \cref{eq:misc_061}).

	With regards to the $d\geq 2$ case, we just mention that oscillations in the angular variables $\theta=\bfx/r$ have finite $\Psi_{\mathrm{de,sc}}$-frequency if and only if they have finite $\Psi_{\mathrm{sc}}$-frequency in the ordinary sense. Thus, an oscillation such as $\exp(i \varrho^{-1} \eta \theta_j )$, for $\eta\in \bbR$, is typical of both frameworks. In the previous sentence, $\theta_1,\dots,\theta_{d-1}$ is some local coordinate chart on $\smash{\bbS^{d-1}_\theta} = \smash{\bbS^{d-1}_{\bfx/r}}$. So, in summary: compared to $\Psi_{\mathrm{sc}}(\bbM)$, the de,sc- calculus detects oscillations that are slower along the light cone, faster across the light cone, and comparable in the (spatial) angular directions.
	\label{rem:reviewer}
\end{remark}

\begin{figure}
	\begin{tikzpicture} 
		\draw[dashed] (-3,-3) -- (-3,3.5) -- (3,3.5) -- (3,-3) -- cycle;
		\node (0) at (0,3.2) {Oscillations detected by $\Psi_{\mathrm{sc}}(\bbM)$ }; 
		\filldraw[fill=gray!10] (0,2.25) -- (2.25,0) -- (0,-2.25) -- (-2.25,0) -- cycle;
		\node (1) at (0,2.5) {Ff};
		\node (1) at (1.5,1.4) {nFf};
		\node (1) at (2.6,0) {Sf};
		\node (1) at (1.5,-1.4) {nPf};
		\node (1) at (0,-2.5) {Pf};
		\draw[orange, decorate, decoration={snake, segment length=1.2mm, amplitude=1mm}] (-.3,.3) -- (-1,1);
		\draw[orange, decorate, decoration={snake, segment length=4mm, amplitude=1mm}] (-1.8,.2) -- (-.2,1.8);
	\end{tikzpicture}
	\quad 
	\begin{tikzpicture} 
		\draw[dashed] (-3,-3) -- (-3,3.5) -- (3,3.5) -- (3,-3) -- cycle;
		\node (0) at (0,3.2) {Oscillations detected by $\Psi_{\mathrm{de,sc}}(\bbO)$ }; 
		\filldraw[fill=gray!10] (0,2.25) -- (2.25,0) -- (0,-2.25) -- (-2.25,0) -- cycle;
		\node (1) at (0,2.5) {Ff};
		\node (1) at (1.5,1.4) {nFf};
		\node (1) at (2.6,0) {Sf};
		\node (1) at (1.5,-1.4) {nPf};
		\node (1) at (0,-2.5) {Pf};
		\draw[orange, decorate, decoration={snake, segment length=4mm, amplitude=1mm}] (-.3,.3) -- (-1,1);
		\draw[orange, decorate, decoration={snake, segment length=1.02mm, amplitude=1mm}] (-1.8,.2) -- (-.2,1.8);
	\end{tikzpicture}
	\caption{A comparison, drawn on the Penrose diagram $\bbP\hookleftarrow \bbR^{1,1}$, of the sorts of oscillations measured by $\Psi_{\mathrm{sc}}(\bbM)$ (left), fast in the outgoing direction and slow in the tangential direction, and the sorts of oscillations measured by $\Psi_{\mathrm{de,sc}}(\bbO)$ (right), fast in the tangential direction and slow in the outgoing direction. In each case, the oscillations get faster as they approach the boundary, and it is really how fast this occurs that matters, but this is not depicted. }
	\label{fig:comp}
\end{figure}

\begin{remark}[Diffraction at null infinity]
	One defect of the de,sc- framework for analyzing the Klein--Gordon equation when the coefficients are well-behaved already on $\bbM$ is that, because the de,sc- Hamiltonian flow over $\mathrm{Sf}^\circ$ tends to fiber infinity as it approaches null infinity, different sc-frequencies over $\mathrm{Sf}^\circ$ are scrambled by the flow. Consequently, it is not possible in the de,sc-framework to track individual frequencies from $\mathrm{Sf}^\circ$ to $\mathrm{Pf}^\circ,\mathrm{Ff}^\circ$. In contrast, this is possible when working with $\Psi_{\mathrm{sc}}(\bbM)$; see \Cref{fig:sc}. Instead, the de,sc- propagation results in \S\ref{sec:propagation} predict that a single frequency over $\mathrm{Sf}^\circ$ propagates to infinitely many frequencies over $\mathrm{Pf}^\circ,\mathrm{Ff}^\circ$. While this possibility cannot be realized when the PDE is analyzable in $\Psi_{\mathrm{sc}}(\bbM)$, we expect it to be realized when the coefficients of the PDE are only well-behaved on $\bbO$ and not on $\bbM$, as is the case when the metric is that of a radiative spacetime.  
	This phenomenon is reminiscent of diffraction, the generation of spherical waves when an incoming wave hits a singular point in a manifold. In this case, the singular ``point'' is $\scrI\subset \partial \bbM$.
\end{remark}

\subsection*{Index of notation}
The following notation is used in more than one section:
\begin{itemize}
	\item $\bbM = \overline{\bbR^{1,d}}$, the radial compactification of spacetime. When we want to indicate the dimension, we use a superscript. For example, $\bbM^{1,1}$ is the radial compactification of $\bbR^{1,1}$. 
	
	The (open) timelike caps of $\bbM$ are denoted $C_\pm$, with $C_+$ the future timelike cap and $C_-$ the past timelike cap. The components are null infinity are $\scrI^\pm$, and spacelike infinity is $i^0$. We typically use subscripts to denote whether an object is related to the future or the past. One exception is $\scrI^\pm$, for which a superscript is standard.
	\item $\bbO \hookleftarrow \bbR^{1,d}$, the octagonal compactification of spacetime, defined in \S\ref{subsec:basespace} by blowing up null infinity in $\bbM$ and then modifying the smooth structure. 
	We use $\bbO_0$ to denote the same manifold-with-corners but without the change of smooth structure. 
	\item $\mathrm{Pf},\mathrm{nPf},\mathrm{Sf},\mathrm{nFf},\mathrm{Ff}$ various boundary hypersurfaces of $\bbO$. In each case, $\mathrm{f}$ stands for face. 
	\item $\mathrm{Tf}$, used to denote one of $\mathrm{Pf},\mathrm{Ff}$, for ``timelike face.''  $\mathrm{nf}$, used to denote one of $\mathrm{nPf},\mathrm{nFf}$, for ``null face.''  $\mathrm{Of}$, used to denote one of $\mathrm{Pf},\mathrm{Sf},\mathrm{Ff}$, for ``other face'' (in contrast to $\mathrm{nf}$).  
	\item $\Omega_{\mathrm{nfTf},\pm , T}$, for $T>0$, used to denote a particular coordinate chart (really, the domain thereof), defined in \S\ref{sec:geometry}, near the timelike corner of null infinity, $\mathrm{Tf}\cap \mathrm{nf}$. Similarly, $\Omega_{\mathrm{nfSf},\pm , R}$, for $R>0$, is the domain of a coordinate chart over the spacelike corner of null infinity. By varying $T,R$, we can cover all of null infinity.
	\item $\varrho_{\mathrm{f}}$, used to denote a boundary-defining-function of the face $\mathrm{f}$. 
	\item $\calV_{\mathrm{b}}(M)$, the $C^\infty(M)$-module of b-vector fields on a manifold-with-corners $M$, i.e.\ the smooth vector fields (where smooth means extendable to a bigger manifold-without-boundary that $M$ can be assumed to be embedded in) tangent to all of the boundary hypersurfaces thereof. The algebra of differential operators generated over $C^\infty(M)$ by the b-vector fields is denoted $\operatorname{Diff}_{\mathrm{b}}(M)$.
	\item $\calV_{\mathrm{sc}}$, the $C^\infty(\bbM)$-module of vector fields on $\bbM$ of the form $(1+r^2+t^2)^{-1/2} V$ for $V\in \calV_{\mathrm{b}}(\bbM)$. The ``scattering vector fields.''
	\item $\calV_{\mathrm{de,sc}}$, the $C^\infty(\bbO)$-module of de,sc-vector fields on $\bbO$, defined locally in \cref{eq:Vdef}, \cref{eq:Vdef_local}, and globally in \cref{eq:Vdef_global}.  
	\item $\Psi_{\mathrm{de,sc}}$, the calculus of pseudodifferential operators arising via quantizing $\calV_{\mathrm{de,sc}}$; see \S\ref{subsec:Psi}. We use $H_{\mathrm{de,sc}}^{m,\mathsf{s}}$ to denote the corresponding scale of Sobolev spaces, where $m\in \bbR,\mathsf{s}\in \bbR^5$. If instead of `$\Psi$' we write $\operatorname{Diff}$, then this indicates the subset consisting of differential operators.

	$\operatorname{WF}_{\mathrm{de,sc}}^{m,\mathsf{s}}$ denotes wavefront set relative to $H_{\mathrm{de,sc}}^{m,\mathsf{s}}$, and $\operatorname{WF}'_{\mathrm{de,sc}}(A)$ denotes the de,sc-notion of essential support of a pseudodifferential operator $A$. $\sigma_{\mathrm{de,sc}}$ is the map taking de,sc- pseudodifferential operators to their principal symbols.
	
	Throughout this work, we use sans-serif typestyle to denote pentuples of real numbers. For example, 
	\begin{equation}
		\Psi_{\mathrm{de,sc}}^{1,\mathsf{1}} = \Psi_{\mathrm{de,sc}}^{1,(1,1,1,1,1)} 
	\end{equation}
	is the set of de,sc- operators that are first order in every sense. Similarly, $\mathsf{0}=(0,0,0,0,0)$, so $\smash{\Psi^{0,\mathsf{0}}_{\mathrm{de,sc}}}$ is the space of de,sc- operators that are zeroth order in every sense, i.e.\ de,sc- \emph{microlocalizers}. The reason for using sans-serif here is that, in microlocal work on the Klein--Gordon and Helmholtz equations following \cite{VasyGrenoble} --- see e.g.\ \cite{VasyKG}\cite{HassellETAL, HassellNL} --- sans-serif is used to denote variable orders. 
	We do not use variable orders here, but in $\smash{\Psi_{\mathrm{de,sc}}^{m,\mathsf{s}}}$, the pentuple $\mathsf{s} \in \bbR^5$ functions a bit like a variable order, allowing us to consider different decay rates at different faces of $\bbO$.
	\item $\operatorname{Diff}_{\mathrm{de,sc}}^{m,\mathsf{s}}$, the elements of $\Psi_{\mathrm{de,sc}}^{m,\mathsf{s}}$ that are differential operators.
	\item ${}^{\mathrm{de,sc}}\overline{T}^* \bbO \hookleftarrow T^* \bbR^{1,d}$, the manifold-with-corners arising from compactifying the fibers of the bundle dual to that whose smooth sections are $\calV_{\mathrm{de,sc}}$. See \S\ref{subsec:phasespace}. The line in $\overline{T}$ indicates the radial compactification in the fibers.
	\item Fiber infinity in  ${}^{\mathrm{de,sc}}\overline{T}^* \bbO \hookleftarrow T^* \bbR^{1,d}$ will be denoted $\mathrm{df}$ or ${}^{\mathrm{de,sc}}\bbS^* \bbO$. 
	\item Throughout, ``$S^\bullet$'' is used to denote various function spaces of symbols, whereas ``$\calA^\bullet$'' is used to denote function spaces of partially polyhomogeneous functions; see the beginning of \S\ref{subsec:asymptotics} for a summary of our notation regarding the latter.
	\item $\zeta,\xi,\eta$ denote various smooth frequency coordinates  on ${}^{\mathrm{de,sc}}T^* \bbO$. Various such choices are in \S\ref{subsec:phasespace}. Always, $\eta$ denotes a frequency variable dual to an angular variable, whereas $\zeta,\xi$ mix the frequency variables dual to $t,r$. 
	\item $\Sigma_{\mathsf{m}}$, the de,sc-characteristic set of the Minkowski Klein--Gordon operator with mass $\mathsf{m}>0$. This set has two connected components, one corresponding to positive-energy solutions and one corresponding to negative-energy solutions. These are distinguished by writing $\Sigma_{\mathsf{m},\pm}$. 
	
	When writing $\Sigma_{\mathsf{m}}[g]$, this means the characteristic set of the Klein--Gordon operator with an asymptotically Minkowski metric $g$ instead of the exact Minkowski metric. These differ only in $\mathrm{df}$, the boundary of  ${}^{\mathrm{de,sc}}\overline{T}^* \bbO$ consisting of ``fiber infinity.'' More generally, when we write ``$\bullet[g]$,'' this denotes that the Minkowski metric should be replaced by $g$ in the definition of $\bullet$. 
	\item $P=P[g]$, the Klein--Gordon operator under consideration. When we write $P$, we do \emph{not} (necessarily) mean $\square+\mathsf{m}^2$. This is therefore an exception to the previous rule: ``$P$'' is not an abbreviation for $P[g_{\bbM}]$. 
	\item $p[g]$, the de,sc- principal symbol of the Klein--Gordon operator under consideration; $p$ just means the symbol of $\square+\mathsf{m}^2$. 
	\item $\tilde{p}[g] = \varrho_{\mathrm{df}}^2 p[g]\in C^\infty({}^{\mathrm{de,sc}}\overline{T}^* \bbO ) $.
	\item $H_p \in \calV(T^* \bbR^{1,d})$, the Hamiltonian vector field associated to a symbol $p$. 
	\item $\mathsf{H}_p$ the Hamiltonian vector field, but multiplied by some boundary-defining-functions so as to become a b-vector field on ${}^{\mathrm{de,sc}}\overline{T}^* \bbO$. See \cref{eq:Hp_rescaled}.
	\item $\calR,\calC,\calN,\calK,\calA$, various subsets of the characteristic set where $\mathsf{H}_p$ vanishes (in the ordinary sense, meaning as a section of the extendable cotangent bundle on the compactified phase space); see \S\ref{sec:dynamics}. The ultimate sources and sinks of the flow are $\calR$, and these are traditionally denoted `$\calR$.' The sets $\calC,\calK$ lie in corners in the $d=1$ case, and the notation stands for ``corner.'' $\calN$ lies over all of null infinity; the notation stands for ``null.'' Finally, $\calA$, the radial set at high angular momentum; the notation can be taken to stand for ``angular.''
	
	In $\calR^\varsigma_\sigma$, for $\varsigma,\sigma$ signs, the subscript denotes whether the radial set is over future or past timelike infinity, and the superscript denotes which component of the two-sheeted characteristic set the radial set is in. Omitting a sign means taking a union over the two possibilities. 
	Similar notation is used for $\calC,\calN,\calK,\calA$. 
	\item $\frakM,\frakN$, modules of operators characteristic on $\calR$ or related sets. These are defined in \cref{eq:misc_mmm}, \cref{eq:misc_nnn}; a few others are defined later in \S\ref{subsec:test_modules}. For example, adding ``$\varsigma,\sigma$'' as a sub or superscript means that these modules are defined using only $\calR^\varsigma_\sigma$. 
	
	 $\frakM^{\kappa,k}_{\varsigma,\sigma}$ denotes the space of differential operators formed by composing $\kappa$ elements of $\frakM$ and $k$ elements of $\frakN$. 
	
	$H_{\mathrm{de,sc}}^{m,\mathsf{s};\kappa,k}$ denotes the Sobolev space of distributions with $\kappa \in \bbN$ units of $\frakM$-regularity and and $k\in \bbN$ units of $\frakN$-regularity relative to $H_{\mathrm{de,sc}}^{m,\mathsf{s}}$, which means 
	\begin{equation}
		u\in H_{\mathrm{de,sc}}^{m,\mathsf{s};\kappa,k} \iff \frakM^{\kappa,k} u \in H_{\mathrm{de,sc}}^{m,\mathsf{s}}.
		\label{eq:final_sob_def}
	\end{equation} 
	Since $1\in  \frakM^{\kappa,k}$, $H_{\mathrm{de,sc}}^{m,\mathsf{s};\kappa,k}\subseteq H_{\mathrm{de,sc}}^{m,\mathsf{s}}$.
	\item $\operatorname{cl}_X\{\bullet\}$, used throughout to denote the closure of $\{\bullet\}$ in the space $X$. 
\end{itemize}

In a few places in this paper, we will have occasion to refer to boundary fibration structures besides the b-,  sc-, and de,sc- structures. For each, we will use the standard Melrosian notation \cite{Melrose}, always analogous to the conventions explained here for the de,sc- case. None of these other structures are essential for understanding the thrust of this paper. Moreover, no b-$\Psi$DOs are required. 

\section{The octagonal compactification $\bbO$}
\label{sec:geometry}

We now discuss the octagonal compactification of $\bbR^{1,d}$. In \S\ref{subsec:basespace}, we describe the compactification itself. In \S\ref{subsec:phasespace}, we briefly discuss the de,sc-phase space --- here, we will describe some coordinate systems that will be used in \S\ref{sec:dynamics}. In \S\ref{subsec:Psi}, we outline the construction and features of the symbolic $\Psi$DO calculus whose symbols live on that phase space. This includes a discussion of de,sc-based Sobolev spaces and associated wavefront sets, with respect to which the analysis in \S\ref{sec:propagation} will be phrased.

\subsection{The Base Space}
\label{subsec:basespace}

We repeat the definitions of some of the important subsets of $\partial \bbM$. 
Let 
\begin{itemize}
	\item $\scrI^\pm = \partial \bbM\cap \mathrm{cl}_\bbM\{(t,\bfx)\in \bbR^{1,d}:\pm t=r\}$,
	\item $C_\pm = (\partial \bbM \cap \mathrm{cl}_\bbM \{(t,\bfx)\in \bbR^{1,d}:\pm t>r\})\backslash \scrI_\pm$, 
	\item and $i^0 = (\partial \bbM \cap \mathrm{cl}_\bbM \{(t,\bfx)\in \bbR^{1,d}:t^2<r^2\}) \backslash (\scrI^- \cup \scrI^+)$.
\end{itemize}
(Note that $i^0,C_\pm$ are relatively open subsets of $\partial \bbM$.) 
``Null infinity,'' $\scrI$, when referring to a subset of $\bbM$, then refers to $\scrI^-\cup \scrI^+$, and timelike infinity refers to $C_-\cup C_+$. Spacelike infinity is $i^0$.  
The sets $\scrI_-,\scrI_+$ are Poincar\'e invariant in the sense that, if $\Lambda:\bbR^{1,d}\to \bbR^{1,d}$ is an element of the Poincar\'e group (i.e.\ the group of affine maps preserving the metric), then $\Lambda$ extends to a diffeomorphism of $\bbM$, under which $\scrI_\pm$ are closed.

In the introduction, we defined the octagonal compactification $\bbR^{1+d}\hookrightarrow \bbO = \bbO^{1,d}$ by
\begin{equation}
\bbO = [\bbM ; \{\scrI_-,\scrI_+\}; 1/2] = [\bbM ; \scrI; 1/2],
\label{eq:Odef}
\end{equation}
i.e.\ we first perform a polar blowup of the boundary submanifolds $\scrI_-, \scrI_+$ (in whichever order -- the two possible orders give rise to equivalent compactifications) and then modify the smooth structure at the front face(s) of the blowups using the coordinate change $\varrho \mapsto \smash{\varrho^{1/2}}$. 
In other words, if we set
\begin{equation} 
	\bbO_0=[\bbM;\{\scrI_-, \scrI_+\}],
\end{equation} 
then $\bbO=\bbO_0$ at the level of sets, and if $\varrho$ denotes a bdf of the front face (or a front face) of this blowup, then $\smash{\varrho^{1/2}}$ denotes a bdf of the corresponding face of $\bbO$.

\begin{figure}[h]
	\begin{tikzpicture}
		\draw[dashed] (-2,-2.2) rectangle (2,2.2); 
		\begin{scope}[xshift=-8pt]
			\fill[lightgray!20] (-1.5,-1.9) rectangle (1.5,1.9);
			\draw[dashed] (-1.5,0) -- (1.5,0);
			\draw (1.5,-1.9) -- (1.5,1.9);
			\node () at (-1,-1) {$\bbM$};
			\fill[black] (1.5,0) circle (.07) node[right] {$\scrI^+$};
			\node[right] () at (1.5,-1) {$i^0$};
			\node[right] () at (1.5,1) {$C_+$};
			\draw[darkcandyapp, ->] (1.4,.1) -- (1.4,1) node[left] {$\frac{t-r}{t+r}$};
			\draw[darkcandyapp, ->] (1.4,.1) -- (.5,.1) node[above] {$\frac{1}{t+r}$};
		\end{scope}
	\end{tikzpicture}
	\begin{tikzpicture}
		\draw[dashed] (-2,-2.2) rectangle (2,2.2); 
		\begin{scope}[xshift=-3pt]
			\fill[lightgray!20] (-1.5,-1.9) rectangle (1.5,1.9);
			\draw[dashed] (-1.5,0) -- (1.5,0);
			\fill[white] (1.5,.75) arc(90:270:.75) -- cycle;
			\draw (1.5,1.9) -- (1.5,.75) arc(90:270:.75) -- (1.5,-1.9);
			\node () at (-1,-1) {$\bbO_0$};
			\draw[darkcandyapp, ->] (1.4,.85) -- (1.4,1.6) node[left] {$\frac{t-r}{t+r} $};
			\draw[darkcandyapp, ->] (1.4,.85) to[out=180, in=75] (.7,.3) node[above left] {$\frac{1}{t-r} $};
			\draw[darkcandyapp, ->] (1.4,-.85) -- (1.4,-1.6) node[left] {$\frac{r-t}{t+r} $};
			\draw[darkcandyapp, ->] (1.4,-.85) to[out=180, in=-75] (.7,-.3) node[below left] {$\frac{1}{r-t} $};
		\end{scope}
	\end{tikzpicture}
	\begin{tikzpicture}
	\draw[dashed] (-2,-2.2) rectangle (2,2.2); 
	\begin{scope}[xshift=-6pt]
		\fill[lightgray!20] (-1.5,-1.9) rectangle (1.5,1.9);
		\draw[dashed] (-1.5,0) -- (1.5,0);
		\fill[white] (1.5,.75) arc(90:270:.75) -- cycle;
		\draw (1.5,1.9) -- (1.5,.75) arc(90:270:.75) -- (1.5,-1.9);
		\node () at (-1,-1) {$\bbO$};
		\node () at (1.2,0) {nFf};
		\node () at (1.8,1.2) {Ff};
		\node () at (1.8,-1.2) {Sf};
		\draw[darkcandyapp, ->] (1.4,.85) -- (1.4,1.5) node[left] {$\sqrt{\frac{t-r}{t+r}} $};
		\draw[darkcandyapp, ->] (1.4,.85) to[out=180, in=75] (.7,.3) node[above left] {$\frac{1}{t-r} $};
		\draw[darkcandyapp, ->] (1.4,-.85) -- (1.4,-1.5) node[left] {$\sqrt{\frac{r-t}{t+r}} $};
		\draw[darkcandyapp, ->] (1.4,-.85) to[out=180, in=-75] (.7,-.3) node[below left] {$\frac{1}{r-t} $};
	\end{scope}
	\end{tikzpicture}
	\caption{A closeup, near $\scrI^+$, of the polar blowup used to create $\bbO$.}
	\label{fig:blowup_extended}
\end{figure}
 
We will only write the ``$1,d$'' label on $\bbO^{1,d}$ when necessary.
Otherwise, $d \geq 1$ should be assumed to be arbitrary. 

We use $\mathrm{bd} : \bbO \to \bbM$ to denote the blowdown map. 

For convenience, we can take $\bbO^\circ=\bbM^\circ = \bbR^{1+d}$, with $\mathrm{bd}|_{\bbO^\circ} = \mathrm{id}_{\bbR^{1+d}}$, along with 
\begin{equation} 
	\bbO\backslash (\mathrm{nPf}\cup \mathrm{nFf}) = \bbM\backslash (\scrI_-\cup \scrI_+).
\end{equation} 
Before, these equalities were instead canonical diffeomorphisms. It is a matter of convenience to promote them to equalities at the level of sets, 
this just making literal some conventional abuses of notation. 

The octagonal compactification is Poincar\'e invariant:
\begin{proposition}
	If $\Lambda$ is any element of the Poincar\'e group, then $\Lambda$ extends to an automorphism of $\bbO$ under which each boundary hypersurface is a closed set. 
	\label{prop:invariance}
\end{proposition} 
\begin{proof}
	We observe, first of all, that it suffices to prove the proposition for $\bbO_0$ in place of $\bbO$. Indeed, suppose that $\Lambda$ extends to an automorphism $\Lambda_{\mathrm{ext}}:\bbO_0\to \bbO_0$ fixing each boundary hypersurface. We only need to check that this map is actually smooth with respect to the smooth structure of $\bbO$. (Then, after applying the same reasoning to the inverse, we can conclude that $\Lambda_{\mathrm{ext}}\in \operatorname{Aut}(\bbO)$.) 
	To see this, note that, letting $\varrho_{\mathrm{f},0}$ denote any bdf for $\mathrm{f} \in \{\mathrm{Pf},\mathrm{nPf},\mathrm{Sf},\mathrm{nFf},\mathrm{Ff}\}$ \textit{in $\bbO_0$},
	\begin{equation}
		\varrho_{\mathrm{f},0} \circ \Lambda_{\mathrm{ext}} \in \varrho_{\mathrm{f},0} C^\infty(\bbO_0;\bbR^+).
		\label{eq:misc_vf0}
	\end{equation} 
	Since we can take $\varrho_{\mathrm{f},0} = \varrho_{\mathrm{f}}$  unless $\mathrm{f} \in \{\mathrm{nPf},\mathrm{nFf}\}$, $\varrho_{\mathrm{f}} \circ \Lambda_{\mathrm{ext}} \in \varrho_{\mathrm{f}} C^\infty(\bbO_0;\bbR^+) \subseteq C^\infty(\bbO)$ for each $\mathrm{f} \in \{\mathrm{Pf},\mathrm{Sf},\mathrm{Ff}\}$. Taking square roots of \cref{eq:misc_vf0}, 
	\begin{equation}
		\varrho_{\mathrm{f}}\circ \Lambda_{\mathrm{ext}} \in \varrho_{\mathrm{f}} C^\infty(\bbO_0;\bbR^+) \subseteq \varrho_{\mathrm{f}} C^\infty(\bbO) 
	\end{equation}
	for $\mathrm{f} \in \{\mathrm{nPf},\mathrm{nFf}\}$. So, indeed, it suffices to prove the claim made in the proposition for $\bbO_0$ in place of $\bbO$.

	We recall \cite[Lemma 1]{MelroseSC} that any invertible affine transformation $\bbR^{1+d}\to \bbR^{1+d}$ extends to a diffeomorphism of $\bbM$ and that any translation extends to a diffeomorphism under which every point of $\partial \bbM$ is fixed.
	Of course, each of $i^0,C_\pm,\scrI^\pm$ is also closed under the action of any element of the Lorentz group. 
	The claim of the proposition (for $\bbO_0$) then follows from the lemma that any diffeomorphism of any manifold-with-boundary $X$ fixing (but not necessarily acting as the identity on) a submanifold $S\subseteq \partial X$ and each of the components of $\partial X \backslash S$ lifts to a diffeomorphism of the mwc $[X;S]$, with the lift fixing each boundary hypersurface. This is a reformulation of the coordinate invariance of polar blowups, which follows from the fact that $[X;S]$ is, in a neighborhood of the lift of $S$, diffeomorphic to the outward pointed normal bundle ${}^+ N^* S$ of $S$ \cite{MelroseCorners}.
\end{proof}

\begin{example}
	Consider the $d=1$ case, and the Lorentz boost 
	\begin{equation}
		\Lambda = \begin{pmatrix}
			\cosh(\beta) & \sinh(\beta) \\ 
			\sinh(\beta) & \cosh(\beta)
		\end{pmatrix}
	\end{equation}
	with rapidity $\beta\in \bbR $. Then, using the local bdfs for $\mathrm{nFf},\mathrm{Ff}$ given in \cref{eq:misc_019}, one calculates 
	$\varrho_{\mathrm{nFf}}\circ \Lambda = \exp(-\beta) \varrho_{\mathrm{nFf}}$ and $\varrho_{\mathrm{Ff}}\circ \Lambda = \exp(\beta) \varrho_{\mathrm{Ff}}$. So, $\Lambda$ extends to a diffeomorphism of some neighborhood of $\mathrm{nFf}\cap\mathrm{Ff}$, where, in the coordinates $(\varrho_{\mathrm{nFf}},\varrho_{\mathrm{Ff}})$, it is linear. 
\end{example}

In the previous proposition, it is key that $\scrI$ was constructed using a \emph{polar} blowup and not some other quasihomogeneous blowup. Indeed, we will see in the next section (see \Cref{rem:parabolic_non-invariant}) that the conclusion of \Cref{prop:invariance} is false for the parabolic blowup $[\bbM;\scrI]_{\mathrm{par}}$, which is related to hyperbolic coordinates.

Next, we define the algebra $\operatorname{Diff}_{\mathrm{de,sc}}(\bbO)$ of de,sc-differential operators.
First, we define $\calV_{\mathrm{de,sc}}$ to be the $C^\infty(\bbO)$-module of smooth vector fields on $\smash{\bbR^{1,d}}$ which, away from null infinity, are just sc-vector fields (i.e.\ in the $C^\infty(\bbM)$-module generated by constant-coefficient vector fields), and which, near a corner $\mathrm{nf}\cap \mathrm{Of}$ of $\bbO$, are of the form 
\begin{equation}
	\psi \varrho_{\mathrm{nf}}^2 \varrho_{\mathrm{Of}} \partial_{\varrho_{\mathrm{nf}} } , \psi \varrho_{\mathrm{nf}} \varrho_{\mathrm{Of}}^2 \partial_{\varrho_{\mathrm{Of}} } ,\psi \varrho_{\mathrm{nf}}^2 \varrho_{\mathrm{Of}} V: V\in \calV(\bbS^{d-1}) ,
	\label{eq:Vdef}
\end{equation}
where $\psi\in C^\infty(\bbO)$ contains no other corners of $\bbO$ in its support. Here, $\mathrm{nf}\in \{\mathrm{nPf},\mathrm{nFf}\}$, $\mathrm{Of} \in \{\mathrm{Pf},\mathrm{Sf},\mathrm{Ff} \}$, depending on which corner is under examination.  

We will have to check that this definition of $\calV_{\mathrm{de,sc}}$ is consistent with \cref{eq:Vdef_global}.

Now define $\operatorname{Diff}_{\mathrm{de,sc}}^{m,\mathsf{0}}(\bbO)$ to be the $C^\infty(\bbO)$-submodule of $\operatorname{Diff}^m(\bbR^{1,d})$ spanned by products of de,sc-vector fields, and let
\begin{align}
\begin{split} 
\operatorname{Diff}_{\mathrm{de,sc}}^{m,\mathsf{s}} &= \rho^{-\mathsf{s}} \operatorname{Diff}_{\mathrm{de,sc}}^{m,\mathsf{0}}, \quad \mathsf{s} = (s_{\mathrm{Pf}},s_{\mathrm{nPf}},s_{\mathrm{Sf}},s_{\mathrm{nFf}},s_{\mathrm{Ff}} ) \in \bbR^5, \\ \rho^{-\mathsf{s}} &= \rho_{\mathrm{Pf}}^{-s_{\mathrm{Pf}} } \rho_{\mathrm{nPf}}^{-s_{\mathrm{nPf}} }\rho_{\mathrm{Sf}}^{-s_{\mathrm{Sf}} }\rho_{\mathrm{nFf}}^{-s_{\mathrm{nFf}} }\rho_{\mathrm{Ff}}^{-s_{\mathrm{Ff}} }.
\end{split} 
\end{align}

For the most part, we work away from $\mathrm{cl}_\bbM \{r=0\} = \mathrm{cl}_\bbO \{r=0\}$. This allows us to work with (spatial) polar coordinates. 
Let $\smash{\dot{\bbO}}=\bbO\backslash \mathrm{cl}_\bbO\{r=0\}$.
This is canonically diffeomorphic to $\hat{\bbO}\times \bbS^{d-1}_{\bfx/r}$, where 
\begin{equation}
	\hat{\bbO} = \bbO^{1,1} \backslash \mathrm{cl}_{\bbO^{1,1}}\{(t,r) \in \bbR^{1,1}_{t,r} : r\leq 0\}. 
	\label{eq:hatO}
\end{equation}
This mwc is noncompact (we do not add a boundary face corresponding to $r=0$). The interior is equal to $\{(t,r)\in \bbR^{1,1} : r> 0\}$. Then, we have a diffeomorphism 
\begin{equation}
	 \bbR\times \bbR^+ \times \bbS^{d-1} \to \dot{\bbO}^\circ = \bbR^{1,d}\backslash \{r=0\}, 
	\label{eq:plr}
\end{equation}
$(t,r,\theta)\mapsto (t,r\theta)$, which extends to a diffeomorphism $\hat{\bbO}\times \bbS^{d-1} \to \dot{\bbO}$. We will abuse notation below and conflate $\hat{\bbO}\times \bbS^{d-1}$ with $\dot{\bbO}$.

We will make use of the following coordinate charts for $\hat{\bbO}$. In the following, if $S\subset \hat{\bbO}$, then $S^\circ$ denotes the usual topological notion of interior, 
\begin{equation} 
	S^\circ=S\backslash \mathrm{cl}_{\hat{\bbO}} (\hat{\bbO} \backslash S).
\end{equation} 
In particular, $S^\circ$ is allowed to have points in $\partial \hat{\bbO}$.
\begin{itemize}
	\item (The timelike corner of null infinity.) For each $T>0$, let 
	\begin{equation} 
		\hat{\Omega}_{\mathrm{nfTf},\pm,T} = (\mathrm{cl}_{\hat{\bbO}} \{ |t|+T>r , \pm t>0\})^\circ,
	\end{equation} 
	and let $\varrho_{\mathrm{nf}} : \hat{\Omega}_{\mathrm{nfTf},\pm,T} \to [0,\infty)$ and $\varrho_{\mathrm{Tf}} : \hat{\Omega}_{\mathrm{nfTf},\pm,T} \to [0,\infty)$ be defined by $\varrho_{\mathrm{nf}} = (|t|-r+T)^{1/2} / (|t|+r+T)^{1/2}$ and $\varrho_{\mathrm{Tf}} = (|t|-r+T)^{-1}$. 
	Then, $(\varrho_{\mathrm{nf}},\varrho_{\mathrm{Tf}}) : \hat{\Omega}_{\mathrm{nfTf},\pm,T} \to [0,\infty)^2$ is a coordinate chart on $\hat{\bbO}$. Solving for $r,t$ in terms of $\varrho_{\mathrm{nf}},\varrho_{\mathrm{Tf}}$, 
	\begin{align}
		\begin{split} 
		t &= \pm ( (2\varrho_{\mathrm{nf}}^{2}\varrho_{\mathrm{Tf}})^{-1}(1+\varrho_{\mathrm{nf}}^{2})-T ), \\
		r &=  (2\varrho_{\mathrm{nf}}^{2}\varrho_{\mathrm{Tf}})^{-1} (1-\varrho_{\mathrm{nf}}^{2} ).
		\end{split}
		\label{eq:tr_inv_def}
	\end{align} 
	For later use, we record the partial derivatives 
	\begin{equation}
		\frac{\partial \varrho_{\mathrm{Tf}}}{\partial t} = \mp \varrho_{\mathrm{Tf}}^2, \quad 
		\frac{\partial \varrho_{\mathrm{nf}}}{\partial t} = \pm \frac{1}{2}  (1-\varrho_{\mathrm{nf}}^2)\varrho_{\mathrm{nf}}\varrho_{\mathrm{Tf}}, 
		\label{eq:misc_83a}
	\end{equation}
	\begin{equation}
		\frac{\partial \varrho_{\mathrm{Tf}}}{\partial r} = \varrho_{\mathrm{Tf}}^2,\quad \frac{\partial \varrho_{\mathrm{nf}}}{\partial r} = - \frac{1}{2} (1+\varrho_{\mathrm{nf}}^2)\varrho_{\mathrm{nf}} \varrho_{\mathrm{Tf}}.
		\label{eq:misc_83b}
	\end{equation}
	\item  (The spacelike corner of null infinity.) For each $R>0$, let 
	\begin{equation} 
		\hat{\Omega}_{\mathrm{nfSf},\pm,R} = (\mathrm{cl}_\bbO \{ |t|<r+R , \pm t>0\})^\circ
	\end{equation}	
	and let $\varrho_{\mathrm{nf}} : \hat{\Omega}_{\mathrm{nfSf},\pm,R} \to [0,\infty)$ and $\varrho_{\mathrm{Sf}} : \hat{\Omega}_{\mathrm{nfSf},\pm,R} \to [0,\infty)$ be defined by $\varrho_{\mathrm{nf}} = (r-|t|+R)^{1/2} / (r+|t|+R)^{1/2}$ and $\varrho_{\mathrm{Sf}} = (r-|t|+R)^{-1}$. 
	Then, $(\varrho_{\mathrm{nf}},\varrho_{\mathrm{Sf}}) : \hat{\Omega}_{\mathrm{nfSf},\pm,R} \to [0,\infty)^2$ is a coordinate chart on $\hat{\bbO}$. Solving for $r,t$ in terms of $\varrho_{\mathrm{nf}},\varrho_{\mathrm{Sf}}$, we have
	\begin{align}
		r&=(2\varrho_{\mathrm{nf}}^{2}\varrho_{\mathrm{Sf}})^{-1}(1+\varrho_{\mathrm{nf}}^{2})-R \\
		t&=\pm (2\varrho_{\mathrm{nf}}^{2}\varrho_{\mathrm{Sf}})^{-1}(1-\varrho_{\mathrm{nf}}^{2}).
	\end{align}
	The partial derivatives are 
	\begin{equation}
		\frac{\partial \varrho_{\mathrm{Sf}}}{\partial t} = \pm \varrho_{\mathrm{Sf}}^2, \quad \frac{\partial \varrho_{\mathrm{nf}}}{\partial t} = \mp \frac{1}{2} (1+\varrho_{\mathrm{nf}}^2)  \varrho_{\mathrm{nf}}\varrho_{\mathrm{Sf}} 
		\label{eq:misc_83c}
	\end{equation}
	\begin{equation}
		\frac{\partial \varrho_{\mathrm{Sf}}}{\partial r} = - \varrho_{\mathrm{Sf}}^2, \quad 
		\frac{\partial \varrho_{\mathrm{nf}}}{\partial r} = \frac{1}{2} (1-\varrho_{\mathrm{nf}}^2) \varrho_{\mathrm{nf}} \varrho_{\mathrm{Sf}}.
		\label{eq:misc_83d}
	\end{equation}
\end{itemize}

Let $\Omega_\bullet =\hat{\Omega}_\bullet\times \bbS^{d-1}$.  

\begin{proposition}
	On $\hat{\bbO}$, the d'Alembertian $\square$ is given by the following:
	\begin{itemize}
		\item in $\Omega_{\mathrm{nfTf},\pm,T}$, 
		\begin{equation}
			\square =- \varrho_{\mathrm{nf}}^4 \varrho_{\mathrm{Tf}}^2 \frac{\partial^2}{\partial \varrho_{\mathrm{nf}}^2} + 2 \varrho_{\mathrm{nf}}^3 \varrho_{\mathrm{Tf}}^3 \frac{\partial^2}{\partial \varrho_{\mathrm{nf}} \partial \varrho_{\mathrm{Tf}} } - \varrho_{\mathrm{nf}}^3 \varrho_{\mathrm{Tf}}^2 \frac{\partial}{\partial \varrho_{\mathrm{nf}}},
			\label{eq:misc_090}
		\end{equation}
		and 
		\item in $\Omega_{\mathrm{nfSf},\pm,R }$, 
		\begin{equation}
			\square =+ \varrho_{\mathrm{nf}}^4 \varrho_{\mathrm{Sf}}^2 \frac{\partial^2}{\partial \varrho_{\mathrm{nf}}^2} - 2 \varrho_{\mathrm{nf}}^3 \varrho_{\mathrm{Sf}}^3 \frac{\partial^2}{\partial \varrho_{\mathrm{nf}} \partial \varrho_{\mathrm{Sf}} } +  \varrho_{\mathrm{nf}}^3 \varrho_{\mathrm{Sf}}^2 \frac{\partial}{\partial \varrho_{\mathrm{nf}}},
			\label{eq:misc_091}
		\end{equation} 
	\end{itemize}
	where we are identifying the coordinate patches $\hat{\Omega}_{\mathrm{nfTf},\pm,T},\hat{\Omega}_{\mathrm{nfSf},\pm,R}$ and their images in $\bbR^2$ under the coordinate charts above. 
	\label{prop:dAlembertian_calculation}
\end{proposition}
\begin{proof}
	The first formula is the result of using \cref{eq:misc_83a}, \cref{eq:misc_83b} to replace
	\begin{align}
		\begin{split} 
		\frac{\partial}{\partial t} = \frac{\partial \varrho_{\mathrm{nf}}}{\partial t} \frac{\partial}{\partial \varrho_{\mathrm{nf}}} + \frac{\partial \varrho_{\mathrm{Tf}}}{\partial t} \frac{\partial}{\partial \varrho_{\mathrm{Tf}}} 
		&= \pm \frac{1}{2} (1-\varrho_{\mathrm{nf}}^2)\varrho_{\mathrm{nf}}\varrho_{\mathrm{Tf}} \frac{\partial}{\partial \varrho_{\mathrm{nf}}} \mp \varrho_{\mathrm{Tf}}^2 \frac{\partial}{\partial \varrho_{\mathrm{Tf}}} \\ 
		\frac{\partial}{\partial r} = \frac{\partial \varrho_{\mathrm{nf}}}{\partial r} \frac{\partial}{\partial \varrho_{\mathrm{nf}}} + \frac{\partial \varrho_{\mathrm{Tf}}}{\partial r} \frac{\partial}{\partial \varrho_{\mathrm{Tf}}} 
		 &= -\frac{1}{2} (1+\varrho_{\mathrm{nf}}^2)\varrho_{\mathrm{nf}}\varrho_{\mathrm{Tf}} \frac{\partial}{\partial \varrho_{\mathrm{nf}}} + \varrho_{\mathrm{Tf}}^2 \frac{\partial}{\partial \varrho_{\mathrm{Tf}}}
		\end{split} 
	\label{eq:misc_062}
	\end{align}
	in $\square = \partial_t^2 - \partial_r^2$. 
	The second formula is the result of using \cref{eq:misc_83c}, \cref{eq:misc_83d} to replace
	\begin{align}
		\begin{split}
			\frac{\partial}{\partial t} = \frac{\partial \varrho_{\mathrm{nf}}}{\partial t} \frac{\partial}{\partial \varrho_{\mathrm{nf}}} + \frac{\partial \varrho_{\mathrm{Sf}}}{\partial t} \frac{\partial}{\partial \varrho_{\mathrm{Sf}}} 
			&= \mp \frac{1}{2} (1+\varrho_{\mathrm{nf}}^2)\varrho_{\mathrm{nf}}\varrho_{\mathrm{Sf}} \frac{\partial}{\partial \varrho_{\mathrm{nf}}} \pm \varrho_{\mathrm{Sf}}^2 \frac{\partial}{\partial \varrho_{\mathrm{Sf}}} \\ 
			\frac{\partial}{\partial r} = \frac{\partial \varrho_{\mathrm{nf}}}{\partial r} \frac{\partial}{\partial \varrho_{\mathrm{nf}}} + \frac{\partial \varrho_{\mathrm{Sf}}}{\partial r} \frac{\partial}{\partial \varrho_{\mathrm{Sf}}} 
			&= +\frac{1}{2} (1-\varrho_{\mathrm{nf}}^2)\varrho_{\mathrm{nf}}\varrho_{\mathrm{Sf}} \frac{\partial}{\partial \varrho_{\mathrm{nf}}} - \varrho_{\mathrm{Sf}}^2 \frac{\partial}{\partial \varrho_{\mathrm{Sf}}}.
		\end{split}
		\label{eq:misc_064}
	\end{align}
\end{proof}
So, in the $d=1$ case, $\square \in \operatorname{Diff}_{\mathrm{de,sc}}^{2,\mathsf{0}}(\bbO)$. 
The first derivative terms in \cref{eq:misc_090}, \cref{eq:misc_091} are subprincipal, being subleading by one order in every possible sense.

\begin{proposition}
	For any $m,s\in \bbR$, $\operatorname{Diff}^{m,s}_{\mathrm{sc}}(\bbM) \subseteq \operatorname{Diff}_{\mathrm{de,sc}}^{m,(s,2s+m,s,2s+m,s)}(\bbO)$.
	\label{prop:sc-desc_conv}
\end{proposition}
\begin{proof}
	From \cref{eq:misc_062}, \cref{eq:misc_064}.
\end{proof}

\begin{proposition}
	$\square \in \operatorname{Diff}^{2,\mathsf{0}}_{\mathrm{de,sc}}(\bbO)$.
\end{proposition}
\begin{proof}
	Since $\square \in \operatorname{Diff}^{2,0}_{\mathrm{sc}}(\bbM)$, it suffices to restrict attention to a neighborhood of null infinity, where we can use spherical coordinates in the spatial variables. Then, the computation above shows that the $\partial_t^2-\partial_r^2$ terms are in $ \operatorname{Diff}^{2,\mathsf{0}}_{\mathrm{de,sc}}(\bbO)$. The angular derivatives have the form $r^{-1}\partial_{\theta_j}$, where $\theta_1,\dots,\theta_{d-1}$ is a coordinate chart on $\bbS^{d-1}$. Since \begin{equation} 
		r^{-1} \in \varrho_{\mathrm{Pf}}\varrho_{\mathrm{nPf}}^2\varrho_{\mathrm{Sf}}\varrho_{\mathrm{nFf}}^2\varrho_{\mathrm{Ff}}C^\infty(\bbO) 
		\label{eq:misc_094}
	\end{equation} 
	near null infinity, this yields $r^{-1}\partial_{\theta_j} \in \calV_{\mathrm{de,sc}}$. Their contribution to $\square$ is therefore in $\operatorname{Diff}^{2,\mathsf{0}}_{\mathrm{de,sc}}(\bbO)$.
	Finally, the spatial Laplacian $\triangle$ in $\square$ has a term $(d-1)r^{-1} \partial_r \in \operatorname{Diff}_{\mathrm{sc}}^{1,-1}$. By \Cref{prop:sc-desc_conv}, this lies in $\operatorname{Diff}_{\mathrm{de,sc}}^{1,-\mathsf{1}}$.
\end{proof}

Recall that if $M$ is a compact mwc, $\calF(M)$ is the set of its faces, and $\varrho_{\mathrm{f}}$ denotes a bdf of $\mathrm{f}\in \calF(M)$, then we have an LCTVS (Hausdorff locally convex topological vector space)
\begin{equation}
	\calS(M) = \bigcap_{s\in \bbR} \Big[ \prod_{\mathrm{f}\in \calF(M)} \varrho_{\mathrm{f}}^s \Big] C^\infty(M) =  \bigcap_{s\in \bbR} \bigcap_{k\in \bbN}\Big[ \prod_{\mathrm{f}\in \calF(M)} \varrho_{\mathrm{f}}^s \Big] C^k(M)
\end{equation}
of ``Schwartz'' functions. When $M= \bbM$, then this is just the usual set of Schwartz functions on $\bbR^{1,d}=\bbR^{1+d}$. 
Identifying smooth functions on mwcs with their restrictions to interiors, $\calS(\bbR^{1,d})=\calS(\bbO)$. Indeed:
\begin{proof}
\begin{itemize}
	\item $\calS(\bbR^{1,d})\subseteq \calS(\bbO)$: The blowdown map $\bbO\to \bbM$ is smooth, so any Schwartz function on $\bbR^{1,d}$ extends to a smooth function on $\bbO$. If $\varrho \in C^\infty(\bbM)$ is a bdf of $\partial \bbM$ in $\bbM$, then $\varrho \in \varrho_{\mathrm{Pf}}\varrho_{\mathrm{nPf}}^2\varrho_{\mathrm{Sf}}\varrho_{\mathrm{nFf}}^2 \varrho_{\mathrm{Ff}} C^\infty(\bbO)$. Consequently, 
	\begin{equation}
		\bigcap_{s\in \bbR} \varrho^s C^\infty(\bbM) \subseteq \calS(\bbO). 
	\end{equation}
	\item $\calS(\bbO)\subseteq \calS(\bbR^{1,d})$:
	Conversely, if $u\in \calS(\bbO)$, then $\operatorname{Diff}_{\mathrm{de,sc}}^{m,\mathsf{s}} u \in L^2(\bbR^{1,d})$ for all $m\in \bbR,\mathsf{s}\in \bbR^5$. By \Cref{prop:sc-desc_conv}, this implies that $\operatorname{Diff}_{\mathrm{sc}}^{m,s}(\bbM)u\in L^2$ for all $m,s\in \bbR$. By the Schwartz representation theorem, this implies that $u\in \calS(\bbR^{1,d})$.  
\end{itemize}
\end{proof}
Consequently, going forwards, we will simply write $\calS$ to refer to the space of Schwartz functions, and we do not need to specify whether we mean on $\bbO$ or $\bbM$. 
Moreover, this holds at the level of TVSs, as the same argument shows. 
Consequently, a tempered distribution on $\bbO$, meaning an element of $\calS'(\bbO)= \calS(\bbO)^*$, is just a tempered distribution on $\bbR^{1,d}$, and vice versa, and so we can unambiguously write $\calS'$ to refer to the space of tempered distributions.

The fact that the two definitions $\calV_{\mathrm{de,sc}}$ given so far, \cref{eq:Vdef_global} and \cref{eq:Vdef}, are equivalent follows from combining the formulas in \cref{eq:misc_062}, \cref{eq:misc_064}:
\begin{proof}
	\begin{itemize}
		\item First, we check that the vector fields in \cref{eq:Vdef_global} are de,sc-vector fields according to \cref{eq:Vdef}. Indeed, \Cref{prop:sc-desc_conv} gives
		\begin{equation}
			\varrho_{\mathrm{nPf}}\varrho_{\mathrm{nFf}}\partial_t, \varrho_{\mathrm{nPf}}\varrho_{\mathrm{nFf}}\partial_{x_j}\in \operatorname{Diff}_{\mathrm{de,sc}}^{1,\mathsf{0}} = \calV_{\mathrm{de,sc}}.
		\end{equation}
		 On the other hand, \cref{eq:misc_062} yields
		 \begin{equation}
		 	\partial_{|t|} + \partial_r = - \varrho_{\mathrm{nf}}^3 \varrho_{\mathrm{Tf}} \frac{\partial}{\partial \varrho_{\mathrm{nf}}} \in \operatorname{Diff}_{\mathrm{de,sc}}^{1,(0,-1,0,-1,0)} = \varrho_{\mathrm{nPf}}\varrho_{\mathrm{nFf}} \calV_{\mathrm{de,sc}}
		 \end{equation}
		 near $\mathrm{nFf}\cap\mathrm{Ff}$, and likewise over the other corners of $\bbO$. Finally, it follows from \cref{eq:misc_094} that $r^{-1}\partial_{\theta_k} \in \calV_{\mathrm{de,sc}}$ where this vector field is defined.
		 \item Because we have included in \cref{eq:Vdef_global} the angular derivatives by hand, it suffices to check the $d=1$ case. Using the formulas above, near $\mathrm{nFf}\cap\mathrm{Ff}$ we can write
		 $\partial_t = \varrho_{\mathrm{nf}}^{-1} V+\varrho_{\mathrm{nf}} W$ and $\partial_r = -\varrho_{\mathrm{nf}}^{-1} V+\varrho_{\mathrm{nf}} W$ for $V,W \in \calV_{\mathrm{de,sc}}$ given by 
		 \begin{align}
		 	\begin{split} 
		 		V &= 2^{-1} \varrho_{\mathrm{nf}}^2\varrho_{\mathrm{Tf}} \partial_{\varrho_{\mathrm{nf}}} - \varrho_{\mathrm{nf}}\varrho_{\mathrm{Tf}}^2 \partial_{\varrho_{\mathrm{Tf}}} \\ 
		 		W &= - \varrho_{\mathrm{nf}}^2 \varrho_{\mathrm{Tf}}\partial_{\varrho_{\mathrm{nf}}}.
		 	\end{split} 
		 	\label{eq:misc_99}
		 \end{align}
		 Notice that $V,W$ are in the right-hand side of \cref{eq:Vdef_global}. Moreover, it follows from \cref{eq:misc_99} that $V,W$ span $\calV_{\mathrm{de,sc}}$ over $C^\infty(\bbO)$, locally. 
	\end{itemize}
\end{proof}
\subsection{The de,sc-Cotangent Bundle}
\label{subsec:phasespace}

We now define the de,sc-tangent bundle $\pi_{\mathrm{de,sc}}:{}^{\mathrm{de,sc}} T \bbO \to \bbO$. As an indexed set, this is ${}^{\mathrm{de,sc}} T \bbO=\{{}^{\mathrm{de,sc}} T_p \bbO \}_{p\in \bbO}$, whose elements are the vector spaces 
\begin{equation}
	{}^{\mathrm{de,sc}} T_p \bbO = \calV_{\mathrm{de,sc}}(\bbO;\bbR) / \calI_p \calV_{\mathrm{de,sc}}(\bbO;\bbR), 
\end{equation}
where $\smash{\calI_p} \subset C^\infty(\bbO;\bbR)$ is the ideal of smooth real-valued functions on $\bbO$ vanishing at $p$. 
Naturally, we can regard $\pi_{\mathrm{de,sc}}:{}^{\mathrm{de,sc}} T \bbO \to \bbO$
as a real vector bundle over $\bbO$. The entire space ${}^{\mathrm{de,sc}} T \bbO$ is a mwc diffeomorphic to $\bbO\times \bbR^{1+d}$.

Then, the de,sc-cotangent bundle 
\begin{equation} 
	{}^{\mathrm{de,sc}}\pi:{}^{\mathrm{de,sc}}T^* \bbO\to \bbO 
\end{equation}
is just defined to be the dual vector bundle to $\pi_{\mathrm{de,sc}}:{}^{\mathrm{de,sc}}T \bbO \to \bbO$. 
For convenience, we can arrange that ${}^{\mathrm{de,sc}}T^*_{\bbR^{1,d}}\bbO = T^* \bbR^{1,d}$ at the level of indexed sets (in which case this is identification is a bundle isomorphism). Typical smooth sections of ${}^{\mathrm{de,sc}}T^* \bbO$ near $\mathrm{nFf}\cap \mathrm{Ff}$ are
\begin{equation}
	\frac{\dd \varrho_{\mathrm{nf}}}{\varrho_{\mathrm{nf}}^2\varrho_{\mathrm{Tf}} }, \quad 
	\frac{\dd \varrho_{\mathrm{Tf}}}{\varrho_{\mathrm{nf}}\varrho_{\mathrm{Tf}}^2 },\quad 
	\frac{\dd \theta_k }{\varrho_{\mathrm{nf}}^2\varrho_{\mathrm{Tf}} }.
\end{equation}
In fact, near $\mathrm{nFf}\cap \mathrm{Ff}$, every smooth section of ${}^{\mathrm{de,sc}}T^* \bbO$ is a linear combination of these de,sc- 1-forms over $C^\infty(\bbO)$. 
The other corners are similar.

Away from null infinity, ${}^{\mathrm{de,sc}}T^* \bbO$ is canonically diffeomorphic to ${}^{\mathrm{sc}}T^* \bbM$.

\begin{remark}[Precise definition of angular frequency]
	Let ${}^{\mathrm{sc}}\pi:{}^{\mathrm{sc}}T^* \bbM\to \bbM$ denote the sc-cotangent bundle --- see \cite{MelroseSC,Melrose, VasyGrenoble}. It can be shown that there exists a diffeomorphism 
	\begin{equation} 
		{}^{\mathrm{sc}}\mathrm{plr} : \hat{\bbM}=\bbM^{1,1}\backslash \mathrm{cl}_{\bbM^{1,1}}\{(t,x)\in \bbR^{1,1}:x\leq 0\} \times \bbR\times \bbR  \times T^* \bbS^{d-1} \to {}^{\mathrm{sc}}T^* \bbM \backslash {}^{\mathrm{sc}} \pi^{-1} \mathrm{cl}_\bbM \{r=0\}
		\label{eq:scplr}
	\end{equation} 
	such that, for all $t,\tau,\Xi \in \bbR$, $r\in \bbR^+$, and $\eta_{\mathrm{sc}} \in T^* \bbS^{d-1}$, 
	\begin{equation}
		{}^{\mathrm{sc}}\mathrm{plr}((t,r),\tau,\Xi,\eta_{\mathrm{sc}}) = \tau \dd t +  \Xi  \dd r + r\operatorname{eulr}^*(\eta_{\mathrm{sc}}), 
		\label{eq:scplrdef}
	\end{equation}
	where $\operatorname{eulr} : \bbR^{1,d}_{t,\bfx} \backslash \{r=0\} \to \bbS^{d-1}_\theta$ is the map $(t,\bfx)\mapsto \bfx/r$.  For the comparison with the de,sc-cotangent bundle, it is slightly better to work with $\mu = \tau+\Xi$ and $\nu = \tau - \Xi$, in terms of which $\tau \dd t + \Xi \dd r = \mu (\mathrm{d} t + \dd r)+\nu(\mathrm{d} t-\dd r)$. 
	
	Similarly, it can be shown that there exists a diffeomorphism 
	\begin{equation} 
		{}^{\mathrm{de,sc}}\mathrm{plr}:\hat{\bbO} \times \bbR\times \bbR \times T^* \bbS^{d-1} \to {}^{\mathrm{de,sc}}T^* \bbO \backslash {}^{\mathrm{de,sc}}\pi^{-1} \operatorname{cl}_{\bbO}\{r=0\}
	\end{equation} 
	such that, for all $t,\mu,\nu \in \bbR$, $r\in \bbR^+$, and $\eta_{\mathrm{sc}} \in T^* \bbS^{d-1}$, 
	\begin{equation}
		{}^{\mathrm{de,sc}}\mathrm{plr}((t,r),\mu,\nu,\eta_{\mathrm{sc}} ) = \frac{\varrho_{\mathrm{nFf}}}{\varrho_{\mathrm{nPf}}} \mu( \mathrm{d} t + \dd r)   + \frac{\varrho_{\mathrm{nPf}}}{\varrho_{\mathrm{nFf}}} \nu(\mathrm{d} t- \dd r)  + r  \operatorname{eulr}^*(\eta_{\mathrm{sc}}). 
	\end{equation}
	(As the subscript indicates, the coordinate $\eta_{\mathrm{sc}}$ should be thought of as keeping track of the spatial angle and of the angular component of sc-frequency. We will drop the subscript `sc' in later sections.) 
	With this diffeomorphism in mind, we set 
	\begin{equation}
		{}^{\mathrm{de,sc}} T^* \hat{\bbO} = \hat{\bbO} \times \bbR_\mu\times \bbR_\nu. 
	\end{equation}
	So, away from ${}^{\mathrm{de,sc}}\pi^{-1}(\mathrm{cl}_\bbO\{r=0\})$, ${}^{\mathrm{de,sc}} T^* \hat{\bbO}\times (T^* \bbS^{d-1})_{\eta_{\mathrm{sc}}} \cong {}^{\mathrm{de,sc}}T^* \bbO$ ``canonically.'' 
	Thus, $(t,r)\in \hat{\bbO}, \mu,\nu\in \bbR, \eta_{\mathrm{sc}} \in T^* \bbS^{d-1}$ serve as coordinates on ${}^{\mathrm{de,sc}}T^* \bbO$ away from ${}^{\mathrm{de,sc}}\pi^{-1}(\mathrm{cl}_\bbO\{r=0\})$, at which spherical coordinates break down.
\end{remark}

In \S\ref{sec:dynamics}, we will use coordinates $\xi,\zeta$, which, over $\hat{\Omega}_{\mathrm{nfTf},\pm,T}$, are associated to points in ${}^{\mathrm{de,sc}}T^* \hat{\bbO}$ via 
\begin{equation}
	(\varrho_{\mathrm{nf}},\varrho_{\mathrm{Tf}}, \xi,\zeta) \mapsto \frac{\xi \dd \varrho_{\mathrm{nf}}}{\varrho_{\mathrm{nf}}^2 \varrho_{\mathrm{Tf}} } + \frac{\zeta \dd \varrho_{\mathrm{Tf}}}{\varrho_{\mathrm{nf}}\varrho_{\mathrm{Tf}}^2 }.
	\label{eq:nfTf_coord_full}
\end{equation}
Over $\hat{\Omega}_{\mathrm{nfSf},\pm,R}$, we use $\xi,\zeta$ to denote the coordinates 
\begin{equation}
	(\varrho_{\mathrm{nf}},\varrho_{\mathrm{Sf}}, \xi,\zeta) \mapsto \frac{\xi \dd \varrho_{\mathrm{nf}}}{\varrho_{\mathrm{nf}}^2 \varrho_{\mathrm{Sf}} } + \frac{\zeta \dd \varrho_{\mathrm{Sf}}}{\varrho_{\mathrm{nf}}\varrho_{\mathrm{Sf}}^2 }
	\label{eq:nfSf_coord_full}
\end{equation} 
on ${}^{\mathrm{de,sc}}T^* \hat{\bbO}$. Because $\varrho_{\mathrm{nf}}$ means something different near the spacelike corner vs.\ the timelike corner of null infinity, the same applies to $\xi,\zeta$. 

Define $\smash{{}^{\mathrm{de,sc}}\overline{T}^*} \bbO$ to be the ball bundle that results from radially compactifying the fibers of ${}^{\mathrm{de,sc}}T^* \bbO$. Going forwards, let 
\begin{equation}
	{}^{\mathrm{de,sc}} \pi : {}^{\mathrm{de,sc}} \overline{T}^* \bbO \to \bbO 
\end{equation}
denote the extension of ${}^{\mathrm{de,sc}}\pi$ to the radial compactified bundle. So, e.g.\ ${}^{\mathrm{de,sc}} \pi^{-1}(\mathrm{nf})$ will denote the union of all compactified fibers over null infinity (or over one component of null infinity, depending on context).
Let $\varrho_{\mathrm{df}}\in \smash{C^\infty({}^{\mathrm{de,sc}}\overline{T}^* \bbO;\bbR^+)}$ denote a bdf for the new face at fiber infinity, which we label df. (We will also consider the bdfs $\varrho_{\mathrm{f}}$ of the faces of $\bbO$ as bdfs of their lifts to the de,sc-cotangent bundle and the radial compactification thereof. That is, we conflate $\varrho_{\mathrm{f}}$ and $\varrho_{\mathrm{f}} \circ \smash{{}^{\mathrm{de,sc}} \pi}$.)

The diffeomorphisms discussed above extend to radial compactifications. 
They (and their extensions) will be left implicit below.

Given $m\in \bbR$ and $\mathsf{s}=(s_{\mathrm{Pf}},s_{\mathrm{nPf}},s_{\mathrm{Sf}},s_{\mathrm{nFf}},s_{\mathrm{Ff}} )\in \bbR^5$, Let 
\begin{equation} 
	S_{\mathrm{de,sc}}^{m,\mathsf{s}}=S_{\mathrm{de,sc}}^{m,\mathsf{s}}({}^{\mathrm{de,sc}}\overline{T}^* \bbO) = \varrho_{\mathrm{df}}^{-m}\varrho^{-\mathsf{s}} S_{\mathrm{de,sc}}^{0,\mathsf{0}},
\end{equation} 
where $S_{\mathrm{de,sc}}^{0,\mathsf{0}}$ is the Fr\'echet space of conormal functions on ${}^{\mathrm{de,sc}}\overline{T}^*\bbO$. These are ``de,sc-symbols,'' and, as usual, the space 
\begin{equation}
	S_{\mathrm{de,sc}} = \bigcup_{m\in \bbR,\mathsf{s}\in \bbR^5}S_{\mathrm{de,sc}}^{m,\mathsf{s}}
\end{equation} 
has the structure of a multigraded Fr\'echet algebra.  

If $L\in \operatorname{Diff}_{\mathrm{de,sc}}^{m,\mathsf{0}}$, define $\sigma_{\mathrm{de,sc}}^{m,\mathsf{0}}(L)$ to be the equivalence class in $S_{\mathrm{de,sc}}^{m,\mathsf{0}} / S_{\mathrm{de,sc}}^{m-1,-\mathsf{1}}$ constructed as follows:
\begin{itemize}
	\item away from null infinity, $\sigma_{\mathrm{de,sc}}^{m,\mathsf{0}}(L)$ is the usual sc-principal symbol; 
	\item near the corners of $\bbO$, this is the equivalence class of functions that results from replacing $\partial_{\varrho_{\mathrm{nf}}}$ by $i\xi$, $\partial_{\varrho_{\mathrm{Of}}} \in \{\partial_{\varrho_{\mathrm{Tf}}},\partial_{\varrho_{\mathrm{Sf}}}\}$ by $i\zeta$, and $r^{-1} Q$ for $Q$ a vector field on $\bbS^{d-1}_\theta$ by its principal symbol.
\end{itemize}
This is a well-defined element of $S_{\mathrm{de,sc}}^{m,\mathsf{0}} / S_{\mathrm{de,sc}}^{m-1,-\mathsf{1}}$. The definition of $\sigma_{\mathrm{de,sc}}^{m,\mathsf{s}}$ for general $\mathsf{s}\in \bbR^5$ is similar.

\begin{proposition}
	The function 
	\begin{equation} 
		p_0: \tau \dd t + \sum_{j=1}^d \Xi_j \dd x_j \mapsto -\tau^2 + \sum_{j=1}^d \Xi_j^2 \in C^\infty(T^* \bbR^{1,d}) 
	\end{equation} 
	is a representative of $\sigma_{\mathrm{de,sc}}^{2,\mathsf{0}}(\square)$. 
	\label{prop:symbol}
\end{proposition}
\begin{proof}
	We already know that $p_0$ is a representative for the sc-principal symbol of $\square$, so it suffices to work near null infinity. 
	Passing to polar coordinates, it suffices to consider the $d=1$ case, working on $\hat{\bbO}$. (We saw above that the $(d-1)r^{-1} \partial_r$ term in the spatial Laplacian lies in $\operatorname{Diff}_{\mathrm{de,sc}}^{1,-\mathsf{1}}$ and therefore does not affect the principal symbol.)
	\begin{itemize}
		\item  In terms of the coordinates $(\varrho_{\mathrm{nf}},\varrho_{\mathrm{Tf}}, \xi,\zeta)$ on $\hat{\Omega}_{\mathrm{nfTf},\pm,T}$, solving for $\xi$ and $\zeta$ in $\varrho_{\mathrm{nf}}^{-2} \varrho_{\mathrm{Tf}}^{-1} \xi \dd \varrho_{\mathrm{nf}} +  \varrho_{\mathrm{nf}}^{-1} \varrho_{\mathrm{Tf}}^{-2} \zeta \dd \varrho_{\mathrm{Tf}} = \tau \dd t + \Xi \dd r$ yields 
		\begin{align}
			\begin{split}
			\pm \tau &= \frac{1}{2\varrho_{\mathrm{nf}}} (1-\varrho_{\mathrm{nf}}^2)  \xi - \frac{\zeta}{\varrho_{\mathrm{nf}}}, \\ 
			\Xi &= - \frac{1}{2 \varrho_{\mathrm{nf}}} (1+\varrho_{\mathrm{nf}}^2) \ \xi + \frac{\zeta}{\varrho_{\mathrm{nf}}}. 
			\end{split}
		\end{align}
		Thus, the symbol $\Xi^2 - \tau^2$ is given by $\xi^2 - 2 \xi \zeta$ with respect to this coordinate system. This is exactly what \cref{eq:misc_090} gives for the de,sc- principal symbol of $\square$ locally. 
		\item In terms of the coordinates $(\varrho_{\mathrm{nf}},\varrho_{\mathrm{Sf}}, \xi,\zeta)$ on $\hat{\Omega}_{\mathrm{nfSf},\pm,R}$, solving for $\xi$ and $\zeta$ in $\varrho_{\mathrm{nf}}^{-2} \varrho_{\mathrm{Sf}}^{-1} \xi \dd \varrho_{\mathrm{nf}} +  \varrho_{\mathrm{nf}}^{-1} \varrho_{\mathrm{Sf}}^{-2} \zeta \dd \varrho_{\mathrm{Sf}} = \tau \dd t + \Xi \dd r$ yields 
		\begin{align}
			\begin{split}
			\pm \tau &= -\frac{1}{2 \varrho_{\mathrm{nf}}}(1+\varrho_{\mathrm{nf}}^2) \xi + \frac{\zeta}{\varrho_{\mathrm{nf}}}, \\ 
			\Xi &= \frac{1}{2\varrho_{\mathrm{nf}}}(1-\varrho_{\mathrm{nf}}^2) \xi - \frac{\zeta}{\varrho_{\mathrm{nf}}}.
			\end{split}
		\end{align}
		Thus, $\Xi^2 - \tau^2$ is given by $-\xi^2 +2 \xi \zeta$ with respect to this coordinate system. 
		This is exactly what \cref{eq:misc_091} gives for the de,sc- principal symbol of $\square$ locally. 
	\end{itemize}
	So, $\Xi^2-\tau^2 \in S_{\mathrm{de,sc}}^{2,\mathsf{0}}$ and is a representative of $\sigma_{\mathrm{de,sc}}^{2,\mathsf{0}}(\square)$.
\end{proof}

\subsection{The de,sc-calculus}
\label{subsec:Psi}

We have already discussed de,sc- differential operators.
Here, we summarize the basic properties of the pseudodifferential calculus $\Psi_{\mathrm{de,sc}}$. 
The details are analogous to those in the construction of the sc-calculus, so we concentrate on the main points. (So, for instance, we will not talk about the topologies of de,sc-pseudodifferential operators, nor about uniform families of operators.)

\begin{remark}
	The results in this section are all special cases of facts from the theory of pseudodifferential operators associated to \emph{scaled bounded geometries}, a concept due to Hintz, who expounded their theory in work \cite{HintzSBG} announced after the completion of the first version of this paper. Indeed, the de,sc- calculus is an example he gives \cite[\S1.2.4]{HintzSBG}, associated to a particular scaled bounded geometry on $\bbR^{1,d}$. In the initial version of this work, we only sketched proofs of the propositions in this subsection. Now, we refer to Hintz's work for the proofs, in much greater generality. Specifically, \cite[Thm.\ 1.4]{HintzSBG} contains the properties of the calculus that we need.
\end{remark}

Since the relevant calculi end up being coordinate invariant, and since the de- and sc-calculi are constructed in \cite{LauterMoroianu}\cite{MelroseSC} respectively, the main order of business is to construct the calculus near the corners of $\bbO$, which we model (using local coordinates $\bmtheta=(\theta_1,\ldots,\theta_{d-1})$ on $\bbS^{d-1}_{\bmtheta}$) by 
\begin{equation}
	\bbR^{d+1}_2 = [0,\infty)_{\varrho_{\mathrm{nf}}} \times [0,\infty)_{\varrho_{\mathrm{Of}}}\times \bbR^{d-1}_{\bmtheta}, 
\end{equation} 
with the face $\{\varrho_{\mathrm{nf}}=0\}$ of the right-hand side corresponding to null infinity. Here `Of' stands for ``other face,'' meaning any of $\mathrm{Pf},\mathrm{Sf},\mathrm{Ff}$, depending on which corner of $\bbO$ is under consideration. Thus, we discuss the construction of 
\begin{equation} 
	\Psi_{\mathrm{de,sc,c}}(\bbR^{d+1}_2) = \bigcup_{m,s,\varsigma\in \bbR}\Psi_{\mathrm{de,sc,c}}^{m,(s,\varsigma)}(\bbR^{d+1}_2), 
\end{equation} 
where $s$ is the ``de-decay order'' at $\{\varrho_{\mathrm{nf}}=0\}$ and $\varsigma$ is the ``sc-decay order'' at $\{\varrho_{\mathrm{Of}}=0\}$. The extra `c' denotes that these operators will have properly supported Schwartz kernels $K$, so that $K(-,\chi) \in \calE'(\bbR_2^{d+1})$ whenever $\chi \in C_{\mathrm{c}}^\infty(\bbR_2^{d+1})$. Roughly speaking, this local de,sc-calculus is the result of quantizing 
\begin{equation}
	\calV_{\mathrm{de,sc}}(\bbR^{d+1}_2) = \operatorname{span}_{C_{\mathrm{c}}^\infty(\bbR^{d+1}_2)}\Big\{ \varrho_{\mathrm{nf}}^2 \varrho_{\mathrm{Of}} \frac{\partial}{\partial \varrho_{\mathrm{nf}}},\varrho_{\mathrm{nf}} \varrho_{\mathrm{Of}}^2 \frac{\partial}{\partial \varrho_{\mathrm{Of}}} , \varrho_{\mathrm{nf}}^2\varrho_{\mathrm{Of}} V : V\in \calV(\bbR^{d-1}_{\bmtheta}) \Big\}. 
	\label{eq:Vdef_local}
\end{equation}
From this Lie algebra, we get the coball-bundle ${}^{\mathrm{de,sc}}\overline{T}^* \bbR^{d+1}_2$ in the usual way. 

For (conormal) symbols $a$ on ${}^{\mathrm{de,sc}}\overline{T}^* \bbR^{d+1}_2$ of sufficiently low order, we can define an element $\operatorname{Op}(a)$ of $\Psi_{\mathrm{de,sc}}(\bbR^{d+1}_2)$ via its Schwartz kernel, 
\begin{equation} 
	K_a \smash{\in \calS'(\bbR_2^{d+1}\times \bbR_2^{d+1})},
\end{equation} 
given by 
\begin{multline}
	K_a(x_{\mathrm{L}},x_{\mathrm{R}}) = \frac{\chi}{(2\pi)^{d+1}}\int_{\bbR^{d+1}} \Big[\exp\Big(\pm i \frac{\zeta}{\varrho_{\mathrm{nf}}^2\varrho_{\mathrm{Of}} }(\varrho_{\mathrm{nf}}-\varrho_{\mathrm{nf}}' ) \pm i \frac{\xi}{\varrho_{\mathrm{nf}}\varrho_{\mathrm{Of}}^2}(\varrho_{\mathrm{Of}}-\varrho_{\mathrm{Of}}' )  \pm \sum_{j=1}^{d-1} \frac{i\eta_j}{\varrho_{\mathrm{nf}}^2\varrho_{\mathrm{Of}} }(\theta_j-\theta_j')  \Big) \\ \times a(\varrho_{\mathrm{nf}},\varrho_{\mathrm{Of}},\bmtheta,\zeta,\xi,\eta )\Big]\dd \zeta \dd \xi \dd^{d-1} \eta, 
	\label{eq:misc_112}
\end{multline}
where $\chi \in C^\infty((\bbR^{d+1}_2)^2_{\mathrm{b}})$ is identically equal to $1$ near the diagonal of the b-double space 
\begin{equation} 
	(\bbR^{d+1}_2)^2_{\mathrm{b}}\cong ([0,\infty)^2_{\mathrm{b}})^2\times \bbR^{2d-2}
\end{equation} 
and identically $0$ near boundary faces disjoint from the diagonal. Here, $x_{\mathrm{L}} = (\varrho_{\mathrm{nf}},\varrho_{\mathrm{Of}},\bmtheta)$ and $x_{\mathrm{R}} = (\varrho_{\mathrm{nf}}',\varrho_{\mathrm{Of}}',\bmtheta')$. 
The choice of sign in the exponent in \cref{eq:misc_112} is to be fixed as a convention.
Actually, in order to establish the basic properties of the calculus, it is useful to introduce spaces of symbols (``two-sided symbols'') which depend on both $x_{\mathrm{L}}$ and $x_{\mathrm{R}}$, these being quantized in the same manner. 
These definitions are extended to symbols of arbitrary order using slightly modified versions of the standard estimates for oscillatory integrals. The initial restriction to $a$ of sufficiently good order is to guarantee that the integral above converges, but standard estimates show this restriction to be unnecessary.

For each $m,s,\varsigma \in \bbR$ and (compactly supported) 
\begin{equation} 
	a \in S^{m,s,\varsigma}({}^{\mathrm{de,sc}}\overline{T}^* \bbR_2^{d+1}),
\end{equation} 
let $\operatorname{Op}(a)=K_a$ as above, and let  $\Psi_{\mathrm{de,sc,c}}^{m,(s,\varsigma)}(\bbR^{d+1}_2)$ denote the set of operators whose Schwartz kernels have the form $K_a + R$ for some properly supported remainder kernel 
\begin{equation}
	R \in \calS(\bbR_2^{d+1}\times \bbR_2^{d+1}). 
\end{equation}
Elements of $\Psi_{\mathrm{de,sc,c}}(\bbR_2^{d+1})$ are initially defined as maps $C_{\mathrm{c}}^\infty(\bbR_2^{d+1}) \to \calS'(\bbR_2^{d+1})$, but they extend (uniquely) to maps 
\begin{equation} 
	\calS'(\bbR_2^{d+1})\to \calS'(\bbR_2^{d+1}),
\end{equation} 
and elements of $\Psi_{\mathrm{de,sc,c}}(\bbR^{d+1}_2)$ can be identified with the corresponding maps. 
This completes our sketch of the definition of $\Psi_{\mathrm{de,sc,c}}(\bbR^{d+1}_2)$. 

We now return to $\bbO$. 
The calculus $\Psi_{\mathrm{de,sc}}=\Psi_{\mathrm{de,sc}}(\bbO)$ behaves very similarly to the sc-calculus. This is because we have \emph{principal symbol maps} 
\begin{equation}
	\sigma_{\mathrm{de,sc}}^{m,\mathsf{s}} : \Psi_{\mathrm{de,sc}}^{m,\mathsf{s}}(\bbO) \to S_{\mathrm{de,sc}}^{[m],[\mathsf{s}]}= S_{\mathrm{de,sc}}^{m,\mathsf{s}}(\bbO)/S_{\mathrm{de,sc}}^{m-1,\mathsf{s}-\mathsf{1}}(\bbO)
\end{equation}
fitting into a short exact sequence 
\begin{equation}
	0 \to \Psi_{\mathrm{de,sc}}^{m-1,\mathsf{s}-1} \hookrightarrow \Psi_{\mathrm{de,sc}}^{m,\mathsf{s}} \to S_{\mathrm{de,sc}}^{[m],[\mathsf{s}]} \to 0
\end{equation}
of vector spaces. 
This interacts with $\operatorname{Op}$ in the expected way: 
\begin{equation} 
	\sigma_{\mathrm{de,sc}}^{m,\mathsf{s}}(\operatorname{Op}(a)+R) = a \bmod S_{\mathrm{de,sc}}^{m-1,\mathsf{s}-\mathsf{1}},
\end{equation} 
whenever $R$ is as above. We have already given the definition of $\sigma_{\mathrm{de,sc}}$ on elements of $\operatorname{Diff}_{\mathrm{de,sc}}$, and it can be checked that this is a special case of the general definition. For example, if we evaluate \cref{eq:misc_112} for $a=\chi_0(x_{\mathrm{L}})\zeta$ for $\chi_0 \in C_{\mathrm{c}}^\infty(\bbR^{d+1}_2)$, then 
\begin{multline}
	K_a(x_{\mathrm{L}},x_{\mathrm{R}} )=K_\zeta(x_{\mathrm{L}},x_{\mathrm{R}} ) 
	= \pm\frac{i\chi \chi_0(x_{\mathrm{L}}) }{(2\pi)^{d+1}} \varrho_{\mathrm{nf}}^2\varrho_{\mathrm{Of}} \frac{\partial}{\partial \varrho_{\mathrm{nf}}'} \int_{\bbR^{d+1}}\Big[ \\ \exp\Big(\pm i \frac{\zeta}{\varrho_{\mathrm{nf}}^2\varrho_{\mathrm{Of}} }(\varrho_{\mathrm{nf}}-\varrho_{\mathrm{nf}}' ) \pm i \frac{\xi}{\varrho_{\mathrm{nf}}\varrho_{\mathrm{Of}}^2}(\varrho_{\mathrm{Of}}-\varrho_{\mathrm{Of}}' )  \pm \sum_{j=1}^{d-1} \frac{i\eta_j}{\varrho_{\mathrm{nf}}^2\varrho_{\mathrm{Of}} }(\theta_j-\theta_j')  \Big)   \dd \zeta \dd \xi \dd^{d-1} \eta \\ 
	\propto  \chi_0(x_{\mathrm{L}})  \varrho_{\mathrm{nf}}^2\varrho_{\mathrm{Of}} \frac{\partial}{\partial \varrho_{\mathrm{nf}}'} \delta(x_{\mathrm{L}},x_{\mathrm{R}} ).
\end{multline}
This is just the Schwartz kernel of $i\chi_0 \partial_{\varrho_{\mathrm{nf}}}$, up to some $i$'s. So, the principal symbol map involves replacing $\partial_{\varrho_{\mathrm{nf}}}$ with $i\zeta$, as described in the previous subsection. The partial derivatives in the other directions are similar.

The set $\bigcup_{m\in \bbR,\mathsf{s}\in \bbR^5} \Psi_{\mathrm{de,sc}}^{m,\mathsf{s}}$ is a a multi-graded algebra. In particular, this means that 
\begin{equation}
	AB \in \Psi_{\mathrm{de,sc}}^{m+m',\mathsf{s}+\mathsf{s}'}
\end{equation}
whenever $A\in \Psi_{\mathrm{de,sc}}^{m,\mathsf{s}}$ and $B\in \Psi_{\mathrm{de,sc}}^{m',\mathsf{s}'}$. That this is true if one ignores the situation at the corners follows from the analogous result for the sc- and de-calculi. However, since the combination of the two has not been studied before, we say a few words on how one can prove this. One way is to use a reduction formula for two-sided symbols, in analogy with what is done in the sc-calculus in \cite{VasyGrenoble}; the quantization of any two-sided symbol is the quantization of a one-sided symbol modulo residual terms. The composition of two de,sc-pseudodifferential operators, one given as the quantization of a left symbol and one as the quantization of the right symbol, is seen to be the quantization of a two-sided symbol. Thus, it is a pseudodifferential operator of the expected orders. Applying the reduction formula gives a formula for the symbol of $AB$ similar to the usual Moyal formula, except written using a different coordinate system. 
One consequence of the Moyal formula is that the principal symbol map is an algebra homomorphism to leading order, in the sense that 
\begin{equation}
	\sigma_{\mathrm{de,sc}}^{m+m',\mathsf{s}+\mathsf{s}'}(AB) = \sigma_{\mathrm{de,sc}}^{m,\mathsf{s}}(A)\sigma_{\mathrm{de,sc}}^{m',\mathsf{s}'}(B) \overset{\mathrm{def}}{=} ab \bmod S_{\mathrm{de,sc}}^{m+m'-1,\mathsf{s}+\mathsf{s}'-\mathsf{1}}, 
	\label{eq:misc_652}
\end{equation}
where $a,b$ are any representatives of $\smash{\sigma_{\mathrm{de,sc}}^{m,\mathsf{s}}(A)}$ and $\smash{\sigma_{\mathrm{de,sc}}^{m',\mathsf{s}'}(B)}$, respectively (the equivalence class of $ab$ not depending on the choice of $a,b$). See \cite[Thm.\ 1.4]{HintzSBG} for details.

We also have a notion of ``de,sc--essential support:'' 
\begin{equation}
	\operatorname{WF}'_{\mathrm{de,sc}}(A) \overset{\mathrm{def}}{=} \operatorname{esssupp}(a) \cap \partial\, {}^{\mathrm{de,sc}}\overline{T}^* \bbO
\end{equation}
whenever $A = \operatorname{Op}(a) + R$ for $R$ as above. (That is, $\operatorname{WF}'_{\mathrm{de,sc}}(A)$ is the closure of the set of points on the boundary of the radially-compactified de,sc-cotangent bundle at which $a$ is not rapidly decaying.) This is well-defined, meaning that if we have $\operatorname{Op}(a)+R = \operatorname{Op}(a')+R'$ for some $a',R'$, then $a,a'$ have the same essential support. We have 
\begin{align}
	\operatorname{WF}'_{\mathrm{de,sc}}(AB) &\subseteq \operatorname{WF}'_{\mathrm{de,sc}}(A) \cap\operatorname{WF}'_{\mathrm{de,sc}}(B) \\ 
	\operatorname{WF}'_{\mathrm{de,sc}}(A+B) &\subseteq \operatorname{WF}'_{\mathrm{de,sc}}(A) \cup\operatorname{WF}'_{\mathrm{de,sc}}(B)
\end{align}
for any $A,B\in \Psi_{\mathrm{de,sc}}$. In particular, $\operatorname{WF}'_{\mathrm{de,sc}}([A,B]) \subseteq \operatorname{WF}'_{\mathrm{de,sc}}(A) \cap\operatorname{WF}'_{\mathrm{de,sc}}(B)$.

The \emph{de,sc-characteristic set} $\operatorname{Char}_{\mathrm{de,sc}}^{m,\mathsf{s}}(A)$ of $A\in \Psi_{\mathrm{de,sc}}^{m,\mathsf{s}}$ is defined as 
\begin{equation}
	\operatorname{Char}_{\mathrm{de,sc}}^{m,\mathsf{s}}(A) = \operatorname{char}_{\mathrm{de,sc}}^{m,\mathsf{s}}(a) = \partial\, {}^{\mathrm{de,sc}}\overline{T}^* \bbO \backslash \operatorname{ell}_{\mathrm{de,sc}}^{m,\mathsf{s}}(a), 
\end{equation}
where the elliptic set is defined in the usual way. We mainly care about the case when $a$ is a classical symbol. Considering the case $m=0$, $\mathsf{s}=\mathsf{0}$, this means that $a$ is a smooth function on the radially-compactified cotangent bundle, in which case the de,sc-characteristic set of $A=\operatorname{Op}(a)$ is just
\begin{equation}
	\operatorname{Char}_{\mathrm{de,sc}}^{0,\mathsf{0}}(A) = a^{-1}(\{0\}) \cap ( \partial\, {}^{\mathrm{de,sc}}\overline{T}^* \bbO),
\end{equation}
the vanishing set of $a$ at the boundary of the compactified de,sc- phase space.

An operator $A\in \Psi_{\mathrm{de,sc}}^{m,\mathsf{s}}$ is said to be \emph{elliptic} (with respect to the de,sc-calculus) if $\smash{\operatorname{Char}_{\mathrm{de,sc}}^{m,\mathsf{s}}(A)} = \varnothing$. (Note that this depends implicitly on $m,\mathsf{s}$.) It should be emphasized that, just as sc-ellipticity is stronger than ordinary ellipticity, de,sc- ellipticity is a stronger notion than the usual notion of ellipticity in $\bbO^\circ = \bbR^{1,d}$, which is equivalent to
\begin{equation}
	\operatorname{Char}_{\mathrm{de,sc}}^{m,\mathsf{s}}(A) \cap \overline{T}^* \bbR^{1,d} = \varnothing. 
\end{equation}
I.e.,\ the ordinary notion of ellipticity in the interior is ellipticity at fiber infinity over the interior. De,sc-ellipticity requires ellipticity in the fibers of the de,sc-cotangent bundle over the boundary of $\bbO$, as well as at fiber infinity over the boundary.

It follows from the principal symbol short exact sequence and the leading order commutativity of the principal symbol map that 
\begin{equation}
	A\in \Psi_{\mathrm{de,sc}}^{m,\mathsf{s}},B\in \Psi_{\mathrm{de,sc}}^{m',\mathsf{s}'} \Rightarrow [A,B] \in \Psi_{\mathrm{de,sc}}^{m+m'-1,\mathsf{s}+\mathsf{s}'-1}. 
\end{equation}
It can be shown (using e.g. the Moyal formula) that 
\begin{equation} 
	\sigma_{\mathrm{de,sc}}^{m+m'-1,\mathsf{s}+\mathsf{s}'-\mathsf{1}}([A,B]) = \pm i \{a,b\} \bmod S_{\mathrm{de,sc}}^{m+m'-2,\mathsf{s}+\mathsf{s}'-\mathsf{2}}, 
\end{equation} 
where the sign depends on our sign convention in defining $\operatorname{Op}$ in \cref{eq:misc_112}. The right-hand side is just the usual Poisson bracket on $T^* \bbR^{1,d}$. (It must, of course, be checked that $\{a,b\}$ is actually a de,sc-symbol of the claimed orders.) This is also part of \cite[Thm.\ 1.4]{HintzSBG}.

It can be shown directly, that, if 
\begin{equation} 
	R\in \Psi_{\mathrm{de,sc}}^{-\infty,-\infty} = \cap_{m,\mathsf{s}}\Psi_{\mathrm{de,sc}}^{m,\mathsf{s}},
\end{equation} 
then $R$ defines a bounded linear map on $L^2(\bbR^{1,d})$. (Indeed, $R\in \Psi_{\mathrm{sc}}^{-\infty,-\infty}(\bbR^{1,d})$.)

H\"ormander's square root trick can be used to extend this to $A \in \smash{\Psi_{\mathrm{de,sc}}^{0,\mathsf{0}}}$. 

The quantization procedure yields, for each $m\in \bbR, \mathsf{s}\in \bbR^5$, a plethora of elliptic elements of $\Psi_{\mathrm{de,sc}}^{m,\mathsf{s}}$. Pick one and call it $\Lambda^{m,\mathsf{s}}$. 
For $m=0$, we can take 
\begin{equation}
\Lambda^{0,\mathsf{s}} = \varrho_{\mathrm{Pf}}^{-s_1}\varrho_{\mathrm{nPf}}^{-s_2}\varrho_{\mathrm{Sf}}^{-s_3}\varrho_{\mathrm{nFf}}^{-s_4}\varrho_{\mathrm{Ff}}^{-s_5} = \varrho^{-\mathsf{s}}. 
\end{equation}
We now define $H^{0,0}_{\mathrm{de,sc}} = L^2(\bbR^{1,d})$ and, for each $\mathsf{s}\in \bbR^5$,  $H^{0,\mathsf{s}}_{\mathrm{de,sc}} = \varrho^{\mathsf{s}} H^{0,0}_{\mathrm{de,sc}}$. 
We can now define Sobolev spaces $H_{\mathrm{de,sc}}^{m,\mathsf{s}}$ as follows: 
\begin{itemize}
	\item if $m> 0$, then we define 
	\begin{align} 
	H^{m,\mathsf{s}}_{\mathrm{de,sc}} &= \{ u \in H^{0,\mathsf{s}}_{\mathrm{de,sc}} : Au\in L^2(\bbR^{1,d}) \text{ for all } A \in \Psi^{m,\mathsf{s}}_{\mathrm{de,sc}} \} \label{eq:misc_hms} \\
	&= \{ u \in H^{0,\mathsf{s}}_{\mathrm{de,sc}} : \Lambda^{m,\mathsf{s}} u\in L^2(\bbR^{1,d}) \}, \label{eq:misc_ter}
	\end{align} 
	with norm $\lVert u \rVert_{H^{m,\mathsf{s}}_{\mathrm{de,sc}}} = \lVert \varrho^{-\mathsf{s}} u \rVert_{L^2} + \lVert \Lambda^{m,\mathsf{s}} u \rVert_{L^2}$, and 
	\item if $m<0$, then we define 
	\begin{align} 
	H^{m,\mathsf{s}}_{\mathrm{de,sc}} &= \{ \Lambda^{-m,-\mathsf{s}} u +v : u \in L^2(\bbR^{1,d}),v \in H_{\mathrm{de,sc}}^{0,\mathsf{s}}  \} \label{eq:misc_hjs} \\
	&= \{ Au : u\in L^2(\bbR^{1,d}) , A \in \Psi_{\mathrm{de,sc}}^{-m,-\mathsf{s}} \}, \label{eq:misc_tjr}
	\end{align} 
	with a corresponding norm. 
\end{itemize}
Each element of $\Psi_{\mathrm{de,sc}}^{m,\mathsf{s}}$ defines a bounded map $H_{\mathrm{de,sc}}^{m',\mathsf{s}'} \to H_{\mathrm{de,sc}}^{m'-m,\mathsf{s}'-\mathsf{s}}$ for any $m\in \bbR$ and $\mathsf{s}'\in \bbR^5$. See \cite[Thm.\ 3.44]{HintzSBG}.

The failure of a $u\in \calS'$ to lie in $H_{\mathrm{de,sc}}^{m,\mathsf{s}}$ is measured by a notion of de,sc-wavefront set, 
\begin{equation} 
	\operatorname{WF}_{\mathrm{de,sc}}^{m,\mathsf{s}}(u) = \bigcap_{A \in \Psi_{\mathrm{de,sc}}^{0,\mathsf{0}}\text{ s.t. } Au\in H_{\mathrm{de,sc}}^{m,\mathsf{s}} } \operatorname{Char}_{\mathrm{de,sc}}^{m,\mathsf{s}}(A). 
\end{equation} 
Thus, $u\in H_{\mathrm{de,sc}}^{m,\mathsf{s}}  \iff 	\operatorname{WF}_{\mathrm{de,sc}}^{m,\mathsf{s}}(u) = \varnothing$. (The $\Rightarrow$ direction is trivial, and the $\Leftarrow$ direction follows via the standard patching argument.) 
Also, let 
\begin{equation} 
	\operatorname{WF}_{\mathrm{de,sc}}(u) = \mathrm{cl}_{{}^{\mathrm{de,sc}}\overline{T}^* \bbO } \Big[\bigcup_{m,\mathsf{s}}\operatorname{WF}_{\mathrm{de,sc}}^{m,\mathsf{s}}(u)\Big].
	\label{eq:WF_def}
\end{equation} 
It can be shown that $u\in \calS \iff \operatorname{WF}_{\mathrm{de,sc}}(u) = \varnothing$, so de,sc-wavefront set measures microlocal obstructions to being Schwartz on $\bbR^{1,d}$, as $\operatorname{WF}_{\mathrm{sc}}(u)$ does. 

The portion of the de,sc- wavefront set of the Green's functions $D_\pm$ (discussed in \Cref{ex:Green}) in the de,sc-characteristic set of $\square+\mathsf{m}^2$ is depicted in \Cref{fig:Greens}. 
\begin{figure}[t]
	\begin{tikzpicture}[scale=.8]
		\coordinate (one) at (-.413\octheight,\octheight) {};
		\coordinate (two) at (-\octheight,.413\octheight) {}; 
		\coordinate (three) at (-\octheight,-.413\octheight) {};
		\coordinate (four) at (-.413\octheight,-\octheight) {};
		\coordinate (five) at (.413\octheight,-\octheight) {};
		\coordinate (six) at (\octheight,-.413\octheight) {};
		\coordinate (seven) at (\octheight,.413\octheight) {}; 
		\coordinate (eight) at (.413\octheight,\octheight) {};
		\coordinate (oner) at (-.413\octheightr,\octheightr) {};
		\coordinate (twor) at (-\octheightr,.413\octheightr) {}; 
		\coordinate (threer) at (-\octheightr,-.413\octheightr) {};
		\coordinate (fourr) at (-.413\octheightr,-\octheightr) {};
		\coordinate (fiver) at (.413\octheightr,-\octheightr) {};
		\coordinate (sixr) at (\octheightr,-.413\octheightr) {};
		\coordinate (sevenr) at (\octheightr,.413\octheightr) {}; 
		\coordinate (eightr) at (.413\octheightr,\octheightr) {};
		\coordinate (oneint) at (-1.25*.413\octheight,1.25\octheight) {};
		\coordinate (eightint) at (.413*1.25\octheight,1.25\octheight) {};
		\coordinate (fourint) at (-.413*1.25\octheight,-1.25\octheight) {};
		\coordinate (fiveint) at (.413*1.25\octheight,-1.25\octheight) {};
		\draw[fill=gray!10] (oner) -- (twor) -- (threer) -- (fourr) -- (fiver) -- (sixr) -- (sevenr) -- (eightr) -- cycle;
		\draw[fill=gray!30] (one) -- (two) -- (three) -- (four) -- (five) -- (six) -- (seven) -- (eight) -- cycle;
		\draw (one) -- (oner);
		\draw (two) -- (twor);
		\draw (three) -- (threer);
		\draw (four) -- (fourr);
		\draw (five) -- (fiver);
		\draw (six) -- (sixr);
		\draw (seven) -- (sevenr);		
		\draw (eight) -- (eightr);
		\draw[gray] (-.706\octheight,-.706\octheight) -- (.706\octheight,.706\octheight); 
		\draw[darkcandyapp] (0,0) -- (.706\octheight,.706\octheight) to[out=45, in=0] (eightint); 
		\draw[darkcandyapp] (-.706\octheightr,.706\octheightr) to[out=-45,in=245] (oneint);
		\begin{scope}[decoration={
				markings,
				mark=at position 0.51 with {\arrow[scale=1.5,>=latex, color=gray]{>}}}]
			\draw[gray, postaction={decorate}] plot[smooth,tension=.1] coordinates {
				(-.636\octheight,-.776\octheight) (.125\octheight,-.125\octheight) (.776\octheight,.636\octheight)};
			\draw[gray, postaction={decorate}] plot[smooth,tension=.1] coordinates {
				(-.776\octheight,-.636\octheight) (-.125\octheight,+.125\octheight) (.636\octheight, .776\octheight)};
			\draw[gray, postaction={decorate}]  plot[smooth,tension=.3] coordinates{ 
				(-.566\octheight,-.846\octheight) (-.456\octheight,-.756\octheight) (.275\octheight,-.275\octheight) (.756\octheight, .456\octheight)  (.846\octheight,.566\octheight) };
			\draw[gray, postaction={decorate}]  plot[smooth,tension=.3] coordinates{ 
				(-.846\octheight,-.566\octheight) (-.756\octheight,-.456\octheight) (-.275\octheight,.275\octheight) (.456\octheight,.756\octheight)  (.566\octheight,.846\octheight) };
			\draw[gray, postaction={decorate}]  plot[smooth,tension=.4] coordinates{ 
				(-.526\octheight,-.886\octheight) (-.146\octheight,-.786\octheight) (.45\octheight,-.45\octheight) (.786\octheight,.146\octheight) (.886\octheight,.526\octheight)    }; 
			\draw[gray, postaction={decorate}]  plot[smooth,tension=.4] coordinates{ 
				(-.886\octheight,-.526\octheight) (-.786\octheight,-.146\octheight) (-.45\octheight,.45\octheight) (.146\octheight,.786\octheight) (.526\octheight,.886\octheight)    }; 
			\draw[postaction={decorate}] (three) -- (two);
			\draw[postaction={decorate}] (threer) -- (three); 
			\draw[postaction={decorate}] (seven) -- (sevenr); 
			\draw[postaction={decorate}] (threer) -- (twor);
			\draw[postaction={decorate}] (fourr) -- (threer);
			\draw[postaction={decorate}] (two) -- (one);
			\draw[postaction={decorate}] (fiver) -- (fourr);
			\draw[postaction={decorate}] (six) -- (seven);
			\draw[postaction={decorate}] (sixr) -- (sevenr);
			\draw[postaction={decorate}] (sevenr) -- (eightr);
			\draw[postaction={decorate}] (eightr) -- (oner);
			\draw[gray, postaction={decorate}]  plot[smooth,tension=.5] coordinates{ (fourint) (-.513\octheightr,-.85\octheightr)  (-.75\octheightr,-.443\octheightr) 	(-.886\octheight,-.526\octheight) };
			\draw[gray, postaction={decorate}]  plot[smooth,tension=.5] coordinates{ 	(.886\octheight,.526\octheight) (.75\octheightr,.443\octheightr)  (.513\octheightr,.85\octheightr)  (eightint)};
			\draw[gray, postaction={decorate}]  plot[smooth,tension=.5] coordinates{ (fiveint) (.413\octheightr,-.65\octheightr)  (.75\octheightr,-.443\octheightr) 	(.886\octheightr,-.526\octheightr) };
			\draw[gray, postaction={decorate}]  plot[smooth,tension=.5] coordinates{ 	(-.886\octheightr,.526\octheightr) (-.75\octheightr,.443\octheightr) (-.413\octheightr,.65\octheightr)  (oneint)};
			\draw[gray, postaction={decorate}] (threer) -- (two);
			\draw[gray, postaction={decorate}] (six) -- (sevenr);
			\draw[gray, postaction={decorate}] (-.2\octheight,-1.25\octheight) -- (five);
			\draw[gray, postaction={decorate}] (.2\octheight,-1.25\octheight) -- (fourr);
			\draw[gray, postaction={decorate}] (one) -- (.2\octheight,1.25\octheight);
			\draw[gray, postaction={decorate}] (eightr) -- (-.2\octheight,1.25\octheight);
		\end{scope}
		\begin{scope}[decoration={
				markings,
				mark=at position 0.51 with {\arrow[scale=1.5,>=latex, color=gray]{>}}}] 
			\draw[postaction={decorate}] (one) -- (eight);
			\draw[postaction={decorate}] (two) -- (one); 
			\draw[postaction={decorate}] (twor) -- (two); 
			\draw[postaction={decorate}] (four) -- (five);
			\draw[postaction={decorate}] (six) -- (sixr); 
			\draw[postaction={decorate}] (five) -- (six);
		\end{scope} 
		\begin{scope}[decoration={
				markings,
				mark=at position 0.75 with {\arrow[scale=1.5,>=latex, color=gray]{>}}}]
			\draw[postaction={decorate}] (fourint) -- (four);
			\draw[postaction={decorate}] (fourint) -- (fourr);
			\draw[postaction={decorate}] (fiveint) -- (five);
			\draw[postaction={decorate}] (fiveint) -- (fiver);
			\draw[postaction={decorate}] (one) -- (oneint);
			\draw[postaction={decorate}] (eight) -- (eightint);
			\draw[postaction={decorate}] (oner) -- (oneint);
			\draw[postaction={decorate}] (eightr) -- (eightint);
			\draw[gray, postaction={decorate}]  plot[smooth,tension=.6] coordinates{ (fourint) (-.63\octheight,-1.1\octheight) (-.526\octheight,-.886\octheight) };
			\draw[gray, postaction={decorate}]  plot[smooth,tension=.6] coordinates{(-.616\octheightr,.796\octheightr) (-.73\octheight,1.19\octheight) (oneint) };
			\draw[gray, postaction={decorate}]  plot[smooth,tension=.6] coordinates{ (fiveint) (.73\octheight,-1.19\octheight) (.616\octheightr,-.796\octheightr) };
			\draw[gray, postaction={decorate}]  plot[smooth,tension=.6] coordinates{(.526\octheight,.886\octheight) (.66\octheight,1.0\octheight) (eightint) };
			\filldraw[color=darkcandyapp] (oneint) circle (2pt);
			\filldraw[color=darkcandyapp] (eightint) circle (2pt);
			\draw[thick, darkcandyapp] (oneint) -- (eightint);
			\draw[gray] (fourint) -- (fiveint);
		\end{scope}
		\node[color=darkcandyapp] (bc) at (-.22\octheightr,.91\octheightr) {$\calR^+_+$};
	\end{tikzpicture}
	\caption{ 
		The de,sc-wavefront set $\color{darkcandyapp}\operatorname{WF}_{\mathrm{de,sc}}(D_+)$ of the forward propagator $D_+$ for $\square+\mathsf{m}^2$, shown together with the de,sc- Hamiltonian flow depicted in \Cref{fig:O}. 
		As we can see, the wavefront set at fiber infinity introduced by the $\delta$-function forcing propagates along null geodesics until it hints the interior of null infinity at fiber infinity. It then enters the interiors of the de-fibers and tends to the radial set $\calR_+$. Only the portion of the wavefront set in the characteristic set of $\square+\mathsf{m}^2$ is shown. Since $(\square+\mathsf{m}^2)D_\pm = \delta$, where $\delta$ is a Dirac function at the spacetime origin, and since $\operatorname{WF}(\delta)$ consists of the whole cosphere fiber over the spacetime origin, it must be the case (by microlocality -- \cref{eq:microlocality}) that $D_\pm$ has wavefront set in the elliptic set there. However, by micro-ellipticity in the de,sc-calculus (\cref{eq:micro-ellipticity}), this occurs only over the spacetime origin.}
	\label{fig:Greens}
\end{figure}

De,sc-$\Psi$DOs are microlocal, which means that 
\begin{equation}
	\operatorname{WF}_{\mathrm{de,sc}}^{m-m',\mathsf{s}-\mathsf{s}'}(Au) \subseteq \operatorname{WF}'_{\mathrm{de,sc}}(A) \cap \operatorname{WF}_{\mathrm{de,sc}}^{m,\mathsf{s}}(u) 
	\label{eq:microlocality}
\end{equation}
for any $u\in \calS'$ and $A\in \Psi_{\mathrm{de,sc}}^{m',\mathsf{s}'}$. On the other hand, the de,sc-version of microlocal elliptic regularity states that 
\begin{equation}
	\operatorname{WF}_{\mathrm{de,sc}}^{m,\mathsf{s}}(u) \subseteq \operatorname{WF}_{\mathrm{de,sc}}^{m-m',\mathsf{s}-\mathsf{s}'}(Au) \cup \operatorname{Char}_{\mathrm{de,sc}}^{m',\mathsf{s}'}(A). 
	\label{eq:micro-ellipticity}
\end{equation}
For example, if $u$ solves the Klein--Gordon equation $\square_gu +\mathsf{m}^2 u \in \calS$, then it follows immediately that 
\begin{equation}
	\operatorname{WF}_{\mathrm{de,sc}}(u) \subseteq \operatorname{Char}_{\mathrm{de,sc}}^{2,\mathsf{0}}(\square_g+\mathsf{m}^2).
\end{equation}
Thus, the de,sc-wavefront set of $u$ is highly restricted; it can only lie in the de,sc- characteristic set of the Klein--Gordon operator, this being a codimension one subset of the boundary of the de,sc- phase space.

To get information on the de,sc- wavefront set within the characteristic set, we will need to use propagation results, which will be discussed later, in \S\ref{sec:propagation}. For now, see \Cref{fig:Greens}, which is suggestive of how de,sc-wavefront set propagates along bicharacteristics of the PDE in question, in this case $\square+\mathsf{m}^2$.

\section{Asymptotics, module regularity, and the Poincar\'e cylinder}
\label{sec:asymptotics}

We now discuss module regularity at $\calR$ (which really means additional regularity \emph{outside} of $\calR$) and its relation to asymptotic expansions at timelike infinity. In \S\ref{subsec:Poincare_cylinder}, we discuss the Poincar\'e cylinder, which is the geometrization of the hyperbolic coordinate system (\Cref{prop:hyperbolic=Minkowski}). 
In \S\ref{subsec:test_modules}, we discuss the test modules of differential operators relevant to the rest of the paper. These are used to define refined Sobolev spaces 
\begin{equation}
	H_{\mathrm{de,sc}}^{m,\mathsf{s};\kappa,k} , \; m\in \bbR,\, \mathsf{s}\in \bbR^5,\, \kappa,k\in \bbN,
\end{equation}
the functions with $m$ orders of de,sc-regularity, $\mathsf{s}$ orders of de,sc-decay, and $\kappa,k$ orders of additional regularity with respect to the relevant test modules.
In \S\ref{subsec:asymptotics}, we prove the main result of this section (\Cref{prop:asymptotic_main}), which states that if $u\in \calS'$ solves the Klein--Gordon equation $Pu=f$ for $f\in \calS'$ with sufficient module regularity, then, if $u$ has sufficient module regularity as well, including rapid decay at null infinity (which at this stage of our analysis is still a hypothesis), then $u$ admits an asymptotic expansion to some specified order on $\bbO$. 

The main theme of this section is the relation between hyperbolic coordinates and the de,sc- machinery. One conceptual way of understanding this relation is that, whereas performing a polar blowup of $\scrI\subset \bbM$ resulted in a compactification $\bbO$ whose front faces $\mathrm{nPf},\mathrm{nFf}$ are parameterized by the light cone coordinate $|t|-r$, performing a \emph{parabolic} blowup of $\scrI$ results in a different compactification whose front face is instead parameterized by the hyperbolic coordinate $\tau^2=t^2-r^2$ in $\{|t|>r\}$. 
So, this second compactification is closely related to the Poincar\'e cylinder; in fact, the two are identifiable above the future-directed light cone. The side of the Poincar\'e cylinder corresponds to null infinity.

\begin{figure}[h!]
	\begin{tikzpicture}
		\draw[dashed] (-2,-2) rectangle (2,2);
		\coordinate (i) at (.75,.75); 
		\fill[lightgray!20] (-1.8,-1.8) -- (1.8,-1.8) arc(0:90:3.6) -- cycle;
		\draw (1.8,-1.8) arc(0:90:3.6);
		\fill[black] (i) circle (.07);
		\node[above right] () at (i) {$\scrI^+$};
		\node () at (1,-1.2) {$\bbM$};
		\draw[darkcandyapp, ->] (.65,.65) to[out=135, in=-35] (-.2,1.3) node[below left] {$\frac{t-r}{t+r}$};
		\draw[darkcandyapp, ->] (.65,.65) -- (0,0) node[right] {$\frac{1}{t+r}$};
	\end{tikzpicture}
	\begin{tikzpicture}
		\draw[dashed] (-2,-2) rectangle (2,2);
		\coordinate (i) at (.75,.75); 
		\fill[lightgray!20] (-1.8,-1.8) -- (1.8,-1.8) arc(0:90:3.6) -- cycle;
		\fill[lightgray!40] (-1.8,-1.8) -- (.745,.745) arc(45:90:3.6) -- cycle;
		\draw (1.8,-1.8) arc(0:90:3.6);
		\fill[white, rotate=45] (i) ellipse (1.4 and .6);
		\begin{scope}
			\clip (-1.81,-1.81) -- (1.81,-1.81) arc(0:90:3.62) -- cycle;
			\filldraw[fill=white, rotate=45] (i) ellipse (1.4 and .6);
		\end{scope}
		\draw[darkcandyapp, ->]  (-.3,-.3) to[out=145, in=240] (-.3,.5) node[left] {$t^2-r^2$}; 
		\node () at (.7,-1.2) {$[\bbM;\scrI]_{\mathrm{par}}$};
	\end{tikzpicture}
	\begin{tikzpicture}
	\draw[dashed] (-2,-2) rectangle (2,2);
	\begin{scope}[yshift=-5pt] 
	\filldraw[fill=lightgray!20] (0,-1.2) ellipse (1 and .3);
	\fill[lightgray!20] (-1,-1.2) rectangle (1,1.2);
	\fill[lightgray!40] (-1,0) rectangle (1,1.2);
	\filldraw[dashed, fill=lightgray!40] (0,0) ellipse (1 and .3);
	\draw (-1,-1.2) -- (-1,1.2);
	\draw (1,-1.2) -- (1,1.2); 
	\draw[darkcandyapp,->] (0,1.2) -- (0,.7) node[right] {$\tau^{-1}$};
	\filldraw[fill=lightgray!40] (0,1.2) ellipse (1 and .3);
	\node () at (0,1.2) {$\bbH^d$};
	\node () at (1.5,0) {$\overline{\bbR}_\tau$};
	\draw[darkcandyapp,->] (0,0) -- (0,.4) node[left] {$\tau$};
	\end{scope} 
	\node () at (0,1.65) {Poincar{\'e} cylinder};
	\end{tikzpicture}
	\caption{
		Whereas a polar blowup of $\scrI^+\subseteq \bbM$ (shown in \Cref{fig:blowup}) resolved the ratio $((t-r)/(t+r))/(1/(t+r)) = t-r$, a parabolic blowup (\textit{middle}) instead resolves the quadratic ratio $((t-r)/(t+r))/(1/(t+r))^2 = t^2-r^2=\tau^2$. The resulting compactification of $\bbR^{1,d}$ is closely related to the Poincar\'e cylinder (\textit{right}), with the shaded regions in $[\bbM;\scrI]_{\mathrm{par}}$ and the Poincar\'e cylinder being identifiable.
	}
	\label{fig:para_blowup}
\end{figure}

We can convert between de,sc-Sobolev regularity/decay on $\bbO$ and a suitable sort of Sobolev regularity on the Poincar\'e cylinder (\Cref{prop:Sobolev_conversion}, \Cref{prop:ultimate_sobolev_conversion}). Upon doing so, one encounters a loss of finitely many orders of decay at null infinity. 
But, if our functions of interest decay \emph{rapidly} at null infinity, then this loss does not matter, and we can pass between the Poincar\'e cylinder and the octagonal compactification freely. This will apply to solutions of the Klein--Gordon IVP with Schwartz forcing and data. Indeed, we saw in the introduction that such solutions decay rapidly at null infinity (though we have yet to prove this).

This is put to work in \S\ref{subsec:asymptotics} for the production of asymptotic expansions. The reason for the passage to hyperbolic coordinates is that the oscillations 
\begin{equation} 
	e^{\pm i \mathsf{m} \sqrt{t^2-r^2}} = e^{\pm i \mathsf{m} \tau }
\end{equation} 
seen in solutions $u$ of the Klein--Gordon equation depend only on the one coordinate $\tau$, so an appeal to ODE theory is possible. Indeed, in hyperbolic coordinates (see e.g.\ \cite{HormanderNL}) the Minkowski d'Alembertian can be written as 
\begin{equation}
	\square = \partial_\tau^2 + d \tau^{-1} \partial_\tau + \tau^{-2} \triangle_{\bbH^d}, 
	\label{eq:Minkowski_hyper}
\end{equation}
where $\triangle_{\bbH^d}$ is the (positive semidefinite) Laplace--Beltrami operator on the Poincar\'e ball $\bbH^d$ (the level sets of $\tau$ in the Poincar\'e cylinder). Module regularity of $u$ will imply that $\triangle_{\bbH^d} u$ is under control. Consequently, $\tau^{-2} \triangle_{\bbH^d}$ has two orders of decay as $\tau\to\infty$ relative to $u$. So, if $\square u + \mathsf{m}^2 u  = f$, where $f$ is (say) Schwartz, then we can understand the large-$\tau$ behavior of $u$ by inverting the ordinary differential operator on the left-hand side of 
\begin{equation}
	(\partial_\tau^2 + d \tau^{-1} \partial_\tau + \mathsf{m}^2) u = - \tau^{-2} \triangle_{\bbH^d} u + f.
\end{equation}
This yields the desired $e^{\pm i \mathsf{m} \tau}$ behavior. The same argument, with additional error terms, applies if we replace $\square+\mathsf{m}^2$ with the Klein--Gordon operators appearing in our main theorem (see \Cref{prop:integration}).

\begin{remark}
	Unlike the octagonal compactification (cf.\ \Cref{prop:invariance}), the parabolic blowup $[\bbM;\scrI]_{\mathrm{par}}$ is not Poincar\'e invariant, because e.g.\ time-shifts do not extend to homeomorphisms. This is a consequence of the fact that, among the shifted light cones $
	\{t =r + t_0\}$, the one $\{t=r\}$ emanating from the origin is distinguished by the property that it asymptotes to the \emph{interior} of the front face of the blowup in $[\bbM;\scrI]_{\mathrm{par}}$. All other shifted light cones asymptote either to the timelike corner or the spacelike corner of the front face, depending on the sign of $t_0$. Indeed, if $t=r+t_0$, then the hyperbolic coordinate $t^2-r^2$ asymptotes to $\pm \infty$ as $r\to\infty$ if $\pm t_0>0$. 
	\label{rem:parabolic_non-invariant}  
\end{remark}

\begin{remark}
	In understanding the relation between the octagonal compactification (we use $\bbO_0$ here instead of $\bbO$ for simplicity) and the Poincar\'e cylinder, it can be helpful to use a compactification that refines both (in the same way that the octagonal compactification refines both the radial compactification and the Penrose diagram). This compactification is 
	\begin{equation}
		[\bbO_0; \mathrm{cl}_{\bbO_0} \{|t|=r\} \cap \partial \bbO_0 ],
	\end{equation}
	i.e.\ perform a polar blowup of where the light cone hits the boundary of $\bbO_0$. To see that this gives the front face of the parabolic blowup, note that it resolves the ratio $(|t|-r)/(1/(|t|+r)) = t^2-r^2$. Alternatively, one can begin with the parabolic compactification and then perform a blowup of the corners. See \Cref{fig:mutual_refine}.
	\label{rem:mutual_refine}
\end{remark}

\begin{figure}[t!]
	\begin{tikzpicture}
		\draw[dashed] (-2.1,-2) rectangle (2.1,2.5);
		\coordinate (i) at (.75,.75); 
		\fill[lightgray!20] (-1.8,-1.8) -- (1.8,-1.8) arc(0:90:3.6) -- cycle;
		\draw (1.8,-1.8) arc(0:90:3.6);
		\fill[white] (i) circle (1);
		\begin{scope}
			\clip (-1.81,-1.81) -- (1.81,-1.81) arc(0:90:3.62) -- cycle;
			\filldraw[fill=white] (i) circle (1);
		\end{scope}
		\node () at (1,-1.2) {$\bbO_0$};
		\draw[dashed] (-1.8,-1.8) -- (.035,.035);
		\fill[black] (.035,.035) circle (.075);
		\draw[darkcandyapp, ->] (-.05,0.05) -- (-.4,-.3) node[left] {$\frac{1}{t+r}$};
		\draw[darkcandyapp, ->] (-.05,.05) to[out=125, in=280] (-.3,.6) node[left] {$t-r$};
	\end{tikzpicture}
	\begin{tikzpicture}
		\draw[dashed] (-2.1,-2) rectangle (2.1,2.5);
		\coordinate (i) at (.75,.75); 
		\coordinate (j) at (.5,.5); 
		\fill[lightgray!20] (-1.8,-1.8) -- (1.8,-1.8) arc(0:90:3.6) -- cycle;
		\draw (1.8,-1.8) arc(0:90:3.6);
		\begin{scope}
			\clip (-1.81,-1.81) -- (1.81,-1.81) arc(0:90:3.62) -- cycle;
			\filldraw[fill=white, rotate=45] (j) ellipse (1.4 and .4);
		\end{scope}
		\fill[white] (i) circle (1);
		\begin{scope}
			\clip (-1.81,-1.81) -- (1.81,-1.81) arc(0:90:3.62) -- cycle;
			\filldraw[fill=white] (i) circle (1);
		\end{scope}
		\fill[white, rotate=45] (j) ellipse (1.38 and .39);
		\node () at (0,2.1) {$[\bbO_0; \mathrm{cl} \{t=r\} \cap \partial \bbO_0 ]$};
		\draw[darkcandyapp, ->] (-.25,.35) to[out=110, in=-105] (-.3,1) node[left] {$t-r$}; 
		\draw[darkcandyapp, ->] (-.25,.35) to[out=230, in=60] (-.55,-.1) node[left] {$\frac{1}{t^2-r^2}$}; 
	\end{tikzpicture}
	\begin{tikzpicture}
	\draw[dashed] (-2.1,-2) rectangle (2.1,2.5);
	\coordinate (i) at (.75,.75); 
	\fill[lightgray!20] (-1.8,-1.8) -- (1.8,-1.8) arc(0:90:3.6) -- cycle;
	\draw (1.8,-1.8) arc(0:90:3.6);
	\fill[white, rotate=45] (i) ellipse (1.4 and .6);
	\begin{scope}
		\clip (-1.81,-1.81) -- (1.81,-1.81) arc(0:90:3.62) -- cycle;
		\filldraw[fill=white, rotate=45] (i) ellipse (1.4 and .6);
	\end{scope}
	\node () at (.7,-1.2) {$[\bbM;\scrI]_{\mathrm{par}}$};
	\fill[black] (.27,1.12) circle (.075);
	\fill[black] (1.13,.26) circle (.075);
	\draw[darkcandyapp, ->] (.15,1.1) to[out=225, in=65] (-.33,.4) node[left] {$\frac{1}{t^2-r^2}$};
	\draw[darkcandyapp, ->] (.15,1.1) to[out=135, in=-20] (-.5,1.46) node[ below] {$\frac{t-r}{t+r}$};
	\end{tikzpicture}
	\caption{A compactification (\textit{middle}) that refines both $\bbO_0=[\bbM;\scrI]$, $[\bbM;\scrI]_{\mathrm{par}}$, as described in \Cref{rem:mutual_refine}. It arises by blowing up the point at null infinity in $\bbO_0$ where the light cone (the one emanating from the origin) hits the front face (\textit{left}). It can also arise by blowing up the corners of $[\bbM;\scrI]_{\mathrm{par}}$ (\textit{right}). This makes it clear that rapid decay at null infinity means rapid decay at the front faces of $\bbO_0$, $[\bbM;\scrI]_{\mathrm{par}}$, these two notions being equivalent to rapid decay at all three faces of the middle compactification at null infinity. }
	\label{fig:mutual_refine}
\end{figure}

\subsection{The Poincar\'e Cylinder and hyperbolic coordinates}
\label{subsec:Poincare_cylinder}

By the (punctured) \emph{Poincar\'e cylinder}, we mean the mwc 
\begin{equation} 
	\overline{\bbR}\setminus \{0\}\times \bbB^d= ([-\infty,0)\cup (0,+\infty])_\tau \times \bbB^d_{\bfy},
\end{equation} 
equipped with the Lorentzian metric $g_{\mathrm{e,sc}}=-\dd \tau^2 + \tau^{2} g_{\bbH}$, where $\smash{g_{\bbH} \in {}^0\!\operatorname{Sym}^2 T^* \bbB}$ is the metric 
\begin{equation}
	g_{\bbH}(y_1,\ldots ,y_d) = \frac{4\sum_{j=1}^d \dd y_{j}^2}{(1-\sum_{j=1}^d y_{j}^2)^2}
\end{equation}
on the unit ball $\smash{\bbB^d_{\bfy}}$, which makes $\bbB^d$ into the Poincar\'e ball model of hyperbolic space. (Here, the `0' superscript in ${}^0\!\operatorname{Sym}^2 T^* \bbB$ refers to the 0-boundary fibration structure, the defining vector fields of which are those extendable vector fields that vanish at the boundary.) 
It is also useful to refer to the unpunctured cylinder $\overline{\bbR}\times \bbB^d$. 
As the subscript indicates, $g_{\mathrm{e,sc}}$ is an ``e,sc-metric'' on $\overline{\bbR}\times \bbB^d$, where 
\begin{itemize}
	\item the ``e'' (for ``edge'') refers to the structure of the metric at $\smash{\overline{\bbR}\times \partial \bbB^d}$, for which $1-y^2$ is a bdf, where $y^2 = \sum_{j=1}^d y_j^2$, and
	\item the ``sc'' refers to the structure at $\smash{\partial \overline{\bbR}\times \bbB^d}$, for which $\langle\tau\rangle^{-1}$ is a bdf. 
\end{itemize}

The set of \emph{e,sc-vector fields} on the unpunctured cylinder is defined by 
\begin{equation}
\calV_{\mathrm{e,sc}} = \operatorname{span}_{C^\infty(\overline{\bbR} \times \bbB^d)}[ \{\partial_\tau\} \cup \langle \tau \rangle^{-1} \calV_0(\bbB^d)] = \operatorname{span}_{C^\infty(\overline{\bbR}\times \bbB^d)}[\{\partial_\tau\} \cup \{\langle\tau\rangle^{-1} (1-y^2)\partial_{y_j} \}_{j=1}^d],
\end{equation}
where $\calV_0(\bbB^d)$ is the set of vector fields on $\bbB^d$ that vanish at $\partial \bbB^d$ (considered as vector fields on the Poincar\'e cylinder that are constant in $\tau$). Likewise for $\calV_{\mathrm{e,sc}}((\overline{\bbR}\setminus\{0\})\times \bbB^d)$.

Let $\bbX\subseteq \bbO$ denote the relatively open subset 
\begin{equation} 
	\bbX=(\operatorname{cl}_\bbO\{ (t,\bfx)\in \bbR^{1,d}:t^2\geq r^2\})^\circ \subseteq \bbO.
\end{equation} 
This is a non-compact sub-mwc (it includes part of the boundary of $\bbO$, but it is disjoint from $\mathrm{cl}_\bbO\{t^2=r^2\}$). Then, $\bbX^\circ = \{(t,\bfx)\in \bbR^{1,d}:t^2>r^2\}$.  In \Cref{fig:o}, $\bbX$ is the shaded region, including the points on the boundary but not including the light cone (and not including the points in $\bbO$ hit by the light cone).

Consider the map 
\begin{equation} 
	\iota:\bbX_{t,\bfx}^\circ \to \bbR_\tau\backslash \{0\} \times \bbB^{d\circ}_\bfy 
\end{equation}
given by 
\begin{align}
	\begin{split} 
		\tau &= ( t^2 - r^2)^{1/2} \operatorname{sign}(t) \in \bbR\backslash \{0\} , \\
		\bfy &= \bfx ( |t| + (t^2-r^2)^{1/2} )^{-1} \operatorname{sign}(t) \in \bbB^{d\circ}. 
	\end{split} 
	\label{eq:hyperbolic_coordinates_definition} 
\end{align}
The set-theoretic inverse $\iota^{-1}: \bbR_\tau \backslash \{0\}\times \bbB^{d\circ}_{\bfy}\to \bbX^\circ_{t,\bfx}$ of $\iota$ is given by 
\begin{equation}
	\bbR_\tau \backslash \{0\} \times \bbB^{d\circ}_{\bfy}\ni (\tau,\bfy)\mapsto \Big( \tau \Big( \frac{1+y^2}{1-y^2} \Big), \frac{2\tau \bfy}{1-y^2}   \Big).
	\label{eq:inverse_hyperbolic_coordinates}
\end{equation} 
It is apparent from \cref{eq:hyperbolic_coordinates_definition}, \cref{eq:inverse_hyperbolic_coordinates} that $\iota$ and and its set-theoretic inverse are both smooth. So, $\iota$ is a diffeomorphism.
So, given any smooth vector field $V$ on $\bbR_\tau\backslash \{0\} \times \bbB^{d\circ}_\bfy$ we can pull back $V$ by $\iota$ to form a vector field $\iota^* V$ on $\bbX_{t,\bfx}^\circ$, and likewise for differential operators. 

The map $\iota$ defines \emph{hyperbolic coordinates} on $\bbX$. The significance of the e,sc- metric $g_{\mathrm{e,sc}}$ defined above is the following standard fact:
\begin{proposition}
	$g_{\mathrm{e,sc}}=\iota_* g_\bbM$. That is, $g_{\mathrm{e,sc}}$ is just the Minkowski metric rewritten in hyperbolic coordinates.
	\label{prop:hyperbolic=Minkowski}
\end{proposition}
For the reader who has not seen this before, we include a direct computation.
\begin{proof}
	For simplicity, we work in $\{t>0\}$ (and omit $\iota$ from the notation). 
	Then, \cref{eq:inverse_hyperbolic_coordinates} gives 
	\begin{align}
		\dd t = \frac{1}{1-y^2}\Big( (1+y^2)  \dd \tau + \frac{4\tau y \dd y}{1-y^2} \Big), \quad 
		\dd x_j =  \frac{2}{1-y^2}\Big( y_j \dd \tau  + \tau \dd y_j +  \frac{2\tau y_j y \dd y}{1-y^2} \Big).
	\end{align}
	So, 
	\begin{equation}
		\dd t^2 = \frac{1}{(1-y^2)^2} \Big(  (1+y^2)^2  \dd \tau^2 +  8\tau y\frac{1+y^2}{1-y^2} \dd \tau \odot \dd y+ \frac{16\tau^2 y^2 \dd y^2}{(1-y^2)^2} \Big), 
	\end{equation}
	\begin{multline}
		\dd x_j^2 = \frac{4}{(1-y^2)^2} \Big[ y_j^2 \dd \tau^2 + \tau^2 \dd y_j^2 +  \frac{4\tau^2 y_j^2 y^2 \dd y^2}{(1-y^2)^2} + 2 y_j \tau \dd \tau \odot \dd y_j + \frac{4\tau y_j^2 y \dd \tau \odot \dd y}{1-y^2}   \\ +  \frac{4\tau^2 y_jy \dd y_j \odot \dd y}{1-y^2}\Big].
	\end{multline}
	Summing this over $j$ yields
	\begin{equation}
		\sum_{j=1}^d \dd x_j^2 = \frac{4}{(1-y^2)^2} \Big[ y^2 \dd \tau^2 + \tau^2 \sum_{j=1}^d\dd y_j^2 +  \frac{4\tau^2 y^2}{(1-y^2)^2} \dd y^2 + \frac{2\tau (1+y^2)}{1-y^2} \dd \tau \odot \dd y  \Big].
	\end{equation}
	So, 
	\begin{equation}
		g_{\bbM}=-\dd t^2 + \sum_{j=1}^d\dd x_j^2 = -\dd \tau^2  + \frac{4\tau^2}{(1-y^2)^2} \sum_{j=1}^d \dd y_j^2 = g_{\mathrm{e,sc}}.
	\end{equation}
\end{proof}
Consequently, 
$\square = \iota^* \square_{g_{\mathrm{e,sc}}}$, where 
\begin{equation}
	\square_{g_{\mathrm{e,sc}}}=\partial_\tau^2 + d \tau^{-1} \partial_\tau + \tau^{-2} \triangle_{\bbH^d} \in \operatorname{Diff}_{\mathrm{e,sc}}^{2,0,0}(\overline{\bbR} \backslash \{0\} \times \bbB^d)
	\label{eq:misc_34b}
\end{equation}
is the d'Alembertian of the Poincar\'e cylinder. (This is just a more precise way of saying \cref{eq:Minkowski_hyper}.)
Another way of deriving this is using
\begin{align}
	\iota^* \frac{\partial}{\partial \tau} &= \frac{t}{(t^2-r^2)^{1/2}} \partial_t + \frac{1}{(t^2-r^2)^{1/2}} \sum_{j=1}^d x_j \partial_{x_j}, 
	\label{eq:misc_ptf} \\ 
	\iota^*\frac{\partial}{\partial y_j} &= x_j  \Big(\frac{t+(t^2-r^2)^{1/2}}{(t^2-r^2)^{1/2}}\Big) \partial_t + (t+ (t^2-r^2)^{1/2}) \partial_{x_j} + \sum_{k=1}^d \frac{x_jx_k}{(t^2-r^2)^{1/2}} \partial_{x_k}.
	\label{eq:misc_pff}
\end{align}

\begin{proposition}
	The map $\iota$ extends to a smooth map $\smash{\bbX \to \overline{\bbR}\times \bbB^d}$. 
	\label{prop:hyperbolic_smoothness}
\end{proposition}
\begin{proof}
	It suffices to work in local coordinate charts near $\partial \bbO$: 
	\begin{enumerate}[label=(\Roman*)]
		\item (Away from null infinity.) For $c>0$ and $\sigma \in \{-1,+1\}$, in the set $\{t^2>(1+c)r^2,\sigma t>+ 1\}$ we can use the coordinates $\rho = 1/|t|$ and $\hat{\bfx} = \bfx/|t|$,  in terms of which 
		\begin{align}
			\tau^{-1} &= \sigma \rho (1-\lVert \hat{\bfx} \rVert^2)^{-1/2}, \\
			\bfy &= \sigma \hat{\bfx}(1+ (1-\lVert \hat{\bfx} \rVert^2)^{1/2})^{-1}.
		\end{align} 
		In $\{t^2>(1+c)r^2,\sigma t>+ 1\}$, we have $\rho < 1$ and $\lVert \hat{\bfx} \rVert < (1+c)^{-1/2}$, so $\tau^{-1}$ and $\bfy$ are smooth functions of $\rho$ and $\hat{\bfx}$, all the way down to $\rho = 0$. 
		\item (Near the timelike corner of null infinity.) In the set $\{t^2<(1+c)r^2, \sigma t>+1\}$, we instead work with the coordinates $\varrho_{\mathrm{Ff}}$, $\varrho_{\mathrm{nFf}}$, along with $\theta = \bfx/r \in \bbS^{d-1}$. In terms of these, 
		\begin{align}
			\tau^{-1} &= \sigma  \varrho_{\mathrm{nf}}\varrho_{\mathrm{nFf}} 
			\label{eq:misc_161}\\
			\bfy &= \sigma \theta (1-\varrho_{\mathrm{nf}})(1+\varrho_{\mathrm{nf}})^{-1}.
			\label{eq:misc_162}
		\end{align}
		Again, we see that, locally, $\tau^{-1}$ and $\bfy$ are smooth functions of $\varrho_{\mathrm{nf}}$ and $\varrho_{\mathrm{nFf}}$, now all the way to $\partial ([0,\infty)_{\varrho_{\mathrm{nf}}}\times [0,\infty)_{\varrho_{\mathrm{nFf}}})$.
	\end{enumerate}
\end{proof}

We denote the extension of $\iota$ using the same symbol.  

Note that, when $\varrho_{\mathrm{nf}}=0$, \cref{eq:misc_161} says that $\langle \tau \rangle^{-1}=0$ and \cref{eq:misc_162} says that $\lVert \bfy \rVert = 1$. So, $\iota$ maps null infinity to the \emph{corner} of the Poincar\'e cylinder. Compare with \Cref{rem:parabolic_non-invariant}.

As a corollary of the previous proposition,
\begin{equation}
f\in C^\infty(\overline{\bbR} \setminus \{0\}\times \bbB^d) \Rightarrow \iota^* f \in C^\infty(\bbX).
\label{eq:pullback}
\end{equation}
For such $f$, $\iota^* f$ is constant on each component of null infinity. Smoothness on $(\overline{\bbR}\setminus \{0\})\times \bbB^d$ is therefore more restrictive than smoothness on $\bbX$. 
If $f\in C^\infty(\bbX^\circ)$, then even though we may not have $f=\iota^* f_0$ for some $f_0\in C^\infty((\overline{\bbR}\backslash \{0\})\times \bbB^{d\circ} )$ we can still form 
\begin{equation} 
	\iota_* f \in C^\infty((\bbR\backslash \{0\})\times \bbB^{d\circ}),
\end{equation} 
since $\iota$ is a diffeomorphism in the interior. 

We read off of the proof of the previous proposition that
\begin{align}
	\iota^* \langle \tau \rangle^{-1} &\in \varrho_{\mathrm{Pf}} \varrho_{\mathrm{nPf}} \varrho_{\mathrm{nFf}} \varrho_{\mathrm{Ff}} C^\infty(\bbX ; \bbR^+ ).  \label{eq:misc_vr1}
	\intertext{Moreover, from the computation $\iota^*(1-y^2) = 4 \varrho_{\mathrm{nf}} (1+\varrho_{\mathrm{nf}})^{-2}$, one gets } 
	\iota^* (1-y^2) &\in  \varrho_{\mathrm{nPf}} \varrho_{\mathrm{nFf}}  C^\infty(\bbX ; \bbR^+ ).
	\label{eq:misc_vr2}
\end{align}
In $\Omega_{\mathrm{nfTf},\pm ,0}$, 
\begin{align}
	\iota^* \frac{\partial}{\partial \tau} &= \mp \varrho_{\mathrm{nf}} \varrho_{\mathrm{Tf}}^2 \frac{\partial}{\partial\varrho_{\mathrm{Tf}}} \label{eq:misc_613} \\ 
	\iota^* \frac{\partial}{\partial y} &= - \frac{(1+\varrho_{\mathrm{nf}})^2}{2} \frac{\partial}{\partial \varrho_{\mathrm{nf}}} + \frac{(1+\varrho_{\mathrm{nf}})^2}{2} \frac{\varrho_{\mathrm{Tf}}}{\varrho_{\mathrm{nf}}} \frac{\partial}{\partial \varrho_{\mathrm{Tf}}}, \label{eq:misc_614}
\end{align}
where $\partial_y$ is shorthand for $\partial_y = y^{-1}\sum_{j=1}^d y_j \partial_{y_j}$. 
From these formulas and the observation that angular derivatives in $\overline{\bbR}\times \bbB^d$ pull  back to angular derivatives on $\bbX$, we read:
\begin{proposition}
	If $V \in \calV_{\mathrm{e,sc}}(\overline{\bbR}\times \bbB^d)$, then $\iota^* V \in  \calV_{\mathrm{de,sc}}(\bbX)$, and 
	the elements of $\{\iota^* V: V\in \calV_{\mathrm{e,sc}}(\overline{\bbR}\times \bbB^d)\}$
	generate $\calV_{\mathrm{de,sc}}(\bbX)$ as a $C^\infty(\bbX)$-module.  
	\label{prop:pullback_structure}
\end{proposition}
\begin{proof}
	In the $d=1$ case, this follows from \cref{eq:misc_613}, \cref{eq:misc_614}. To handle the $d\geq 2$ case, it just needs to be seen that the angular derivatives 
	\begin{equation} 
		\langle \tau \rangle^{-1}(1-y^2)\partial_{\theta_k} \in \calV_{\mathrm{e,sc}}(\overline{\bbR}\times \dot{\bbB}^d)
	\end{equation} 
	(the dot in $\smash{\dot{\bbB}^d_{\bfy}}$ denoting that the ball has been punctured at $\bfy=0$)
	are, when pulled back via $\iota$, proportional to $\varrho_{\mathrm{Pf}}\varrho_{\mathrm{nPf}}^2\varrho_{\mathrm{nFf}}^2 \varrho_{\mathrm{Ff}} \partial_{\theta_k}$ (using \cref{eq:misc_vr1}, \cref{eq:misc_vr2}), where the constant of proportionality is nonvanishing at the boundary. 
\end{proof}
Let
\begin{equation} 
	\operatorname{Diff}_{\mathrm{e,sc}}^{m,s,\varsigma}(\overline{\bbR} \times \bbB^d) = \langle \tau \rangle^s (1-y^2)^{-\varsigma} \operatorname{Diff}_{\mathrm{e,sc}}^{m,0,0}(\overline{\bbR} \times \bbB^d)
\end{equation} 
denote the set of e,sc-differential operators of order at most $m$, weighted by $\langle \tau \rangle^s$ and $(1-y^2)^{-\varsigma}$. In addition, let 
\begin{equation}
	\operatorname{Diff}_{\mathrm{de,sc}}^{m,( s,s+\varsigma,s+\varsigma,s)}(\bbX) = \operatorname{span}_{C^\infty_{\mathrm{c}}(\bbX)} \operatorname{Diff}_{\mathrm{de,sc}}^{m,( s,s+\varsigma,\infty,s+\varsigma,s)}(\bbO).
\end{equation}
\begin{proposition}
	Given any $L\in \operatorname{Diff}_{\mathrm{e,sc}}^{m,s,\varsigma}(\overline{\bbR} \times \bbB^d)$, 
	\begin{equation} 
		\chi \iota^* L\in  \operatorname{Diff}_{\mathrm{de,sc}}^{m, (s,s+\varsigma,\infty,s+\varsigma,s)}(\bbO),
	\end{equation} 
	for any $\chi \in C_{\mathrm{c}}^\infty(\bbX)$.
	Moreover, the differential operators on $\bbX$ of this form
	generate the  $C^\infty_{\mathrm{c}}(\bbX)$-module $\operatorname{Diff}_{\mathrm{de,sc}}^{m,( s,s+\varsigma,s+\varsigma,s)}(\bbX)$.
	\label{prop:scedge->oct}
\end{proposition}
\begin{proof}
	The subset of $\operatorname{Diff}_{\mathrm{e,sc}}^{\infty,\infty,\infty}(\overline{\bbR}\times \bbB^d)$ consisting of $L$ such that $\chi \iota^* L \in \operatorname{Diff}_{\mathrm{de,sc}}^{m, (s,s+\varsigma,s+\varsigma,s)}(\bbX)$ whenever $\smash{L \in \operatorname{Diff}_{\mathrm{e,sc}}^{m,s,\varsigma}(\overline{\bbR} \times \bbB^d)}$ is a subring of $\smash{\operatorname{Diff}_{\mathrm{e,sc}}^{\infty,\infty,\infty}(\overline{\bbR}\times \bbB^d)}$. So, in order to prove the first part of the proposition, it suffices to check the case
	\begin{equation}
		L \in C^\infty(\overline{\bbR}\times \bbB^d) \cup \{ \langle \tau \rangle^{-s} (1-y^2)^{-\varsigma} : s,\varsigma \in \bbR\} \cup \{\partial_\tau \} \cup \{ \langle \tau \rangle^{-1} (1-y^2) \partial_{y_j}\}_{j=1}^d, 
	\end{equation}
	as the elements of the set on the right-hand side generate $\smash{\operatorname{Diff}_{\mathrm{e,sc}}^{\infty,\infty,\infty}(\overline{\bbR}\times \bbB^d)}$ as a ring. 
	Each of these cases we have already checked, as recorded in \cref{eq:pullback}, \cref{eq:misc_vr1}, \cref{eq:misc_vr2}, and \Cref{prop:pullback_structure}.
	Likewise, the second part of the proposition follows from the second clause of \Cref{prop:pullback_structure}. 
\end{proof}

For each $m\in \bbN$ and $s,\varsigma \in \bbR$, let 
\begin{equation}
	H_{\mathrm{e,sc}}^{m,s,\varsigma}(\overline{\bbR}\times \bbB^d) 
	=\{u \in L^2(\overline{\bbR}\times \bbB^d) : Lu \in L^2(\overline{\bbR}\times \bbB^d,g_{\mathrm{e,sc}})\text{ for all }L\in\operatorname{Diff}_{\mathrm{e,sc}}^{m,s,\varsigma}(\overline{\bbR} \times \bbB^d)\}.
\end{equation}
Also, let $H_{\mathrm{de,sc,loc}}^{m,(s,s+\varsigma,s+\varsigma,s)}(\bbX)$ denote the set of distributions which lie in $H_{\mathrm{de,sc}}^{m,(s,s+\varsigma,\infty,s+\varsigma,s)}(\bbX)$ upon multiplication by an element of $C_{\mathrm{c}}^\infty(\bbX)$. 

\begin{proposition}
	For any $m\in \bbN$ and $s,\varsigma\in \bbR$, 
	\begin{equation} 
		\iota^*H_{\mathrm{e,sc}}^{m,s,\varsigma}(\overline{\bbR}\times \bbB^d) \subseteq H_{\mathrm{de,sc,loc}}^{m,(s,s+\varsigma,s+\varsigma,s)}(\bbX).
	\end{equation} 
	
	Conversely, if $\chi \in C_{\mathrm{c}}^\infty(\bbX)$, then $u\in H_{\mathrm{de,sc}}^{m,(s,s+\varsigma,\infty,s+\varsigma,s)}(\bbO) \Rightarrow  \iota_*  \chi u \in H_{\mathrm{e,sc}}^{m,s,\varsigma}(\overline{\bbR}\times \bbB^d)$.
	\label{prop:Sobolev_conversion}
\end{proposition}

\begin{proof}
	The $m,s,\varsigma=0$ case follows immediately from \Cref{prop:hyperbolic=Minkowski}. 
	The case of general $m\in \bbN$ and $s,\varsigma \in \bbR$ follows from the already considered case along with \Cref{prop:scedge->oct}.
\end{proof}

\subsection{Test Modules}
\label{subsec:test_modules}

There are six test modules $\frakM^\varsigma_\sigma,\frakN_\sigma \subseteq \Psi_{\mathrm{de,sc}}^{1,\mathsf{1}}$ that we use, where $\varsigma,\sigma \in \{-,+\}$. These are defined by 
\begin{align}
	\frakM_{\sigma}^\varsigma &= \{A \in \Psi_{\mathrm{de,sc}}^{1,\mathsf{1}} : \operatorname{char}_{\mathrm{de,sc}}^{1,\mathsf{1}}(A) \supseteq \calR^\varsigma_\sigma \}, 
	\label{eq:misc_mmm}\\
	\frakN_{\sigma} &= \{A \in \Psi_{\mathrm{de,sc}}^{1,\mathsf{1}} : \operatorname{char}_{\mathrm{de,sc}}^{1,\mathsf{1}}(A) \supseteq \operatorname{span}_{\bbR}\calR^\varsigma_\sigma \} \subseteq \frakM_{\sigma}^-\cap \frakM_{\sigma}^+
	\label{eq:misc_nnn}
\end{align}
at the level of sets, and we consider them as $\Psi_{\mathrm{de,sc}}^{0,\mathsf{0}}$-bimodules. 

Recall that the sets $\calR_\sigma^\varsigma$ were defined in the introduction. A more concrete definition can be found in the next section. \Cref{eq:R_alt} gives each component of $\calR$ as the graph of a multiple of $\dd \tau$ over one of $\mathrm{Pf},\mathrm{Ff}$. 
So, in \cref{eq:misc_nnn}, 
\begin{equation} 
	\operatorname{span}_\bbR\calR^\varsigma_\sigma = \bbR  \iota^* \dd\tau( \mathrm{Tf}). 
	\label{eq:misc_bt4}
\end{equation} 
Here, we are interpreting the de,sc-1-form $\iota^*\dd \tau$ as a function $\bbX\to {}^{\mathrm{de,sc}}T^* \bbO$.

We have $\Psi_{\mathrm{de,sc}}^{0,\mathsf{0}}\subset \frakN_\sigma$, so $1\in \frakN_\sigma,\frakM_{\sigma}^\varsigma$.

Let $\frakN_0,\frakM_0$ denote the $C^\infty(\overline{\bbR}\times \bbB^d)$-submodules of $\operatorname{Diff}_{\mathrm{e,sc}}^{1,1,0}(\overline{\bbR}\times \bbB^d)$ given by 
\begin{align}
	\frakN_0 &= \operatorname{span}_{C^\infty(\overline{\bbR}\times \bbB^d)}(\{1,\partial_\tau\}\cup \{ (1-y^2)\partial_{y_j} \}_{j=1}^d ) \supset \operatorname{Diff}^{1,0,0}_{\mathrm{e,sc}}(\overline{\bbR}\times \bbB^d), \label{eq:N0_def} \\ 
	\frakM_0 &= \operatorname{span}_{C^\infty(\overline{\bbR}\times \bbB^d)}(\{1,\tau \partial_\tau\}\cup \{ (1-y^2)\partial_{y_j} \}_{j=1}^d ) = \operatorname{Diff}_{\mathrm{b}}^1(\overline{\bbR}\times \bbB^d).  \label{eq:M0_def}
\end{align}

Fix $\chi\in C_{\mathrm{c}}^\infty(\bbX)$ that is identically equal to $1$ near timelike infinity.
From the computations in the previous subsection, we get: 
\begin{propositionp} 
	For each choice of sign $\sigma \in \{-,+\}$, the $\Psi_{\mathrm{de,sc}}^{0,\mathsf{0}}$-module $\frakN_\sigma$ is generated as both a left and right module by the elements of 
	\begin{equation}
		\{1\}\cup \{ (1-1_{\sigma t>0}\chi) V : V\in \operatorname{Diff}_{\mathrm{de,sc}}^{1,\mathsf{1}} \} \cup \{ 1_{\sigma t>0}\chi \iota^* V_0 : V_0 \in \frakN_0 \}. 
	\end{equation}
	Consequently, $\frakN_\sigma$ admits a finite generating set consisting of differential operators. 
	For each choice of pair of signs $\sigma,\varsigma \in \{-,+\}$, $\frakM_{\sigma}^\varsigma$ is generated (as both a left and right module) by the elements of $\frakN_\sigma$ along with
	\begin{equation} 
		V_\pm =  \chi \tau(\partial_\tau \mp i \mathsf{m}),
	\end{equation} 
	where the sign depends on $\varsigma$.
	\label{prop:module_generators}
\end{propositionp}

By \cref{eq:misc_ptf}, $V_\pm$ is a weighted version of the vector field in \cref{eq:misc_010}.

We fix a finite generating set of $\frakN_\sigma$ consisting of differential operators. Let 
\begin{align}
	\begin{split} 
	A_0 &= 1 \\
	A_j &= 1_{\sigma t>0} \chi \iota^* (1-y^2)\partial_{y_j} \text{ for $j=1,\ldots,d$}, \\
	A_{1+d} &= 1_{\sigma t>0} \chi \iota^* \partial_\tau.	
	\end{split} 
\end{align}
In addition, taking $N\in \bbN$ sufficiently large, let $A_{2+d},\ldots,A_N$ be elements of $ \operatorname{Diff}_{\mathrm{de,sc}}^{1,\mathsf{1}}$ supported away from $\mathrm{Tf}$ that together with $A_0,\ldots,A_{1+d}$ generate $\frakN_\sigma$. 
We notationally suppress the $\sigma$ dependence.

As a consequence of the observation that the commutators of the generators above all lie in the selfsame modules, we get a direct proof of:
\begin{propositionp}
	Each of $\frakM^{\varsigma}_\sigma$ and $\frakN_\sigma$ is closed under the taking of commutators. 
	\label{prop:closure_under_commutators}
\end{propositionp}

For each $k,\kappa \in \bbN$,
\begin{itemize}
	\item let $\frakM^{\kappa,k}_{\varsigma,\sigma}$ denote the $\Psi_{\mathrm{de,sc}}^{0,\mathsf{0}}$-bimodule generated by the de,sc-$\Psi$DOs of the form $L_1 \cdots L_{k+\kappa}$, where $\kappa$ of the $L_\bullet$'s are in $\frakM_{\varsigma,\sigma}$ and the remainder are in $\frakN_\sigma$, and
	\item let $\frakN_\sigma^{k}$ denote the $\Psi_{\mathrm{de,sc}}^{0,\mathsf{0}}$-bimodule generated by the $k$-fold products of members of $\frakN_\sigma$. 
\end{itemize}
Products of the form $B_1\cdots B_k$ for $B_\bullet \in \{A_0,\ldots,A_N\}$ generate $\frakN^{k}_{\sigma}$, as an inductive argument shows. Similarly, products of the form $B_1\cdots B_{\kappa+k}$ for $B_\bullet \in \{A_0,\ldots,A_N,V_\varsigma\}$ with at most $\kappa$ of the $B_\bullet$ being $V_\varsigma$ generate $\frakM^{\kappa,k}_{\varsigma,\sigma}$. 
Let 
\begin{equation} 
	\frakN^k = \frakN_-^k\cap \frakN^k_+.
\end{equation} 
Conventionally, $\frakN^0 = \Psi_{\mathrm{de,sc}}^{0,\mathsf{0}}$. 

Similarly, let $\frakN_0^k,\frakM_0^k$ denote the sets of sums of $k$-fold products of elements of $\frakN_0,\frakM_0$, respectively.

\begin{lemma}
	Fix $\varsigma,\sigma \in \{-,+\}$.
	If $A\in \frakM_{\varsigma,\sigma}^{\kappa,k}$ and $B\in \frakM_{\varsigma,\sigma}^{\varkappa,j}$, then 
	\begin{equation} 
		[A,B] \in 1_{k+j>0}\frakM_{\varsigma,\sigma}^{\kappa+\varkappa,\max\{k+j-1,0\}} + 1_{\kappa+\varkappa>0}\frakM_{\varsigma,\sigma}^{\max\{\kappa+\varkappa-1,0\},k+j}  + \Psi_{\mathrm{de,sc}}^{0,\mathsf{0}}. 
		\label{eq:misc_lol}
	\end{equation} 	
	\label{prop:lol}
\end{lemma}
\begin{proof}
	We proceed inductively: 
	\begin{itemize}
		\item If all of $k,\kappa,j,\varkappa$ are $0$, then the result just states the fact that  $\Psi_{\mathrm{de,sc}}^{0,\mathsf{0}}$  is closed under the taking of commutators. 
		\item Suppose that $\kappa,k$ are both $0$ and $\varkappa+j = 1$ (and the case where $\varkappa,j$ are both $0$ and $\kappa+k=1$ is similar). In this case, the result states that $[A,B] \in \Psi_{\mathrm{de,sc}}^{0,\mathsf{0}}$ for all $A \in \Psi_{\mathrm{de,sc}}^{0,\mathsf{0}}$ and $B\in \frakM^{\varsigma}_\sigma$. This just holds because 
		\begin{equation}
			[\Psi_{\mathrm{de,sc}}^{1,\mathsf{1}},\Psi_{\mathrm{de,sc}}^{0,\mathsf{0}}] \subseteq \Psi_{\mathrm{de,sc}}^{0,\mathsf{0}}. 
		\end{equation}
		\item Suppose that $\kappa+k=1$ and $\varkappa+j=1$. There are three essentially different cases to consider.
		\begin{itemize}
			\item If $\kappa,\varkappa=1$, then the result states that $[A,B] \in \frakM_\sigma^\varsigma$ for all $A,B\in \frakM_\sigma^{\varsigma}$, which is part of \Cref{prop:closure_under_commutators}. 
			\item Likewise, if $k,j=1$, then the result states that $[A,B]\in \frakN_\sigma$ for all $A,B\in \frakN_\sigma$, which is also part of \Cref{prop:closure_under_commutators}. 
			\item If $\kappa,j=1$ (with the case $\varkappa,k=1$ being similar), then the result states that 
			\begin{equation}
				[A,B] \in \frakM_{\varsigma,\sigma}^{1,0} + \frakM^{0,1}_{\varsigma,\sigma} = \frakM^{1,0}_{\varsigma,\sigma} 
			\end{equation}
			for all $A\in \frakM^\varsigma_\sigma$ and $B\in \frakN_\sigma$. This is a weaker statement than the result in the $\kappa,\varkappa=1$ case. 
		\end{itemize}
		\item We now handle the case when $\kappa+k \geq 2$ or $\varkappa+j\geq 2$, proceeding inductively on $\kappa+\varkappa+k+j$. We discuss the case $\kappa+k\geq 2$, and the (overlapping) case $\varkappa+j\geq 2$ is similar. 
		
		Since  $\frakM_{\varsigma,\sigma}^{\kappa,k}$  is spanned by elements of the form
		\begin{equation} 
			A=  1_{\kappa>0} A_0A' + 1_{k>0} A_1 A''
			\label{eq:misc_mfa}
		\end{equation} 
		for $A_0\in \frakM^{\max\{\kappa-1,0\},k}_{\varsigma,\sigma}$, $A'\in \frakM_{\sigma}^\varsigma$, $A_1 \in  \frakM^{\kappa,\max\{k-1,0\}}_{\varsigma,\sigma}$, and $A''\in \frakN_\sigma$, it suffices to prove the claim for such $A$. We have
		\begin{equation}
			[A,B] = 1_{\kappa>0}(A_0[A',B]  +  [A_0,B]A' ) + 1_{k>0} (A_1[A'',B]  +  [A_1,B]A'').
		\end{equation}
		These satisfy 
		\begin{align}
			\begin{split} 
				1_{\kappa>0} A_0 [A',B] &\in 1_{\kappa>0} \frakM^{\max\{\kappa-1,0\},k}_{\varsigma,\sigma}(1_{j>0} \frakM_{\varsigma,\sigma}^{\varkappa+1,\max\{j-1,0\}} + \frakM_{\varsigma,\sigma}^{\varkappa,j} ) \\ 
				&\subseteq 1_{k+j>0} \frakM_{\varsigma,\sigma}^{\kappa+\varkappa,\max\{k+j-1,0\}} + 1_{\kappa+\varkappa>0} \frakM^{\max\{\kappa+\varkappa-1,0\},k+j}, \\
				1_{\kappa>0} [A_0,B] A' &\in 1_{\kappa>0} (1_{k+j>0}\frakM_{\varsigma,\sigma}^{\max\{\kappa+\varkappa-1,0\},\max\{k+j-1,0\} } + 1_{\kappa+\varkappa>1} \frakM_{\varsigma,\sigma}^{\max\{\kappa+\varkappa-2,0\},k+j} ) \frakM_\sigma^{\varsigma} \\ 
				&\subseteq 1_{k+j>0} \frakM^{\kappa+\varkappa,\max\{k+j-1,0\}} + 1_{\kappa+\varkappa>0} \frakM_{\varsigma,\sigma}^{\max\{\kappa+\varkappa-1,0\},k+j}, \\ 
				1_{k>0} A_1 [A'',B] &\in 1_{k>0} \frakM_{\varsigma,\sigma}^{\kappa,\max\{k-1,0\}} ( \frakM_{\varsigma,\sigma}^{\varkappa,j} + 1_{\varkappa>0} \frakM_{\varsigma,\sigma}^{\max\{\varkappa-1,0\},j+1 } ) \\ 
				&\subseteq 1_{k+j>0} \frakM_{\varsigma,\sigma}^{\kappa+\varkappa,\max\{k+j-1,0\}} + 1_{\kappa+\varkappa>0} \frakM_{\varsigma,\sigma}^{\max\{\kappa+\varkappa-1,0\},k+j} \\ 
				1_{k>0} [A_1,B]A'' &\in 1_{k>0} ( 1_{k+j>1} \frakM_{\varsigma,\sigma}^{\kappa+\varkappa,\max\{k+j-2,0\}} +1_{\kappa+\varkappa>0} \frakM_{\varsigma,\sigma}^{\max\{\kappa+\varkappa-1,0\}, \max\{k+j-1,0\} } ) \frakN_\sigma \\
				&\subseteq 1_{k+j>0} \frakM_{\varsigma,\sigma}^{\kappa+\varkappa,\max\{k+j-1,0\}} + 1_{\kappa+\varkappa>0} \frakM_{\varsigma,\sigma}^{\max\{\kappa+\varkappa-1,0\},k+j}. 
			\end{split} 
		\end{align}
		So, we can conclude that \cref{eq:misc_lol} holds. 
	\end{itemize} 
\end{proof}
Let 
\begin{equation} 
	\Psi_{\mathrm{de,sc}}^{m,\mathsf{s}} \frakM^{\kappa,k}_{\varsigma,\sigma} = \operatorname{span}_{\bbC} \{ AB: A\in \Psi_{\mathrm{de,sc}}^{m,\mathsf{s}}, B\in \frakM^{\kappa,k}_{\varsigma,\sigma} \},
\end{equation} 
and analogously  
\begin{equation}
	\frakM^{\kappa,k}_{\varsigma,\sigma} \Psi_{\mathrm{de,sc}}^{m,\mathsf{s}} = \operatorname{span}_{\bbC} \{ BA: A\in \Psi_{\mathrm{de,sc}}^{m,\mathsf{s}}, B\in \frakM^{\kappa,k}_{\varsigma,\sigma} \} .
\end{equation}

\begin{lemma}
	\hfill 
	\begin{itemize}
		\item If $A\in \Psi_{\mathrm{de,sc}}^{m,\mathsf{s}}$ and $B\in \frakM^{\kappa,k}_{\varsigma,\sigma}$, then $[A,B] \in \Psi_{\mathrm{de,sc}}^{m,\mathsf{s}} (1_{\kappa>0}\frakM^{\max\{\kappa-1,0\},k}_{\varsigma,\sigma} + 1_{k>0} \frakM_{\varsigma,\sigma}^{\kappa,\max\{k-1,0\}} + \Psi_{\mathrm{de,sc}}^{0,\mathsf{0}} )$. 
		\item $\Psi_{\mathrm{de,sc}}^{m,\mathsf{s}} \frakM^{\kappa,k}_{\varsigma,\sigma}=\frakM^{\kappa,k}_{\varsigma,\sigma}\Psi_{\mathrm{de,sc}}^{m,\mathsf{s}}$.		
	\end{itemize}
	\label{prop:module_commutator_lemma}
\end{lemma}
\begin{proof}
	We prove the result via simultaneous induction on $\kappa,k\in \bbN$. The case $\kappa+k\leq 1$ is an immediate consequence of the algebraic properties of the de,sc-calculus.

	Suppose now that $\kappa,k\in \bbN$ satisfy $\kappa+k\geq 2$, and suppose that we have proven the result for all pairs $\kappa_0,k_0$ such that $\kappa_0+ k_0 < \kappa+k$.
	\begin{itemize}
		\item For $B\in \frakM_{\varsigma,\sigma}^{\kappa,k}$, we can write, like \cref{eq:misc_mfa}, 
		\begin{equation} 
			B=  1_{\kappa>0} B_0B' + 1_{k>0} B_1 B''
			\label{eq:misc_mfr}
		\end{equation} 
		for $B_0\in \frakM^{\max\{\kappa-1,0\},k}_{\varsigma,\sigma}$, $B'\in \frakM_{\sigma}^\varsigma$, $B_1 \in  \frakM^{\kappa,\max\{k-1,0\}}_{\varsigma,\sigma}$, and $B''\in \frakN_\sigma$. For $A\in \Psi_{\mathrm{de,sc}}^{m,\mathsf{s}}$, 
		\begin{equation}
			[A,B] = 1_{\kappa>0}([A,B_0] B' + B_0 [A,B'] ) + 1_{k>0} ([A,B_1] B'' + B_1 [A,B'']).
		\end{equation}
		The terms on the right-hand side satisfy 
		\begin{align}
			\begin{split} 
				1_{\kappa>0}[A,B_0] B'&\in 1_{\kappa>0}\Psi_{\mathrm{de,sc}}^{m,\mathsf{s}}( 1_{\kappa>1}\frakM_{\varsigma,\sigma}^{\max\{\kappa-2,0\},k } + 1_{k>0}\frakM_{\varsigma,\sigma}^{\max\{\kappa-1,0\},\max\{k-1,0\} } ) \frakM_\sigma^{\varsigma}\\ &\subseteq 1_{\kappa>0} \Psi_{\mathrm{de,sc}}^{m,\mathsf{s}} \frakM_{\varsigma,\sigma}^{\max\{\kappa-1,0\}, k} + 1_{k>0}\Psi_{\mathrm{de,sc}}^{m,\mathsf{s}}\frakM_{\varsigma,\sigma}^{\kappa,\max\{k-1,0\}}, \\ 
				1_{\kappa>0} B_0 [A,B']&\in  1_{\kappa>0}\frakM_{\varsigma,\sigma}^{\max\{\kappa-1,0\},k} \Psi_{\mathrm{de,sc}}^{m,\mathsf{s}} = 1_{\kappa>0}\Psi_{\mathrm{de,sc}}^{m,\mathsf{s}} \frakM_{\varsigma,\sigma}^{\max\{\kappa-1,0\},k}, \\
				1_{k>0}[A,B_1] B''&\in 1_{k>0} \Psi_{\mathrm{de,sc}}^{m,\mathsf{s}} ( 1_{\kappa>0}\frakM_{\varsigma,\sigma}^{\max\{\kappa-1,0\},\max\{k-1,0\}} + 1_{k>1} \frakM_{\varsigma,\sigma}^{\kappa,\max\{k-2,0\}} ) \frakN_\sigma\\ 
				&\subseteq  1_{\kappa>0} \Psi_{\mathrm{de,sc}}^{m,\mathsf{s}} \frakM_{\varsigma,\sigma}^{\max\{\kappa-1,0\},k} + 1_{k>0} \Psi_{\mathrm{de,sc}}^{m,\mathsf{s}} \frakM_{\varsigma,\sigma}^{\kappa,\max\{k-1,0\}}, \\ 
				1_{k>0}B_1 [A,B'']&\in 1_{k>0} \frakM_{\varsigma,\sigma}^{\kappa,\max\{k-1,0\}} \Psi_{\mathrm{de,sc}}^{m,\mathsf{s}}= 1_{k>0} \Psi_{\mathrm{de,sc}}^{m,\mathsf{s}}  \frakM_{\varsigma,\sigma}^{\kappa,\max\{k-1,0\}},
			\end{split} 
		\end{align}
		where we used the inductive hypothesis. So, 
		\begin{equation}
			[A,B] \in \Psi_{\mathrm{de,sc}}^{m,\mathsf{s}} (1_{\kappa>0}\frakM^{\max\{\kappa-1,0\},k}_{\varsigma,\sigma} + 1_{k>0} \frakM_{\varsigma,\sigma}^{\kappa,\max\{k-1,0\}} + \Psi_{\mathrm{de,sc}}^{0,\mathsf{0}} ).
			\label{eq:misc_zoz}
		\end{equation}
		\item 
		If $A\in \Psi_{\mathrm{de,sc}}^{m,\mathsf{s}}$ and $B\in \frakM^{\kappa,k}_{\varsigma,\sigma}$, then, using \cref{eq:misc_zoz}, 
		\begin{align}
			AB = BA + [A,B] &\in \frakM^{\kappa,k}_{\varsigma,\sigma}\Psi_{\mathrm{de,sc}}^{m,\mathsf{s}} + \Psi_{\mathrm{de,sc}}^{m,\mathsf{s}} (\frakM^{\max\{\kappa-1,0\},k}_{\varsigma,\sigma}+\frakM^{\kappa,\max\{k-1,0\}}_{\varsigma,\sigma})=   \frakM^{\kappa,k}_{\varsigma,\sigma} \Psi_{\mathrm{de,sc}}^{m,\mathsf{s}} \\  
			BA = AB - [A,B] &\in \Psi_{\mathrm{de,sc}}^{m,\mathsf{s}} \frakM^{\kappa,k}_{\varsigma,\sigma}  + \Psi_{\mathrm{de,sc}}^{m,\mathsf{s}} (\frakM^{\max\{\kappa-1,0\},k}_{\varsigma,\sigma}+\frakM^{\kappa,\max\{k-1,0\}}_{\varsigma,\sigma}) = \Psi_{\mathrm{de,sc}}^{m,\mathsf{s}} \frakM^{\kappa,k}_{\varsigma,\sigma}, 
		\end{align} 
		which together show that 
		\begin{equation}
			\Psi_{\mathrm{de,sc}}^{m,\mathsf{s}} \frakM^{\kappa,k}_{\varsigma,\sigma}=\frakM^{\kappa,k}_{\varsigma,\sigma}\Psi_{\mathrm{de,sc}}^{m,\mathsf{s}}. 
		\end{equation}
	\end{itemize}	 
	
\end{proof}
Consequently, combining \Cref{prop:lol} and \Cref{prop:module_commutator_lemma}: 
\begin{corollary}
	If $A \in \Psi_{\mathrm{de,sc}}^{m,\mathsf{s}}$, $B\in \frakM_{\varsigma,\sigma}^{\kappa,k}$, and $C\in \frakM_{\varsigma,\sigma}^{\varkappa,j}$, then  
	\begin{multline}
		[AB,C] = A[B,C] + [A,C] B \in \Psi_{\mathrm{de,sc}}^{m,\mathsf{s}} (1_{k+j>0}\frakM_{\varsigma,\sigma}^{\kappa+\varkappa,\max\{k+j-1,0\}} + 1_{\kappa+\varkappa>0}\frakM_{\varsigma,\sigma}^{\max\{\kappa+\varkappa-1,0\},k+j} \\  + \Psi_{\mathrm{de,sc}}^{0,\mathsf{0}}).
	\end{multline}
	\label{cor:module_commutator_lemma} 
\end{corollary}

Let $P$ denote an arbitrary de,sc-$\Psi$DO of the form $P=\square+\lambda+R$
for some $R\in \Psi_{\mathrm{de,sc}}^{2,-\mathsf{1}}$ and $\lambda\in \bbC$. We now show that the module $\frakN_\pm$ is ``$P$-critical'' at $\smash{\calR_\pm = \calR_\pm^+\cup \calR_\pm^-}$:
\begin{proposition}
	There exists, for each $\sigma\in \{-,+\}$, a collection $\{C_{j,k}\}_{j,k=1}^N \subset \Psi_{\mathrm{de,sc}}^{1,\mathsf{0}}$, depending on $R$, such that 
	\begin{itemize}
		\item $i \varrho^{-1} [P,A_j] = \sum_{k=0}^N C_{j,k} A_k$, 
		\item if $k\neq 0$, $\sigma_{\mathrm{de,sc}}^{1,\mathsf{0}}(C_{j,k}) = 0$ on $\calR_\sigma$,
	\end{itemize}
	where $\varrho = \varrho_{\mathrm{Pf}}\varrho_{\mathrm{nPf}}\varrho_{\mathrm{Sf}}\varrho_{\mathrm{nFf}}\varrho_{\mathrm{Ff}}$.
\label{prop:criticality}
\end{proposition}
\begin{remark*}
	This version of ``$P$-criticality'' is a slightly weaker statement than the one offered in \cite{HassellETAL} in the context of the Schr\"odinger--Helmholtz equation, as they require  $\smash{\sigma_{\mathrm{de,sc}}^{1,\mathsf{0}}(C_{j,0})} = 0$ on $\calR_\pm$, but this is unnecessary for the proof of radial point and propagation estimates involving module regularity in later sections. Proving only this weaker statement seems to allow us to require slightly less of $R$. 
\end{remark*}
\begin{proof}
	We only consider the `$+$' case explicitly. 
	
	It suffices to consider the case $\lambda,R=0$. Indeed, $\lambda$ clearly does not matter. Reducing to the $R=0$ case requires an argument: 
	\begin{itemize}
		\item If we have found $C_{j,k}(\square)$ satisfying the conclusion of the proposition when $P=\square$, then for general $P=\square+ R$, 
		\begin{equation}
			i \varrho^{-1} [P,A_j] = i \varrho^{-1} [\square,A_j] + i \varrho^{-1} [R,A_j] = R_j+ \sum_{k=0}^N C_{j,k}(\square) A_k, 
			\label{eq:misc_iva}
		\end{equation}
		where $R_j = i \varrho^{-1} [R,A_j] \in  \Psi_{\mathrm{de,sc}}^{2,\mathsf{0}}$. Choose $G\in \Psi_{\mathrm{de,sc}}^{-\infty,\mathsf{0}}$ satisfying 
		\begin{equation}
			\operatorname{WF}_{\mathrm{de,sc}}'(1-G) \cap \calR^+_+=\varnothing,
		\end{equation} 
		which we can do because $\calR^+_+$ is disjoint from fiber infinity.  
		Let $\Lambda \in \Psi_{\mathrm{de,sc}}^{1,\mathsf{0}}$ be elliptic and $\Lambda_{-1} \in \Psi_{\mathrm{de,sc}}^{-1,\mathsf{0}}$ be a parametrix for $\Lambda$, so that $1=\Lambda \Lambda_{-1}+E$ for $E\in \smash{\Psi_{\mathrm{de,sc}}^{-\infty,-\infty}}$.
		We can rewrite \cref{eq:misc_iva} as
		\begin{align}
			\begin{split} 
			i \varrho^{-1} [P,A_j] = (1-G) \Lambda&\Lambda_{-1} (1-G) R_j \\ &+ (1-(1-G) \Lambda \Lambda_{-1} (1-G) )R_j+ \sum_{k=0}^N C_{j,k}(\square) A_k.
			\end{split}
			\label{eq:misc_ivb}
		\end{align}
		Since $\Lambda_{-1} (1-G) R_j \in \Psi_{\mathrm{de,sc}}^{1,\mathsf{0}}\subseteq \frakN_+$, we can write
		\begin{equation} 
			\Lambda_{-1} (1-G) R_j = \sum_{k=0}^N D_{j,k} A_k
			\label{eq:misc_eva}
		\end{equation} 
		for some $D_{j,k} \in \Psi_{\mathrm{de,sc}}^{0,\mathsf{0}}$. 
		Also, 
		\begin{align}
			\begin{split} 
			(1 - (1-G)\Lambda \Lambda_{-1} (1-G)) R_j &= ((1-G)E (1-G)  + 2G-G^2)R_j \\
			&\in \Psi_{\mathrm{de,sc}}^{-\infty,\mathsf{0}}. 
			\end{split}
		\end{align}
		By \cref{eq:misc_ivb} and \cref{eq:misc_eva}, $i \varrho^{-1} [P,A_j] = \sum_{k=0}^N C_{j,k}(P) A_k$ holds for 
		\begin{equation}
			C_{j,k}(P) = 
			\begin{cases}
				C_{j,0}(\square) + (1-G)\Lambda D_{j,0} + (1 - (1-G)\Lambda \Lambda_{-1} (1-G)) R_j   & (k=0) \\
				C_{j,k}(\square) + (1-G)\Lambda D_{j,k} & (k\neq 0),
			\end{cases}
		\end{equation}
		this being in $\Psi_{\mathrm{de,sc}}^{1,\mathsf{0}}$. 
		Since the essential support of $1-G$ is disjoint from $\calR_\sigma$, we have, if $k\neq 0$, $\sigma_{\mathrm{de,sc}}^{1,\mathsf{0}}(C_{j,k}(P)) = 0$ on $\calR_\sigma$. 
	\end{itemize}
	So, it suffices to consider the case $P=\square$.

	Suppose now that $P=\square$. We focus on the situation for $\frakN_+$ near $\mathrm{Ff}\cap \mathrm{nFf}$, with the situation in the other regions either similar (e.g.\ at $\mathrm{nPf}\cap \mathrm{Pf}$) or strictly easier (e.g.\ the spacelike corner of null infinity). 
	\begin{itemize}
		\item First consider $A\in \{A_{2+d},\ldots,A_N\}$. Then, $\varrho^{-1}[ P,A]\in \operatorname{Diff}_{\mathrm{de,sc}}^{2,\mathsf{1}}$ and is supported away from $\mathrm{Ff}$. We can write this as 
		\begin{equation}
			\Lambda (\Lambda_{-1} \varrho^{-1}[ P,A]) + E \varrho^{-1}[ P,A].
		\end{equation}
		As $\Lambda_{-1} \varrho^{-1}[ P,A] \in \Psi_{\mathrm{de,sc}}^{1,\mathsf{1}}$ and is microsupported away from $\calR_+$, it lies in $\frakN_+$. Thus, we can write 
		\begin{equation}
			\Lambda_{-1} \varrho^{-1}[ P,A] = \sum_{k=0}^N C^{(1)}_k A_k 
		\end{equation}
		for some $C^{(1)}_k \in \Psi_{\mathrm{de,sc}}^{0,\mathsf{0}}$ which we can choose to be microsupported away from $\calR_+$. Thus, $i\varrho^{-1}[ P,A] = \sum_{k=0}^N C_k A_k$
		for $C_k \in \Psi_{\mathrm{de,sc}}^{1,\mathsf{0}}$ given by 
		\begin{equation}
			C_k = 
			\begin{cases}
				i\Lambda C^{(1)}_0 + iE \varrho^{-1} [P,A] & (k=0), \\ 
				i\Lambda C^{(1)}_k & (k\neq 0),
			\end{cases}
		\end{equation}
		and these are microsupported away from $\calR_+$. 
		\item We now check $A_1,\ldots,A_{d}$ (leaving only $A_{1+d}$ to be done). Using \cref{eq:misc_34b} and the fact that $\triangle_{\bbH^d}$ is a 0-operator (and the fact that $A_1,\ldots,A_d$ are all 0-operators on the Poincar\'e ball as well), 
		\begin{equation}
			\varrho^{-1} [P,A_j] \in \varrho^{-1} \chi_0 \iota^* \tau^{-2} \operatorname{Diff}^2_{0}(\bbB^d) \subset \varrho_{\mathrm{Pf}}\varrho_{\mathrm{nPf}}\varrho_{\mathrm{Sf}}^\infty \varrho_{\mathrm{nFf}} \varrho_{\mathrm{Ff}} \chi_0 \iota^* \frakN_0^2
			\label{eq:misc_203}
		\end{equation}
		for each $j\in \{1,\ldots,d\}$, where $\chi_0 \in C_{\mathrm{c}}^\infty(\bbX)$ is identically $1$ on the support of $\chi$. 
		Observe that $\chi_0 \iota^*\frakN_0\subset \frakN_\sigma$.  
		This implies that, for $j\in \{1,\ldots,d\}$, we can write $i\varrho^{-1}[P,A_j] =\sum_{k=0}^N C_{j,k}^{(2)} A_k$ for 
		\begin{equation} 
			C_{j,k}^{(2)} \in \varrho_{\mathrm{Pf}}\varrho_{\mathrm{nPf}}\varrho_{\mathrm{Sf}}^\infty \varrho_{\mathrm{nFf}} \varrho_{\mathrm{Ff}}  \chi_0  \frakN_\sigma \subseteq \Psi_{\mathrm{de,sc}}^{1,\mathsf{0}}.
			\label{eq:misc_204}
		\end{equation} 
		The factor of $\frakN_\sigma$ in \cref{eq:misc_204} means that $\sigma_{\mathrm{de,sc}}^{1,\mathsf{0}}(C_{j,k})$ is vanishing on $\calR_\sigma$. 
		\item 
		On the other hand, 
		\begin{equation}
			[ P, \iota^* \partial_\tau] =   \iota^*[ \square_{g_{\mathrm{e,sc}}}, \partial_\tau] =  \iota^* \Big[\frac{d}{\tau^2} \frac{\partial}{\partial \tau}  + \frac{2}{\tau^2} \triangle_{\bbH^d} \Big].
		\end{equation}
		\begin{itemize}
			\item 	The contribution of $\iota^* (\tau^{-2} \triangle_{\bbH^d})$ to $\varrho^{-1} [P,A_{1+d}]$ has the same form as \cref{eq:misc_203}.
			\item On the other hand, the contribution from $\iota^* (\tau^{-2} \partial_\tau)$ lies in $\varrho_{\mathrm{Pf}}\varrho_{\mathrm{nPf}}\varrho_{\mathrm{Sf}}^\infty \varrho_{\mathrm{nFf}} \varrho_{\mathrm{Ff}}  \chi_0  \frakN_\sigma$. This then has the form $\sum_{k=0}^N C_{j,k}^{(3)} A_k$ for $C_{j,k}^{(3)} \in \Psi_{\mathrm{de,sc}}^{0,-\mathsf{1}}$, which satisfies 
			\begin{equation}
				\sigma_{\mathrm{de,sc}}^{1,\mathsf{0}}(C_{j,k})=0.
			\end{equation}
		\end{itemize}
		So, $\varrho^{-1} [P,A_{1+d}]$ has the desired form.
	\end{itemize}
\end{proof}

From the formula for the d'Alembertian in hyperbolic coordinates, we get the following formula for $\square+\mathsf{m}^2$ that will be used in \S\ref{sec:radialpoint}: 
\begin{propositionp}
	\begin{equation} 
		\square  + \mathsf{m}^2 = \chi\cdot  (\iota^*\tau^{-2}) [ V_{\pm} V_{\mp} +  (d-1) V_\pm\pm\tau i \mathsf{m} (d-2)] + \varrho^2 R_\pm
		\label{eq:misc_o8b}
	\end{equation} 
	for some $R_\pm\in \frakN^2$.
	\label{prop:square_module_form}
\end{propositionp}
\begin{proof} We check the $+$ case, the $-$ case following via complex conjugation.
	Let $\chi\in C_{\mathrm{c}}^\infty(\bbX)$ be identically $1$ near timelike infinity, and decompose  
	\begin{equation}
		\square+\mathsf{m}^2 = \chi(\square+\mathsf{m}^2) + (1-\chi) (\square+\mathsf{m}^2). 
	\end{equation}
	The second term satisfies $(1-\chi) (\square+\mathsf{m}^2) \in  \varrho^2 \frakN^2$.
	On the other hand, $\chi (\square+\mathsf{m}^2) = \chi \iota^* (\square_{g_{\mathrm{e,sc}}}+\mathsf{m}^2)$, and the right-hand side is computed as 
	\begin{equation}
		\chi \iota^* (\square_{g_{\mathrm{e,sc}}}+\mathsf{m}^2)= \chi \cdot (\iota^* \tau^{-2} )\Big[V_+ V_- + (d-1) V_+ + \iota^* \triangle_{\bbH^d} + \tau i \mathsf{m}(d-2)\Big]. 
	\end{equation}
	Using the computations in the previous subsection, $(\iota^* \tau^{-2}) \chi \iota^*\triangle_{\bbH^d} \in \varrho^2 \frakN^2$ as well. We conclude that \cref{eq:misc_o8b} holds with $R_+ = \varrho^{-2}( (1-\chi)(\square+\mathsf{m}^2) + (\iota^* \tau^{-2}) \chi\iota^* \triangle_{\bbH^d})$.
\end{proof}

For each $m\in \bbR$, $\mathsf{s}\in \bbR^5$, and $\kappa,k\in \bbN$, let 
\begin{equation} 
	H_{\mathrm{de,sc};\varsigma,\sigma }^{m,\mathsf{s};\kappa,k} = H_{\mathrm{de,sc};\varsigma,\sigma }^{m,\mathsf{s};\kappa,k}(\bbO)
\end{equation} 
denote the Sobolev space consisting of elements of $H_{\mathrm{de,sc}}^{m,\mathsf{s}}$ which remain in this space under the application of any element of $\frakM_{\varsigma,\sigma}^{\kappa,k}$; see \cref{eq:final_sob_def}. 

Correspondingly, for $m,s,\varsigma \in \bbR$ and $\kappa,k\in \bbN$, let 
\begin{equation} 
	H_{\mathrm{e,sc}}^{m,s,\varsigma;\kappa,k}= H_{\mathrm{e,sc}}^{m,s,\varsigma;\kappa,k}(\overline{\bbR}\times \bbB^d)
\end{equation}
be the set of elements $u\in H_{\mathrm{e,sc}}^{m,s,\varsigma}(\overline{\bbR}\times \bbB^d)$ such that $Au$ is also in this space when $A$ is a product of at most $\kappa$ elements of $\frakM_0$ and $k$ elements of $\frakN_0$. 

\begin{proposition}
	Fix a sign $\varepsilon \in \{-,+\}$. 
	If $\chi \in C_{\mathrm{c}}^\infty(\bbX)$ has support in $\mathrm{cl}_\bbX\{ \varepsilon t>0\}$ and $u\in \calS'$, then, for any $s,\varsigma,\sigma \in \bbR$, 
	\begin{equation} 
		\chi u \in H_{\mathrm{de,sc};\pm,\varepsilon }^{m,(s,s+\varsigma,\sigma,s+\varsigma,s);\kappa,k}(\bbO)
	\end{equation} 
	if and only if $e^{\mp i  \mathsf{m} \tau } \iota_* \chi u \in H_{\mathrm{e,sc}}^{m,s,\varsigma;\kappa,k}(\overline{\bbR}\times \bbB^d)$. 
	
	In particular, if $u\in H_{\mathrm{de,sc};\pm,\varepsilon }^{m,(s,\infty,\infty,\infty,s);\infty,\infty}(\bbO)$, then $e^{\mp i  \mathsf{m} \tau } \iota_* \chi u \in H_{\mathrm{e,sc}}^{m,s,\infty;\infty;\infty}(\overline{\bbR}\times \bbB^d)$.
	\label{prop:ultimate_sobolev_conversion}
\end{proposition}
\begin{proof}
	The $\kappa=k=0$ case of this follows from \Cref{prop:Sobolev_conversion} and the observation that multiplication by $e^{\mp i \mathsf{m}\tau}$ defines an isomorphism of e,sc-Sobolev spaces. This latter fact follows from the $m,s,\varsigma=0$ case (which is just that multiplication by $e^{\mp i \mathsf{m} \tau}$ is an $L^2$-isometry) and the observation that if $V \in \calV_{\mathrm{e,sc}}$, then 
	\begin{equation} 
		e^{\pm i \mathsf{m} \tau} Ve^{\mp i \mathsf{m}\tau} \in \operatorname{Diff}_{\mathrm{e,sc}}^{1,0,0}(\overline{\bbR}\times \bbB^d). 
	\end{equation} 
	
	To deduce the case where $\kappa>0$ or $k>0$, it suffices to note that the elements of $\frakN_{\varepsilon}$ push forward to elements of $\frakN_0$ via $\iota$, and elements of $\frakM_{-,\varepsilon},\frakM_{+,\varepsilon}$, when conjugated by $\exp(\mp i \mathsf{m}\tau)$, push forward to elements of $\frakM_0$. For example, if 
	\begin{equation} 
		\chi u \in H_{\mathrm{de,sc};\pm,\varepsilon }^{m,(s,s+\varsigma,\sigma,s+\varsigma,s);1,0}(\bbO),
	\end{equation} 
	then, considering an arbitrary element 
	\begin{equation} 
		L = a \tau \partial_\tau + \sum_{j=1}^d a_j (1-y^2)\partial_{y_j}
	\end{equation} 
	of $\frakM_0$, where $a,a_j \in C^\infty(\overline{\bbR}\times \bbB^d)$: for any $a_0\in C^\infty(\overline{\bbR}\times \bbB^d)$,
	\begin{equation} 
		(a_0+L) e^{\mp i \mathsf{m} \tau} \iota_* \chi u = e^{\mp i \mathsf{m} \tau} \iota_* \Big[ \Big(\iota^* a_0+(\iota^* a) V_\pm  + \sum_{j=1}^d (\iota^* a_j) \iota^*\Big((1-y^2) \frac{\partial}{\partial y_j} \Big) \Big)\chi u \Big], 
	\end{equation}
	so that 
	\begin{equation}
		(a_0+L) e^{\mp i \mathsf{m} \tau} \iota_* \chi u \in e^{\mp i \mathsf{m} \tau} \iota_* H_{\mathrm{de,sc};\pm,\varepsilon}^{m,(s,s+\varsigma,\sigma,s+\varsigma,s)} \subseteq  e^{\mp i \mathsf{m} \tau}  H_{\mathrm{e,sc}}^{m,s,\varsigma} \subseteq  H_{\mathrm{e,sc}}^{m,s,\varsigma}.
	\end{equation}
	So, we can conclude that $e^{\mp i  \mathsf{m} \tau } \iota_* \chi u \in H_{\mathrm{e,sc}}^{m,s,\varsigma;1,0}(\overline{\bbR}\times \bbB^d)$. Conversely, if 
	\begin{equation}
		e^{\mp i  \mathsf{m} \tau } \iota_* \chi u \in H_{\mathrm{e,sc}}^{m,s,\varsigma;1,0}(\overline{\bbR}\times \bbB^d), 
	\end{equation}
	then, for any $a_0,a,a_1,\ldots,a_d \in C^\infty(\bbX)$, 
	\begin{equation}
		\Big[ a_0 + a V_\pm +\sum_{j=1}^d a_j A_j \Big] \chi u = e^{\pm i \mathsf{m}\iota^* \tau} \Big[ a_0 \iota^* + a \iota^* \Big( \tau \frac{\partial}{\partial \tau} \Big) +\sum_{j=1}^d a_j \iota^* (1-y^2)\partial_{y_j} \Big]  e^{\mp i \mathsf{m}\tau} \iota_* \chi u \Big], 
	\end{equation}
	so that 
	\begin{equation}
			\Big[ a_0 + a V_\pm +\sum_{j=1}^d a_j A_j \Big] \chi u  \in  C^\infty_{\mathrm{c}}(\bbX) \iota^*  H_{\mathrm{e,sc}}^{m,s,\varsigma} \subseteq H_{\mathrm{de,sc}}^{m,(s,s+\varsigma,\sigma,s+\varsigma,s)}.
	\end{equation}
	So, we conclude that $\chi u \in H_{\mathrm{de,sc};\pm,\varepsilon }^{m,(s,s+\varsigma,\sigma,s+\varsigma,s);1,0}(\bbO)$. 
	The proof for general $\kappa,k$ is analogous. 
\end{proof}

\subsection{Asymptotics from Module Regularity}
\label{subsec:asymptotics}

We apply standard notation to refer to spaces of extendable distributions on $\overline{\bbR}\times \bbB^d$ with partial polyhomogeneous expansions. So, 
\begin{align}
	\calA^{0,\infty}(\overline{\bbR}\times \bbB^d) &= \{u\in L^\infty : Lu \in  (1-y^2)^{\varsigma}L^\infty(\overline{\bbR}\times \bbB^d)\;\forall\;L\in \operatorname{Diff}_{\mathrm{b}}(\overline{\bbR}\times \bbB^d),\varsigma \in \bbR \} \\
	\calA^{s,\infty}(\overline{\bbR}\times \bbB^d) &= \langle \tau \rangle^{-s}  \calA^{0,\infty}(\overline{\bbR}\times \bbB^d).
\end{align}
(Note here that, in $\calA^{s,\infty}$, higher $s$ means higher decay as $|\tau|\to\infty$, and the $\infty$ means Schwartz behavior at the side of the cylinder.)
In general, if $\calE \subseteq \bbC\times \bbN$ is an index set, we write 
\begin{equation}
	\calA^{(\calE,\alpha), \infty }(\overline{\bbR}\times \bbB^d) 
\end{equation}
to denote the Fr\'echet space of distributions which have partial polyhomogeneous expansions with index set $\calE$ at $(\partial \overline{\bbR})\times \bbB^d$ (with a conormal remainder of order $\alpha$), the terms of which are Schwartz at $\overline{\bbR}\times \partial \bbB^d$. We can also work with LCTVSs of functions lying locally in one of these spaces, in specified open sets. For example, it is often convenient to exclude $\{\tau=0\}$, so we write 
\begin{equation}
	\calA^{(\calE,\alpha), \infty }_{\mathrm{loc}}((\overline{\bbR}\backslash \{0\})\times \bbB^d) 
\end{equation}
to denote the set of functions $u:(\overline{\bbR}\backslash \{0\})\times \bbB^d\to \bbC$ such that $\chi f \in \calA^{(\calE,\alpha), \infty }(\overline{\bbR}\times \bbB^d)$ for any $\chi \in C_{\mathrm{c}}^\infty((\overline{\bbR}\backslash \{0\})\times \bbB^d)$. 

Often, $(0,0)$ is used as an abbreviation for the index set $\{(n,0):n\in \bbN\}$ (i.e.\ the index set ``generated'' by $(0,0)$ ). This index set is significant because 
\begin{equation} 
	\calA^{(0,0),\infty}(\overline{\bbR}\times \bbB^d) = C^\infty(\overline{\bbR}_\tau ; \calS(\bbB^d_{\bfy}))
\end{equation} 
is just the set of smooth functions on the Poincar\'e cylinder that are Schwartz at the sides. We will not have much use for more exotic index sets here. In fact, below we only use the spaces 
\begin{equation}
	\calA^{((0,0),\alpha),\infty}(\overline{\bbR}\times \bbB^d) = \calA^{(0,0),\infty}(\overline{\bbR}\times \bbB^d) + \calA^{\alpha,\infty}(\overline{\bbR}\times \bbB^d), 
\end{equation}
the elements of which are just sums of elements of $ C^\infty(\overline{\bbR}_\tau ; \calS(\bbB^d_{\bfy}))$ and $ \calA^\alpha(\overline{\bbR}_\tau ; \calS(\bbB^d_{\bfy}))$. Note that the conormal spaces are indexed relative to $L^\infty$. So,  
\begin{equation}
	\calA^{((0,0),\alpha),\infty}(\overline{\bbR}\times \bbB^d) \subseteq \langle \tau \rangle^{-\min\{0,\alpha\}  } L^\infty.
	\label{eq:misc_247}
\end{equation}

\begin{lemma}
	If $f\in C^\infty_{\mathrm{c}}(\bbX)$ and $u\in \calA^{((0,0),\alpha),\infty}(\overline{\bbR}\times \bbB^d)$ for some $\alpha>0$, then $(\iota_* f)u\in \calA^{((0,0),\alpha),\infty}(\overline{\bbR}\times \bbB^d)$.
	\label{lem:corner_conversion}
\end{lemma}
The reason this is not trivial is that $\iota_* f$ will not be smooth on the Poincar\'e cylinder (because $f$ cannot be constant at null infinity). However, since $u$ is Schwartz at the sides of the cylinder, this does not cause a problem:
\begin{proof}
	Write $u = \langle \tau \rangle^{-\alpha} r+ \sum_{n\in \bbN, n< \alpha} \langle \tau \rangle^{-n} u_n$ for $u_n \in \calA^{(0,0),\infty}(\overline{\bbR}\times \bbB^d)$ and $r \in \calA^{0,\infty}(\overline{\bbR}\times \bbB^d)$. 
	Via Leibniz and \cref{eq:misc_613} and \cref{eq:misc_614}, if $L\in \operatorname{Diff}_{\mathrm{b}}(\overline{\bbR}\times \bbB^d)$, then, for any $\varsigma \in \bbR$, 
	\begin{equation} 
		(1-y^2)^\varsigma L (r \iota_* f) \in L^\infty(\overline{\bbR}\times \bbB^d),
	\end{equation} 
	so $r\iota_* f \in \calA^{0,\infty}(\overline{\bbR}\times \bbB^d)$. Similarly, for any $L\in \operatorname{Diff}_{\mathrm{E}}(\overline{\bbR}\times \bbB^d)$, 
	\begin{equation} 
		(1-y^2)^\varsigma L (u_n \iota_* f) \in L^\infty(\overline{\bbR}\times \bbB^d),
	\end{equation} 
	so $u_n \iota_* f \in  \calA^{(0,0),\infty}(\overline{\bbR}\times \bbB^d)$. 
	From these observations, we conclude that $(\iota_* f)u\in \calA^{((0,0),\alpha),\infty}(\overline{\bbR}\times \bbB^d)$. 
\end{proof}

Another way of thinking about this is using the compactification $[\bbO_0; \mathrm{cl}_{\bbO_0}\{|t|=r\}\cap \partial \bbO]$ described in \Cref{rem:mutual_refine} (and depicted in \Cref{fig:mutual_refine}). If $u$ is fully polyhomogeneous (discussing the full case for simplicity) on $[\bbM;\scrI]_{\mathrm{par}}$, then it is also polyhomogeneous on this more elaborate compactification, since it comes from blowing up the corner of $[\bbM;\scrI]_{\mathrm{par}}$. Moreover, $u$ is Schwartz at all of the faces corresponding to null infinity. By similar reasoning, $f$ is smooth on this compactification. So, $uf$ is smooth on this compactification and Schwartz at both of the faces corresponding to null infinity. The Schwartz behavior allows us to blow down faces to conclude that $uf$ is polyhomogeneous on $[\bbM;\scrI]_{\mathrm{par}}$.

Let $P_{\mathrm{e,sc}}$ denote a differential operator on $\overline{\bbR} \backslash \{0\}\times \bbB^d$ of the form 
\begin{equation} 
	P_{\mathrm{e,sc}}= \partial_\tau^2 + d \tau^{-1} \partial_\tau + \tau^{-2} ( \iota_* \Upsilon )R  + \mathsf{m}^2
	\label{eq:misc_psc}
\end{equation} 
for $R\in \frakN^2_0$ and $\Upsilon \in C_{\mathrm{c}}^\infty(\bbX)$, which the Klein--Gordon operators studied in the rest of the paper have, up to pre-multiplication of the error term by an element of $C^\infty_{\mathrm{c}}(\bbX)$. The function $\iota_* \Upsilon$ appearing here can be fairly badly behaved as a function on the Poincar\'e cylinder, but as long as we are applying $P_{\mathrm{e,sc}}$ to functions that are Schwartz at the sides of the Poincar\'e cylinder (which is all we do), this will not matter. 

\begin{proposition}
	Suppose that $u=e^{\pm i \mathsf{m}\tau} \tau^{-d/2} u_0$ for 
	\begin{equation} 
		u_0 \in \calA^{((0,0),\alpha),\infty}_{\mathrm{loc}}(\overline{\bbR}\backslash \{0\}\times \bbB^d)
	\end{equation} 
	for $\alpha\in \bbR$ with $\alpha>-1$, and suppose that $Pu=f$ for $f$ of the form $f=e^{\pm i \mathsf{m} \tau} \tau^{-d/2-2}f_0$, where
	\begin{equation} 
		f_0 \in \calA^{((0,0),\alpha),\infty}_{\mathrm{loc}}(\overline{\bbR}\backslash \{0\}\times \bbB^d).
	\end{equation}
	Then $u_0 \in \calA^{((0,0),\alpha+1),\infty}_{\mathrm{loc}}(\overline{\bbR}\backslash \{0\}\times \bbB^d)$. 
	\label{prop:integration}
\end{proposition}
\begin{proof}
	Let $\tilde{P}_{\mathrm{e,sc}} = \tau^{d/2} e^{\mp i \mathsf{m}\tau} P_{\mathrm{e,sc}} e^{\pm i \mathsf{m} \tau} \tau^{-d/2}$ denote the result of conjugating $P_{\mathrm{e,sc}}$ by the multiplication operator $\exp(\pm i \mathsf{m} \tau) \tau^{-d/2}$. 
	
	Since $\frakN^2_0$ is closed under conjugations of this form (as seen from the definition \cref{eq:N0_def}), 
	\begin{equation}
		\tilde{P}_{\mathrm{e,sc}} = \partial_\tau^2 \pm 2 i \mathsf{m} \partial_\tau + \tau^{-2} ( \iota_* \Upsilon ) \tilde{R} 
		\label{eq:misc_tsc}
	\end{equation} 
	for some $\tilde{R} \in \frakN^2_0$. Since $P_{\mathrm{e,sc}}u=f$, we have $\tilde{P}_{\mathrm{e,sc}} u_0 = \tau^{-2} f_0$. Thus, 
	\begin{equation} 
		\pm 2 i \mathsf{m} \partial_\tau u_0 = \tau^{-2} f_1, \quad f_1 =  f_0 - \tau^2 \partial_\tau^2 u_0 -  (\iota_* \Upsilon )\tilde{R} u_0.
		\label{eq:misc_kh1}
	\end{equation} 
	Under the hypotheses above, 
	\begin{equation} 
		f_1 \in \calA^{((0,0),\alpha),\infty}_{\mathrm{loc}}(\overline{\bbR}\backslash \{0\}\times \bbB^d),
		\label{eq:misc_250}
	\end{equation} 
	just like $f_0$. Indeed, $\tau^2 \partial_\tau^2$ is b-differential operator, so it preserves $\calA_{\mathrm{loc}}^{((0,0),\alpha),\infty}$. Likewise, we can read off \cref{eq:N0_def} that 
	\begin{equation} 
		\frakN_0\subset \operatorname{Diff}^1_{\mathrm{b}}(\overline{\bbR}\times \bbB^d),
	\end{equation} 
	so $\tilde{R}$ preserves $\calA_{\mathrm{loc}}^{((0,0),\alpha),\infty}$ as well. \Cref{lem:corner_conversion} says that that multiplying by $\iota_* \Upsilon$ does not change the conclusion. So, \cref{eq:misc_250} is justified.
	 
	Integrating \cref{eq:misc_kh1}, 
	\begin{equation}
		u_0(\tau,\bfy)   = u_0(1,\bfy) \pm \frac{1}{2i \mathsf{m}} \int_1^\tau \frac{f_1(s,\bfy)}{s^2} \dd s .
	\end{equation}
	The term $u_0(1,\bfy)$ is Schwartz at $y=1$. 
	Using that $\alpha>-1$ (which means that $f_1(\tau,\bfy)=O(\tau^{1-\epsilon})$ as $\tau\to\infty$, for some $\epsilon>0$, and likewise for b-derivatives thereof), it follows that 
	\begin{equation} 
		u_0 \in \calA^{((0,0),\alpha+1),\infty}_{\mathrm{loc}}(\overline{\bbR}\backslash \{0\}\times \bbB^d),
	\end{equation} 
	because integration $f_1\mapsto \int_1^\tau s^{-2} f_1(s) \dd s$  maps conormal functions (with sufficient decay as $\tau\to\infty$ to avoid a logarithmic or worse divergence) to partially polyhomogeneous functions with a conormal error with one less order of decay. We are integrating $f_1(\tau,\bfy)/\tau^2$, which has two orders of decay relative to what $u_0$ is assumed to have, so overall we gain an order of decay in the conormal part of $u_0$. 
\end{proof}

This proposition is the key to extracting $\tau\to\infty$ asymptotics; using it inductively allows us to deduce, starting from the conormality of $u_0$, that $u_0$ possesses an expansion in powers of $1/\tau$.  Indeed:

\begin{proposition}
	Let $\epsilon>0$.
	Suppose that $u_0 \in \calA_{\mathrm{loc}}^{-1+\epsilon,\infty}(\overline{\bbR}\backslash \{0\}\times \bbB^d)$ satisfies 
	\begin{equation} 
		P_{\mathrm{e,sc}} (e^{\pm i \mathsf{m} \tau} \tau^{-d/2} u_0)=f
	\end{equation} 
	for some $f\in \calA_{\mathrm{loc}}^{\infty,\infty}(\overline{\bbR}\backslash \{0\}\times \bbB^d)$. Then, $
		u_0 \in \calA_{\mathrm{loc}}^{(0,0),\infty}(\overline{\bbR}\backslash \{0\}\times \bbB^d)$.
	\label{prop:asymptotic_extraction}
\end{proposition}
Here, $\calA_{\mathrm{loc}}^{\infty,\infty}(\overline{\bbR}\backslash \{0\}\times \bbB^d)$ is essentially the space $\calS(\overline{\bbR}\times \bbB^d)$ of Schwartz functions, except we are not restricting $\tau\to 0$ behavior in the former.
\begin{proof}
	Since $f$ is Schwartz, $f_0 = e^{\mp i \mathsf{m} \tau} \tau^{d/2+2} f$ is in $\calA_{\mathrm{loc}}^{\infty,\infty}(\overline{\bbR}\backslash \{0\}\times \bbB^d)$ as well. This result therefore follows from \Cref{prop:integration}, 
	 via an inductive argument on $\alpha$, starting with $\alpha = -1+\epsilon$. 
\end{proof}

By the Sobolev embedding theorem, we can relate the conormal function spaces to e,sc-Sobolev spaces with module regularity:
\begin{proposition}
	For any $m\in \bbN$ and $s,\varsigma \in \bbR$,  $
		H_{\mathrm{e,sc}}^{m,s+,\infty;\infty,\infty} (\overline{\bbR}\times \bbB^d) \subseteq \calA^{s+(d+1)/2,\infty}(\overline{\bbR}\times \bbB^d) \subseteq H_{\mathrm{e,sc}}^{m,s-,\infty;\infty,\infty}(\overline{\bbR}\times \bbB^d)$.
	\label{prop:Sobolev_embedding}
\end{proposition}
Here,  $H_{\mathrm{e,sc}}^{m,s,\infty;\infty,\infty} =\bigcap_{k,\kappa\in \bbN} \bigcap_{\varsigma\in \bbR} H_{\mathrm{e,sc}}^{m,s,\varsigma;\kappa,k}$, $H_{\mathrm{e,sc}}^{m,s+,\infty;\infty,\infty} = \bigcap_{k,\kappa \in \bbN}\bigcap_{\varsigma_0\in \bbR}\bigcup_{s_0>s} H_{\mathrm{e,sc}}^{m,s_0,\varsigma_0;\infty,\infty}$, and $H_{\mathrm{e,sc}}^{m,s-,\infty;\infty,\infty} = \bigcap_{k,\kappa \in \bbN}\bigcap_{\varsigma_0\in \bbR}\bigcap_{s_0<s} H_{\mathrm{e,sc}}^{m,s_0,\varsigma_0;\infty,\infty}$. Similar notation will be used for other hierarchies of function spaces below.
\begin{proof}
	It suffices to consider the $s=0$ case. 
	\begin{itemize}
		\item Let $u\in H_{\mathrm{e,sc}}^{m,0+,\infty;\infty,\infty}(\overline{\bbR}\times \bbB^d)$. 
		Since \cref{eq:M0_def} gives
		\begin{equation} 
			\operatorname{Diff}_{\mathrm{b}}(\overline{\bbR}\times \bbB^d) \subseteq \bigcup_{\ell\in \bbN} \frakM_0^\ell, 
		\end{equation} 
		$L u\in H_{\mathrm{e,sc}}^{m,0+,\infty;\infty,\infty}(\overline{\bbR}\times \bbB^d) \subseteq (1-y^2)^{\varsigma} L^2(\overline{\bbR}\times \bbB^d,g_{\mathrm{e,sc}})$ for each  $L\in \operatorname{Diff}_{\mathrm{b}}(\overline{\bbR}\times \bbB^d)$ and $\varsigma \in \bbR$. Since 
		\begin{equation} 
			L^2(\overline{\bbR}\times \bbB^d,g_{\mathrm{e,sc}}) = \langle \tau \rangle^{-(d+1)/2} (1-y^2)^{(d-1)/2} L^2_{\mathrm{b}}(\overline{\bbR}\times \bbB^d),
		\end{equation} 
		we deduce that $u$ lies in the $L^2_{\mathrm{b}}$-based conormal space
		\begin{multline}
			\calI^{(d+1)/2+,\infty}(\overline{\bbR}\times \bbB^d) = \{u\in L^2_{\mathrm{b}}(\overline{\bbR}\times \bbB^d) : Lu \in  \langle \tau \rangle^{-\epsilon-(d+1)/2} (1-y^2)^{\varsigma}L^2_{\mathrm{b}}(\overline{\bbR}\times \bbB^d)\\ \text{ for all }L\in \operatorname{Diff}_{\mathrm{b}}(\overline{\bbR}\times \bbB^d)\text{ and }\varsigma \in \bbR,\epsilon>0 \}.
		\end{multline}
		The Sobolev embedding theorem implies that $\calI^{(d+1)/2+,\infty}(\overline{\bbR}\times \bbB^d) \subseteq \calA^{(d+1)/2,\infty}(\overline{\bbR}\times \bbB^d)$.
		\item Conversely, suppose that $u\in \calA^{(d+1)/2,\infty}(\overline{\bbR}\times \bbB^d)$. Then, because $L^\infty(\overline{\bbR}\times \bbB^d)\subseteq \langle \tau \rangle^{\epsilon}(1-y^2)^{-\epsilon} L^2_{\mathrm{b}}(\overline{\bbR}\times \bbB^d)$ for any $\epsilon>0$, 
		\begin{equation}
			Lu \in \langle \tau \rangle^{\epsilon-(d+1)/2} (1-y^2)^{\varsigma} L^2_{\mathrm{b}}(\overline{\bbR}\times \bbB^d)  = \langle \tau \rangle^{\epsilon} (1-y^2)^{\varsigma -(d-1)/2} L^2(\overline{\bbR}\times \bbB^d,g_{\mathrm{e,sc}})
		\end{equation}
		for any $\varsigma \in \bbR$ and $L\in (1-y^2)^{\varsigma_0} \operatorname{Diff}_{\mathrm{b}}(\overline{\bbR}\times \bbB^d)$, for any $\varsigma_0\in \bbR$.  
		Since $\partial_\tau,\langle \tau \rangle^{-1} (1-y^2) \partial_{y_j} \in \operatorname{Diff}_{\mathrm{b}}(\overline{\bbR}\times \bbB^d )$, 
		\begin{equation}
			\operatorname{Diff}_{\mathrm{e,sc}}(\overline{\bbR}\times \bbB^d) \subseteq  \operatorname{Diff}_{\mathrm{b}} (\overline{\bbR}\times \bbB^d). 
		\end{equation}
		Combining these observations, $L_0 L u \in \langle \tau \rangle^{\epsilon} (1-y^2)^{\varsigma } L^2(\overline{\bbR}\times \bbB^d,g_{\mathrm{e,sc}})$ for any $L_0 \in \operatorname{Diff}_{\mathrm{e,sc}}(\overline{\bbR}\times \bbB^d)$ and $L\in \operatorname{Diff}_{\mathrm{b}}(\overline{\bbR}\times \bbB^d)$. That is, 
		\begin{equation}
			L u \in H_{\mathrm{e,sc}}^{m,-\epsilon ,\varsigma}(\overline{\bbR}\times \bbB^d),
		\end{equation}
		for any $m\in \bbN$. 
		Since $\frakM_0 \subset \operatorname{Diff}_{\mathrm{b}} (\overline{\bbR}\times \bbB^d)$, we can apply this for all $L\in \bigcup_{\ell\in \bbN} \frakM^\ell_0$ to conclude that $u\in H_{\mathrm{e,sc}}^{m,-\epsilon,\varsigma;\infty,\infty}$. Taking $\epsilon \to 0^+$ and $\varsigma \to \infty$, we conclude the result. 
	\end{itemize}
\end{proof}

Now let $P = \square + \mathsf{m}^2 + R_0$ for some $R_0\in \operatorname{Diff}_{\mathrm{de,sc}}^{2,-\mathsf{2}}(\bbO)$. 

As in \Cref{thm:main}, let $\chi \in C^\infty(\bbO)$ denote a function supported in $\bbX$ and identically equal to $1$ in some neighborhood of $\mathrm{Pf}\cup \mathrm{Ff}$.
\begin{proposition}
	Let $m\geq 0$ and $s>-3/2$. Suppose that $u\in \calS'(\bbR^{1,d})$ satisfies $Pu=f$ for some $f\in \calS(\bbR^{1,d})$. Then:
	\begin{itemize}
		\item If $u\in H_{\mathrm{de,sc};\pm ,-}^{m,(s,\infty,\infty,\infty,\infty);\infty,\infty}(\bbO)$, then $u$ has the form $u=w + \chi e^{\pm i \mathsf{m} \tau} \langle \tau \rangle^{-d/2} v$
		for some $w\in \calS(\bbR^{1,d})$ and $v\in \varrho_{\mathrm{nPf}}^\infty \varrho_{\mathrm{Sf}}^\infty \varrho_{\mathrm{nFf}}^\infty \varrho_{\mathrm{Ff}}^\infty C^\infty(\bbO)$. 
		\item If $u\in H_{\mathrm{de,sc};\pm ,+}^{m,(\infty,\infty,\infty,\infty,s);\infty,\infty}(\bbO)$, then $u$ has the form $u=w + \chi e^{\pm i \mathsf{m} \tau} \langle \tau \rangle^{-d/2} v$
		for some $w\in \calS(\bbR^{1,d})$ and $v\in \varrho_{\mathrm{Pf}}^\infty \varrho_{\mathrm{nPf}}^\infty \varrho_{\mathrm{Sf}}^\infty \varrho_{\mathrm{nFf}}^\infty C^\infty(\bbO)$. 
	\end{itemize}
	\label{prop:asymptotic_main}
\end{proposition}
\begin{proof}
	Note that, if we let $P_0 = \square+\mathsf{m}^2+\psi R_0$ for $\psi\in C_{\mathrm{c}}^\infty(\bbX)$ such that $\psi=1$ identically on $\operatorname{supp} \chi$, then $u$ satisfies $P_0 u =f_0$ for some $f_0\in \calS(\bbR^{1,d})$. 
	Now observe that $P_0$ can be written as the pullback $P=\iota^* P_{\mathrm{e,sc}}$ for an e,sc-operator $P_{\mathrm{e,sc}}$ on the Poincar\'e cylinder, with $P_{\mathrm{e,sc}}$ satisfying the conditions above. Indeed, we just need that pulling back $R_0$ results in a linear combination of elements of $\tau^{-2} \frakN_0^2$ multiplied by pullbacks of elements of $C_{\mathrm{c}}^\infty(\bbX)$. Indeed, \Cref{prop:scedge->oct} tells us that $R_0$ is a sum of operators of the form 
	\begin{equation} 
		C_{\mathrm{c}}^\infty(\bbX)\iota_* \operatorname{Diff}_{\mathrm{e,sc}}^{2,-2,0},
	\end{equation} 
	so the pullback of $R_0$ is a sum of operators of the form $(\iota^* C^\infty_{\mathrm{c}}(\bbX)) \operatorname{Diff}_{\mathrm{e,sc}}^{2,-2,0}$. Now, 
	\begin{equation}
		\operatorname{Diff}_{\mathrm{e,sc}}^{2,-2,0} = \tau^{-2} \operatorname{Diff}_{\mathrm{e,sc}}^{2,0,0} \subseteq \tau^{-2} \frakN_0^2,
		\label{eq:misc_262}
	\end{equation}
	where the final $\subseteq$ used $\operatorname{Diff}_{\mathrm{e,sc}}^{1,0,0}\subset \frakN_0$ (\cref{eq:N0_def}).

	We consider the case when
	\begin{equation}
		u \in H_{\mathrm{de,sc};+,+}^{m,(\infty,\infty,\infty,\infty,s_{\mathrm{Tf}});\infty,\infty}(\bbO),
	\end{equation}
	and the others are similar. 
	Let $g = \chi_0 f + [P_0,\chi_0] u$ for $\chi_0\in C_{\mathrm{c}}^\infty(\bbX)$ such that $\operatorname{supp} \chi_0 \Subset \chi^{-1}(\{1\})$. We have $P_0 (\chi_0 u) =g$. Since $[P_0,\chi_0]$ is supported away from timelike infinity, where the de,sc-wavefront set of $u$ is, we have $g\in \calS(\bbR^{1,d})$. Pushing forward to the Poincar\'e cylinder, 
	\begin{equation}
		P_{\mathrm{e,sc}}(\iota_* \chi_0 u) =  (\iota_* P_0) (\iota_* \chi_0 u) = \iota_* (P_0(\chi_0 u)) = \iota_* g.
	\end{equation}
	By the hypothesis and \Cref{prop:ultimate_sobolev_conversion}, $\iota_* \chi_0 u = e^{+i \mathsf{m} \tau} \langle \tau\rangle^{-d/2} u_0$ for 
	\begin{equation} 
		u_0 \in H_{\mathrm{e,sc}}^{m,s-d/2,\infty;\infty,\infty}(\overline{\bbR}\times \bbB^d).
	\end{equation} 
	Likewise, $\iota_* g \in \calS^\infty(\overline{\bbR}\times \bbB^d)$.
	
	By \Cref{prop:Sobolev_embedding}, $u_0 \in \calA^{s+1/2-,\infty}(\overline{\bbR}\times \bbB^d)$.
	Since $s>-3/2$, we can appeal to \Cref{prop:asymptotic_extraction} to deduce that 
	\begin{equation} 
		u_0 \in \calA^{(0,0),\infty}(\overline{\bbR}\times \bbB^d).
	\end{equation} 
	Letting $v =\chi^{-1} \iota^*    u_0$, this being well-defined because of the support condition on $\chi_0$, we have $\chi_0 u = \chi e^{+i\mathsf{m}\tau} \langle \tau \rangle^{-d/2} v$, so setting $w= (1-\chi_0)u \in \calS(\bbR^{1,d})$, we have $u = w + \chi e^{+ i \mathsf{m} \tau} \langle \tau \rangle^{-d/2} v$. 
\end{proof}

The following result will be useful when discussing the scattering problem in \S\ref{sec:proofs}:
\begin{proposition}
	Suppose that $v_\pm$ are Schwartz functions on either the the past or the future timelike cap of $\bbM$, not necessarily the same cap. Then, there exist some 
	\begin{equation} 
		u_\pm \in \varrho_{\mathrm{nPf}}^\infty \varrho_{\mathrm{Sf}}^\infty \varrho_{\mathrm{nFf}}^\infty C^\infty(\bbO)
	\end{equation} 
	such that 
	\begin{itemize}
		\item $u_\pm$ has support  disjoint from all of the faces of $\bbO$ except the cap on which $v_\pm$ is given and the adjacent component of $\mathrm{nf}$, 
		\item $u_\pm$, when restricted to that cap, is $v_\pm$, and 
		\item $ P( \chi \varrho_{\mathrm{Pf}}^{d/2}\varrho_{\mathrm{Ff}}^{d/2} e^{\pm i \mathsf{m} \sqrt{t^2-r^2}} u_{\pm}) \in \calS(\bbR^{1,d})$, 
	\end{itemize} 
	for each choice of sign.
	\label{prop:borel} 
\end{proposition}
\begin{proof}
	We consider the case where the timelike cap is the future one, with the past case being analogous, and we consider only the plus case of the theorem, the minus case being analogous. 
	We work on $\overline{\bbR}\times \bbB^d$, considering $v_+ \in \calS(\bbB^d_{\bfy})$. It suffices to construct 
	\begin{equation} 
		w_+ \in (1-y^2)^\infty C^\infty(\overline{\bbR}\times \bbB^d)
	\end{equation} 
	supported in $[1,\infty]_\tau\times \bbB^d$ such that $w_+|_{\{\infty\}\times \bbB^d} = v_+$ and $P_{\mathrm{e,sc}}(\tau^{-d/2} e^{+i\mathsf{m}\tau } w_+ ) \in \calS(\overline{\bbR}\times \bbB^d)$. Indeed, given this, set 
	\begin{equation} 
		u_+ = \varrho_{\mathrm{Pf}}^{-d/2} (\varrho_{\mathrm{Ff}} \iota^* \tau)^{-d/2} \iota^* w_+ \in C^\infty(\bbO) 
	\end{equation}
	(which is supported away from $\mathrm{Pf}\cup \mathrm{nPf}\cup \mathrm{Sf}$). Then, since $P_0 = \iota^* P_{\mathrm{e,sc}}$, where $P_0$ is as in the proof of the previous proposition, 
	\begin{multline}
		 P( \chi \varrho_{\mathrm{Pf}}^{d/2}\varrho_{\mathrm{Ff}}^{d/2} e^{+ i \mathsf{m} \sqrt{t^2-r^2}} u_{+})  = [P_0,\chi ] ( \iota^* \tau^{-d/2} e^{+i\mathsf{m}\tau} w_+ ) + \chi \iota^*  P_{\mathrm{e,sc}}(\tau^{-d/2} e^{+i\mathsf{m}\tau } w_+ ) \\ 
		 + (1-\psi) R_0 ( \chi \varrho_{\mathrm{Pf}}^{d/2}\varrho_{\mathrm{Ff}}^{d/2} e^{+ i \mathsf{m} \sqrt{t^2-r^2}} u_{+})  \in \calS(\bbR^{1,d}). 
	\end{multline}

	The construction of $w_+$ is a straightforward term-by-term construction using the structure of $P_{\mathrm{e,sc}}$ described in \cref{eq:misc_psc}. 	
	Consider a formal series
	\begin{equation}
		w_{+,\Sigma}(\tau,\bfy) = \sum_{k=0}^\infty w_{+,k}(\bfy) \tau^{-k} \in \calS(\bbB^d_y)[[1/\tau]]. 
		\label{eq:misc_224}
	\end{equation}
	Formally applying $P_{\mathrm{e,sc}}$ to $\tau^{-d/2} e^{+i\mathsf{m} \tau} w_{+,\Sigma}$ yields $\tau^{-d/2} e^{+i\mathsf{m} \tau} \tilde{P}_{\mathrm{e,sc}} w_{+,\Sigma}$, where $\tilde{P}_{\mathrm{e,sc}} = \partial_\tau^2 + 2i \mathsf{m} \partial_\tau + \tau^{-2} (\iota_* \Upsilon)\tilde{R}$, as in \cref{eq:misc_tsc}. 
	In order to make sense of $\tilde{R} w_{+,\Sigma}$, we consider the Taylor expansion 
	\begin{multline}
		(\iota_* \Upsilon) \tilde{R}  \sim \sum_{k=0}^\infty \tau^{-k} \Big( c_k \partial_\tau^2 + \sum_{j=1}^d c_{k,j} (1-y^2)\partial_\tau \partial_{y_j} \\ + \sum_{j,\ell=1}^d (1-y^2)^2 c_{k,j,\ell}\partial_{y_j} \partial_{y_k}+  d_k \partial_\tau + \sum_{j=1}^d d_{k,j} (1-y^2)\partial_{y_j}  + e_k  \Big), 
	\end{multline}
	where $c_k,c_{k,j},c_{k,j,\ell},d_k,d_{k,j},e_k \in C^\infty(\bbB^{d\circ})$, with a polynomial rate of growth at the boundary.  
	Applying $\tilde{P}_{\mathrm{e,sc}}$ to $w_{+,\Sigma}$, the result is the formal series in $1/\tau$, the $k$th term of which is a linear combination of the $w_{+,0},\ldots,w_{+,k-1}$ with coefficients in $C^\infty(\bbB^{d\circ})$ having polynomial growth at the boundary, and with the coefficient of $w_{+,k-1}$  being $-2i\mathsf{m}(k-1) w_{+,k-1}$. Thus, we can recursively define a sequence 
	\begin{equation} 
		\{w_{+,k}\}_{k=0}^\infty \subset \calS(\bbB^d)
	\end{equation} 
	such that $w_{+,0} = v_+$ and such that $\tilde{P}_{\mathrm{e,sc}} w_{+,\Sigma} = 0$, formally. Via the smooth Borel summation lemma, there exists a $w_+ \in (1-y^2)^\infty C^\infty(\overline{\bbR}\times \bbB^d)$ whose Taylor series at $\tau=\infty$ is given by \cref{eq:misc_224}. Multiplying by a cutoff, we can assume that $w_+$ is supported in $[1,\infty]_\tau \times \bbB^d$. The formal manipulations above make sense at the level of computing the Taylor series of $f=P_{\mathrm{e,sc}} (\tau^{-d/2} e^{+i\mathsf{m}\tau} w_{+,\Sigma})$, which is a priori in $\tau^{-d/2} e^{+i\mathsf{m} \tau} (1-y^2)^\infty C^\infty(\overline{\bbR}\times \bbB^d)$. The formal manipulations show that the Taylor series of $f$ at $\tau=\infty$ vanishes, which suffices to conclude that $f$ is actually Schwartz. 
\end{proof}

\section{Classical dynamics on the de,sc- phase space}
\label{sec:dynamics}
We now study the (appropriately scaled) Hamiltonian flow of the d'Alembertian -- i.e.\ the null geodesic flow -- of an admissible metric on the de,sc-phase space, near null infinity. Attention is restricted to the characteristic set of the Klein--Gordon operator, which depends on $\mathsf{m}$. 
As seen above, the symbol 
\begin{equation} 
	p = p[g_{\bbM}]= - \tau^2 + \Xi^2 + \eta^2 + \mathsf{m}^2 \in C^\infty(T^* \bbR^{1,d} ) 
\end{equation} 
of the Minkowski d'Alembertian is a classical symbol on the de,sc-cotangent bundle of order zero at each face except df, where it is second order (that is, growing quadratically).

Let $p[g]$ denote a representative of the principal symbol of $\square_g$. If $g$ is admissible (see \S\ref{sec:proofs}), then 
\begin{equation} 
	p[g]=p+\varrho_{\mathrm{Pf}}\varrho_{\mathrm{nPf}}^2\varrho_{\mathrm{Sf}}\varrho_{\mathrm{nFf}}^2 \varrho_{\mathrm{Ff}}   p_{\mathrm{lo}}
	\label{eq:p1}
\end{equation} 
for some 
\begin{equation} 
	p_{\mathrm{lo}} \in S_{\mathrm{de,sc}}^{2,\mathsf{0}}(\bbO).
	\label{eq:p1_z}
\end{equation} 
The `lo' stands for ``lower order.'' However, $p_{\mathrm{lo}}$ is not subleading at $\mathrm{df}$, so it does enter the principal symbol. This would be true even if $g_\bbM-g$ were Schwartz.
\begin{remark*}
	Actually, the $g$ considered in \S\ref{sec:proofs} satisfy 
	\begin{equation} 
		p[g]=p+ S_{\mathrm{de,sc}}^{2,-\mathsf{2}}(\bbO).
	\end{equation} 
	However, in this section (and therefore in \S\ref{sec:propagation}), the weaker \cref{eq:p1} suffices. We highlight this because 
	\begin{equation} 
		\varrho_{\mathrm{Pf}}\varrho_{\mathrm{nPf}}^2\varrho_{\mathrm{Sf}}\varrho_{\mathrm{nFf}}^2 \varrho_{\mathrm{Ff}}\sim 1/(1+r^2+t^2)^{-1/2} ,
	\end{equation} 
	so \cref{eq:p1} allows natural long-range terms which are excluded later.

	In fact, most of the discussion in this section goes through for 
	\begin{equation}
		p[g]=p+ S_{\mathrm{de,sc}}^{2,-\mathsf{1}}(\bbO).
		\label{eq:weakest}
	\end{equation}
	The only exceptions are the results relating to the source/sink structure of $\calN$. Thus, while discussing $\calA,\calC,\calK$, we will only use \cref{eq:weakest}.
	In work-in-progress, Molodyk--Vasy \cite{M-V-Feynman} are extending the analysis in this and the next section to even more general long-range metrics than those satisfying \cref{eq:weakest}. What ultimately matters for the arguments in \S\ref{sec:propagation} is the source/sink structure of the flow at null infinity, which is not affected by adding to $p$ decaying terms. However, if only \cref{eq:weakest} (or something weaker) holds, then it may be necessary to use different symbols than those defined below to probe the source/sink structure.
\end{remark*}

Let 
\begin{equation} 
	\tilde{p}[g] = \varrho_{\mathrm{df}}^2 p[g] \in C^\infty({}^{\mathrm{de,sc}}\overline{T}^* \bbO ).
	\label{eq:ptilde}
\end{equation} 
Then, the characteristic set $\operatorname{Char}_{\mathrm{de,sc}}^{2,\mathsf{0}}(P)=\Sigma_{\mathsf{m}}[g]$ of any $P\in \operatorname{Diff}_{\mathrm{de,sc}}^{2,\mathsf{0}}$ with principal symbol $p[g]$ is given by
\begin{equation}
	\Sigma_{\mathsf{m}}[g] = \tilde{p}[g]^{-1}(0) \cap \partial ({}^{\mathrm{de,sc}}\overline{T}^* \bbO ), 
\end{equation}
that is the portion of the vanishing set of $\tilde{p}[g]$ contained in the boundary of the de,sc-phase space.
Over the boundary of $\bbO$, $\Sigma_{\mathsf{m}}[g]$ does not depend on $g$. In each fiber over the boundary $\Sigma_{\mathsf{m}}[g]$ consists of a two-sheeted hyperboloid (note that this notion does not depend on the choice of coordinates in the base). By the admissibility criteria, which imply time orientability, $\Sigma_{\mathsf{m}}[g]$ has two connected components,
\begin{equation}
	\Sigma_{\mathsf{m},\pm}[g] = \Sigma_{\mathsf{m}}[g] \cap \mathrm{cl}_{{}^{\mathrm{de,sc}}\overline{T}^* \bbO} \{\pm \tau \geq 0\}. 
\end{equation}
Recall that if we drop `$[g]$', then this just means evaluated for $g=g_{\bbM}$ the Minkowski metric; this notation applies throughout this section.

\begin{figure}[t]
	\centering
	\begin{tikzpicture}[scale=.85]
		\def\mm{5.5}
		\def\fram{7}
		
		\def\a{1}
		\def\b{1} 	
		\begin{scope}[shift={(-5,0)}]
			\begin{axis}[
				hide axis,
				axis equal image,
				xmin=-\fram,xmax=\fram,
				ymin=-\fram,ymax=\fram]
				
				\draw[fill=gray!10] (0,0) circle (\mm);
				\draw[dotted] (0,-\mm) -- (0, +\mm);
				\draw[dotted] (-\mm,0) -- (+\mm,0);
				
				\begin{scope}[rotate around={58:(0,0)}]
					\draw[darkcandyapp] (0,2.95) ellipse (53.2pt and 3pt);
					\draw[darkcandyapp,dashed] (0,2.2) ellipse (37pt and 2pt);
					\draw[darkcandyapp,dashed] (0,1.5) ellipse (21pt and 2pt);
					\draw[darkblue] (0,-2.95) ellipse (53.2pt and 3pt);
					\draw[darkblue,dashed] (0,-2.2) ellipse (37.5pt and 2pt);
					\draw[darkblue,dashed] (0,-1.5) ellipse (21.5pt and 2pt);
				\end{scope}
				
				\addplot [domain=-2.03:1.28, color=darkcandyapp] ({-1.41\b*cosh(\x)-1.41\b*sinh(\x)+.15},{-1.41\b*sinh(\x)});
				\addplot [domain=-2.03:1.28, color=darkblue] ({1.41\b*cosh(\x)+1.41\b*sinh(\x)-.15},{1.41\b*sinh(\x)});
				\draw[->] (0,0) -- (0,2.5);
				\node[above right] (nu) at (0,1.7) {$+\zeta$};
				\draw[->] (0,0) -- (2.5,0);
				\node[above right] (aleph) at (.5,0) {$\xi$};
				\draw[->] (0,0) -- (-1.2,-1.2);
				\node[below left] (eta) at (-1.2,-1.2) {$\eta $};
				
				\filldraw[black] (0,\mm) circle (2pt); 
				\filldraw[black] (0,-\mm) circle (2pt); 
			\end{axis}
		\end{scope}
		\begin{scope}[shift={(1,0)}]
			\begin{axis}[
				hide axis,
				axis equal image,
				xmin=-\fram,xmax=\fram,
				ymin=-\fram,ymax=\fram]

				\draw[fill=gray!10] (0,0) circle (\mm);
				\draw[dotted] (0,-\mm) -- (0, +\mm);
				\draw[dotted] (-\mm,0) -- (+\mm,0);
				
				\begin{scope}[rotate around={30:(0,0)}]
					\draw[darkcandyapp] (0,4.74) ellipse (32.5pt and 3pt);
					\draw[darkcandyapp,dashed] (0,3.5) ellipse (21.5pt and 3pt);
					\draw[darkcandyapp,dashed] (0,2.3) ellipse (10pt and 2pt);
					\draw[darkblue] (0,-4.74) ellipse (32.5pt and 3pt);
					\draw[darkblue,dashed] (0,-3.5) ellipse (21.2pt and 3pt);
					\draw[darkblue,dashed] (0,-2.3) ellipse (10pt and 2pt);
				\end{scope}
				\coordinate (ar1) at (0,0);
				\addplot [domain=-2.03:1.25, color=darkblue] ({1.41\b*cosh(\x)+1.41\b*sinh(\x)-.15},{-1.41\b*cosh(\x)});
				\addplot [domain=-2.03:1.25, color=darkcandyapp] ({-1.41\b*cosh(\x)-1.41\b*sinh(\x)+.15},{1.41\b*cosh(\x)});
				\draw[->] (0,0) -- (0,2.5);
				\node[above right] (nu) at (0,1.7) {$-\zeta$};
				\draw[->] (0,0) -- (2.5,0);
				\node[above right] (aleph) at (2,0) {$\xi$};
				\draw[->] (0,0) -- (-1.2,-1.2);
				\node[below left] (eta) at (-1.2,-1.2) {$\eta $};
				\filldraw[black] (0,\mm) circle (2pt); 
				\filldraw[black] (0,-\mm) circle (2pt); 
			\end{axis}
		\end{scope}
	\end{tikzpicture}
	\caption{The characteristic set $\Sigma_{\mathsf{m}} = {\color{darkblue}\Sigma_{\mathsf{m},-}}\cup {\color{darkcandyapp}\Sigma_{\mathsf{m},+}}$, over $\alpha \in \mathrm{nFf}$, depicted using the momenta dual to $(\varrho_{\mathrm{nf}},\varrho_{\mathrm{Sf}})$ (\emph{left}, if $\alpha \in \Omega_{\mathrm{nfSf},+,R}$) and the momenta dual to $(\varrho_{\mathrm{nf}},\varrho_{\mathrm{Tf}})$ (\emph{right}, if $\alpha \in \Omega_{\mathrm{nfTf},\pm,T}$). 
		The vertical axis is oriented so that page-up corresponds to positive timelike momentum.	
		The `$\bullet$' marks the submanifolds $\calN^+_+,\calN^-_+$. Over $\mathrm{nPf}$, the situation is similar. Warning: in each of the coordinate charts $\Omega_{\bullet,+,1}$, we use ``$\xi,\zeta$'' to label momenta coordinates, but the meaning depends on whether $\bullet$ reads nfTf or nfSf; otherwise, the two plots above would be the same. }
	\label{fig:Sigma}
\end{figure}

Recapping the proof of \Cref{prop:symbol}, and adding back in the $\eta$ (the sc-angular momentum coordinate) dependence: 
\begin{itemize}
	\item in $\Omega_{\mathrm{nfTf},\pm,T}$, where we can use the coordinate system $(\varrho_{\mathrm{nf}},\varrho_{\mathrm{Tf}},\theta,\xi,\zeta,\eta) \mapsto \varrho_{\mathrm{nf}}^{-2} \varrho_{\mathrm{Tf}}^{-1} \xi \mathrm{d} \varrho_{\mathrm{nf}} + \varrho_{\mathrm{nf}}^{-1} \varrho_{\mathrm{Tf}}^{-2} \zeta \mathrm{d} \varrho_{\mathrm{Tf}} + \varrho_{\mathrm{nf}}^{-2}\varrho_{\mathrm{Tf}}^{-1} \eta \dd \theta$, \begin{equation}
		p=\xi^2 - 2 \xi \zeta  + \eta^2 + \mathsf{m}^2, 
		\label{eq:misc_258}
	\end{equation} 
	and 
	\item in $\Omega_{\mathrm{nfSf},\pm,R}$, using the coordinates $(\varrho_{\mathrm{nf}},\varrho_{\mathrm{Sf}},\theta,\xi,\zeta,\eta) \mapsto \varrho_{\mathrm{nf}}^{-2} \varrho_{\mathrm{Sf}}^{-1} \xi \mathrm{d} \varrho_{\mathrm{nf}} + \varrho_{\mathrm{nf}}^{-1} \varrho_{\mathrm{Sf}}^{-2} \zeta \mathrm{d} \varrho_{\mathrm{Sf}} + \varrho_{\mathrm{nf}}^{-2}\varrho_{\mathrm{Sf}}^{-1} \eta \dd \theta$, 
	\begin{equation} 
		p= -\xi^2 + 2\xi \zeta + \eta^2 + \mathsf{m}^2.
		\label{eq:misc_259}
	\end{equation}  
\end{itemize}
Thus, over $\mathrm{nf}\in \{\mathrm{nPf},\mathrm{nFf}\}$, and letting $\sigma \in \{-1,+1\}$ be defined by $\sigma = +1$ if $\mathrm{nf}=\mathrm{nFf}$ and $\sigma = -1$ if $\mathrm{nf}=\mathrm{nPf}$, the set $\Sigma_{\mathsf{m},\pm}[g] \cap {}^{\mathrm{de,sc}}\pi^{-1}(\mathrm{nf})$ is given by 
\begin{align}
	\Sigma_{\mathsf{m},\pm}[g] \cap \Omega_{\mathrm{nfTf},\sigma,T} \cap {}^{\mathrm{de,sc}}\pi^{-1}(\mathrm{nf}) &= \mathrm{cl}_{{}^{\mathrm{de,sc}}\overline{T}^*_{\mathrm{nf}} \bbO } \{ \zeta = (2\xi)^{-1} (\xi^2 + \eta^2+\mathsf{m}^2), \mp \sigma \xi > 0  \} \label{eq:p_nfTf}
	\intertext{with respect to the first coordinate system and}
	\Sigma_{\mathsf{m},\pm}[g] \cap \Omega_{\mathrm{nfSf},\sigma,R} \cap {}^{\mathrm{de,sc}}\pi^{-1}(\mathrm{nf})  &= \mathrm{cl}_{{}^{\mathrm{de,sc}}\overline{T}^*_{\mathrm{nf}} \bbO } \{ \zeta = (2\xi)^{-1} (\xi^2 - \eta^2-\mathsf{m}^2), \mp \sigma \xi > 0  \} 
\end{align}
with respect to the second. These hyperboloids are depicted in \Cref{fig:Sigma} in the $d=2$ case. 

\subsection{The de,sc- Hamiltonian vector field}

We now discuss the (properly scaled) Hamiltonian vector field $\mathsf{H}_{p[g]}$. 
Defining $\mathsf{H}_{p[g]} = \varrho_{\mathrm{df}} (\varrho_{\mathrm{Pf}}\varrho_{\mathrm{nPf}}\varrho_{\mathrm{Sf}}\varrho_{\mathrm{nFf}}\varrho_{\mathrm{Ff}})^{-1}H_{p[g]}$, 
\begin{equation}
	\mathsf{H}_{p[g]} = \mathsf{H}_{p} \bmod \varrho_{\mathrm{Pf}}\varrho_{\mathrm{nPf}}^2\varrho_{\mathrm{Sf}}\varrho_{\mathrm{nFf}}^2\varrho_{\mathrm{Ff}} \calV_{\mathrm{b}}({}^{\mathrm{de,sc}}\overline{T}^* \bbO).
	\label{eq:misc_255}
\end{equation}
Since our focus is on the situation over $\partial \bbO$, we will calculate the Minkowski case $\mathsf{H}_p$ explicitly, and the error $\mathsf{H}_{p[g]}-\mathsf{H}_p$ will turn out to be negligible.

In the Cartesian coordinate system $(t,\bfx,\tau,\bmxi)\mapsto \tau \dd t + \sum_{i=1}^d \xi_i \dd x_i \in T^* \bbR^{1,d}$, the Hamiltonian vector field of $p$ is
\begin{equation}
 	H_p = 2 \tau \frac{\partial}{\partial t} - 2 \sum_{i=1}^d \xi_i \frac{\partial}{\partial x_i}
\end{equation}
using our sign convention. 

\begin{remark*}
	The reader may see the Hamiltonian vector field defined elsewhere using the opposite sign convention. This does not affect any arguments, and the sign choice is merely conventional. Our convention is the one such that, on the positive frequency sheet $\Sigma_{\mathsf{m},+}$, the Hamiltonian flow propagates \emph{forwards} in time (as depicted in \Cref{fig:O}).
\end{remark*}

With respect to the coordinate system $(\varrho_{\mathrm{nf}},\varrho_{\mathrm{Tf}},\theta,\xi,\zeta,\eta) \mapsto \varrho_{\mathrm{nf}}^{-2} \varrho_{\mathrm{Tf}}^{-1} \xi \mathrm{d} \varrho_{\mathrm{nf}} + \varrho_{\mathrm{nf}}^{-1} \varrho_{\mathrm{Tf}}^{-2} \zeta \mathrm{d} \varrho_{\mathrm{Tf}} + \varrho_{\mathrm{nf}}^{-2}\varrho_{\mathrm{Tf}}^{-1} \eta \dd \theta $, the rescaled Hamiltonian flow $\mathsf{H}_p = \mathsf{H}_{p[g_{\bbM}]}$, defined by \cref{eq:Hp_rescaled}, is given by  
\begin{equation}
	2^{-1} \varrho_{\mathrm{df}}^{-1} \mathsf{H}_p = (\zeta - \xi) \varrho_{\mathrm{nf}}\frac{\partial}{\partial \varrho_{\mathrm{nf}}} + \xi \varrho_{\mathrm{Tf}} \frac{\partial}{\partial \varrho_{\mathrm{Tf}}} + ( 2\eta^2  + \xi^2 - \xi \zeta) \frac{\partial}{\partial \xi} + (\eta^2 + (\xi-\zeta)^2 ) \frac{\partial}{\partial \zeta} + (2\zeta-\xi) V_{\bbS^{d-1}}, 
	\label{eq:H_nfTf}
\end{equation}
where $V_{\bbS^{d-1}}$ is the generator of dilations on $T^* \bbS^{d-1}$. (I.e.\ $V_{\bbS^{d-1}} = \sum_{i=1}^{d-1} \eta_i \partial_{\eta_i}$ with respect to any local coordinate system $\theta_1,\ldots,\theta_{d-1}$ on the sphere at infinity.)

On the other hand, with respect to the coordinate system 
$(\varrho_{\mathrm{nf}},\varrho_{\mathrm{Sf}},\theta,\xi,\zeta,\eta) \mapsto \varrho_{\mathrm{nf}}^{-2} \varrho_{\mathrm{Sf}}^{-1} \xi \mathrm{d} \varrho_{\mathrm{nf}} + \varrho_{\mathrm{nf}}^{-1} \varrho_{\mathrm{Sf}}^{-2} \zeta \mathrm{d} \varrho_{\mathrm{Sf}} + \varrho_{\mathrm{nf}}^{-2}\varrho_{\mathrm{Sf}}^{-1} \eta \dd \theta $,
$\mathsf{H}_p$ is given by 
\begin{equation}
	 2^{-1}\varrho_{\mathrm{df}}^{-1} \mathsf{H}_p = (\xi-\zeta)\varrho_{\mathrm{nf}} \frac{\partial}{\partial \varrho_{\mathrm{nf}}} - \xi \varrho_{\mathrm{Sf}}\frac{\partial}{\partial \varrho_{\mathrm{Sf}}} + ( 2\eta^2  - \xi^2 + \xi \zeta) \frac{\partial}{\partial \xi} + (\eta^2 - (\xi-\zeta)^2 ) \frac{\partial}{\partial \zeta} + (\xi-2\zeta) V_{\bbS^{d-1}}.
	\label{eq:H_nfSf}
\end{equation}

The radial set 
$\calR^-_+ \cup \calR^+_+ \subseteq {}^{\mathrm{de,sc}}T^*_{\mathrm{Ff}}\bbO$ defined by \cref{eq:R_def} is given over $\mathrm{nFf}\cap \mathrm{Ff}$ by $\{\xi=\zeta , |\xi| = \mathsf{m}, \eta=0\}$, and likewise $\calR^-_- \cup \calR^+_- \subseteq {}^{\mathrm{de,sc}}T^*_{\mathrm{Pf}}\bbO $ is given over $\mathrm{nPf}\cap \mathrm{Pf}$ by $\{\xi=\zeta, |\xi|=\mathsf{m}, \eta=0\}$. 

Likewise, for $\varsigma \in \{-,+\}$, the radial sets $\calN_\pm^\varsigma$ are the subsets of $\Sigma_{\mathsf{m},\varsigma}$ defined by
\begin{equation}
	\calN^{\varsigma}_\pm \cap {}^{\mathrm{de,sc}}\overline{T}^*_\alpha \bbO  = 
	\begin{cases}
		\Sigma_{\mathsf{m},\varsigma}\cap {}^{\mathrm{de,sc}}\bbS^*_\alpha \bbO \cap \mathrm{cl}_{{}^{\mathrm{de,sc}}\overline{T}^*_\alpha \bbO }\{\xi=0\} & (\alpha \in \Omega_{\mathrm{nfTf},\pm,T } ), \\
		\Sigma_{\mathsf{m},\varsigma} \cap {}^{\mathrm{de,sc}}\bbS^*_\alpha \bbO \cap \mathrm{cl}_{{}^{\mathrm{de,sc}}\overline{T}^*_\alpha \bbO }\{\xi=0\} & (\alpha \in \Omega_{\mathrm{nfSf},\pm,R } ),
	\end{cases}
\end{equation}
these two definitions agreeing on their overlap. Here $\alpha\in \mathrm{nf}$. The radial sets $\calC_\pm^\varsigma, \calK^\varsigma_\pm \subset \Sigma_{\mathsf{m},\varsigma}$ are 
\begin{align}
	\calC^\varsigma_\pm  &= (\Sigma_{\mathsf{m},\varsigma}  \cap {}^{\mathrm{de,sc}}\bbS^* \bbO \cap {}^{\mathrm{de,sc}}\pi^{-1} (\mathrm{nf} \cap \mathrm{Tf}) \cap \mathrm{cl}_{{}^{\mathrm{de,sc}}\overline{T}^*_{\mathrm{nf}} \bbO }\{\eta=0\})\backslash \calN^\varsigma_\pm, \\
	\calK^\varsigma_\pm  &= (\Sigma_{\mathsf{m},\varsigma} \cap {}^{\mathrm{de,sc}}\bbS^* \bbO \cap {}^{\mathrm{de,sc}}\pi^{-1} (\mathrm{nf} \cap \mathrm{Sf}) \cap \mathrm{cl}_{{}^{\mathrm{de,sc}}\overline{T}^*_{\mathrm{nf} } \bbO }\{\eta=0\})\backslash \calN^\varsigma_\pm,
\end{align}
for 
\begin{itemize}
	\item $\mathrm{Tf} = \mathrm{Ff}$ and $\mathrm{nf} = \mathrm{nFf}$ in the $+$ case ($\sigma=+$ in $\calN^\varsigma_\sigma,\calC^\varsigma_\sigma,\dots$),
	\item $\mathrm{Tf}=\mathrm{Pf}$  and $\mathrm{nf} = \mathrm{nPf}$ in the $-$ case.
\end{itemize}
We can now define the final radial sets $\calA_\pm^\varsigma \subset \Sigma_{\mathsf{m},\varsigma}$ to be the components of the remaining vanishing set of $\mathsf{H}_p$ in $\Sigma_{\mathsf{m},\varsigma}$, which can be seen to lie at fiber infinity. We will compute shortly that these are nice submanifolds.

Consider fiber infinity over $\mathrm{nFf} \cap \mathrm{Ff}$, using the coordinate system in the half-space $\{\zeta<0\}$ over $\Omega_{\mathrm{nfTf},+,T}$ given by 
\begin{equation}
	\rho = -\frac{1}{\zeta }, \qquad s =   \frac{\xi}{\zeta} , \qquad \hat{\eta} = -\frac{\eta}{\zeta}. 
	\label{eq:misc_co1}
\end{equation}
Rewriting the formula \cref{eq:misc_258} for $\tilde{p}$ in these coordinates, 
\begin{equation} 
	\tilde{p} = \rho^{-2} \varrho_{\mathrm{df}}^2 ( s^2-2s+\hat{\eta}^2 + \rho^2 \mathsf{m}^2),
\end{equation} 
which is $s^2-2s+\hat{\eta}^2+\rho^2\mathsf{m}^2 = (s-1)^2 + \hat{\eta}^2 -1 + \rho^2 \mathsf{m}^2$ up to a smooth, nonvanishing factor in a neighborhood of the part of the characteristic set under consideration. We therefore have
\begin{equation} 
	\Sigma_{\mathsf{m},+} = \{ (s-1)^2 + \hat{\eta}^2 = 1 - \rho^2 \mathsf{m}^2 \}
\end{equation} 
locally.
Rewriting \cref{eq:H_nfTf}, we get 
\begin{multline}
	2^{-1} \rho \varrho_{\mathrm{df}}^{-1} \mathsf{H}_p = - (1 - s)  \varrho_{\mathrm{nf}}\frac{\partial}{\partial \varrho_{\mathrm{nf}}} - s  \varrho_{\mathrm{Tf}} \frac{\partial}{\partial \varrho_{\mathrm{Tf}}} + \Big[  \hat{\eta}^2  + (s-1)^2 \Big] \rho \frac{\partial}{\partial \rho} - (2-s)\Big[ \hat{\eta}^2 + s(s-1) \Big] \frac{\partial}{\partial s}  \\+ \Big( \hat{\eta}^2 +s^2-s-1 \Big) V_{\bbS^{d-1}}   .
	\label{eq:H_nfTf_df}
\end{multline}
(In local coordinates for the sphere at spatial infinity, $V_{\bbS^{d-1}}=\sum_{j=1}^{d-1} \hat{\eta}_j \partial_{\hat{\eta}_j}$.)
This only vanishes over the boundary of $\bbO$. 

Let us examine the situation over future null infinity, $\mathrm{nFf}$, where $\varrho_{\mathrm{nf}} = \varrho_{\mathrm{nFf}}= 0$. Then:
\begin{itemize}
	\item if $\varrho_{\mathrm{Tf}}\neq 0$, then $\mathsf{H}_p$ only vanishes on $\Sigma_{\mathsf{m},+}$ at $\Sigma_{\mathsf{m},+}\cap \{s=0\} = \Sigma_{\mathsf{m},+}\cap \{s=0=\rho,\hat{\eta}\}$, which is just the set $\calN_+^+$, 
	\item over the corner $\mathrm{nFf} \cap \mathrm{Ff}$, if $\rho\neq 0$ then $\mathsf{H}_p$ can only vanish on $\Sigma_{\mathsf{m},+}$ if $\hat{\eta} = 0$ and $s=1$, which corresponds to $\calR_+^+$, 
	\item at $\rho = 0$, $\mathsf{H}_p$ vanishes on $\Sigma_{\mathsf{m},+}$ only if $s=2$ or $s=0$ (the latter as already noted), in which case $\hat{\eta}=0$. The former possibility corresponds to $\calC^+_+$. 
\end{itemize}
So, in these coordinates, 
\begin{align}
	\begin{split} 
	\calN^+_+ &= \{\varrho_{\mathrm{nf}}  = \rho = \hat{\eta} = s = 0 \}, \\
	\calR^+_+ &= \{\varrho_{\mathrm{Tf}}  = \hat{\eta} =0, s=1, \rho=\mathsf{m}^{-1} \}, \\
	\calC^+_+ &= \{\varrho_{\mathrm{nf}} =\varrho_{\mathrm{Tf}} = \rho = \hat{\eta} = 0, s=2\}.
	\end{split} 
\end{align}
In the case $d=2$, $\mathsf{H}_p$ restricted to $\Sigma_{\mathsf{m},+} \cap {}^{\mathrm{de,sc}}\pi^{-1}(\mathrm{nFf}\cap \mathrm{Ff})$ is depicted in \Cref{fig:flowplot}.

The situation on $\Sigma_{\mathsf{m},-}$, and over past null infinity, is similar.

\begin{figure}[t]
	\centering 
	\includegraphics[scale=.6]{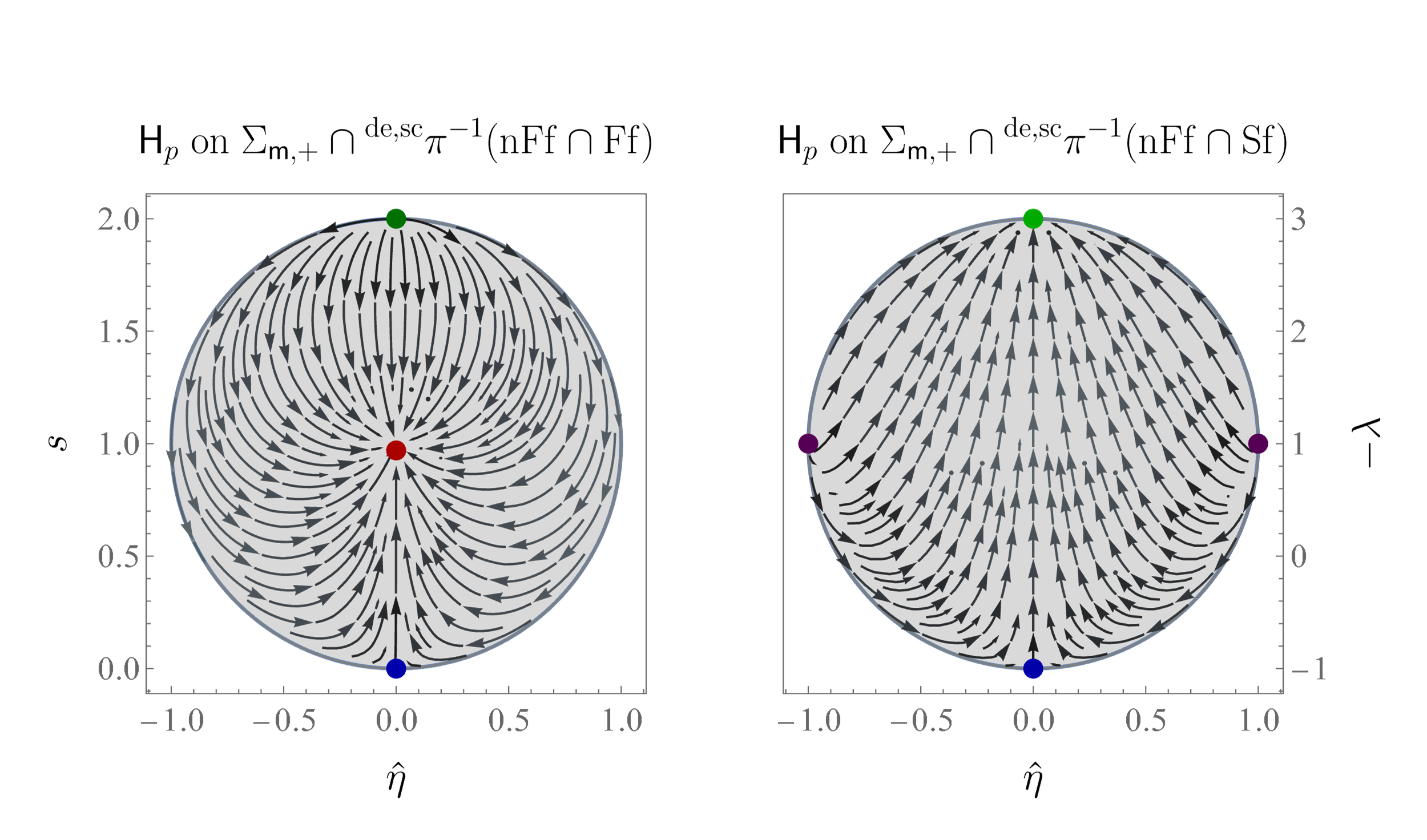}
	\caption{The vector field $\mathsf{H}_p$ plotted (in the case $d=2$) on the hyperboloid $\Sigma_{\mathsf{m},+}$ over $\mathrm{nFf}\cap \mathrm{Ff}$ (\emph{left}) and $\mathrm{nFf}\cap \mathrm{Sf}$ (\emph{right}), versus the coordinates $\hat{\eta}$ and $s$ or $\hat{\eta}$ and $-\lambda$. Increasing $\rho$ corresponds to decreasing radii from the plot origin, and the boundary of the disk lies at fiber infinity; thus, the gray disks in the figure are the compactified hyperboloid viewed ``from above.'' In the left plot, we can see the portions over $\mathrm{nFf}\cap \mathrm{Ff}$ of the radial sets ${\color{darkblue}\calN^+_+},{\color{darkcandyapp}\calR^+_+},{\color{darkgreen}\calC^+_+}$, located at $\hat{\eta} = 0$ and $s=0,1,2$, respectively. In the right plot, we can see the portions of $\mathrm{nFf}\cap \mathrm{Sf}$ of the radial sets ${\color{darkblue}\calN^+_+},{\color{darkpurple}\calA^+_+},{\color{mygreen}\calK^+_+}$, located at $\hat{\eta} = 0,\pm 1,0$ and $-\lambda=-1,1,3$, respectively. Restricting to the invariant subset $\{\hat{\eta}=0\}$, the plots show the same thing as \Cref{fig:O} over the relevant corners of $\bbO$.}
	\label{fig:flowplot}
\end{figure}

We now discuss the situation near the spacelike corner of null infinity.
Using instead the $\xi,\zeta$ coordinates defined over $\Omega_{\mathrm{nfSf},\pm,R}$,  the characteristic set crosses $\{\zeta=0\}$, as shown in
\Cref{fig:Sigma}. 
Consequently, $\rho=1/\zeta$ is not well-defined on the whole characteristic set.
Instead, consider the coordinate system in the half-space $\{\zeta-\xi >0\}$ over $\Omega_{\mathrm{nfSf},+,R}$ given by 
\begin{equation}
	\rho = \frac{1}{\zeta-\xi}, \qquad \lambda = \frac{\zeta+\xi}{\zeta-\xi}, \qquad \hat{\eta} = \frac{\eta}{\zeta-\xi}. 
	\label{eq:misc_co2}
\end{equation}
In terms of these coordinates, $\tilde{p} = \rho^2 \varrho_{\mathrm{df}}^2( -4^{-1}(1-\lambda)(3+\lambda) + \hat{\eta}^2 + \rho^2 \mathsf{m}^2)$, which is $-4^{-1}(1-\lambda)(3+\lambda) + \hat{\eta}^2 + \rho^2 \mathsf{m}^2 = 4^{-1}(\lambda+1)^2 + \hat{\eta}^2 -1 + \rho^2 \mathsf{m}^2$ up to a smooth, nonvanishing factor. Therefore, 
\begin{equation}
	\Sigma_{\mathsf{m},+} = \{ 4^{-1}(\lambda+1)^2 + \hat{\eta}^2 = 1 - \rho^2 \mathsf{m}^2 \}
\end{equation}
locally. Rewriting \cref{eq:H_nfSf} in terms of these coordinates, 
\begin{multline}
		2^{-1} \rho \varrho_{\mathrm{df}}^{-1} \mathsf{H}_p  = - \varrho_{\mathrm{nf}} \frac{\partial}{\partial \varrho_{\mathrm{nf}}} + \frac{1}{2}(1-\lambda) \varrho_{\mathrm{Sf}}\frac{\partial}{\partial \varrho_{\mathrm{Sf}}} + \frac{1}{2} \Big[ 2 \hat{\eta}^2+\lambda+1 \Big]\rho \frac{\partial}{\partial \rho} + \frac{1}{2}(\lambda+3)\Big[ 2 \hat{\eta}^2+\lambda-1 \Big] \frac{\partial}{\partial \lambda} \\ 
			+ ( \hat{\eta}^2 -1 ) V_{\bbS^{d-1}}  . 
		\label{eq:H_nfSf_df}
\end{multline}
Then: 
\begin{itemize}
	\item if $\varrho_{\mathrm{nf}} \neq 0$, then $\mathsf{H}_p$ is nonvanishing, so we have no radial set over the interior of $\mathrm{Sf}$. We could also have seen this from the fact that the situation over $\mathrm{Sf}^\circ$ is canonically identifiable with the situation in the sc-phase space, where no radial set lies over spacelike infinity.  
	\item Moreover, if $\rho \neq 0$ then $\mathsf{H}_p$ is also nonvanishing. This is because, since $\lambda+3>0$ on $\Sigma_{\mathsf{m},+}\backslash \mathrm{df}$, the coefficients of the $\partial_\lambda,\partial_\rho$ terms in \cref{eq:H_nfSf_df}  do not vanish simultaneously outside of $\mathrm{df}$; an alternative justification is that the $\partial_{\hat{\eta}}$ terms vanish outside of $\mathrm{df}$ only if $\hat{\eta}=0$ (in the right panel of \Cref{fig:flowplot}, the vector field is vertical only when $\lVert \hat{\eta} \rVert \in \{0,1\}$, with the latter possibility occurring only at fiber infinity), and then the coefficient of the $\partial_\lambda$ term is nonvanishing, since $\lambda<1$ outside of $\mathrm{df}$. 
	
	So, the radial set must be at fiber infinity. 
	\item At fiber infinity, $\mathsf{H}_p$ vanishes only if $\lambda \in\{ 1,-1,-3\}$. If $\lambda =1,-3$, then $\hat{\eta} = 0$, and, if $\lambda=-1$, then $\lVert \hat{\eta} \rVert^2= 1$. These possibilities correspond to $\calN^+_+$, $\calK^+_+$, and $\calA^+_+$, respectively, where $\mathsf{H}_p$ does, in fact, vanish. 
\end{itemize}
Thus,
\begin{align}
	\begin{split} 
	\calN^+_+ &= \{\varrho_{\mathrm{nf}}  = \rho =0, \lambda=1,\hat{\eta}=0  \}, \\
	\calA^+_+ &= \{\varrho_{\mathrm{nf}} = \varrho_{\mathrm{Sf}} = \rho =0, \lambda=-1,\lVert \hat{\eta} \rVert =1  \}, \\
	\calK^+_+ &= \{\varrho_{\mathrm{nf}} = \varrho_{\mathrm{Sf}} = \rho =0, \lambda=-3,\hat{\eta}=0  \}. 
	\end{split} 
	\label{eq:NAK_defs}
\end{align}
The situation on $\Sigma_{\mathsf{m},-}$, and over past null infinity, is similar. The $d=2$ case is depicted in \Cref{fig:flowplot}.

\begin{remark*}
	One feature of the dynamics that can be seen from \Cref{fig:flowplot} and \Cref{fig:globalflowplot} is the flow from $\calN^+_+ \cap {}^{\mathrm{de,sc}}\pi^{-1}(\mathrm{Sf}\cap \mathrm{nFf})$ to $\calK^+_+$ through finite de,sc-frequencies, across $\mathrm{nFf}$ along fiber infinity to $\calC^+_+$, and then around to $\calN^+_+ \cap {}^{\mathrm{de,sc}}\pi^{-1}(\mathrm{Ff}\cap \mathrm{nFf})$ along fiber infinity (in the $\hat{\eta}$ direction). (This is in addition to the other sort of path from one end of $\calN$ to the other shown in \Cref{fig:O}. That path, however, crosses over the timelike caps.)
	
	This bolsters the conclusion, forewarned in the introduction, that in order to control our solution at $\calN^+_+ \cap {}^{\mathrm{de,sc}}\pi^{-1}(\mathrm{Ff}\cap \mathrm{nFf})$, we need to already have control at $\calN^+_+ \cap {}^{\mathrm{de,sc}}\pi^{-1}(\mathrm{Sf}\cap \mathrm{nFf})$. 
\end{remark*}

We now proceed with a few elementary computations in preparation for the propagation and radial point estimates in the next section. 

\subsection{Flow across null infinity}
The most basic of these, which captures the fact that the Hamiltonian flow moves us along $\mathrm{nf}$ (except at $\calN$), is:

\begin{proposition}
	On $\Sigma_{\mathsf{m},\varsigma} \cap {}^{\mathrm{de,sc}}\pi^{-1}(\mathrm{nf}^\circ ) \backslash \calN^\varsigma_\pm$, $\alpha = |t| - r$ satisfies $\pm \varsigma  \mathsf{H}_{p[g]} \alpha >0$. 
	\label{prop:propagation_lemma}
\end{proposition}
\begin{proof}
	We only discuss $\calN^+_+$, as the others differ by sign switches.
	We cover $\Sigma_{\mathsf{m},+} \cap {}^{\mathrm{de,sc}}\pi^{-1}(\mathrm{nf}^\circ )$ by ${}^{\mathrm{de,sc}}\pi^{-1}(\Omega_{\mathrm{nfTf},+,T } ) \cup {}^{\mathrm{de,sc}}\pi^{-1}(\Omega_{\mathrm{nfSf},+,R } )$. In the former, we can write $\alpha = \varrho_{\mathrm{Tf}}^{-1}-T$, so, by \cref{eq:H_nfTf},
	\begin{equation}
		\mathsf{H}_p \alpha = - \frac{2 \varrho_{\mathrm{df}} \xi }{\varrho_{\mathrm{Tf}}} 
		\label{eq:misc_278}
	\end{equation}
	over $\mathrm{nf}^\circ$, and $ \xi < 0$ on $\Sigma_{\mathsf{m},+} \cap {}^{\mathrm{de,sc}}\pi^{-1}( \Omega_{\mathrm{nfTf},+,T} ) \backslash \calN^+_+$ (see \Cref{fig:Sigma}). So, the right-hand side of \cref{eq:misc_278} is positive.
	
	On the other hand, over $\Omega_{\mathrm{nfSf},+,R}$, we write $\alpha = - \varrho_{\mathrm{Sf}}^{-1}+R$, so by \cref{eq:H_nfSf} we have 
	\begin{equation}
		\mathsf{H}_p \alpha = -  \frac{2\xi \varrho_{\mathrm{df}}  }{\varrho_{\mathrm{Sf}}}
	\end{equation}
	over $\mathrm{nf}^\circ$. It is also the case that, using the definition of $\xi$ relevant to $\Omega_{\mathrm{nfSf},+,R}$, $\xi<0$ on $\Sigma_{\mathsf{m},+} \cap {}^{\mathrm{de,sc}}\pi^{-1}( \Omega_{\mathrm{nfSf},+,R} ) \backslash \calN^+_+$ (see \Cref{fig:Sigma} again). So, $\mathsf{H}_{p} \alpha >0$ on $\Sigma_{\mathsf{m},+} \cap {}^{\mathrm{de,sc}}\pi^{-1}(\mathrm{nf}^\circ ) \backslash \calN^+_+$. The same therefore holds for $\mathsf{H}_{p[g]}$. 
\end{proof}

\begin{figure}
	\floatbox[{\capbeside\thisfloatsetup{capbesideposition={left,bottom},capbesidewidth=8cm,capbesidesep=none}}]{figure}[\FBwidth]
	{\caption{The vector field $\mathsf{H}_p$ plotted (in the case $d=2$) at the part of fiber infinity in $\Sigma_{\mathsf{m},+}$ over $\mathrm{nFf}$. 
			Only the $\hat{\eta}>0$ (corresponding to, say, clockwise motion in $\bbS^1_\theta$) half is shown. The horizontal axis is parametrized by $\operatorname{arctan}(t-r)$, so the left end $\{\operatorname{arctan}(t-r) = -\pi/2\}$ is over $\mathrm{nFf}\cap \mathrm{Sf}$ and the right end $\{\operatorname{arctan}(t-r) = +\pi/2\}$ is over $\mathrm{nFf}\cap \mathrm{Ff}$, and the vertical axis is parametrized by an appropriate coordinate interpolating between the coordinates $s$ and $-\lambda$ used in \Cref{fig:flowplot}.	
			The radial sets are colored as in \Cref{fig:flowplot}:  ${\color{darkblue}\calN^+_+}$ along the bottom, ${\color{mygreen}\calK^+_+}$ in the top left, ${\color{darkgreen}\calC^+_+}$ in the top right, and ${\color{darkpurple}\calA^+_+}$ in the center left.}
		\label{fig:globalflowplot}}
	{\includegraphics[scale=.45]{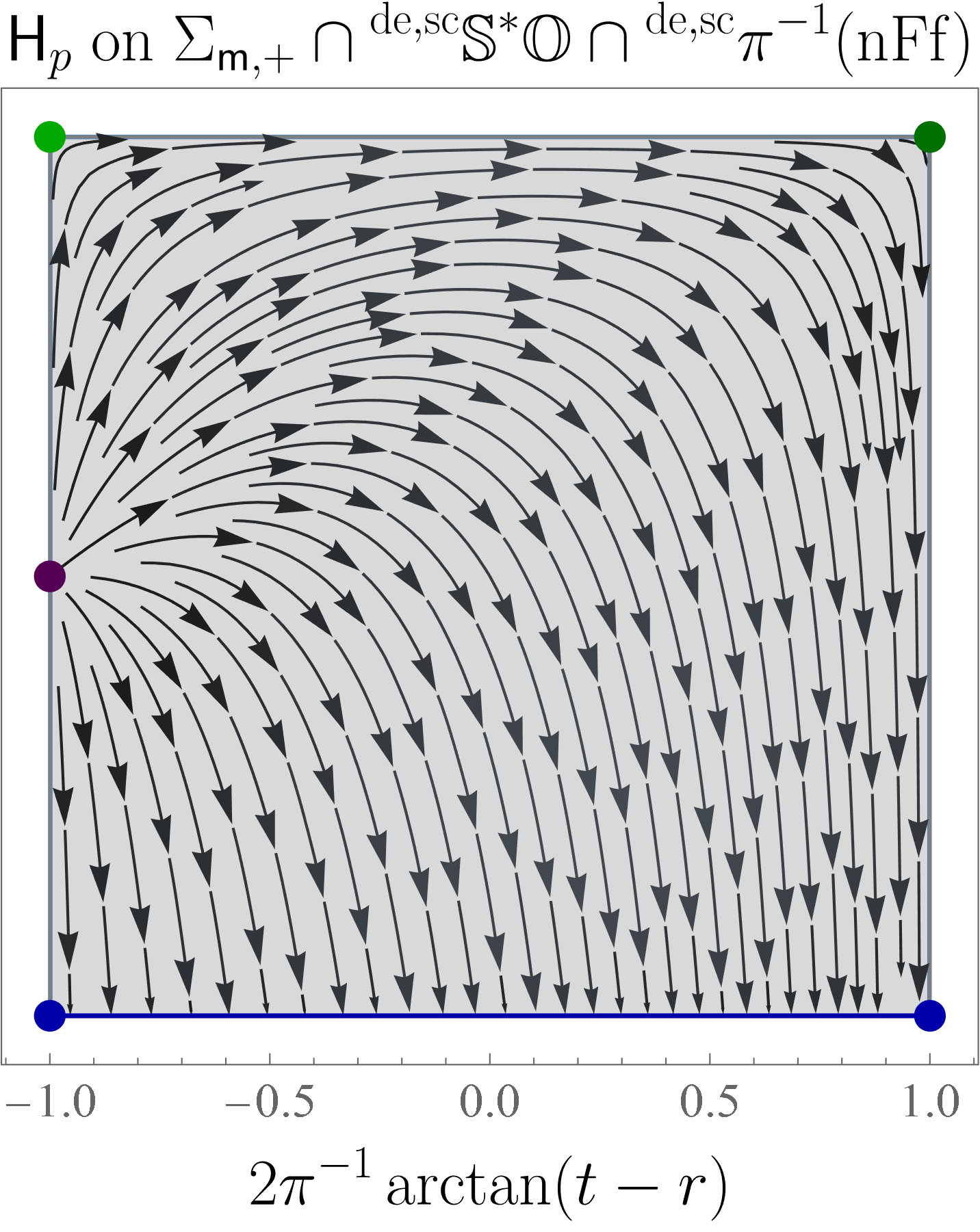}}
\end{figure}

Next is the source/sink structure of the flow at the various radial sets. 

To simplify the discussions of the radial sets $\calA,\calK,\calC$ over the corners of $\bbO$, we can assume that $\varrho_{\mathrm{df}}$ is given by the coordinates labeled $\rho$ in \cref{eq:misc_co1} and \cref{eq:misc_co2} near the corners of $\bbO$. Changing to another bdf multiplies $\mathsf{H}_{p[g]}$ by a nonvanishing factor, so does not change the conclusions of the propositions below. Discussing $\calN$ requires a bit more care --- the independence of the result on $\varrho_{\mathrm{df}}$ will require proof.

\subsection{$\calA$}

Consider the linearization of $\mathsf{H}_p[g]$ at $\calA^+_+$. This is the same as the linearization of $\mathsf{H}_p$ at that set, which recall is given in local coordinates by \cref{eq:NAK_defs} (so $\lambda=-1$, $\lVert \hat{\eta} \rVert=1$, and $\varrho_{\mathrm{nf}},\varrho_{\mathrm{Sf}},\rho=0$). \Cref{eq:H_nfSf_df} says 
\begin{multline}
2^{-1} \mathsf{H}_p = -\varrho_{\mathrm{nf}} \frac{\partial}{\partial \varrho_{\mathrm{nf}}} + \frac{1}{2}(2-{\color{lightgray}(\lambda+1)}) \varrho_{\mathrm{Sf}}\frac{\partial}{\partial \varrho_{\mathrm{Sf}}} + \frac{1}{2} \Big[{\color{lightgray} 2 (\hat{\eta}^2-1)+(\lambda+1)}+2 \Big]\rho \frac{\partial}{\partial \rho} \\ + \frac{1}{2}({\color{lightgray}(\lambda+1)}+2)\Big[ 2 (\lVert\hat{\eta}\rVert-1)({(\color{lightgray}\lVert\hat{\eta}\rVert -1)} +2)+\lambda+1 \Big] \frac{\partial}{\partial \lambda} 
+( \lVert \hat{\eta} \rVert-1 )( {\color{lightgray}(\lVert\hat{\eta}\rVert-1)} +2 ) ({\color{lightgray}(\lVert \hat{\eta} \rVert -1)} + 1) \frac{\partial}{\partial \lVert \hat{\eta} \rVert} , 
\end{multline}
where we have swapped from using $\hat{\eta}$ as a coordinate to using $\lVert \hat{\eta} \rVert$ and $\hat{\eta} / \lVert \hat{\eta} \rVert \in \bbS^{d-2}$ (the latter only if $d\geq 3$). 
The terms in light gray are those that are dropped when linearizing: 
\begin{equation}
	2^{-1} \mathsf{H}_p \approx -\varrho_{\mathrm{nf}} \frac{\partial}{\partial \varrho_{\mathrm{nf}}} +  \varrho_{\mathrm{Sf}}\frac{\partial}{\partial \varrho_{\mathrm{Sf}}} +  \rho \frac{\partial}{\partial \rho} \\ + \Big[ -4 (1-\lVert\hat{\eta}\rVert)+\lambda+1 \Big] \frac{\partial}{\partial \lambda} 
	+2(\lVert \hat{\eta} \rVert -1) \frac{\partial}{\partial \lVert \hat{\eta} \rVert}
	\label{eq:linearization_at_A}
\end{equation}
near $\calA^+_+$. The right-hand side can be written 
\begin{equation}
	\begin{pmatrix}
		\partial_{\varrho_{\mathrm{nf}}} \\ \partial_{\varrho_{\mathrm{Sf}}} \\ \partial_{\rho} \\ \partial_{\lambda+1} \\ \partial_{1-\lVert \hat{\eta} \rVert} 
	\end{pmatrix}^\intercal 
	\begin{pmatrix}
		- 1 & 0 & 0 & 0 & 0 \\ 
		0 & 1 & 0 & 0 & 0 \\
		0 & 0 & 1 & 0 & 0 \\
		0 & 0 & 0 & 1 &  -4 \\ 
		0 & 0 & 0 & 0 & 2
	\end{pmatrix} \begin{pmatrix}
	\varrho_{\mathrm{nf}} \\ \varrho_{\mathrm{Sf}} \\ \rho \\ \lambda+1 \\ 1-\lVert \hat{\eta} \rVert 
\end{pmatrix}
	.
\end{equation}
If we ignore the upper-left entry, the matrix in the middle has all positive eigenvalues.
So, $\calA^+_+$ is a source in all of the coordinate directions except $\varrho_{\mathrm{nf}}$, in which it is a sink, as depicted in the various figures above.

Precisely:
\begin{proposition}
	Fix signs $\varsigma,\sigma \in \{-,+\}$. 
	Letting $\aleph \in C^\infty({}^{\mathrm{de,sc}}\overline{T}^* \bbO)$ satisfy $\aleph = \varrho_{\mathrm{Sf}}^2 + \rho^2 + (\lambda+1)^2$ near $\calA^\varsigma_\sigma$, 
	the symbol 
	\begin{equation} 
		F_1 =F_1[g] = \mathsf{H}_{p[g]} \aleph - 4\varsigma \sigma \aleph   \in C^\infty({}^{\mathrm{de,sc}}\overline{T}^* \bbO)
	\end{equation} 
	vanishes cubically at $\calA^\varsigma_\sigma$ within $\Sigma_{\mathsf{m}}[g]$, in the sense that 
	\begin{equation} 
		F_1 \in \aleph^{3/2} L^\infty + \varrho_{\mathrm{nf}} \aleph L^\infty+ \tilde{p}[g] \aleph^{1/2} L^\infty
		\label{eq:misc_m7o}
	\end{equation} 
	locally.
	\label{prop:aleph}
\end{proposition}
\begin{proof}
	We only consider the case of $\calA^+_+$, the other three cases being analogous. Before doing so, it is useful to reduce to the case where $g$ is the Minkowski metric. Working in some local coordinate chart $\theta_1,\ldots,\theta_{d-1}$ on $\smash{\bbS^{d-1}_{\bmtheta}}$, we can write 
		\begin{equation}
			\frac{\mathsf{H}_{p[g]} - \mathsf{H}_p}{\varrho_{\mathrm{Pf}}\varrho_{\mathrm{nPf}}\varrho_{\mathrm{Sf}}\varrho_{\mathrm{nFf}}\varrho_{\mathrm{Ff}} } = V_{\mathrm{nf}} \varrho_{\mathrm{nf}}\frac{\partial}{\partial \varrho_{\mathrm{nf}}} + V_{\mathrm{Sf}} \varrho_{\mathrm{Sf}}\frac{\partial}{\partial \varrho_{\mathrm{Sf}}} + V_\rho \rho \frac{\partial}{\partial \rho} + V_\lambda \frac{\partial}{\partial \lambda} + \sum_{i=1}^{d-1} \Big( V_{\theta_i} \frac{\partial}{\partial \theta_i} + V_{\hat{\eta}_i} \frac{\partial}{\partial \hat{\eta}_i} \Big)
		\end{equation}
		for some $V_{\mathrm{nf}},V_{\varrho_{\mathrm{Sf}}},V_\rho,V_\lambda,V_{\theta_i},V_{\hat{\eta}_i} \in C^\infty({}^{\mathrm{de,sc}}\overline{T}^* \bbO )$. Applying this to $\aleph$, the result is 
		\begin{equation}
			\frac{\mathsf{H}_{p[g]} - \mathsf{H}_p}{\varrho_{\mathrm{Pf}}\varrho_{\mathrm{nPf}}\varrho_{\mathrm{Sf}}\varrho_{\mathrm{nFf}}\varrho_{\mathrm{Ff}} } \aleph = 2 V_{\mathrm{Sf}} \varrho_{\mathrm{Sf}}^2 + 2 V_\rho \rho^2 + 2 V_\lambda (\lambda+1). 
		\end{equation}
		The ratios $\varrho_{\mathrm{Sf}}^2 / \aleph$, $\rho^2 / \aleph$, and $\varrho_{\mathrm{Sf}}(\lambda+1)/\aleph$ all lie in $L^\infty$ (locally). Thus, we can absorb $(\mathsf{H}_{p[g]} - \mathsf{H}_p) \aleph$ into the $\varrho_{\mathrm{nf}} \aleph L^\infty$ term in \cref{eq:misc_m7o}. Also, using \cref{eq:weakest},
		\begin{equation}
			\tilde{p} \aleph^{1/2} L^\infty = \tilde{p}[g_{\bbM}] \aleph^{1/2} L^\infty \subseteq \tilde{p}[g] \aleph^{1/2} L^\infty + \varrho_{\mathrm{nf}} \varrho_{\mathrm{Sf}}\aleph^{1/2}  L^\infty \subseteq \tilde{p}[g] \aleph^{1/2} L^\infty + \varrho_{\mathrm{nf}} \aleph L^\infty.
		\end{equation}
		Consequently, it suffices to consider only the case where $g$ is the Minkowski metric. 
	
	By \cref{eq:H_nfSf_df}, 
	\begin{equation}
		2^{-1}\mathsf{H}_p \aleph =  (1-\lambda)\varrho_{\mathrm{Sf}}^2 + (2\hat{\eta}^2+\lambda+1) \rho^2 + (\lambda+3)(2\hat{\eta}^2+\lambda-1)(\lambda+1)  
		\label{eq:misc_alg}
	\end{equation}
	near $\calA^+_+$. We write the right-hand side as $2\aleph  + F_{1,0}$ for 
	\begin{equation}
		F_{1,0} = -(1+\lambda)\varrho_{\mathrm{Sf}}^2+ (2\hat{\eta}^2 +\lambda -1)( \rho^2 +(\lambda+1)^2) + 4 (\hat{\eta}^2-1) (\lambda+1),
	\end{equation}
	which vanishes cubically at $\calA^+_+$ within $\Sigma_{\mathsf{m}}$. In order to see this, we write $F_{1,0}=F_{1,1}+\tilde{p}F_{1,2}$ for 
	\begin{multline}
		F_{1,1} = -(1+\lambda)\varrho_{\mathrm{Sf}}^2+ 2^{-1} (1-\lambda)(\lambda+1) (\rho^2+ (\lambda+1)^2) - 2 \rho^2 \mathsf{m}^2(\rho^2+ (\lambda+1)^2) \\ - (\lambda+1)^3 -4 \rho^2 \mathsf{m}^2(\lambda+1)
	\end{multline}
	and $F_{1,2} = 2 \rho^2 + 2 (\lambda+1)^2 +4(\lambda+1) \in S_{\mathrm{de,sc}}^{0,\mathsf{0}} \cap \aleph^{1/2} L^\infty$. Term-by-term, we see that $F_{1,1} \in \aleph^{3/2}L^\infty$. 
\end{proof}

\subsection{$\calN$}
The assumption \cref{eq:p1} guarantees that the linearization of $\mathsf{H}_{p[g]}$ at $\calN$ agrees with that of $\mathsf{H}_p$. (Note that this would not follow from the weaker \cref{eq:weakest}.)
Let us examine the linearization of $\mathsf{H}_p$ at $\calN^+_+$. 

First, using the coordinates in which \cref{eq:H_nfSf_df}, is written, \cref{eq:NAK_defs} says that the radial set in question is located at $\lambda=1$ and all other coordinates besides $\varrho_{\mathrm{Sf}}$ zero. Thus, in 
\begin{multline}
	2^{-1} \mathsf{H}_p = -\varrho_{\mathrm{nf}} \frac{\partial}{\partial \varrho_{\mathrm{nf}}} + \frac{1}{2}(1-\lambda) \varrho_{\mathrm{Sf}}\frac{\partial}{\partial \varrho_{\mathrm{Sf}}} + \frac{1}{2} \Big[{\color{lightgray} 2 \hat{\eta}^2+(\lambda-1)}+2 \Big]\rho \frac{\partial}{\partial \rho} \\ + \frac{1}{2}({\color{lightgray}(\lambda-1)}+4)\Big[ {\color{lightgray}2\hat{\eta}^2}+\lambda-1 \Big] \frac{\partial}{\partial \lambda} 
	+({\color{lightgray}\hat{\eta}^2}-1) V_{\bbS^{d-1}} ,  
\end{multline}
the light gray terms are dropped when linearizing: 
\begin{equation}
	2^{-1} \mathsf{H}_{p[g]} \approx 2^{-1} \mathsf{H}_p  \approx  -\varrho_{\mathrm{nf}} \frac{\partial}{\partial \varrho_{\mathrm{nf}}} + \frac{1}{2}(1-\lambda) \varrho_{\mathrm{Sf}}\frac{\partial}{\partial \varrho_{\mathrm{Sf}}}  +  \rho \frac{\partial}{\partial \rho}  + 2(\lambda-1) \frac{\partial}{\partial \lambda}  - V_{\bbS^{d-1}} 
	\label{eq:misc_280}
\end{equation}
near $\calN^+_+$. 
From this, we see that $\calN^+_+$ is a sink in the $\varrho_{\mathrm{nf}}$ and $\hat{\eta}$-directions and a source in the $\rho,\lambda$ directions, as depicted in the various figures above. The $\varrho_{\mathrm{Sf}}$ direction (along $\calN$) is neutral.

The situation in the coordinates in \cref{eq:H_nfTf_df} is similar. Indeed, \cref{eq:H_nfTf_df} says 
\begin{multline}
	2^{-1} \mathsf{H}_p = - (1 - {\color{lightgray}s})  \varrho_{\mathrm{nf}}\frac{\partial}{\partial \varrho_{\mathrm{nf}}} - s  \varrho_{\mathrm{Tf}} \frac{\partial}{\partial \varrho_{\mathrm{Tf}}} +  \Big[  {\color{lightgray}\hat{\eta}^2 + s^2-2s}  +1 \Big] \rho \frac{\partial}{\partial \rho} \\ - (2-{\color{lightgray}s})\Big[ {\color{lightgray}\hat{\eta}^2 + s^2}-s \Big] \frac{\partial}{\partial s}  + \Big( {\color{lightgray}\hat{\eta}^2 +s^2-s}-1 \Big) V_{\bbS^{d-1}},
\end{multline}
where as above the light gray terms are the ones dropped when linearizing at $\calN^+_+$  (which in these coordinates is located where all coordinates except $\varrho_{\mathrm{Tf}}$ vanish):
\begin{equation}
	2^{-1} \mathsf{H}_p \approx  -\varrho_{\mathrm{nf}} \frac{\partial}{\partial \varrho_{\mathrm{nf}}} - s  \varrho_{\mathrm{Tf}} \frac{\partial}{\partial \varrho_{\mathrm{Tf}}}  +  \rho \frac{\partial}{\partial \rho}  + 2s \frac{\partial}{\partial s}  - V_{\bbS^{d-1}}, 
	\label{eq:misc_282}
\end{equation}
thus corroborating the conclusion above and showing that, in the relevant sense, it holds ``all the way up to'' $\mathrm{Ff}$.

We now prove several precise symbolic consequences:
\begin{proposition}
	Fix signs $\varsigma,\sigma \in \{-,+\}$. 
	Setting $a(m,\ell) = \varrho_{\mathrm{df}}^{m}\varrho_{\mathrm{nf}}^{\ell}$, the symbol $\alpha=\alpha[g](m,\ell) \in C^\infty({}^{\mathrm{de,sc}}\overline{T}^* \bbO )$ defined by $\mathsf{H}_{p[g]} a = \alpha a$ satisfies 
	\begin{enumerate}
		\item $\varsigma \sigma \alpha>0$ on $\calN^\varsigma_\sigma$ if $m>\ell$, and
		\item $\varsigma \sigma \alpha<0$ on $\calN^\varsigma_\sigma$ if $m<\ell$.
	\end{enumerate}
	\label{prop:alpham0s0}
\end{proposition}
\begin{proof}
	We check the case of $\calN^+_+$, with the other three being analogous. Before doing so, it is useful to reduce to the simplest case:
	\begin{itemize}
		\item 	We have $\alpha[g] = \alpha[g_{\bbM} ] + a^{-1} (\mathsf{H}_{p[g]} - \mathsf{H}_p )a$. Because $\mathsf{H}_{p[g]} - \mathsf{H}_p \in \varrho_{\mathrm{Pf}}\varrho_{\mathrm{nPf}}\varrho_{\mathrm{Sf}}\varrho_{\mathrm{nFf}}\varrho_{\mathrm{Ff}} \calV_{\mathrm{b}}({}^{\mathrm{de,sc}}\overline{T}^* \bbO)$,
		\begin{equation} 
			(\mathsf{H}_{p[g]} - \mathsf{H}_p )a = \alpha_1[g] a
		\end{equation} 
		for some $\alpha_1[g] \in C^\infty({}^{\mathrm{de,sc}}\overline{T}^* \bbO )$ vanishing over the boundary of $\bbO$. 
		So, $\alpha[g] = \alpha[g_{\bbM}] + \alpha_1[g]$ satisfies the conditions in the proposition if and only if $\alpha[g_{\bbM}]$ does. It therefore suffices to prove the result in case when $g$ is the Minkowski metric. 
		\item If $\varrho_{\mathrm{df},0}$ is another choice of bdf of $\mathrm{df}$, then 
		\begin{equation}
			\mathsf{H}_{p[g]}  \varrho_{\mathrm{df},0}^m \varrho_{\mathrm{nf}}^\ell = \Big(\alpha + \Big(\frac{\varrho_{\mathrm{df}}}{\varrho_{\mathrm{df},0}} \Big)^m \mathsf{H}_{p[g]} \Big(\frac{\varrho_{\mathrm{df},0}}{\varrho_{\mathrm{df}}} \Big)^m \Big) \varrho_{\mathrm{df},0}^m \varrho_{\mathrm{nf}}^\ell,
		\end{equation}
		assuming that $\alpha$ satisfies the conclusion of the proposition with the original choice, $\varrho_{\mathrm{df}}$, of bdf.
		As $\mathsf{H}_{p[g]}$ vanishes (as an element of $\calV_{\mathrm{E}}({}^{\mathrm{de,sc}}\overline{T}^* \bbO)$) on the radial sets, the second term in the parentheses vanishes at the radial set in question, so the proposition applies regarding the modified bdf with 
		\begin{equation} 
		\alpha_0 = \alpha + \Big(\frac{\varrho_{\mathrm{df}}}{\varrho_{\mathrm{df},0}} \Big)^m \mathsf{H}_{p[g]} \Big(\frac{\varrho_{\mathrm{df},0}}{\varrho_{\mathrm{df}}} \Big)^m 
		\end{equation} 
		 in place of $\alpha$, the two agreeing on $\calN$. 
		 
		We can now calculate $\alpha$ over $\Omega_{\mathrm{nfTf},+,T}$ and over $\Omega_{\mathrm{nfSf},+,R}$, using over each whichever choice of $\varrho_{\mathrm{df}}$ makes the computation simplest. (And these do not need to be the same choice between $\Omega_{\mathrm{nfTf},+,T}$ and $\Omega_{\mathrm{nfSf},+,R}$.)
	\end{itemize}
	With these simplifications in mind, we compute:
	\begin{itemize}
		\item  Over $\Omega_{\mathrm{nfTf},+,T}$, we use the coordinates \cref{eq:misc_co1}, and we can take $\varrho_{\mathrm{df}} = \rho$ locally. In this case, 
		\begin{equation}
			\mathsf{H}_p a   = 2(m (\hat{\eta}^2 + (s-1)^2)  - \ell (1-s))a. 
		\end{equation}
		Thus, $\alpha = 2m (\hat{\eta}^2 + (s-1)^2)  - 2\ell (1-s)$ locally. 
		At $\calN^+_+$, $s=0$ and $\hat{\eta} = 0$, so $\alpha = 2(m-\ell)$.
		\item  Over $\Omega_{\mathrm{nfSf},+,R}$, we use the coordinates \cref{eq:misc_co2}, and we can take $\varrho_{\mathrm{df}}=\rho$ locally. In this case, 
		\begin{equation}
			\mathsf{H}_p a =( m(2\hat{\eta}^2+\lambda+1)-2 \ell )a ,
		\end{equation}
		so $\alpha =  m(2\hat{\eta}^2+\lambda+1)-2 \ell$ locally. At $\calN^+_+$, $\lambda=1$ and $\hat{\eta}=0$, so $\alpha = 2(m - \ell)$ there as well. 
	\end{itemize}
\end{proof}

\begin{lemma}
	Fix signs $\varsigma,\sigma\in  \{-,+\}$. 
	\begin{itemize}
		\item Given any compact subset $K\subseteq \calN^\varsigma_\sigma\cap{}^{\mathrm{de,sc}}\pi^{-1}( \Omega_{\mathrm{nfTf},\sigma,T})$, there exist symbols $s_0,s_1,s_2 \in C^\infty({}^{\mathrm{de,sc}}\overline{T}^* \bbO)$ such that $s=s_0$ near $K$, using the coordinates \cref{eq:misc_co1}, and 
		\begin{equation} 
			s_0 = s_1 \tilde{p} + s_2 (\hat{\eta}^2+\mathsf{m}^2 \varrho^2_{\mathrm{df}})
			\label{eq:misc_h41}
		\end{equation} 
		globally, with $s_2>0$ on $\calN^\varsigma_\sigma$.
		\item Given any compact subset $K\subseteq \calN^\varsigma_\sigma \cap {}^{\mathrm{de,sc}}\pi^{-1}( \Omega_{\mathrm{nfSf},\sigma,R})$, there exist symbols $\lambda_0,\lambda_1,\lambda_2 \in C^\infty({}^{\mathrm{de,sc}}\overline{T}^* \bbO)$ such that $\lambda=\lambda_0$ near $K$, using the coordinates \cref{eq:misc_co2}, and 
		\begin{equation} 
			\lambda_0 =1+ \lambda_1 \tilde{p} - \lambda_2 (\hat{\eta}^2+\mathsf{m}^2 \varrho_{\mathrm{df}}^2)
			\label{eq:misc_h42}
		\end{equation} 
		globally, with $\lambda_2>0$ on $\calN^\varsigma_\sigma $. 
	\end{itemize}
	\label{lem:s_alt}
\end{lemma}
This lemma encodes, in a symbolic fashion, the fact that, on the characteristic set $\Sigma_{\mathsf{m}}$ of $p$ (though not $p[g]$), $s$ and $1-\lambda$ have a semidefinite sign; see \Cref{fig:flowplot}.
\begin{proof} 
	The proofs of the two parts are similar, so we only write up the first, and we only consider $\calN^+_+$, the other three cases being similar. 
	In the coordinates \cref{eq:misc_co1}, we have 
	\begin{equation} 
		s = 1 -(1+\tilde{p}-\hat{\eta}^2-\mathsf{m}^2 \rho^2)^{1/2}
	\end{equation} 
	near $K$, assuming without loss of generality that $\varrho_{\mathrm{df}} = \rho$ locally. It is key that this holds with a single choice of sign on the square root (near the other radial set $\calK^+_+$, we instead have $s = 1 +(1+\tilde{p}-\hat{\eta}^2-\mathsf{m}^2 \rho^2)^{1/2}$, and the transition between the two formulas happens away from these two radial sets).
	We can write 
	\begin{equation} 
		(1-y+z)^{1/2} = (1-y)^{1/2} + zR(y,z) 
	\end{equation} 
	for $R(y,z)$ smooth near $\{y=0,z=0\}$. Applying this with $z=\tilde{p}$ and $y =  \hat{\eta}^2+\mathsf{m}^2 \rho^2$, 
	\begin{align}
		\begin{split} 
			s &= 1- (1-\hat{\eta}^2-\mathsf{m}^2 \rho^2)^{1/2} - \tilde{p} R(\hat{\eta}^2 + \mathsf{m}^2 \rho^2,\tilde{p}) \\ 
			&= R(0,-\hat{\eta}^2 - \mathsf{m}^2 \rho^2)(\hat{\eta}^2 + \mathsf{m}^2\rho^2 ) - \tilde{p} R(\hat{\eta}^2 + \mathsf{m}^2 \rho^2,\tilde{p}). 
		\end{split} 
	\end{align}
	Note that the functions $R(\hat{\eta}^2 + \mathsf{m}^2 \rho^2,\tilde{p})$ and $R(0,-\hat{\eta}^2 - \mathsf{m}^2 \rho^2)(\hat{\eta}^2 + \mathsf{m}^2 \rho^2)$  are or can be extended to smooth functions on some neighborhood of $K$ on ${}^{\mathrm{de,sc}}\overline{T}^* \bbO$. Thus, we can find $s_0,s_1,s_2$ such that \cref{eq:misc_h41} holds, with 
	\begin{equation} 
		s_1 = -R(\hat{\eta}^2 + \mathsf{m}^2 \rho^2,\tilde{p})
	\end{equation} 
	and $s_2 = R(0,-\hat{\eta}^2 - \mathsf{m}^2 \rho^2)$ locally. (Away from $K$, $s_0$ is not constrained, so all we need to do is extend $s_1,s_2$ to smooth functions on the whole radially compactified de,sc-cotangent bundle to satisfy \cref{eq:misc_h41} globally, taking \cref{eq:misc_h41} as a global definition of $s_0$.) Since $R(0,0) = 1/2$,  we have $s_2>0$ near $K$, so we can arrange $s_2>0$ globally. 
\end{proof}

So far, the definition of $\hat{\eta}$ has depended on which corner of $\bbO$ is included in the coordinate chart under consideration. For the next proposition, we use an almost-global definition:
away from $\mathrm{cl}_\bbO\{r=0\}$, let 
\begin{equation} 
	\hat{\eta}^2 = \eta^2 \varrho_{\mathrm{df}}^2,
\end{equation} 
where $\eta^2 = g_{\bbS^{d-1}}^{-1}(\eta,\eta)$ is defined using the standard spherical metric $g_{\bbS^{d-1}}$. If $\varrho_{\mathrm{df}}$ is chosen so that it is given by one of the coordinates called $\rho$ above in the coordinate chart under consideration, then $\hat{\eta}^2$ agrees with what we called ``$\hat{\eta}^2$'' previously.

\begin{proposition}
	Fix $\varsigma,\sigma \in \{-,+\}$. 
	Letting $\beth \in C^\infty({}^{\mathrm{de,sc}} \overline{T}^* \bbO)$ be defined by $\beth = \varrho_{\mathrm{nf}}^2 + \hat{\eta}^2$ near null infinity, the symbol 
	\begin{equation}
		 F_2= \mathsf{H}_{p[g]} \beth +4\varsigma \sigma \beth    \in C^\infty({}^{\mathrm{de,sc}}\overline{T}^* \bbO ) 
	\end{equation}
	vanishes cubically at $\calN^\varsigma_\sigma$ within $\Sigma_{\mathsf{m}}[g]$ in the sense that $F_2 \in \beth^{3/2}L^\infty + \varrho_{\mathrm{df}} \beth L^\infty + \tilde{p}[g] L^\infty$ locally. 
	\label{prop:beth}
\end{proposition}

\begin{proof}
	In fact, we prove the slightly stronger statement that each of $\varrho_{\mathrm{nf}}^2$ and $\hat{\eta}^2$ have the same property: 
	\begin{equation}
		\mathsf{H}_{p[g]} \varrho_{\mathrm{nf}}^2 + 4 \varsigma \sigma \varrho_{\mathrm{nf}}^2, \mathsf{H}_{p[g]} \hat{\eta}^2 + 4 \varsigma \sigma \hat{\eta}^2 \in \beth^{3/2}L^\infty + \varrho_{\mathrm{df}} \beth L^\infty + \tilde{p}[g]L^\infty 
	\end{equation}
	locally. This version of the proposition has the advantage that it manifestly does not depend on the choice of $\varrho_{\mathrm{df}}$, which affects $\beth$ through $\hat{\eta}$. 
	Indeed, if $\varrho_{\mathrm{df},0}$ is some other boundary-defining function of $\mathrm{df}$, then 
	\begin{equation}
		\mathsf{H}_{p[g]} \Big( \hat{\eta}^2 \frac{\varrho_{\mathrm{df},0}^2}{\varrho_{\mathrm{df}}^2} \Big)   + 4 \varsigma \sigma \hat{\eta}^2\frac{\varrho_{\mathrm{df},0}^2}{\varrho_{\mathrm{df}}^2} =  \frac{\varrho_{\mathrm{df},0}^2}{\varrho_{\mathrm{df}}^2} (\mathsf{H}_{p[g]} \hat{\eta}^2 + 4 \varsigma \sigma \hat{\eta}^2) + \hat{\eta}^2 \mathsf{H}_{p[g]} \Big( \frac{\varrho_{\mathrm{df},0}^2}{\varrho_{\mathrm{df}}^2} \Big).
	\end{equation}
	Since $\mathsf{H}_{p[g]}$ vanishes at the radial set $\calN^\varsigma_\sigma$, which has $\varrho_{\mathrm{nf}}^2 + \hat{\eta}^2 + \varrho_{\mathrm{df}}^2$ as a quadratic defining function within some neighborhood of itself within the characteristic set, 
	\begin{equation}
		\hat{\eta}^2 \mathsf{H}_{p[g]} \Big( \frac{\varrho_{\mathrm{df},0}^2}{\varrho_{\mathrm{df}}^2} \Big) \in \beth^{3/2}L^\infty + \varrho_{\mathrm{df}} \beth L^\infty + \tilde{p}[g] L^\infty. 
	\end{equation}
	Thus, this new term vanishes cubically at $\calN^\varsigma_\sigma$ within $\Sigma_{\mathsf{m}}[g]$, as desired. 
	
	In addition, by similar reasoning used in the proof of \Cref{prop:aleph},  it suffices to prove that
	\begin{equation}
			\mathsf{H}_{p} \varrho_{\mathrm{nf}}^2 + 4 \varsigma \sigma \varrho_{\mathrm{nf}}^2, \mathsf{H}_{p} \hat{\eta}^2 + 4 \varsigma \sigma \hat{\eta}^2 \in \beth^{3/2}L^\infty + \varrho_{\mathrm{df}} \beth L^\infty + \tilde{p} \beth L^\infty, 
			\label{eq:misc_jkz}
	\end{equation}
	i.e.\  a slight strengthening (stronger because $\tilde{p} \beth L^\infty$ is in place of $\tilde{p} L^\infty$) of the desired result in the Minkowski case. Indeed: 
	\begin{itemize}
		\item since $\mathsf{H}_p-\mathsf{H}_{p[g]}$ is $O(\varrho_{\mathrm{nf}}^2)$ as a b-vector field, the differences that result from replacing $p$ with $p[g]$ in $\mathsf{H}_p \smash{\varrho_{\mathrm{nf}}^2}, \mathsf{H}_p \smash{\hat{\eta}^2}$ lie in $\beth^{3/2} L^\infty$;
		\item  $\tilde{p} \beth L^\infty \subseteq \tilde{p}[g] L^\infty + \beth^{3/2} L^\infty $.
	\end{itemize}

	We now prove \cref{eq:misc_jkz}. Consider the case of $\calN^+_+$, the other three being analogous.
	\begin{itemize}
		\item First consider the situation over $\Omega_{\mathrm{nfTf},+,T}$, using the coordinates \cref{eq:misc_co1}, taking $\varrho_{\mathrm{df}}=\rho$ locally. Then, $\mathsf{H}_p \varrho_{\mathrm{nf}}^2 = -4 (1-s) \varrho_{\mathrm{nf}}^2$ and $\mathsf{H}_p \hat{\eta}^2 = 4 (\hat{\eta}^2+s^2-s-1)\hat{\eta}^2$ locally. Thus, the claim is that 
		\begin{equation}
		s \varrho_{\mathrm{nf}}^2, (\hat{\eta}^2+s^2-s)\hat{\eta}^2 \in \beth^{3/2}L^\infty + \varrho_{\mathrm{df}} \beth L^\infty + \tilde{p} \beth L^\infty
		\label{eq:misc_300}
		\end{equation}
		locally. Indeed, this follows from \Cref{lem:s_alt}; $\hat{\eta}^4 \in \beth^{3/2} L^\infty$, and using the cited lemma to replace $s$ with a linear combination of $\tilde{p}, \hat{\eta}^2, \rho^2=\varrho_{\mathrm{df}}^2$, we see that the remainder of the left-hand side of \cref{eq:misc_300} lies in the set on the right-hand side.
		\item Over $\Omega_{\mathrm{nfSf},+,R}$, we use the coordinates \cref{eq:misc_co2}, taking $\varrho_{\mathrm{df}} = \rho$. Then, $\mathsf{H}_p \varrho_{\mathrm{nf}}^2 = -4 \varrho_{\mathrm{nf}}^2$ and $\mathsf{H}_p \hat{\eta}^2 = 4(\hat{\eta}^2-1) \hat{\eta}^2 $ locally. Thus, the claim is that $\hat{\eta}^4$ vanishes cubically at $\calN^+_+\cap {}^{\mathrm{de,sc}}\pi^{-1}(\Omega_{\mathrm{nfSf},+,R})$, and this is true.
	\end{itemize}
\end{proof}

\subsection{$\calK$}
Linearizing $\mathsf{H}_{p[g]}\approx \mathsf{H}_p$ near $\calK^+_+$ (given by \cref{eq:NAK_defs} in the coordinate system in terms of which \cref{eq:H_nfSf_df} is written), the light gray terms in 
\begin{multline}
	2^{-1} \mathsf{H}_p = -\varrho_{\mathrm{nf}} \frac{\partial}{\partial \varrho_{\mathrm{nf}}} + \frac{1}{2}(4-{\color{lightgray}(\lambda+3)}) \varrho_{\mathrm{Sf}}\frac{\partial}{\partial \varrho_{\mathrm{Sf}}} + \frac{1}{2} \Big[{\color{lightgray} 2 \hat{\eta}^2+(\lambda+3)}-2 \Big]\rho \frac{\partial}{\partial \rho} \\ + \frac{1}{2}(\lambda+3)\Big[ {\color{lightgray}2\hat{\eta}^2+(\lambda+3)}-4 \Big] \frac{\partial}{\partial \lambda} 
	+({\color{lightgray}\hat{\eta}^2}-1) V_{\bbS^{d-1}} 
\end{multline}
can be dropped, leading to:
\begin{equation}
	2^{-1} \mathsf{H}_{p[g]} \approx 2^{-1} \mathsf{H}_p \approx  -\varrho_{\mathrm{nf}} \frac{\partial}{\partial \varrho_{\mathrm{nf}}}  + 2\varrho_{\mathrm{Sf}} \frac{\partial}{\partial \varrho_{\mathrm{Sf}}} -\rho \frac{\partial}{\partial \rho} - 2 (\lambda+3)\frac{\partial}{\partial \lambda} - V_{\bbS^{d-1}}.
	\label{eq:misc_34x}
\end{equation}
Thus, $\calK^+_+$ is a source in the $\varrho_{\mathrm{Sf}}$ direction and a sink in the others, as depicted in the various figures above.

\begin{proposition}
	Fix signs $\varsigma,\sigma \in \{-,+\}$. Letting $\gimel \in C^\infty({}^{\mathrm{de,sc}}\overline{T}^* \bbO)$ satisfy $\gimel = \varrho_{\mathrm{nf}}^2+\rho^2+(\lambda+3)^2$ near $\calK^\varsigma_\sigma$ in the coordinates \cref{eq:misc_co2}, there exist $F_3,E_3 \in  C^\infty({}^{\mathrm{de,sc}}\overline{T}^* \bbO)$ such that 
	\begin{equation}
		\mathsf{H}_{p[g]} \gimel =\varsigma \sigma(- 4  \gimel   -E_3)+ F_3, 
		\label{eq:misc_330}
	\end{equation} 
	$E_3\geq 0$ everywhere, and 
	$F_3$
	vanishes cubically at $\calK^\varsigma_\sigma$ within $\Sigma_{\mathsf{m}}[g]$ in the sense that $F_3 \in \gimel^{3/2} L^\infty + \varrho_{\mathrm{Sf}} \gimel L^\infty + \tilde{p}[g] \gimel L^\infty$ locally. 
	\label{prop:gimel}
\end{proposition}
\begin{proof}
	By similar reasoning to that in the proof of \Cref{prop:beth}, it suffices to consider the case when $g$ is the Minkowski metric. 
	We consider the case of $\varsigma,\sigma = +$, that is of $\calK^+_+$, the other three being analogous. Then,
	\begin{align}
		\begin{split} 
		\mathsf{H}_p \gimel &= -4 \varrho_{\mathrm{nf}}^2 + 2 (2\hat{\eta}^2+\lambda+1)\rho^2 + 2(2\hat{\eta}^2+\lambda-1)(\lambda+3)^2\\ 
		&= -4 \gimel + 2(2\hat{\eta}^2+\lambda+3)\rho^2 + 2(2\hat{\eta}^2+\lambda+1)(\lambda+3)^2 .
		\end{split}
	\end{align}
	Choose a symbol $E_3\geq 0$ such that $E_3 = 4 (\lambda+3)^2$ near $\calK^+_+$. 
	Then, \cref{eq:misc_330} is satisfied by setting $F_3 = 2(2\hat{\eta}^2+\lambda+3)(\rho^2 +(\lambda+3)^2)$, and this vanishes cubically at $\calK^+_+$ in the desired sense. 
\end{proof}

\subsection{$\calC$}
Linearizing $\mathsf{H}_{p[g]}\approx \mathsf{H}_p$ near $\calC^+_+$ (given by $s=2$ with other coordinates vanishing in the coordinate system in terms of which \cref{eq:H_nfTf_df} is written), the light gray terms in 
\begin{multline}
	2^{-1}  \mathsf{H}_p =  (1 + {\color{lightgray}(s-2)})  \varrho_{\mathrm{nf}}\frac{\partial}{\partial \varrho_{\mathrm{nf}}} - ({\color{lightgray}(s-2)}+2)  \varrho_{\mathrm{Tf}} \frac{\partial}{\partial \varrho_{\mathrm{Tf}}} + \Big[ {\color{lightgray} \hat{\eta}^2}  + ({\color{lightgray}(s-2)}+1)^2 \Big] \rho \frac{\partial}{\partial \rho} \\ - (2-s)\Big[ {\color{lightgray}\hat{\eta}^2} + ({\color{lightgray}(s-2)}+2)^2-{\color{lightgray}(s-2)}-2 \Big] \frac{\partial}{\partial s}  + \Big( {\color{lightgray}\hat{\eta}^2 }+({\color{lightgray}(s-2)}+2)^2-{\color{lightgray}(s-2)}-3 \Big) V_{\bbS^{d-1}} 
\end{multline}
are  dropped. Thus, 
\begin{equation}
	2^{-1} \mathsf{H}_{p[g]} \approx 2^{-1} \mathsf{H}_{p} \approx \varrho_{\mathrm{nf}} \frac{\partial}{\partial \varrho_{\mathrm{nf}}} -2 \varrho_{\mathrm{Tf}} \frac{\partial}{\partial \varrho_{\mathrm{Tf}}} + \rho \frac{\partial}{\partial \rho} +2(s-2) \frac{\partial}{\partial s} + V_{\bbS^{d-1}}
	\label{eq:misc_3fx}
\end{equation}
near $\calC^+_+$. So, $\calC^+_+$ is a source in the $\varrho_{\mathrm{nf}}$ direction, as well as the $\rho,s,\hat{\eta}$ directions, and a sink in the $\varrho_{\mathrm{Tf}}$ direction, as depicted in the various figures above.

\begin{proposition}
	Fix signs $\varsigma,\sigma \in \{-,+\}$. Letting $\daleth \in C^\infty({}^{\mathrm{de,sc}}\overline{T}^* \bbO)$ satisfy $\daleth = \varrho_{\mathrm{nf}}^2+\rho^2+(s-2)^2$ near $\calC^\varsigma_\sigma$ in the coordinates \cref{eq:misc_co1}, there exist $F_4,E_4 \in  C^\infty({}^{\mathrm{de,sc}}\overline{T}^* \bbO)$ such that 
	\begin{equation}
		\mathsf{H}_p \daleth = \varsigma \sigma (4 \daleth   +E_4)+ F_4, 
		\label{eq:misc_334}
	\end{equation} 
	$E_4\geq 0$ everywhere, and 
	$F_4$
	vanishes cubically at $\calC^\varsigma_\sigma$ within $\Sigma_{\mathsf{m}}[g]$ in the sense that $F_4 \in \daleth^{3/2} L^\infty + \varrho_{\mathrm{Tf}} \daleth L^\infty + \tilde{p}[g] \daleth L^\infty$ locally. 
	\label{prop:daleth}
\end{proposition}
\begin{proof}
	By similar reasoning to that in the proof of \Cref{prop:beth}, it suffices to consider the case when $g$ is the Minkowski metric. 
	We consider the case of $\varsigma,\sigma = +$, that is of $\calC^+_+$, the other three being analogous. Then,
	\begin{equation}
		\mathsf{H}_p \daleth =  4(s-1) \varrho_{\mathrm{nf}}^2+ 4 (\hat{\eta}^2+(s-1)^2) \rho^2  + 4(s-2)^2 (\hat{\eta}^2+s(s-1)).
	\end{equation}
	Choose a symbol $E_4\geq 0$ such that $E_4 = 4 (s-2)^2$ near $\calC^+_+$. Then, \cref{eq:misc_334} is satisfied by setting $F_4 =  4(s-2) \varrho_{\mathrm{nf}}^2 + 4(\hat{\eta}^2 + s(s-2)) \rho^2 + 4(s-2)^2(\hat{\eta}^2+(s-2)(s+1))$, and this vanishes cubically at $\calC^+_+$ in the desired sense.  
\end{proof}

\subsection{$\calR$}
Consider the final radial set. We postpone the rigorous statements until \S\ref{sec:radialpoint}. Here we just discuss the linearization of $\mathsf{H}_{p[g]}$ at $\calR^+_+$. We only examine the situation near $\mathrm{nFf}\cap \mathrm{Ff}$, since the situation over the interior of the timelike cap $\mathrm{Ff}$ can be identified with the situation in the sc-phase space. 

We rewrite \cref{eq:H_nfTf_df} as
\begin{multline}
	2^{-1}  \mathsf{H}_p = - (1 - s)  \varrho_{\mathrm{nf}}\frac{\partial}{\partial \varrho_{\mathrm{nf}}} - ({\color{lightgray}(s-1)}+1)  \varrho_{\mathrm{Tf}} \frac{\partial}{\partial \varrho_{\mathrm{Tf}}} + {\color{lightgray}\Big[  \hat{\eta}^2  + (s-1)^2 \Big] \rho \frac{\partial}{\partial \rho}} \\ + ({\color{lightgray}(s-1)}-1)\Big[ {\color{lightgray}\hat{\eta}^2} + ({\color{lightgray}(s-1)}+1)(s-1) \Big] \frac{\partial}{\partial s}  + \Big( {\color{lightgray}\hat{\eta}^2} +({\color{lightgray}(s-1)}+1)^2-{\color{lightgray}(s-1)}-2 \Big) V_{\bbS^{d-1}},
\end{multline}
where, since $\calR$ is locally the set where $s=1$ and the other coordinates except $\varrho_{\mathrm{Tf}}$ vanish, the light gray terms are the ones dropped in the linearization:
\begin{equation}
	2^{-1} \mathsf{H}_{p[g]} \approx 2^{-1} \mathsf{H}_{p} \approx - (1 - s)  \varrho_{\mathrm{nf}}\frac{\partial}{\partial \varrho_{\mathrm{nf}}} - \varrho_{\mathrm{Tf}} \frac{\partial}{\partial \varrho_{\mathrm{Tf}}}  - (s-1) \frac{\partial}{\partial s} - V_{\bbS^{d-1}}.
\end{equation}
So, $\calR^+_+$ is a sink in the $\varrho_{\mathrm{Tf}},s,\hat{\eta}$ directions and neutral in the $\varrho_{\mathrm{nf}}$ and $\rho$ directions. The neutral directions are the expected ones: the $\varrho_{\mathrm{nf}}$ direction is along $\calR$, and the $\rho$ direction is orthogonal to the characteristic set at $\calR$ (which means, since $\mathsf{H}_pp=0$, the $\rho$ direction had to be neutral). Thus, $\calR^+_+$ is a sink within $\Sigma_{\mathsf{m},+}$, as shown in \Cref{fig:O}, \Cref{fig:flowplot}.

\section{Propagation through null infinity}
\label{sec:propagation}

In this section, let $P\in \operatorname{Diff}_{\mathrm{de,sc}}^{2,\mathsf{0}}(\bbO)$ denote a de,sc-differential operator such that 
\begin{equation}
	P = \square_g + \mathsf{m}^2 + \operatorname{Diff}_{\mathrm{de,sc}}^{1,-\mathsf{2}}(\bbO)
	\label{eq:misc_273}
\end{equation}
for an admissible metric $g$. Thus, the symbol $p[g]:\sum_{i=0}^d \xi_i \dd x_i \mapsto g^{-1}(\bmxi,\bmxi)+\mathsf{m}^2$ of $\square_g+\mathsf{m}^2$ is a representative of $\smash{\sigma_{\mathrm{de,sc}}^{2,\mathsf{0}}(P)}$.

We state in \S\ref{subsec:nf} the microlocal version of the proposition that a lack of decay of solutions to $Pu=f$, for nice $f$, as measured by $\mathrm{de}$-wavefront set on the boundary of the Penrose diagram, propagates along null infinity. 
We then prove a series of radial point estimates, two at each of the radial sets $\calA,\calN,\calC,\calK$ lying over null infinity: $\calA$ in \S\ref{subsec:calA}, $\calN$ in \S\ref{subsec:calN}, $\calK$ in \S\ref{subsec:calK}, and finally $\calC$ in \S\ref{subsec:calC}. Each of these is a saddle point of the de,sc-Hamiltonian flow. The radial set $\calR$, which lies instead over the timelike caps, is postponed until the next section.
The two main results of this section, gotten by concatenating the various estimates, are:

\begin{theorem}
	Suppose that $m\in \bbR$ and  $\mathsf{s}=(s_{\mathrm{Pf}},s_{\mathrm{nPf}},s_{\mathrm{Sf}},s_{\mathrm{nFf}},s_{\mathrm{Ff}})\in \bbR^5$ satisfy
	\begin{itemize}
		\item $m > s_{\mathrm{nf}} +1$, 
		\item $2s_{\mathrm{Sf}}> \max\{ - 2m+2s_{\mathrm{nf}}+1, m + s_{\mathrm{nf}}-1 \}$, 
		\item $2 s_{\mathrm{Tf}} < m + s_{\mathrm{nf}}-1$
	\end{itemize}
	when $(s_{\mathrm{nf}},s_{\mathrm{Tf}}) = (s_{\mathrm{nFf}},s_{\mathrm{Ff}})$ and when $(s_{\mathrm{nf}},s_{\mathrm{Tf}}) = (s_{\mathrm{nPf}},s_{\mathrm{Pf}})$.
	Suppose that $u\in \calS'$ is a solution to $P u=f$ such that, for some $T\in \bbR$, 
	\begin{equation} 
		\operatorname{WF}_{\mathrm{sc}}^{m,s_{\mathrm{Sf}}}(u) \cap {}^{\mathrm{sc}}\pi^{-1}\operatorname{cl}_\bbM\{t=T\} = \varnothing
	\end{equation} 
	and such that $\operatorname{WF}_{\mathrm{de,sc}}^{m-1,\mathsf{s}+1}(f) \subseteq \calR$. Then, $\operatorname{WF}_{\mathrm{de,sc}}^{m,\mathsf{s}}(u) \subseteq \calR$ as well. 
	\label{thm:propagation_through_scrI_1}
\end{theorem}
\begin{remark*}
	Using the ordinary sc-calculus, it can be shown that, given the setup of the theorem, $\operatorname{WF}_{\mathrm{de,sc}}^{m,\mathsf{s}}(u) \subseteq \calR \cup {}^{\mathrm{de,sc}}\pi^{-1}(\mathrm{nf})$. Of course, the refinement is only over null infinity.
\end{remark*}
\begin{theorem}
		Suppose that $m\in \bbR$ and  $\mathsf{s}=(s_{\mathrm{Pf}},s_{\mathrm{nPf}},s_{\mathrm{Sf}},s_{\mathrm{nFf}},s_{\mathrm{Ff}})\in \bbR^5$ satisfy
	\begin{itemize}
		\item $m<s_{\mathrm{nf}}+1$, 
		\item $2s_{\mathrm{Sf}}< \min\{ - 2m+2s_{\mathrm{nf}}+1, m + s_{\mathrm{nf}}-1 \}$,
		\item $2 s_{\mathrm{Tf}} > m + s_{\mathrm{nf}}-1$
	\end{itemize}
	when $(s_{\mathrm{nf}},s_{\mathrm{Tf}}) = (s_{\mathrm{nFf}},s_{\mathrm{Ff}})$ and when $(s_{\mathrm{nf}},s_{\mathrm{Tf}}) = (s_{\mathrm{nPf}},s_{\mathrm{Pf}})$.
	Suppose that $u\in \calS'$ is a solution to $P u=f$ such that, for some neighborhood $U\subseteq {}^{\mathrm{de,sc}}\overline{T}^* \bbO$ of $\calR$,  
	\begin{equation} 
		\operatorname{WF}_{\mathrm{de,sc}}^{m,\mathsf{s}}(u)\cap U\subseteq \calR,
	\end{equation} 
	and likewise suppose that $\operatorname{WF}_{\mathrm{de,sc}}^{m-1,\mathsf{s}+1}(f) \subseteq \calR$. Then, $\operatorname{WF}_{\mathrm{de,sc}}^{m,\mathsf{s}}(u) \subseteq \calR$ as well.
	\label{thm:propagation_through_scrI_2}
\end{theorem}

\Cref{thm:propagation_through_scrI_1} is proven by propagating control through the radial sets in the order described in \cref{eq:Cauchy_order} (except for $\calR$, which we will not discuss until \S\ref{sec:radialpoint}). \Cref{thm:propagation_through_scrI_2} is proven by instead propagating control through the radial sets in the order described in \cref{eq:scat_order_1}, \cref{eq:scat_order_2}. (As it is written, \Cref{thm:propagation_through_scrI_2} is for propagating control from all of $\calR$, not just $\calR_-$. However, if we only know that $u$ is under control near say $\calR_-$, then we can cut off $u$ near $\calR_+$ and apply the theorem to the cutoff version.)

The e,b-analogues of the results in this section can be found in \cite[\S4]{HintzVasyScriEB}. As the arguments below are very similar to those there (and de,sc-analogues of the standard sc-results described in \cite{VasyGrenoble} anyways), we will only sketch the key points.
To handle the situation away from null infinity, we can simply cite the propagation results established using the sc-calculus, and so this will be described in even less detail.

Note that, by the definition of admissibility, $\square_g$ differs from the Minkowski d'Alembertian $\square$ by an element of $\operatorname{Diff}^{2,-\mathsf{2}}_{\mathrm{de,sc}}(\bbO)$ with real principal symbol. A consequence is that 
\begin{equation}
	P-P^* \in \operatorname{Diff}^{1,-\mathsf{2}}_{\mathrm{de,sc}}(\bbO),
	\label{eq:misc_276}
\end{equation}
where $P^*$ is the formal adjoint with respect to the $L^2(\bbR^{1+d})$ inner product, with respect to which $\square=\square^*$. This restriction on $P-P^*$ simplifies the radial point estimates, which, as in in \cite{VasyGrenoble}, would otherwise depend on the values of 
\begin{equation} 
	\varrho_{\mathrm{df}}^1 \varrho_{\mathrm{Pf}}^{-1}\varrho_{\mathrm{nPf}}^{-1}\varrho_{\mathrm{Sf}}^{-1}\varrho_{\mathrm{nFf}}^{-1}\varrho_{\mathrm{Ff}}^{-1} \cdot \sigma_{\mathrm{de,sc}}^{1,-\mathsf{1}}(P-P^*) \in S_{\mathrm{de,sc}}^{[0,\mathsf{0}]} (\bbO) 
\end{equation} 
along the various radial sets. 
Let 
\begin{equation}
	p_1 \in \smash{S_{\mathrm{de,sc}}^{1,-\mathsf{2}}(\bbO)} \in i\smash{\sigma_{\mathrm{de,sc}}^{1,-\mathsf{2}}(P-P^*)}.
	\label{eq:p1def}
\end{equation}
We can choose $p_1$ to be real-valued, since $P-P^* = -(P-P^*)^*$. While $p_1$ will be insignificant for the proposition statements below, we must keep track of it in the proofs, since it is subprincipal (i.e.\ only subleading by one order) in the differential sense, and subprincipal symbols enter positive commutator arguments.

\subsection{Propagation Between the Radial Sets}
\label{subsec:nf}

Using the nonzero component of $\mathsf{H}_p$ along the punctured fibers ${}^{\mathrm{de,sc}}\pi^{-1}(\mathrm{nf})\backslash \calN$, we get the following: 
\begin{proposition}
	Suppose that $u\in \calS'$ satisfies $\operatorname{WF}_{\mathrm{de,sc}}^{m,\mathsf{s}}(u)\cap {}^{\mathrm{de,sc}}\pi^{-1}(\mathrm{nf}^\circ)\cap (\calN \cup  {}^{\mathrm{de,sc}}\pi^{-1} (\operatorname{cl}_{\bbO}\{|t|-r= v\}) ) = \varnothing$ for some $v\in \bbR$. Suppose further that, for some $v_1,v_2\in \bbR$ satisfying $v_1<v<v_2$, 
	\begin{equation}
		\operatorname{WF}_{\mathrm{de,sc}}^{m-1,\mathsf{s}+1}(P u)  \cap {}^{\mathrm{de,sc}}\pi^{-1}( \operatorname{cl}_{\bbO}\{v_1\leq |t|-r\leq v_2\}) \subseteq \calN.
	\end{equation}
	Then, $\operatorname{WF}_{\mathrm{de,sc}}^{m,\mathsf{s}}(u) \cap {}^{\mathrm{de,sc}}\pi^{-1} (\operatorname{cl}_{\bbO}\{t_1\leq |t|-r\leq v_2\}) =\varnothing$.
\end{proposition}
\begin{proof}[Proof sketch]
	As seen by rewriting it in terms of $\varrho_{\mathrm{Sf}}$ in $\Omega_{\mathrm{nfSf},\sigma,R}$ and in terms of $\varrho_{\mathrm{Tf}}$ in $\Omega_{\mathrm{nfTf},\sigma,T}$, the function $|t|-r$ is monotone under $\mathsf{H}_{p[g]}$ on each component of 
	\begin{equation} 
		(\Sigma_{\mathsf{m},\pm}  \cap {}^{\mathrm{de,sc}}\pi^{-1} (\operatorname{cl}_{\bbO}\{v_1<|t|-r<v_2\}) )\backslash \calN
	\end{equation} 
	(see \Cref{prop:propagation_lemma}).  
	The proposition therefore follows via the usual proof of Duistermaat--H\"ormander type estimates, using elliptic regularity off of the characteristic set. 
	The point is that any integral curve of $\mathsf{H}_{p[g]}$ in the relevant region of the characteristic set has to have one end (see \Cref{fig:O}) at one of the sets $\calN$ or 
	\begin{equation} 
		{}^{\mathrm{de,sc}}\pi^{-1} (\operatorname{cl}_{\bbO}\{|t|-r= v\}),
	\end{equation} 
	where we are assuming control.
\end{proof}

\begin{proposition}
	Let $m\in \bbR$ and $\mathsf{s}\in \bbR^5$, and suppose that $u\in \calS'$ satisfies  $\operatorname{WF}_{\mathrm{de,sc}}^{m,\mathsf{s}}(u)\cap \calA = \varnothing$. Then, if $\operatorname{WF}_{\mathrm{de,sc}}^{m-1,\mathsf{s}+1}(P u)\cap ({}^{\mathrm{de,sc}}\pi^{-1}(\mathrm{nf} \backslash\mathrm{Tf})) \subseteq  \{\hat{\eta}=0\}$,
	\begin{equation}
		\operatorname{WF}_{\mathrm{de,sc}}^{m,\mathsf{s}}(u)\cap
		({}^{\mathrm{de,sc}}\pi^{-1}(\mathrm{nf} \backslash  \mathrm{Tf}))
		\subseteq  \{\hat{\eta}=0\}. 
		\label{eq:h411}
	\end{equation}
	Moreover, if $\operatorname{WF}_{\mathrm{de,sc}}^{m,\mathsf{s}}(u)\cap (\calA\cup \calN) \cap {}^{\mathrm{de,sc}}\pi^{-1}(\mathrm{nf}\cap \mathrm{Sf}) = \varnothing$ and $\operatorname{WF}_{\mathrm{de,sc}}^{m-1,\mathsf{s}+1}(P u)\cap {}^{\mathrm{de,sc}}\pi^{-1}(\mathrm{nf}\cap \mathrm{Sf}) \subseteq  \calK$, then $\operatorname{WF}_{\mathrm{de,sc}}^{m,\mathsf{s}}(u)\cap {}^{\mathrm{de,sc}}\pi^{-1}(\mathrm{nf}\cap \mathrm{Sf}) \subseteq  \calK$.
\end{proposition}
\begin{proof}[Proof sketch]
	By \cref{eq:H_nfSf_df}, $\rho_{\mathrm{Sf}}$ is monotone with respect to $\mathsf{H}_{p[g]}$ along the invariant set $\{\lVert \hat{\eta}\rVert =1\}\cap \Sigma_{\mathsf{m},\pm}$. (This invariant set is one of the integral curves in \Cref{fig:globalflowplot}. Which it is depends on the parameter $R$ in the definition of the coordinate system.) So the assumption $\operatorname{WF}_{\mathrm{de,sc}}^{m,\mathsf{s}}(u)\cap \calA = \varnothing$ allows us to conclude, using a Duistermaat--H\"ormander estimate, that 
	\begin{equation} 
		\operatorname{WF}_{\mathrm{de,sc}}^{m,\mathsf{s}}(u)\cap \{\lVert \hat{\eta}\rVert=1\} = \varnothing,
	\end{equation} 
	using an elliptic estimate off $\Sigma_{\mathsf{m},\pm}$.

	By \cref{eq:H_nfSf_df}, the function $\hat{\eta}^2$ is monotone under $\mathsf{H}_{p[g]}$ on 
	$\Sigma_{\mathsf{m},\pm} \cap  ({}^{\mathrm{de,sc}}\pi^{-1}(\mathrm{nf} \backslash \mathrm{Tf}) )\backslash \{\lVert\hat{\eta}\rVert =0,1\}$ (see \Cref{fig:flowplot}). 
	We can therefore propagate the control on  $\operatorname{WF}_{\mathrm{de,sc}}^{m,\mathsf{s}}(u)\cap \{\lVert|\hat{\eta}\rVert|=1\}$ to conclude \cref{eq:h411}.
	
	To get the second part of the proposition, we use another propagation estimate, this time based on the monotonicity of $\lambda$ under $\mathsf{H}_{p[g]}$ on $\{\hat{\eta}=0\} \cap {}^{\mathrm{de,sc}}\pi^{-1}(\mathrm{nf}\cap \mathrm{Sf}) \backslash (\calN\cup \calK)$, which we also read off \cref{eq:H_nfSf_df} (again see see \Cref{fig:flowplot}).
\end{proof}
Propagating in the reverse direction:
\begin{propositionp}
	Let $m\in \bbR$ and $\mathsf{s}\in \bbR^5$, and suppose that $u\in \calS'$ satisfies  $\operatorname{WF}_{\mathrm{de,sc}}^{m,\mathsf{s}}(u)\cap \calK = \varnothing$. Then, if $\operatorname{WF}_{\mathrm{de,sc}}^{m-1,\mathsf{s}+1}(P u)\cap {}^{\mathrm{de,sc}}\pi^{-1}(\mathrm{nf}\cap \mathrm{Sf}) \subseteq  \{\rho = 0  \text{ and } \lambda \in [-1,+1]\}$, then
	\begin{equation}
		\operatorname{WF}_{\mathrm{de,sc}}^{m,\mathsf{s}}(u)\cap {}^{\mathrm{de,sc}}\pi^{-1}(\mathrm{nf}\cap \mathrm{Sf}) \subseteq  \{\rho = 0 \text{ and } \lambda \in [-1,+1]\}. 
	\end{equation}
	If, in addition, $\operatorname{WF}_{\mathrm{de,sc}}^{m,\mathsf{s}}(u)\cap \calN \cap {}^{\mathrm{de,sc}}\pi^{-1}(\mathrm{nf}\cap \mathrm{Sf}) = \varnothing$ and $\operatorname{WF}_{\mathrm{de,sc}}^{m-1,\mathsf{s}+1}(P u)\cap {}^{\mathrm{de,sc}}\pi^{-1}(\mathrm{nf}\cap \mathrm{Sf}) \subseteq  \calA$, then $\operatorname{WF}_{\mathrm{de,sc}}^{m,\mathsf{s}}(u)\cap {}^{\mathrm{de,sc}}\pi^{-1}(\mathrm{nf}\cap \mathrm{Sf}) \subseteq  \calA$.
\end{propositionp}

Over the other corner:
\begin{proposition}
	Let $m\in \bbR$ and $\mathsf{s} \in \bbR^5$, and suppose that $u\in \calS'$ satisfies $\operatorname{WF}_{\mathrm{de,sc}}^{m,\mathsf{s}}(u) \cap \calC = \varnothing$. Then, if 
	\begin{equation} 
		\operatorname{WF}_{\mathrm{de,sc}}^{m-1,\mathsf{s}+1}(P u) \cap {}^{\mathrm{de,sc}}\pi^{-1}(\mathrm{nf}\cap \mathrm{Tf}) \subseteq \{ \hat{\eta}=0, s\leq 1\},
	\end{equation} 
	then $\operatorname{WF}_{\mathrm{de,sc}}^{m,\mathsf{s}}(u) \cap {}^{\mathrm{de,sc}}\pi^{-1}(\mathrm{nf}\cap \mathrm{Tf}) \subseteq \{ \hat{\eta}=0, s\leq 1\}$ as well. If, in addition, $\operatorname{WF}_{\mathrm{de,sc}}^{m,\mathsf{s}}(u) \cap \calN \cap  {}^{\mathrm{de,sc}}\pi^{-1}(\mathrm{nf}\cap \mathrm{Tf}) = \varnothing$ and 
	\begin{equation}
		\operatorname{WF}_{\mathrm{de,sc}}^{m-1,\mathsf{s}+1}(P u) \cap {}^{\mathrm{de,sc}}\pi^{-1}(\mathrm{nf}\cap \mathrm{Tf}) \subseteq \calR, 
	\end{equation}
	then $\operatorname{WF}_{\mathrm{de,sc}}^{m,\mathsf{s}}(u) \cap {}^{\mathrm{de,sc}}\pi^{-1}(\mathrm{nf}\cap \mathrm{Tf}) \subseteq \calR$. 
\end{proposition}
\begin{proof}[Proof sketch]
	The argument is slightly different than the previous in that each part involves two propagation steps. For the first step, we propagate control to the rest of 
	\begin{equation} 
		\Sigma_{\mathsf{m},+} \cap {}^{\mathrm{de,sc}}\pi^{-1}(\mathrm{nf} \cap \mathrm{Tf}) \cap {}^{\mathrm{de,sc}}\bbS^* \bbO
	\end{equation} 
	using $s$ as monotone function, which, according to \cref{eq:H_nfTf_df}, is monotone under $\mathsf{H}_{p[g]}$ on 
	\begin{equation} 
		(\Sigma_{\mathsf{m},+} \cap {}^{\mathrm{de,sc}}\pi^{-1}(\mathrm{nf} \cap \mathrm{Tf}) \cap {}^{\mathrm{de,sc}}\bbS^* \bbO)\backslash (\calC\cup \calN). 
	\end{equation} 
	Having done this, we conclude an absence of wavefront set \emph{at fiber infinity} except possibly at $\calN$. Next, this control can be propagated to the rest of $\Sigma_{\mathsf{m},+} \cap {}^{\mathrm{de,sc}}\pi^{-1}(\mathrm{nf} \cap \mathrm{Tf}) \backslash \{\hat{\eta}=0,s \leq 1\}$ using $\rho$ as a monotone function, which is monotone in the interior of the fibers according to \cref{eq:H_nfTf_df}. (See \Cref{fig:flowplot}.)
	For the second part of the proposition, the argument is the same, except after the first step we conclude an absence of wavefront set at all of fiber infinity, including at $\calN$, and then the second step propagates control to everywhere except $\calR$.  
\end{proof}

Propagating in the reverse direction:
\begin{propositionp}
	Let $m\in \bbR$ and $\mathsf{s}\in \bbR^5$, and suppose that $u\in \calS'$ satisfies $\operatorname{WF}_{\mathrm{de,sc}}^{m,\mathsf{s}}(u) \cap \calR = \varnothing$. Then, if 
	\begin{equation}
		\operatorname{WF}_{\mathrm{de,sc}}^{m-1,\mathsf{s}+1}(P u) \cap {}^{\mathrm{de,sc}} \pi^{-1}(\mathrm{nf}\cap \mathrm{Tf}) \cap {}^{\mathrm{de,sc}}T^* \bbO = \varnothing 
	\end{equation}
	(that is, the de,sc-wavefront over $\mathrm{nf}\cap \mathrm{Tf}$ is at fiber infinity), then $\operatorname{WF}_{\mathrm{de,sc}}^{m,\mathsf{s}}(u) \cap {}^{\mathrm{de,sc}} \pi^{-1}(\mathrm{nf}\cap \mathrm{Tf}) \cap {}^{\mathrm{de,sc}}T^* \bbO = \varnothing$ as well. If, in addition, $\operatorname{WF}_{\mathrm{de,sc}}^{m,\mathsf{s}}(u) \cap \calN \cap {}^{\mathrm{de,sc}}\pi^{-1}(\mathrm{nf}\cap \mathrm{Tf}) = \varnothing$ and 
	\begin{equation} 
		\operatorname{WF}_{\mathrm{de,sc}}^{m-1,\mathsf{s}+1}(P u) \cap {}^{\mathrm{de,sc}} \pi^{-1}(\mathrm{nf}\cap \mathrm{Tf}) \subseteq \calC,
	\end{equation} 
	then $\operatorname{WF}_{\mathrm{de,sc}}^{m,\mathsf{s}}(u) \cap {}^{\mathrm{de,sc}} \pi^{-1}(\mathrm{nf}\cap \mathrm{Tf})\subseteq \calC$ as well. 
\end{propositionp}

\subsection{Propagation Through $\calA$}
\label{subsec:calA}

As seen above, $\calA^\pm_\pm$ (with the signs the same) is a source in the fiberwise directions and with respect to the direction along the null face, but it is a sink in the $\partial_{\varrho_{\mathrm{nf}}}$ direction. The same holds for $\smash{\calA^\pm_\mp}$ (with the signs opposite), with ``source'' and ``sink'' switched. Thus, we can prove two estimates:
\begin{enumerate}
	\item propagation from a band $\{\epsilon_1 < \varrho_{\mathrm{nf}}<\epsilon_2 \}$ hitting spacelike infinity into $\calA$, and 
	\item propagation from an appropriate annular set defined using the other coordinates around $\calA$ into $\calA$. 
\end{enumerate}

\begin{proposition}
	Fix signs $\varsigma,\sigma \in \{-,+\}$. 
	Suppose that $m\in \bbR$ and $\mathsf{s} = (s_{\mathrm{Pf}},s_{\mathrm{nPf}},s_{\mathrm{Sf}},s_{\mathrm{nFf}},s_{\mathrm{Ff}})\in \bbR^5$ satisfy $m-s_{\mathrm{nf}}+s_{\mathrm{Sf}}>1/2$, where $s_{\mathrm{nf}}\in \{s_{\mathrm{nPf}},s_{\mathrm{nFf}}\}$, depending on $\sigma$ in the usual way. 
	For any $\epsilon_1>0$, there exists some $\epsilon_0 \in (0,\epsilon_1)$ such that, if $u\in \calS'$ satisfies 
	\begin{itemize}
		\item $\operatorname{WF}_{\mathrm{de,sc}}^{m-1,\mathsf{s}+1}(P u) \cap \calA^\varsigma_\sigma = \varnothing$, 
		\item $\operatorname{WF}_{\mathrm{de,sc}}^{m,\mathsf{s}}(u) \cap \{ \rho^2 +(\lVert\hat{\eta}\rVert-1)^2+(\lambda+1)^2 + \varrho_{\mathrm{Sf}}^2< \epsilon_1, \epsilon_2<\varrho_{\mathrm{nf}} < \epsilon_1 \}=\varnothing $ for some $\epsilon_2 \in (0,\epsilon_0)$, 
	\end{itemize}
	it is the case that $\operatorname{WF}_{\mathrm{de,sc}}^{m,\mathsf{s}}(u)\cap \calA^\varsigma_\sigma = \varnothing$. 
	\label{prop:A_first_propagation}
\end{proposition}
\begin{remark*}
	Here, $\{ \rho^2 +(\lVert\hat{\eta}\rVert-1)^2+(\lambda+1)^2 + \varrho_{\mathrm{Sf}}^2< \epsilon_1, \epsilon_2<\varrho_{\mathrm{nf}} < \epsilon_1 \}$ denotes a subset of ${}^{\mathrm{de,sc}} \pi^{-1}(\Omega_{\mathrm{nfSf},\sigma,R})$. Similar notational conventions will be used below. 
\end{remark*}

Note that the condition 
\begin{equation} 
	\operatorname{WF}_{\mathrm{de,sc}}^{m,\mathsf{s}}(u) \cap \{ \rho^2 +(\lVert\hat{\eta}\rVert-1)^2+(\lambda+1)^2 + \varrho_{\mathrm{Sf}}^2< \epsilon_1, \epsilon_2<\varrho_{\mathrm{nf}} < \epsilon_1 \} = \varnothing
\end{equation} 
can be rewritten in terms of $\operatorname{WF}_{\mathrm{sc}}^{m,s_{\mathrm{Sf}}}(u)$, as the second set is disjoint from null infinity.

The linear combination of $m,s_{\mathrm{nf}},s_{\mathrm{Sf}}$ showing up in the threshold hypothesis
\begin{equation}
	m-s_{\mathrm{nf}}+s_{\mathrm{Sf}}>1/2
	\label{eq:A_threshold}
\end{equation}
in \Cref{prop:A_first_propagation}
can be read off of the linearization of $\mathsf{H}_{p[g]}$ at $\calA$ presented in \cref{eq:linearization_at_A} (remembering that $\rho=\varrho_{\mathrm{df}}$). Indeed, the coefficients of the linear combination are exactly those in \cref{eq:linearization_at_A}. The same sort of observation (\textit{mutatis mutandis}) will apply to all of the propositions in this section. This can be used as a simple sanity check.

The direction of the threshold condition \cref{eq:A_threshold} can be figured out using the heuristic that, when propagating control through a saddle point, we need to assume more ``incoming'' regularity/decay than outgoing regularity/decay. Since \Cref{prop:A_first_propagation} allows us to propagate control from fiber infinity and $\mathrm{Sf}$ into $\mathrm{nf}$, this means that we must assume that $m+s_{\mathrm{Sf}}$ is sufficiently large relative to $s_{\mathrm{nf}}$. 

\begin{proof}
	We handle the case $\varsigma,\sigma = +$, the other three being analogous. 
	Consider the symbol 
	\begin{equation} 
		a_0 = \varrho_{\mathrm{df}}^{m_0}\varrho_{\mathrm{nFf}}^{s_0}\varrho_{\mathrm{Sf}}^{\ell_0}
	\end{equation} 
	for  $m_0,s_0,\ell_0\in \bbR$ given by $m_0 = 1-2m$, $s_0 = -1-2s_{\mathrm{nFf}}$, and $\ell_0 = -1-2s_{\mathrm{Sf}}$, where, as above, we arrange for convenience that near $\smash{\calA^+_+}$, $\varrho_{\mathrm{df}}=\rho$, $\varrho_{\mathrm{nFf}}=\varrho_{\mathrm{nf}}$ (we will always do this below).   
	Then, by \cref{eq:H_nfSf_df}, 
	\begin{equation}
		\mathsf{H}_p a_0 = ( m_0 (2\hat{\eta}^2+\lambda+1) - 2s_0 + \ell_0  (1-\lambda) )a_0.
		\label{eq:misc_hpa}
	\end{equation}
	Thus, we can write $\mathsf{H}_{p[g]} a_0 = -\alpha a_0$ for a symbol $\alpha$ given by $\alpha =  -m_0 (2\hat{\eta}^2+\lambda+1) + 2s_0 - \ell_0  (1-\lambda)$ over null infinity. 
	Exactly at $\calA^+_+$, $\lambda=-1$ and $\hat{\eta}^2 = 1$, so 
	\begin{equation}
		\alpha|_{\calA^+_+}  =2 (-m_0 + s_0 - \ell_0) = 2(-1+2m -2 s_{\mathrm{nFf}}  +2s_{\mathrm{Sf}} )>0, 
	\end{equation}
	with the last inequality coming from our assumption that $m-s_{\mathrm{nFf}}+s_{\mathrm{Sf}}>1/2$.

	Let $\chi \in C_{\mathrm{c}}^\infty$ be such that $- \operatorname{sgn}(t) \chi'(t)\chi(t) = \chi_0^2(t)$ for some $\chi_0 \in C_{\mathrm{c}}^\infty(\bbR)$ and such that $\chi=1$ identically in some neighborhood of the origin (the construction by modifying $e^{-1/t}$ is standard --- see e.g.\ \cite{VasyGrenoble}). For $\digamma\in \bbR^+$, let $\chi_\digamma(t)=\chi(\digamma t)$, and correspondingly let 
	\begin{equation} 
		\chi_{0,\digamma}(t)=\smash{\digamma^{1/2} \chi_0(\digamma t)},
	\end{equation} 
	so that $- \operatorname{sgn}(t) \chi_\digamma'(t) \chi_\digamma(t) = \chi_{0,\digamma}(t)^2$. 
	Modify $a_0$ by using the $\chi_\digamma(t)$ to localize near $\smash{\calA^+_+}$: 
	define $a\in C^\infty({}^{\mathrm{de,sc}}\overline{T}^* \bbO)$ by 
	\begin{equation} 
		a = \chi_{\digamma'}(\tilde{p}[g])^2 \chi_\digamma(\varrho_{\mathrm{nf}})^2 \chi_\digamma(\aleph)^2 a_0
		\label{eq:misc_adf}
	\end{equation} 
	near $\calA^+_+$, where $\aleph$ is as in \Cref{prop:aleph}. 
	(For convenience, taking $\digamma$ sufficiently large, we arrange that $a$ is identically zero outside of the region for which the definition \cref{eq:misc_adf} is taken.) This does indeed localize near $\smash{\calA^+_+}$, in the sense that given any open neighborhood $U \supset \calA^+_+$, we can choose $\digamma,\digamma'>0$ sufficiently large so that $\operatorname{supp} a \subseteq U$.

	Letting $F_1$ be as in \Cref{prop:aleph}, 
	\begin{multline}
		\mathsf{H}_{p[g]} a = -\alpha a + 2(1+\varrho_{\mathrm{nf}} g)\chi_{\digamma'}(\tilde{p}[g])^2\chi_{0,\digamma}(\varrho_{\mathrm{nf}})^2 \chi(\aleph)^2 \varrho_{\mathrm{nf}} a_{0}   - 2\chi_{\digamma'}(\tilde{p}[g])^2\chi_\digamma(\varrho_{\mathrm{nf}})^2\chi_{0,\digamma}(\aleph)^2 ( 4\aleph  + F_1 ) a_0 \\ +2 \chi_{\digamma'}'(\tilde{p}[g])\chi_{\digamma'}(\tilde{p}[g])\chi_\digamma(\varrho_{\mathrm{nf}})^2\chi_\digamma(\aleph)^2  \tilde{q} \tilde{p}[g] a_0,
		\label{eq:misc_291}
	\end{multline}
	where $\tilde{q} = \varrho_{\mathrm{df}}^{-2} \mathsf{H}_{p[g]} \varrho_{\mathrm{df}}^2 \in C^\infty({}^{\mathrm{de,sc}}\overline{T}^* \bbO)$, so that $\mathsf{H}_{p[g]} \tilde{p}[g] = \tilde{q} \tilde{p}[g]$, and 
	\begin{equation} 
		g = - \varrho_{\mathrm{nf}}^{-2}(\mathsf{H}_{p[g]}-\mathsf{H}_p ) \varrho_{\mathrm{nf}} \in C^\infty({}^{\mathrm{de,sc}} \overline{T}^* \bbO).
	\end{equation} 
	Let $w= \varrho_{\mathrm{df}}^{-1}\varrho_{\mathrm{Pf}}\varrho_{\mathrm{nPf}}\varrho_{\mathrm{Sf}}\varrho_{\mathrm{nFf}}\varrho_{\mathrm{Ff}}$, so that $\mathsf{H}_{p[g]} = w^{-1} H_{p[g]}$.
	
	For all $\digamma>0$ sufficiently large, for $\digamma'>0$ sufficiently large, for $\delta$ sufficiently small we can define symbols $b,e,f,h,z\in S_{\mathrm{de,sc}}^{0,\mathsf{0}}$ such that 
	\begin{align}
		\begin{split} 
		\mathsf{H}_{p[g]} a + w^{-1} p_1 a &=  (-\delta a_0^{-2}a^2-b^2+e^2 \varrho_{\mathrm{nf}} - f^2 + h)a_0, \\
		H_{p[g]} a+p_1 a &=   (-\delta a_0^{-2} a^2-b^2+e^2 \varrho_{\mathrm{nf}} - f^2 + h )wa_0 
		\end{split}
		\label{eq:misc_hpp}
	\end{align}
	everywhere (where recall that $p_1$ was defined in \cref{eq:p1def}), with $b = \chi_{\digamma'}(\tilde{p}[g]) \chi_\digamma(\varrho_{\mathrm{nf}}) \chi_\digamma(\aleph) (\alpha-\delta a_0^{-1} a - w^{-1} p_1)^{1/2}$, 
	\begin{align} 
		\begin{split} 
		e &= \sqrt{2(1+\varrho_{\mathrm{nf}} g )} \chi_{\digamma'}(\tilde{p}[g]) \chi_{0,\digamma}(\varrho_{\mathrm{nf}}) \chi_\digamma(\aleph), \\ 
		f &= \sqrt{2} \chi_{\digamma'}(\tilde{p}[g]) \chi_\digamma(\varrho_{\mathrm{nf}}) \chi_{0,\digamma}(\aleph) (4\aleph+F_1)^{1/2},
		\end{split} 
	\end{align} 
	and 
	\begin{equation} 
		h=2 \chi'_{\digamma'}(\tilde{p}[g])\chi_\digamma(\tilde{p}[g]) \chi_\digamma(\varrho_{\mathrm{nf}})^2\chi_\digamma(\aleph)^2 \tilde{q} \tilde{p}[g]
		\label{eq:misc_377}
	\end{equation} 
	near $\calA^+_+$. It is because 
	\begin{equation} 
		w^{-1} p_1 \in S_{\mathrm{de,sc}}^{0,-\mathsf{1}}
	\end{equation} 
	vanishes at $\calA^+_+$ (and in fact, over all of the faces of $\bbO$) that we can take $\digamma,\digamma'$ sufficiently large so that the $-w^{-1} p_1$ term under the square root in the definition of $b$ is guaranteed to not spoil the sign. Likewise, as long as $\digamma$ is sufficiently large, and $\digamma'$ is sufficiently large relative to $\digamma$, then, from the description of $F_1$ in  \Cref{prop:aleph}, $4\aleph > F_1$ on the support of $f$, which is why $f$ is well-defined. Similarly, as long as $\digamma$ is sufficiently large, $e$ is well-defined.
	
	Quantizing, we get $A = (1/2)(\operatorname{Op}(a)+\operatorname{Op}(a)^*) \in  \Psi_{\mathrm{de,sc}}^{-m_0,(-\infty,-\infty,-\ell_0,-s_0,-\infty)}$, this being self-adjoint (here, we are just using the $L^2(\bbR^{1,d})$ inner product, not $L^2(\bbR^{1,d},g)$), and 
	\begin{align}
		\begin{split} 
			B = \operatorname{Op}(w^{1/2} a_0^{1/2} b) &\in  \Psi_{\mathrm{de,sc}}^{m,(-\infty,-\infty,s_{\mathrm{Sf}},s_{\mathrm{nFf}},-\infty)}, \\
			E = \operatorname{Op}(w^{1/2} a_0^{1/2} \varrho_{\mathrm{nf}}^{1/2} e) &\in \Psi_{\mathrm{de,sc}}^{m,(-\infty,-\infty,s_{\mathrm{Sf}},-\infty,-\infty)}, \\
			F = \operatorname{Op}(w^{1/2} a_0^{1/2} f) &\in  \Psi_{\mathrm{de,sc}}^{m,(-\infty,-\infty,s_{\mathrm{Sf}},s_{\mathrm{nFf}},-\infty)}, 
		\end{split}
	\label{eq:misc_295}
	\end{align}
	and $H = \operatorname{Op}(wa_0 h) \in \Psi_{\mathrm{de,sc}}^{2m,(-\infty,-\infty,2s_{\mathrm{Sf}},2s_{\mathrm{nFf}},-\infty)}$, such that
	\begin{equation}
		-i [ P, A] + i(P-P^*) A = -\delta A\Lambda^2 A - B^*B +E^*E - F^*F +H  + R 
		\label{eq:misc_k31}
	\end{equation}
	for some $R\in \Psi_{\mathrm{de,sc}}^{-m_0,(-\infty,-\infty,-2-\ell_0,-2-s_0,-\infty)}$. Above, 
	\begin{equation} 
		\Lambda = (1/2)( \operatorname{Op}(w^{1/2} a_0^{-1/2})+\operatorname{Op}(w^{1/2} a_0^{-1/2})^*) \in \Psi_{\mathrm{de,sc}}^{1-m,(-1/2,-1/2,-1-s_{\mathrm{Sf}},-1-s_{\mathrm{nFf}},-1/2)}.
	\end{equation} 
	 The quantization procedure can be arranged so as to preserve essential supports, so that 
	 \begin{equation}
	 	\operatorname{WF}'_{\mathrm{de,sc}}(B),\operatorname{WF}'_{\mathrm{de,sc}}(E), \operatorname{WF}'_{\mathrm{de,sc}}(F), \operatorname{WF}'_{\mathrm{de,sc}}(H) \subseteq \operatorname{WF}'_{\mathrm{de,sc}}(A) \subseteq \operatorname{supp}(a) 
	 \end{equation}
	 which, via the definition \cref{eq:misc_k31} of $R$, also forces $\operatorname{WF}'_{\mathrm{de,sc}}(R)\subseteq \operatorname{WF}'_{\mathrm{de,sc}}(A)$. 
	 We have $\operatorname{WF}'_{\mathrm{de,sc}} (E) \subseteq \operatorname{supp}(\chi_\digamma(\tilde{p}[g]) \chi_{0,\digamma}(\varrho_{\mathrm{nf}}) \chi_\digamma(\aleph))$ specifically.
	 For each $\epsilon_2>0$, by taking $\digamma$ sufficiently large, 
	 \begin{equation} 
	 	\operatorname{WF}'_{\mathrm{de,sc}} (E)  \subseteq \{ \rho^2 +(\lVert\hat{\eta}\rVert-1)^2+(\lambda+1)^2 + \varrho_{\mathrm{Sf}}^2< \epsilon_1, \epsilon_2<\varrho_{\mathrm{nf}} < \epsilon_1 \}
	 \end{equation}
 	as long as $\epsilon_1$ is sufficiently small relative to $\digamma$. 
	
	Computing
	\begin{equation}
		2i \Im \langle Au,Pu \rangle_{L^2} = \langle P^* Au,u \rangle_{L^2} - \langle APu, u \rangle_{L^2} = \langle ([P,A] + (P^* - P) A) u,u \rangle_{L^2}
	\end{equation}
	and assuming temporarily that $u$ is Schwartz, we get 
	\begin{align}
		\begin{split}
		-2 \Im \langle Au,P u \rangle_{L^2} &= \langle (-i[P,A]+i(P-P^*) A) u,u \rangle_{L^2}  \\ &=  -\delta \lVert \Lambda A u\rVert_{L^2}^2 - \lVert Bu \rVert_{L^2}^2 + \lVert Eu \rVert_{L^2}^2 - \lVert F u \rVert_{L^2}^2   + \langle Hu,u \rangle_{L^2} + \langle Ru ,u \rangle_{L^2} . 
		\end{split} 
		\label{eq:misc_z32}
	\end{align}
	Thus, 
	\begin{equation}
		\lVert Bu \rVert_{L^2}^2 + \delta \lVert \Lambda A u \rVert_{L^2}^2 \leq \lVert E u \rVert_{L^2}^2 + |\langle Hu,u\rangle_{L^2}| + |\langle Ru,u \rangle_{L^2}|  +2 |\langle Au,Pu \rangle_{L^2}|.
		\label{eq:misc_301}
	\end{equation}
 	From this, it can be deduced that, for each $N\in \bbN$, for some $\tilde{B}\in \Psi_{\mathrm{de,sc}}^{0,\mathsf{0}}$ elliptic at $\calA^+_+$ and $\tilde{E}\in \Psi_{\mathrm{de,sc}}^{0,\mathsf{0}}$ with 
 	\begin{equation}
 		\operatorname{WF}'_{\mathrm{de,sc}}(\tilde{E})=\operatorname{WF}'_{\mathrm{de,sc}}(E) ,
 	\end{equation}
 	the estimate 
	\begin{multline}
		\lVert \tilde{B} u \rVert_{H_{\mathrm{de,sc}}^{m,(N,N,s_{\mathrm{Sf}},s_{\mathrm{nFf}},-N)} }^2 \lesssim  \lVert \tilde{E} u \rVert_{H_{\mathrm{de,sc}}^{m,(-N,-N,s_{\mathrm{Sf}},-N,-N)} }^2 + \lVert GP u \rVert_{H_{\mathrm{de,sc}}^{m-1,(-N,-N, s_{\mathrm{Sf}}+1,s_{\mathrm{nFf}}+1,-N )} }^2  \\ 
		+ \lVert G u \rVert_{H_{\mathrm{de,sc}}^{m-1/2,(-N,-N,s_{\mathrm{Sf}}-1/2,s_{\mathrm{nFf}}-1/2,-N)} }^2 + \lVert u \rVert_{H_{\mathrm{de,sc}}^{-N,-N}}^2, 
		\label{eq:Aest}
	\end{multline}
	holds
	for some $G\in \Psi_{\mathrm{de,sc}}^{0,\mathsf{0}}$ having essential support in a small neighborhood of $\smash{\calA^+_+}$ (that can be taken arbitrarily small by making $\digamma$ arbitrarily large, but also at the cost of making the essential support of $\tilde{B},\tilde{E}$ smaller), chosen so that 
	\begin{equation}
		\operatorname{WF}'_{\mathrm{de,sc}}(1-G) \cap \operatorname{WF}'_{\mathrm{de,sc}}(A)   = \varnothing.
		\label{eq:misc_388}
	\end{equation} 
	Since we will use (and leave implicit) similar arguments below, we explain once the details of the deduction of \cref{eq:Aest}.
	Indeed, we can choose $\tilde{B},\tilde{E}$ to differ from $B,E$ by some elliptic factors, so that
	\begin{align}
		\lVert \tilde{B} u \rVert_{H_{\mathrm{de,sc}}^{m,(N,N,s_{\mathrm{Sf}},s_{\mathrm{nFf}},N)} } &\lesssim \lVert Bu \rVert_{L^2} + \lVert u \rVert_{H_{\mathrm{de,sc}}^{-N,-N}}, \\ 
		\lVert E u \rVert_{L^2} &\lesssim \lVert \tilde{E} u \rVert_{H_{\mathrm{de,sc}}^{m,(-N,-N,s_{\mathrm{Sf}},-N,-N)} }  + \lVert u \rVert_{H_{\mathrm{de,sc}}^{-N,-N}}.
	\end{align}
	Referring to \cref{eq:misc_377}, the essential support of $H$ is in the elliptic set of $P$, so an elliptic estimate controls the $\langle Hu,u \rangle_{L^2}$ term in \cref{eq:misc_301}:
	\begin{align}
		\begin{split} 
		|\langle Hu,u \rangle_{L^2}| &\lesssim  \lVert GP u \rVert^2_{H_{\mathrm{de,sc}}^{m-2,(-N,-N,s_{\mathrm{Sf}},s_{\mathrm{nFf}},-N )} } + \lVert u \rVert^2_{H_{\mathrm{de,sc}}^{-N,-N} } \\
		&\lesssim   \lVert GP u \rVert^2_{H_{\mathrm{de,sc}}^{m-1,(-N,-N,s_{\mathrm{Sf}}+1,s_{\mathrm{nFf}}+1,-N )} } + \lVert u \rVert^2_{H_{\mathrm{de,sc}}^{-N,-N} } .
		\end{split} 
	\end{align}
	The $\langle Ru,u\rangle_{L^2}$ term in \cref{eq:misc_301} is just estimated with Cauchy--Schwarz (and an elliptic estimate):
	\begin{equation}
		|\langle Ru,u \rangle_{L^2}| \lesssim \lVert G u \rVert^2_{H_{\mathrm{de,sc}}^{m-1/2,(-N,-N,s_{\mathrm{Sf}}-1/2,s_{\mathrm{nFf}}-1/2,-N )} } + \lVert u \rVert^2_{H_{\mathrm{de,sc}}^{-N,-N} } .
	\end{equation}
	Finally, we can choose $\tilde{G} \in \Psi_{\mathrm{de,sc}}^{0,\mathsf{0}}$ also satisfying \cref{eq:misc_388} such that
	\begin{equation} 
		\lVert \Lambda_{-1} \tilde{G} P u \rVert_{L^2} \lesssim \lVert G P u \rVert_{H_{\mathrm{de,sc} }^{m-1,(-N,-N,s_{\mathrm{Sf}}+1,s_{\mathrm{nFf}}+1,-N )}}, 
	\end{equation} 
	where $\Lambda_{-1}$ is a parametrix for $\Lambda$, and then we can bound the $\langle A u,Pu \rangle_{L^2}$ term in \cref{eq:misc_301} as follows: 
	\begin{equation} 
		|\langle Au,Pu\rangle_{L^2}| \leq |\langle Au,\tilde{G} Pu\rangle_{L^2}|+ |\langle Au,(1-\tilde{G})Pu\rangle_{L^2}|,
	\end{equation} 
	\begin{align}
		|\langle Au,(1-\tilde{G})Pu\rangle_{L^2}| &\lesssim \lVert u \rVert_{H_{\mathrm{de,sc}}^{-N,-N} }^2,  \\
		\begin{split}
			|\langle Au,\tilde{G} Pu\rangle_{L^2}| &\lesssim \lVert u \rVert_{H_{\mathrm{de,sc}}^{-N,-N} }^2+ \lVert \Lambda A u \rVert_{L^2} \lVert \Lambda_{-1} \tilde{G} P u \rVert_{L^2} \\
			\lVert \Lambda A u \rVert_{L^2} \lVert \Lambda_{-1} \tilde{G} P u \rVert_{L^2}, 
			&\leq 2^{-1} \varepsilon \lVert \Lambda A u \rVert_{L^2}^2 + 2^{-1} \varepsilon^{-1} \lVert \Lambda_{-1} \tilde{G} P u \rVert_{L^2}^2 \\ 
			&\lesssim  2^{-1} \varepsilon \lVert \Lambda A u \rVert_{L^2}^2 + 2^{-1} \varepsilon^{-1} \lVert G P u \rVert_{H_{\mathrm{de,sc} }^{m-1,(-N,-N,s_{\mathrm{Sf}}+1,s_{\mathrm{nFf}}+1,-N )}}^2  
		\end{split}
	\end{align}
	for any $\varepsilon>0$, where the bound is independent of $\varepsilon$. 
	If $\varepsilon$ is sufficiently small, then we can absorb the $2^{-1} \varepsilon \lVert \Lambda A u\rVert_{L^2}$ term into the $\delta \lVert \Lambda A u \rVert_{L^2}$ term in \cref{eq:misc_301}, yielding \cref{eq:Aest}, as claimed. The constant implicit in \cref{eq:Aest} depends on all of the operators involved, and on $\delta$ and $N$, but does not depend on $u$. Thus, assuming that $u$ is Schwartz, we have quantitatively controlled $u$ microlocally near $\calA^+_+$ in terms of the quantities on the right-hand side of the estimate. 
	
	The standard regularization argument \cite{VasyGrenoble}\cite{HintzVasyScriEB} allows us to make sense of the estimate for general $u\in \calS'$, with the conclusion being that if the right-hand side of \cref{eq:Aest} is finite, then the left-hand side is too, with the stated inequality holding. 
	One key point is that we can regularize in both the differential sense and the decay sense:
	\begin{itemize}
		\item first regularizing only in $s_{\mathrm{nFf}}$ (which we can do by an arbitrarily large number of orders because, in the threshold inequality \cref{eq:A_threshold}, decreasing $s_{\mathrm{nFf}}$ does not break the inequality), we conclude the desired estimate under the assumption
		\begin{equation}
			Gu\in H_{\mathrm{de,sc}}^{m,(\infty, \infty,s_{\mathrm{Sf}},-N_0,\infty )}, 
		\end{equation}
		where $N_0$ can be arbitrarily large.
		\item Apply the same basic argument, but regularize in $m$ and $s_{\mathrm{Sf}}$ instead to control
		\begin{equation} 
			\lVert Gu \rVert_{H_{\mathrm{de,sc}}^{m,(0,0,s_{\mathrm{Sf}},-N_0,0 )}}.
		\end{equation} 
		For each value of $N_0$, we can only regularize by finitely many orders, essentially because decreasing $m,s_{\mathrm{Sf}}$ can break \cref{eq:A_threshold}: in order to not spoil the signs involved in the construction of $b$, we must assume that 
		\begin{equation} 
			Gu \in H_{\mathrm{de,sc}}^{-N_1,(\infty,\infty,-N_1,-N_0,\infty)}
		\end{equation} 
		for $N_1$ satisfying $-2N_1 + N_0 >1/2$,
		which is the threshold condition for the regularized orders. 
	\end{itemize} 
	Combining the two regularization steps, we end up with the estimate 
	\begin{multline}
		\lVert \tilde{B} u \rVert_{L^2}^2 \lesssim  \lVert \tilde{E} u \rVert_{L^2}^2 + \lVert GP u \rVert_{H_{\mathrm{de,sc}}^{m-1,(-N,-N, s_{\mathrm{Sf}}+1,s_{\mathrm{nFf}}+1,-N )} }^2  + \lVert G u \rVert_{H_{\mathrm{de,sc}}^{m-1/2,(-N,-N,s_{\mathrm{Sf}}-1/2,s_{\mathrm{nFf}}-1/2,-N)} }^2 \\ +
		\lVert u \rVert_{H_{\mathrm{de,sc}}^{-N_1,(-N,-N,-N_1,-N_0,-N)}}^2, 
	\end{multline}
	for any $N,N_0\in \bbR$ and $N_1$ satisfying $-2N_1+N_0 > 1/2$, this holding in the strong sense that, if $u\in \calS'$ is such that the right-hand side is finite, then the left-hand side is as well. 
	By taking $N_0$ sufficiently large, we can choose $N_1$ such that $\min\{N_0,N_1\}>N$.
	Hence, \cref{eq:Aest} holds for all $u\in \calS'$. 
	
	Alternatively, we can regularize in both senses simultaneously with a careful choice of regularizer: for each $\varepsilon,K>0$, consider the locally-defined symbol 
	\begin{equation}
		\varphi_{\varepsilon,K} =  \Big(1 + \frac{\varepsilon}{\rho^{m_1} \varrho_{\mathrm{nf}}^{s_1}\varrho_{\mathrm{Sf}}^{\ell_1} } \Big)^{-K}, 
		\label{eq:regularizer}
	\end{equation}
	for to-be-decided $m_1,s_1,\ell_1>0$. 
	We can then define a symbol $a_{\varepsilon,K} = \smash{\varphi_{\varepsilon,K}^2}a$. The Lie bracket $\mathsf{H}_{p[g]} a_{\varepsilon,K}$ is the same as \cref{eq:misc_291}, with an extra factor of $\smash{\varphi_{\varepsilon,K}^2}$ on the right-hand side, except we have to add the term $2 \varphi_{\varepsilon,K} a \mathsf{H}_{p[g]} \varphi_{\varepsilon,K}$, which is equal to 
	\begin{equation}
		 \frac{4 K \varepsilon}{\varepsilon + \rho^{m_1} \varrho_{\mathrm{nf}}^{s_1}\varrho_{\mathrm{Sf}}^{\ell_1} } \varphi_{\varepsilon,K}^2  a \Big(\frac{\ell_1}{2} (1-\lambda) + \frac{m_1}{2} (2\hat{\eta}^2+\lambda+1)-s_1 \Big) 
		 \label{eq:misc_318}
	\end{equation}
	at $\mathrm{nFf}$. Note that, at $\calA^+_+$, the bracketed term is given by $\ell_1 + m_1 - s_1$. Choose $s_1=2$ and $m_1,\ell_1=1/2$. 
	Then, 
	\begin{equation} 
		 \varphi_{\varepsilon,K}^{-1} \mathsf{H}_{p[g]}\varphi_{\varepsilon,K} <0
		\end{equation}  
	in a neighborhood of $\calA^+_+$ which does not depend on $\varepsilon,K$. 
	Then, for all $\digamma>0$ sufficiently large, for $\delta$ sufficiently small (neither depending on $\varepsilon,K$), 
	\begin{equation}
		b_{\varepsilon}  = \varphi_{\varepsilon,K}   \chi_{\digamma'}(\tilde{p}[g]) \chi_\digamma(\varrho_{\mathrm{nf}}) \chi_\digamma(\aleph) \sqrt{\alpha-\delta a_0^{-1} a_\varepsilon - w^{-1}  p_1   - \frac{2}{\varphi_{\varepsilon,K}} \mathsf{H}_{p[g]} \varphi_{\varepsilon,K}}
	\end{equation}
	is a well-defined symbol near $\calA^+_+$. Defining $e_\varepsilon = \varphi_{\varepsilon,K} e$, $f_\varepsilon = \varphi_{\varepsilon,K} f$, and so on, 
	\begin{equation}
		\mathsf{H}_{p[g]} a_\varepsilon + w^{-1} p_1 a_\varepsilon =  (-\delta a_0^{-2}a_\varepsilon^2-b^2_\varepsilon+e^2_\varepsilon \varrho_{\mathrm{nf}} - f^2_\varepsilon + h_\varepsilon)a_0.
	\end{equation}
	Quantizing (with the extra weights of powers of $w,a_0$ thrown in as in \cref{eq:misc_295}) we get operators $A_\varepsilon,B_\varepsilon,F_\varepsilon,H_\varepsilon,R_\varepsilon$, with similar essential support properties to their non-regularized counterparts, such that 
	\begin{equation}
		-i [ P, A_\varepsilon] + i(P-P^*) A_\varepsilon = -\delta A_\varepsilon\Lambda^2 A_\varepsilon - B^*_\varepsilon B_\varepsilon +E_\varepsilon^*E_\varepsilon - F_\varepsilon^*F_\varepsilon +H_\varepsilon  + R_\varepsilon.
	\end{equation}
	Each of these is a uniform family of de,sc-operators with the same orders as their non-regularized counterparts. But, for each individual $\varepsilon>0$, they are regularizing operators. 
	For each $N\in \bbR$ and tempered distribution 
	\begin{equation} 
		u\in \smash{H_{\mathrm{de,sc}}^{-N,-N}} = H_{\mathrm{de,sc}}^{-N,(-N,-N,-N,-N,-N)},
		\label{eq:misc_411}
	\end{equation} 
	we can choose $K$ sufficiently large such that the algebraic manipulations above are all justified, and via the usual strong convergence argument the estimate \cref{eq:Aest} follows, now contingent only on the weak hypothesis that \cref{eq:misc_411} holds.
	Since $N$ was arbitrary, we can conclude that \cref{eq:Aest} holds in the strong sense, for any $u\in \calS'$. 
	
	We now finish the conclusion of the proposition from the strong estimate \cref{eq:Aest}. 
	Suppose that $u\in \calS'$ satisfies the hypotheses:
	\begin{itemize}
		\item $\operatorname{WF}_{\mathrm{de,sc}}^{m-1,\mathsf{s}+1}(P u) \cap \calA^\varsigma_\sigma = \varnothing$, 
		\item $\operatorname{WF}_{\mathrm{de,sc}}^{m,\mathsf{s}}(u) \cap \{ \rho^2 +(\lVert\hat{\eta}\rVert-1)^2+(\lambda+1)^2 + \varrho_{\mathrm{Sf}}^2< \epsilon_1, \epsilon_2<\varrho_{\mathrm{nf}} < \epsilon_1 \} $. 
	\end{itemize}
	Then, the first two terms on the right-hand side of \cref{eq:Aest} are finite, for any $N$, as long as 
	\begin{itemize}
		\item 
		$\digamma,\digamma'$ are sufficiently large (and correspondingly $\operatorname{WF}'_{\mathrm{de,sc}}(G)$ is taken sufficiently small) and
		\item $\epsilon_0$ is sufficiently small. 
	\end{itemize}
	The second assumption is required to control the $\tilde{E} u$ term.
	If $N$ is sufficiently large then, $u\in\smash{H_{\mathrm{de,sc}}^{-N,-N}}$, so the final term is finite as well. If
	\begin{equation}
		\operatorname{WF}_{\mathrm{de,sc}}^{m-1/2,(-N,-N,s_{\mathrm{Sf}}-1/2,s_{\mathrm{nFf}}-1/2,-N)} (u)\cap \calA^+_+ = \varnothing,
		\label{eq:misc_z1z}
	\end{equation}
	then, for $G$ with sufficiently small essential support (which we can arrange by taking $\digamma'\digamma$ sufficiently large), the third term on the right-hand side of \cref{eq:Aest} is finite. 
	Having checked that each term on the right-hand side of \cref{eq:Aest} is finite, we conclude that the left-hand side is finite as well. Since $\tilde{B}$ is elliptic at the radial set, we conclude that 
	\begin{equation} 
		\operatorname{WF}_{\mathrm{de,sc}}^{m,\mathsf{s}}(u)\cap \calA^+_+ = \varnothing.
		\label{eq:misc_z1y}
	\end{equation} 
	
	The condition 
	\begin{equation}
		\operatorname{WF}_{\mathrm{de,sc}}^{m-1/4,(-N,-N,s_{\mathrm{Sf}}-1/4,s_{\mathrm{nFf}}-1/2,-N)} (u)\cap \calA^+_+ = \varnothing
		\label{eq:misc_z1a}
	\end{equation}
	implies \cref{eq:misc_z1z}, but has the advantage that the orders in \cref{eq:misc_z1a} satisfy the threshold condition if and only if the originals do (as we are assuming as a hypothesis of the proposition). 
	The orders in \cref{eq:misc_z1a} are (for $N$ sufficiently large) a quarter-order smaller than those in \cref{eq:misc_z1y}, so the proposition follows via an inductive argument (taking the case when all of the orders are $\leq -N$ as the base case). 
\end{proof}

Similarly:
\begin{propositionp}
	Fix signs $\varsigma,\sigma \in \{-,+\}$. 
	Suppose that $m\in \bbR$ and $\mathsf{s} = (s_{\mathrm{Pf}},s_{\mathrm{nPf}},s_{\mathrm{Sf}},s_{\mathrm{nFf}},s_{\mathrm{Ff}})\in \bbR^5$ satisfying $m-s_{\mathrm{nf}}+s_{\mathrm{Sf}}<1/2$, where $s_{\mathrm{nf}}\in \{s_{\mathrm{nPf}},s_{\mathrm{nFf}}\}$, depending on $\sigma$. For any $\epsilon_1>0$, there exists some $\epsilon_0 \in (0,\epsilon_1)$ such that, if $u\in \calS'$ satisfies $\epsilon_2 \in (0,\epsilon_1)$, then, if $u\in \calS'$ satisfies 
	\begin{itemize}
		\item $\operatorname{WF}_{\mathrm{de,sc}}^{m-1,\mathsf{s}+1}(P u) \cap \calA^\varsigma_\sigma = \varnothing$, 
		\item $\operatorname{WF}_{\mathrm{de,sc}}^{m,\mathsf{s}}(u) \cap \{ \epsilon_2<\rho^2 +(\lVert\hat{\eta}\rVert-1)^2+(\lambda+1)^2 + \varrho_{\mathrm{Sf}}^2< \epsilon_1,  \varrho_{\mathrm{nf}} < \epsilon_1 \} = \varnothing$ for some $\epsilon_1>0$ and $\epsilon_2 \in (0,\epsilon_0)$, 
	\end{itemize}
	it is the case that $\operatorname{WF}_{\mathrm{de,sc}}^{m,\mathsf{s}}(u)\cap \calA^\varsigma_\sigma = \varnothing$. 
	\label{prop:A_second_propagation}
\end{propositionp}
	The threshold condition  $m-s_{\mathrm{nf}}+s_{\mathrm{Sf}}<1/2$ is the reverse of that, \cref{eq:A_threshold}, from the previous proposition, because we are propagating control in the opposite direction, so, compared to the previous proposition, the notions of incoming and outgoing regularity/decay are switched.

	The proof is the same as that in the previous proposition, with a few sign switches, which result in switching the signs of the $b,B$-terms terms. (The signs of the terms proportional to $\delta$ then have to be switched as well.) Thus, instead of the $\lVert E u \rVert_{L^2}$ term in \cref{eq:Aest}, we need to keep the $\lVert F u \rVert_{L^2}$ term, since it has the opposite sign of $\lVert Bu \rVert_{L^2}$. This results, in the $\smash{\calA^+_+}$ case, in an estimate (holding in the strong sense, for all $u\in \calS'$) of the form 
	\begin{multline}
		\lVert \tilde{B} u \rVert_{H_{\mathrm{de,sc}}^{m,(N,N,s_{\mathrm{Sf}},s_{\mathrm{nFf}},N)} }^2 \lesssim  \lVert \tilde{F} u \rVert_{H_{\mathrm{de,sc}}^{m,(-N,-N,s_{\mathrm{Sf}},s_{\mathrm{nFf}},-N)} }^2 + \lVert GP u \rVert_{H_{\mathrm{de,sc}}^{m-1,(-N,-N, s_{\mathrm{Sf}}+1,s_{\mathrm{nFf}}+1,-N )} }^2  \\ 
		+ \lVert G u \rVert_{H_{\mathrm{de,sc}}^{m-1/2,(-N,-N,s_{\mathrm{Sf}}-1/2,s_{\mathrm{nFf}}-1/2,-N)} }^2 + \lVert u \rVert_{H_{\mathrm{de,sc}}^{-N,-N}}^2, 
		\label{eq:misc_325}
	\end{multline}
	for $\tilde{F}\in \Psi_{\mathrm{de,sc}}^{0,\mathsf{0}}$ with \begin{equation} 
		\operatorname{WF}'_{\mathrm{de,sc}}(\tilde{F})= \operatorname{WF}'_{\mathrm{de,sc}}(F).
	\end{equation}
	The argument is analogous, with a few minor modifications. For instance, instead of taking $s_1=2$ and $m_1,\ell_1=1/2$ in the regularizer \cref{eq:regularizer}, we can take $m_1,s_1,\ell_1=1$, so that \cref{eq:misc_318} has the opposite sign, which matches the switched signs of the $b,B$-terms. From \cref{eq:misc_325} (with the parameters $\digamma,\digamma',G$ chosen appropriately, as above), the statement of the proposition follows.

\subsection{Propagation Through $\calN$}
\label{subsec:calN}

We now prove two different radial point estimates at $\calN$. By a ray, we mean a subset of $\calN^\varsigma_\sigma$ of the form 
\begin{equation}
	\calN_{I} = \calN^\varsigma_\sigma \cap \mathrm{cl}_{{}^{\mathrm{de,sc}}\overline{T}^* \bbO }\{|t|-r \in I \} 
\end{equation}
for some closed interval $I\subseteq [-\infty,+\infty]$ such that at least one of $\pm\infty$ in $I$.  So, for instance, $\calN_{\{-\infty\}} = \calN \cap {}^{\mathrm{de,sc}}\pi^{-1}(\mathrm{Sf})$ is a ray.
If $-\infty\in I$, we will call $\calN_I$ \emph{spacelike-adjacent}, and if $+\infty\in I$, we will call $\calN_I$ \emph{timelike-adjacent}. As we will see below, we can only propagate in one direction on each type of ray for each pair of admissible Sobolev orders $(m,\mathsf{s}) \in \bbR\times \bbR^5$ (without some more complicated argument). The one exception is $\calN^\varsigma_\sigma=\calN_{[-\infty,+\infty]}$ itself, which is both spacelike-adjacent and timelike-adjacent. We say that $\calN_I$ is strictly spacelike-adjacent or strictly timelike-adjacent if $\calN_I\neq \calN$. 

The timelike-adjacent case (which we use when studying the scattering problem) is:
\begin{proposition}
	Fix signs $\varsigma,\sigma \in \{-,+\}$, and let $\calN_I$ denote a strictly timelike-adjacent ray of $\calN^\varsigma_\sigma$, which we can write using the coordinates \cref{eq:misc_co1} (over $\Omega_{\mathrm{nfTf},\sigma,T}$, for some large $T>0$) as 
\begin{equation}
	\calN_I = \{\varrho_{\mathrm{Tf}} \leq \bar{\varrho}_{\mathrm{Tf}}, \rho = 0, s= 0 ,\beth=0\}
\end{equation} 
for some $\bar{\varrho}_{\mathrm{Tf}}>0$, where $\beth$ is as in \Cref{prop:beth}.  
Let $m\in \bbR$ and $\mathsf{s}\in \bbR^5$ satisfy $m<s_{\mathrm{nf}}+1$, where $s\in \{s_{\mathrm{nPf}},s_{\mathrm{nFf}}\}$, depending on the sign $\sigma$. 
Suppose that $u\in \calS'$ satisfies 
\begin{itemize}
	\item $\operatorname{WF}_{\mathrm{de,sc} }^{m-1,\mathsf{s}+1}(P u) \cap \calN_I =\varnothing$ and 
	\item $\operatorname{WF}_{\mathrm{de,sc}}^{m,\mathsf{s}}(u) \cap \{ \beth, s^2 < \epsilon_1 , \varrho_{\mathrm{Tf}}\leq \bar{\varrho}_{\mathrm{Tf}}+\epsilon_1, \epsilon_2< \rho < \epsilon_1 \} = \varnothing$ 
\end{itemize}
for some $\epsilon_1>0$ and sufficiently small $\epsilon_2 \in (0, \epsilon_1)$. Then, $\operatorname{WF}_{\mathrm{de,sc} }^{m,\mathsf{s}}(u) \cap \calN_I = \varnothing$.
\label{prop:beth_propagation_one}
\end{proposition}
The fact that the threshold condition involves the linear combination $m-s_{\mathrm{nf}}$ of orders can be read off \cref{eq:misc_282}.
\begin{proof}
	We handle the case $\varsigma,\sigma = +$, the others being analogous. 
	Let $a_0 = \varrho_{\mathrm{df}}^{m_0}\varrho_{\mathrm{nFf}}^{s_0}\varrho_{\mathrm{Tf}}^{\ell_0}$, where $m_0 = 1-2m$, $s_0 = -1-2s_{\mathrm{nf}}$, $\ell_0 = -1-2s_{\mathrm{Ff}}$. 
	Then, by \Cref{prop:alpham0s0}, we have 
	\begin{equation} 
		\mathsf{H}_{p[g]} a_0 = \tilde{\alpha}  a_0 
	\end{equation} 
	for a symbol $\tilde{\alpha}$ such that $\tilde{\alpha} = \alpha$ at $\calN$, where $\alpha=\alpha(m_0,s_0)$ is as defined in that proposition.
	The difference $\tilde{\alpha}-\alpha$ comes from hitting the $\varrho_{\mathrm{Tf}}$ term in $a_0$ by $\mathsf{H}_{p[g]}$. But, since the coefficient in $\mathsf{H}_p$ of $\varrho_{\mathrm{Tf}}\partial_{\varrho_{\mathrm{Tf}}}$ (see \cref{eq:H_nfTf_df}) vanishes at $\calN$ (where $s=0$), and since $\mathsf{H}_{p[g]}-\mathsf{H}_p$ vanishes as a b-vector field over $\partial \bbO$, the difference $\tilde{\alpha}-\alpha$ vanishes at $\calN$. 
	
	By \Cref{prop:alpham0s0} (and the observation that $m_0>s_0 \iff m<s_{\mathrm{nFf}}+1 $), $\tilde{\alpha}>0$ at $\calN_I$.

	Let $\Upsilon\in\bbR$, $\digamma,\digamma',\digamma''>0$, and  $\tilde{\varrho}_{\mathrm{Tf}}=\Upsilon \varrho_{\mathrm{nf}} + \varrho_{\mathrm{Tf}}$.
	Define $a\in C^\infty({}^{\mathrm{de,sc}}\overline{T}^* \bbO )$ by 
	\begin{equation}
		a= \chi_{\digamma''}(\tilde{p}[g])^2  \chi_{\digamma}(\rho_{\mathrm{df}})^2\chi_{\digamma'}(\beth)^2 \chi_\digamma( \max\{0, \tilde{\varrho}_{\mathrm{Tf}} - \bar{\varrho}_{\mathrm{Tf}} \} )^2 a_0
	\end{equation}
	near $\calN_I$, and we can take $a$ to be supported nearby. Then, near $\calN_I$, 
	\begin{multline}
		\mathsf{H}_{p[g]} a = \tilde{\alpha} a - 2\alpha_{\mathrm{df}} \chi_{\digamma''}(\tilde{p}[g])^2 \chi_{0,\digamma}(\varrho_{\mathrm{df}})^2 \chi_{\digamma'}(\beth)^2 \chi_\digamma(\max\{0,\tilde{\varrho}_{\mathrm{Tf}}-\bar{\varrho}_{\mathrm{Tf}}\})^2  \varrho_{\mathrm{df}} a_0 \\ +2 \chi_{\digamma''}(\tilde{p}[g])^2 \chi_{\digamma}(\varrho_{\mathrm{df}})^2 \chi_{0,\digamma'}(\beth)^2 \chi_\digamma(\max\{0,\tilde{\varrho}_{\mathrm{Tf}}-\bar{\varrho}_{\mathrm{Tf}}\})^2  (4\beth - F_2) a_0  
		\\
		+4( s\varrho_{\mathrm{Tf}}+(1-s)\Upsilon \varrho_{\mathrm{nf}}+\varrho_{\mathrm{nf}}^2  \varrho_{\mathrm{Tf}}^2 c)  \chi_{\digamma''}(\tilde{p}[g])^2  \chi_{\digamma}(\varrho_{\mathrm{df}})^2\chi_{\digamma'}(\beth)^2  \chi_{0,\digamma}(\max\{0,\tilde{\varrho}_{\mathrm{Tf}}-\bar{\varrho}_{\mathrm{Tf}}\})^2 a_0
		\\ +2 \chi'_{\digamma''}(\tilde{p}[g])\chi_{\digamma''}(\tilde{p}[g]) \chi_{\digamma}(\varrho_{\mathrm{df}})^2 \chi_{\digamma'}(\beth)^2 \chi_\digamma(\max\{0,\tilde{\varrho}_{\mathrm{Tf}}-\bar{\varrho}_{\mathrm{Tf}}\})^2  \tilde{q} \tilde{p}[g] a_0, 
	\end{multline}
	where 
	\begin{itemize}
		\item 	$\alpha_{\mathrm{df}} = \alpha[g](1,0)$,  $\alpha_{\mathrm{nf}} = \alpha[g](0,1)$, where $\alpha[g]$ is as in \Cref{prop:alpham0s0},
		\item $F_2$ is as in \Cref{prop:beth}, 
		\item  and $c\in C^\infty({}^{\mathrm{de,sc}}\overline{T}^* \bbO)$ comes from applying $\varrho_{\mathrm{nf}}^{-1}\varrho_{\mathrm{Tf}}^{-1}(\mathsf{H}_{p[g]}-\mathsf{H}_p)$ to $\tilde{\varrho}_{\mathrm{Tf}}$. The main $s\varrho_{\mathrm{Tf}}$ term there is read off \cref{eq:H_nfTf_df}, 
		\item $\tilde{q}$ is as in the proof of \Cref{prop:A_first_propagation}.
	\end{itemize}

	By \Cref{prop:alpham0s0}, $\alpha_{\mathrm{df}}>0$ near $\calN^+_+$.  By choosing $\digamma,\digamma',\digamma''$ sufficiently large, we can, by \Cref{lem:s_alt}, write $s=s_1 \tilde{p} + s_2(\hat{\eta}^2 + \mathsf{m}^2 \rho^2)$ nearby, where $s_1,s_2$ are as in the lemma.
	We want to work with $\tilde{p}[g]$, not $\tilde{p}$, so we will write this as 
	\begin{equation}
	s=s_1 \tilde{p}[g]+ \varrho_{\mathrm{nf}}^2\varrho_{\mathrm{Tf}}^2 c_2 + s_2(\hat{\eta}^2 + \mathsf{m}^2 \rho^2)
	\label{eq:441}
	\end{equation}
	for $c_2\in C^\infty({}^{\mathrm{de,sc}}\overline{T}^* \bbO)$ defined by $c_2 = s_1 (\tilde{p} -\tilde{p}[g])  \varrho_{\mathrm{nf}}^{-2}\varrho_{\mathrm{Tf}}^{-2}$. The key feature of \cref{eq:441} is that each term on the right-hand side is amenable to the positive commutator argument: 
	\begin{itemize}
		\item The term proportional to $\tilde{p}[g]$, when quantized and applied to $u$, will yield a term involving the forcing and therefore under control.
		\item The term involving $c_2$ is suppressed by a positive integer power of $\varrho_{\mathrm{nf}}$, which we will be able to dominate by a term of semidefinite sign by choosing $\Upsilon\neq 0$.
		\item The terms $s_2\hat{\eta}^2 ,s_2 \mathsf{m}^2 \rho^2$ have a semidefinite sign, because $s_2$ does (as part of \Cref{lem:s_alt}).
	\end{itemize}

	For all $\Upsilon>0$ sufficiently large and $\digamma,\digamma'>0$ sufficiently large, for $\digamma''>0$ sufficiently large, for $\delta$ sufficiently small (depending on $\digamma,\digamma'$), we can define symbols 
	\begin{equation} 
		b,e,f,g,h,z \in S_{\mathrm{de,sc}}^{0,\mathsf{0}}
	\end{equation} 
	such that 
	\begin{equation}
		\mathsf{H}_{p[g]} a + w^{-1} p_1 a = \Big(\delta a_0^{-2} a^2+ b^2 -\varrho_{\mathrm{df}} e^2  + f^2 + \rho_{\mathrm{nf}} z^2+ \varrho_{\mathrm{Tf}} \sum_{i=1}^d g^2_i+ h\tilde{p}[g]\Big)a_0 
	\end{equation}
	everywhere, with the following definitions:
	\begin{align}
		\begin{split}
		b &= \chi_{\digamma''}(\tilde{p}[g])\chi_{\digamma}(\varrho_{\mathrm{df}})\chi_{\digamma'}(\beth)\chi_\digamma(\max\{0,\tilde{\varrho}_{\mathrm{Tf}}-\bar{\varrho}_{\mathrm{Tf}}\}) (\tilde{\alpha}-\delta a_0^{-1} a + w^{-1} p_1 )^{1/2}, \\
		e &= \sqrt{2} \alpha_{\mathrm{df}}^{1/2} \chi_{\digamma''}(\tilde{p}[g])\chi_{0,\digamma}(\varrho_{\mathrm{df}}) \chi_{\digamma'}(\beth)\chi_\digamma(\max\{0,\tilde{\varrho}_{\mathrm{Tf}}-\bar{\varrho}_{\mathrm{Tf}}\}), \\
		f &= \sqrt{2} \chi_{\digamma''}(\tilde{p}[g]) \chi_{\digamma}(\varrho_{\mathrm{df}})\chi_{0,\digamma'}(\beth) \chi(\max\{0,\tilde{\varrho}_{\mathrm{Tf}}-\bar{\varrho}_{\mathrm{Tf}}\}) (4\beth - F_2)^{1/2}, 
		\end{split}
	\end{align}
	and, for $i=1,\ldots,d-1$, 
	\begin{align}
		g_i &= 2 \sqrt{s_2 } \chi_{\digamma''}(\tilde{p}[g])  \chi_{\digamma}(\varrho_{\mathrm{df}})\chi_{\digamma'}(\beth) \chi_{0,\digamma}(\max\{0,\tilde{\varrho}_{\mathrm{Tf}}-\bar{\varrho}_{\mathrm{Tf}}\}) \hat{\eta}_i \\ 
		g_{d} &= 2 \sqrt{s_2 } \chi_{\digamma''}(\tilde{p}[g])  \chi_{\digamma}(\varrho_{\mathrm{df}})\chi_{\digamma'}(\beth) \chi_{0,\digamma}(\max\{0,\tilde{\varrho}_{\mathrm{Tf}}-\bar{\varrho}_{\mathrm{Tf}}\}) \mathsf{m} \rho, 
	\end{align} 
	(well-defined because $s_2>0$ on $\calN^+_+$), 
	and, finally,
	\begin{equation}
	z = 2 \sqrt{(1-s)\Upsilon+ \varrho_{\mathrm{nf}}\varrho_{\mathrm{Tf}}^2 (c+c_2) } \cdot\chi_{\digamma''}(\tilde{p}[g]) \chi_{\digamma}(\varrho_{\mathrm{df}})\chi_{\digamma'}(\beth)  \chi_{0,\digamma}(\max\{0,\tilde{\varrho}_{\mathrm{Tf}}-\bar{\varrho}_{\mathrm{Tf}}\})
	\end{equation}
	(as long as $\Upsilon>0$, then $(1-s)\Upsilon+ \varrho_{\mathrm{nf}}\varrho_{\mathrm{Tf}}^2 (c+c_2) = \Upsilon>0$ on $\calN$)
	and 
	\begin{multline}
		h =  2 \chi'_{\digamma''}(\tilde{p}[g])\chi_{\digamma''}(\tilde{p}[g]) \chi_{\digamma}(\varrho_{\mathrm{df}})^2\chi_{\digamma'}(\beth)^2  \chi_\digamma (\max\{0,\tilde{\varrho}_{\mathrm{Tf}}-\bar{\varrho}_{\mathrm{Tf}}\})^2  \tilde{q} \\ + 4 s_1\varrho_{\mathrm{Tf}} \chi_\digamma(\tilde{p}[g])^2 \chi_{\digamma}(\varrho_{\mathrm{df}})^2\chi_{\digamma'}(\beth)^2 \chi_{0,\digamma}(\max\{0,\tilde{\varrho}_{\mathrm{Tf}}-\bar{\varrho}_{\mathrm{Tf}}\})^2
	\end{multline}
	near $\calN^+_+$. For instance, the properties of $F_2$ as specified in \Cref{prop:beth} give that $4\beth > F_2$ if $\digamma,\digamma',\digamma''$ are sufficiently large and $\digamma''$ is sufficiently large relative to these.

	If $\digamma,\digamma',\digamma''$, are sufficiently large, then $\operatorname{WF}_{\mathrm{de,sc}}'(e)\subseteq \{ \beth, s^2 < \epsilon_1 , \varrho_{\mathrm{Tf}}\leq \bar{\varrho}_{\mathrm{Tf}}+\epsilon_1, \epsilon_2< \rho < \epsilon_1 \}$, as long as $\epsilon_2$ is sufficiently small relative to $\digamma,\digamma',\digamma''$.

	Quantizing, we get $A = (1/2)(\operatorname{Op}(a)+\operatorname{Op}(a)^*) \in  \Psi_{\mathrm{de,sc}}^{-m_0,(-\infty,-\infty,-\infty,-s_0,-\ell_0)},$
	\begin{align}
		\begin{split} 
			B = \operatorname{Op}(w^{1/2} a_0^{1/2} b) &\in  \Psi_{\mathrm{de,sc}}^{m,(-\infty,-\infty,-\infty,s_{\mathrm{nFf}},s_{\mathrm{Ff}}) }, \\
			E = \operatorname{Op}(w^{1/2} a_0^{1/2} \varrho_{\mathrm{df}}^{1/2} e) &\in \Psi_{\mathrm{de,sc}}^{-\infty,(-\infty,-\infty,-\infty,s_{\mathrm{nFf}},s_{\mathrm{Ff}} )}, \\
			F = \operatorname{Op}(w^{1/2} a_0^{1/2} f) &\in  \Psi_{\mathrm{de,sc}}^{m,(-\infty,-\infty,-\infty, s_{\mathrm{nFf}},s_{\mathrm{Ff}} )}, \\
			G_i = \operatorname{Op}(w^{1/2} a_0^{1/2} \varrho_{\mathrm{Tf}}^{1/2} g_i) &\in  \Psi_{\mathrm{de,sc}}^{m,(-\infty,-\infty,-\infty,s_{\mathrm{nFf}}, -\infty)}, \\
			Z= \operatorname{Op}(w^{1/2} a_0^{1/2} \varrho_{\mathrm{nf}}^{1/2} z ) &\in \Psi_{\mathrm{de,sc}}^{m,(-\infty,-\infty,-\infty,s_{\mathrm{nFf}}-1/2,-\infty)}
		\end{split}
	\end{align}
	$H = \operatorname{Op}(wa_0 h) \in \Psi_{\mathrm{de,sc}}^{2m,(-\infty,-\infty,-\infty,2s_{\mathrm{nFf}},2s_{\mathrm{Ff}} )}$, 
	and $R\in \Psi_{\mathrm{de,sc}}^{-m_0,(-\infty,-\infty,-\infty,-2-s_0,-2-\ell_0)}$ such that 
	\begin{equation}
		-i [P, A] + i (P^* - P)= \delta A\Lambda^2 A + B^*B -E^*E + F^*F  + \sum_{i=1}^d G_i^* G_i + Z^* Z +H  \tilde{P}  + R 
		\label{eq:misc_432}
	\end{equation}
	for 
	\begin{equation} 
		\Lambda = (1/2)( \operatorname{Op}(w^{1/2} a_0^{-1/2})+\operatorname{Op}(w^{1/2} a_0^{-1/2})) \in  \Psi_{\mathrm{de,sc}}^{1-m,(-1/2,-1/2,-1/2,-1-s_{\mathrm{nFf}},-1-s_{\mathrm{Ff}})},
	\end{equation} 
	with the operators $A,B,E,F,G_i,Z,H$ all having essential supports contained within $\operatorname{supp} a$. Here, $\tilde{P}=\operatorname{Op}(\tilde{p}[g]) \in \Psi_{\mathrm{de,sc}}^{0,\mathsf{0}}$.
	
	The argument proceeds as usual from here, where the key observation is that the $F^*F$ term, $G_i^* G_i$, $Z^* Z$ terms have the same sign as the $B^* B$ term (and therefore can ultimately be discarded from the estimate) except we need a different estimate for the $H$ term (because $H\tilde{P}$ appears in \cref{eq:misc_432} instead of just $H$). Its contribution $\smash{\langle u, H\tilde{P} u \rangle}$ is estimated in the following way: for $u\in \calS$ and $N\in \bbN$, 
	\begin{equation} 
		|\langle u, H\tilde{P} u \rangle| \lesssim \lVert O u \rVert_{H_{\mathrm{de,sc}}^{m-1,\mathsf{s}-1}}^2 + \lVert O P u \rVert_{H_{\mathrm{de,sc}}^{m-1,\mathsf{s}+1}}^2 + \lVert P u \rVert^2_{H_{\mathrm{de,sc}}^{-N,-N} } 
	\end{equation} 
	for some $O\in \Psi_{\mathrm{de,sc}}^{0,\mathsf{0}}$ with essential support contained near $\calN^+_+$. Thus, rather than controlling this term with elliptic regularity as before, we use the assumption 
	\begin{equation} 
		\operatorname{WF}_{\mathrm{de,sc} }^{m-1,\mathsf{s}+1}(P u) \cap \calN_I =\varnothing,
	\end{equation}
	which implies that 
	\begin{equation} 
		\lVert O P u \rVert_{H_{\mathrm{de,sc}}^{m-1,\mathsf{s}+1}}<\infty
	\end{equation} 
	as long as $\digamma$ is sufficiently large. (And the $\smash{\lVert O u \rVert_{H_{\mathrm{de,sc}}^{m-1,\mathsf{s}-1}}}$ term will eventually be controlled using an inductive argument.)
	
	The end result, after carrying out the regularization argument and the typical inductive argument, is the estimate, holding in the strong sense for all $u\in \calS'$, 
	\begin{multline}
		\lVert \tilde{B} u \rVert_{H_{\mathrm{de,sc}}^{m,(N,N,N,s_{\mathrm{nFf}},s_{\mathrm{Ff}})} }^2 \lesssim  \lVert \tilde{E} u \rVert_{H_{\mathrm{de,sc}}^{m,(-N,-N,-N,s_{\mathrm{nFf}},s_{\mathrm{Ff}})} }^2 + \lVert QP u \rVert_{H_{\mathrm{de,sc}}^{m-1,(-N,-N,-N,s_{\mathrm{nFf}}+1,s_{\mathrm{Ff}}+1 )} }^2  \\ 
		 +\lVert u \rVert_{H_{\mathrm{de,sc}}^{-N,-N}}^2, 
		 \label{eq:misc_344}
	\end{multline}
	for some $\tilde{B} \in \smash{\Psi_{\mathrm{de,sc}}^{0,\mathsf{0}}}$ elliptic along $\calN_I$, $\tilde{E} \in  \smash{\Psi_{\mathrm{de,sc}}^{0,\mathsf{0}}}$ satisfying
	\begin{equation}
		\operatorname{WF}'_{\mathrm{de,sc}}(\tilde{E}) = \operatorname{WF}'_{\mathrm{de,sc}}(E),
	\end{equation}
	and some $Q$ dependent on the other operators whose essential support can be made to be an arbitrarily small neighborhood of $\calN_I$ by making $\digamma,\digamma',\digamma''$ larger. The estimate \cref{eq:misc_344} finishes the proof. 
\end{proof}
\begin{propositionp}
	Fix signs $\varsigma,\sigma \in \{-,+\}$, and let $\calN_I$ denote a strictly spacelike-adjacent ray of $\calN^\varsigma_\sigma$, which we can write using the coordinates \cref{eq:misc_co2} (over $\Omega_{\mathrm{nfSf},\sigma,R}$, for some large $R>0$) as 
	\begin{equation}
		\calN_I = \{\varrho_{\mathrm{Sf}} \leq \bar{\varrho}_{\mathrm{Sf}}, \rho = 0, \lambda = 1 ,\beth=0\}
	\end{equation} 
	for some $\bar{\varrho}_{\mathrm{Sf}}>0$, where $\beth$ is as in \Cref{prop:beth}.  
	Let $m\in \bbR$ and $\mathsf{s}\in \bbR^5$ satisfy $m>s_{\mathrm{nFf}}+1$, where $s\in \{s_{\mathrm{nPf}},s_{\mathrm{nFf}}\}$, depending on the sign $\sigma$. 
	Suppose that $u\in \calS'$ satisfies 
	\begin{itemize}
		\item $\operatorname{WF}_{\mathrm{de,sc} }^{m-1,\mathsf{s}+1}(P u) \cap \calN_I =\varnothing$ and 
		\item $\operatorname{WF}_{\mathrm{de,sc}}^{m,\mathsf{s}}(u) \cap \{\epsilon_2 < \beth < \epsilon_1 , \varrho_{\mathrm{Sf}}\leq \bar{\varrho}_{\mathrm{Sf}}+\epsilon_1, \rho^2 + (\lambda-1)^2 < \epsilon_1 \} = \varnothing$ 
	\end{itemize}
	for some $\epsilon_1>0$ and sufficiently small $\epsilon_2 \in (0, \epsilon_1)$. Then, $\operatorname{WF}_{\mathrm{de,sc} }^{m,\mathsf{s}}(u) \cap \calN_I = \varnothing$.
	\label{prop:beth_propagation_two}
\end{propositionp}
The proof is analogous to that above, with the usual sign switches.

For the special case of the full ray, the conclusions of both propositions hold:
\begin{propositionp}
	Fix signs $\varsigma,\sigma \in \{-,+\}$. 
	Suppose that $u\in \calS'$ satisfies $\operatorname{WF}_{\mathrm{de,sc} }^{m-1,\mathsf{s}+1}(P u) \cap \calN^\varsigma_\sigma =\varnothing$ and at least one of 
	\begin{itemize}
		\item $m<s_{\mathrm{nFf}}+1$ and $\operatorname{WF}_{\mathrm{de,sc}}^{m,\mathsf{s}}(u)$ is disjoint from an annulus around $\calN^\varsigma_\sigma$ of the form $\{ \beth, s^2 < \epsilon_1 ,  \epsilon_2< \rho < \epsilon_1 \}$ nearby, 
		\item $m>s_{\mathrm{nFf}}+1$ and $\operatorname{WF}_{\mathrm{de,sc}}^{m,\mathsf{s}}(u)$ is disjoint from an annulus around $\calN^\varsigma_\sigma$ of the form $\{\epsilon_2 < \beth < \epsilon_1 , \rho^2 + (\lambda-1)^2 < \epsilon_1 \}$ nearby, 
	\end{itemize}
	hold, for some $\epsilon_1>0$ and sufficiently small $\epsilon_2 \in (0, \epsilon_1)$. 
	Then, $\operatorname{WF}_{\mathrm{de,sc} }^{m,\mathsf{s}}(u) \cap \calN^\varsigma_\sigma = \varnothing$.
	\label{prop:beth_propagation_three}
\end{propositionp}
The proof is analogous to those above, except we no longer need the cutoff along null infinity, and we must make sure that the symbols are well-defined in both coordinate patches.

\subsection{Propagation Through $\calK$}
\label{subsec:calK}

\begin{proposition}
	Fix signs $\varsigma,\sigma \in \{-,+\}$.
	Suppose that $m\in \bbR$ and $\mathsf{s} = (s_{\mathrm{Pf}},s_{\mathrm{nPf}},s_{\mathrm{Sf}},s_{\mathrm{nFf}},s_{\mathrm{Ff}})\in \bbR^5$ satisfying $m+s_{\mathrm{nf}}-2s_{\mathrm{Sf}}> 1$, where $s_{\mathrm{nf}}\in \{s_{\mathrm{nPf}},s_{\mathrm{nFf}}\}$, depending on $\sigma$. 
	Then, if $u\in \calS'$ satisfies 
	\begin{itemize}
		\item $\operatorname{WF}_{\mathrm{de,sc}}^{m-1,\mathsf{s}+1}(P u) \cap \calK^\varsigma_\sigma = \varnothing$, 
		\item $\operatorname{WF}_{\mathrm{de,sc}}^{m,\mathsf{s}}(u) \cap \{ \rho^2 +\hat{\eta}^2+(\lambda+3)^2 + \varrho_{\mathrm{nf}}^2< \epsilon_1, \epsilon_2 < \varrho_{\mathrm{Sf}} < \epsilon_1 \}  = \varnothing$ for some $\epsilon_1>0$ and $\epsilon_2 \in (0,\epsilon_1)$ sufficiently small, 
	\end{itemize}
	it is the case that $\operatorname{WF}_{\mathrm{de,sc}}^{m,\mathsf{s}}(u)\cap \calK^\varsigma_\sigma = \varnothing$. 
	\label{prop:K_first_propagation}
\end{proposition}

The fact that the threshold condition involves the linear combination $m+s_{\mathrm{nf}}-2s_{\mathrm{Sf}}$ of orders can be read off \cref{eq:misc_34x}.
\begin{proof}
	We discuss the case $\varsigma,\sigma = +$, the other three being analogous. 
	Let $a_0 = \varrho_{\mathrm{df}}^{m_0}\varrho_{\mathrm{nFf}}^{s_0}\varrho_{\mathrm{Sf}}^{\ell_0}$, where $m_0=1-2m$, $s_0 = -1-2s_{\mathrm{nFf}}$, $\ell_0 = -1 - 2s_{\mathrm{Sf}}$. Exactly at $\calK^+_+$, $\lambda=-3$ and $\hat{\eta}^2=0$, so \cref{eq:H_nfSf_df} yields 
	\begin{equation} 
		\mathsf{H}_{p[g]} a_0 = \alpha a_0
	\end{equation}  
	for some $\alpha \in C^\infty({}^{\mathrm{de,sc}}\overline{T}^* \bbO)$ equal to $2(-m_0  -s_0 + 2 \ell_0) = 4(m+s_{\mathrm{nFf}} - 2s_{\mathrm{Sf}}-1)$ at $\calK^+_+$. Note that $\alpha>0$ near $\calK^+_+$. 
	
	Let $\gimel$ be as in \Cref{prop:gimel}.
	Define $a\in C^\infty({}^{\mathrm{de,sc}}\overline{T}^* \bbO)$ by 
	\begin{equation} 
		a = \chi_{\digamma}(\tilde{p}[g])^2 \chi_\digamma(\varrho_{\mathrm{Sf}})^2 \chi_\digamma(\gimel)^2 a_0
	\end{equation} 
	near $\calK^+_+$, with $a$ supported near $\calK^+_+$. We compute 
	\begin{multline}
		\mathsf{H}_{p[g]} a = \alpha a - 2(1-\lambda+ \varrho_{\mathrm{nf}}c )\chi_{\digamma} (\tilde{p}[g])^2\chi_{0,\digamma}(\varrho_{\mathrm{Sf}})^2 \chi_\digamma(\gimel)^2 \varrho_{\mathrm{Sf}} a_0   \\ + 2\chi_{\digamma}(\tilde{p}[g])^2\chi_\digamma(\varrho_{\mathrm{Sf}})^2\chi_{0,\digamma}(\gimel)^2 ( 4\gimel +E_3 - F_3 ) a_0  +2 \chi'_{\digamma}(\tilde{p}[g])\chi_{\digamma}(\tilde{p}[g])\chi_\digamma(\varrho_{\mathrm{Sf}})^2\chi_\digamma(\gimel)^2  \tilde{q} \tilde{p}[g] a_0,
	\end{multline}
	where 
	\begin{itemize}
		\item $E_3,F_3$ are as in \Cref{prop:gimel},
		\item $\tilde{q}$ is as in the proofs of the previous propositions,
		\item $c\in C^\infty({}^{\mathrm{de,sc}}\overline{T}^* \bbO)$ comes from hitting $\varrho_{\mathrm{Sf}}$ with $\mathsf{H}_{p[g]}-\mathsf{H}_p$. 
	\end{itemize}
	For all $\digamma>0$ sufficiently large, for $\delta$ sufficiently small, we can define symbols $b,e,f,h\in S_{\mathrm{de,sc}}^{0,\mathsf{0}}$ such that 
	\begin{align}
		\begin{split} 
		\mathsf{H}_{p[g]} a + w^{-1} p_1 a &= (\delta a_0^{-2}a^2+b^2-e^2 \varrho_{\mathrm{Sf}} + f^2 + h)a_0, \\
		H_{p[g]} a + p_1 a &= (\delta a_0^{-2}a^2+b^2-e^2 \varrho_{\mathrm{Sf}} + f^2 + h)wa_0
		\end{split} 
	\end{align}
	everywhere, with $b = \chi_{\digamma}(\tilde{p}[g]) \chi_\digamma(\varrho_{\mathrm{Sf}}) \chi_\digamma(\gimel) (\alpha-\delta a_0^{-1} a + w^{-1} p_1 )^{1/2}$,
	\begin{equation} 
		e = \sqrt{2(1-\lambda+\varrho_{\mathrm{nf}}c )} \chi_{\digamma}(\tilde{p}[g]) \chi_{0,\digamma}(\varrho_{\mathrm{Sf}}) \chi_\digamma(\gimel)
	\end{equation} 
	(as $\lambda=-3$ at $\calK^+_+$, the function under the square root is positive near the radial set), 
	\begin{equation} 
		f = \sqrt{2} \chi_{\digamma}(\tilde{p}[g]) \chi_\digamma(\varrho_{\mathrm{Sf}}) \chi_{0,\digamma}(\gimel) (4\gimel+E_3-F_3)^{1/2},
	\end{equation} 
	and $h = 2 \chi'_{\digamma}(\tilde{p}[g])\chi_\digamma(\tilde{p}[g])\chi_\digamma(\varrho_{\mathrm{Sf}})^2 \chi_\digamma(\gimel)^2 \tilde{q} \tilde{p}[g]$
	near $\calK^+_+$. The reason $f$ is well-defined is that, as long as $\digamma$ is large, then $4\gimel + E_3> F_3$. Indeed, since $E_3\geq 0$, it only helps, and the cubic vanishing of $F_3$ (as described in \Cref{prop:gimel}) suffices to guarantee that, for $\digamma$ sufficiently large, $4\gimel > F_3$ in the region under consideration.
	
	Quantizing, we get $A = (1/2)(\operatorname{Op}(a)+\operatorname{Op}(a)^*) \in  \Psi_{\mathrm{de,sc}}^{-m_0,(-\infty,-\infty,-\ell_0,-s_0,-\infty)},$
	\begin{align}
		\begin{split} 
			B = \operatorname{Op}(w^{1/2} a_0^{1/2} b) &\in  \Psi_{\mathrm{de,sc}}^{m,(-\infty,-\infty,s_{\mathrm{Sf}},s_{\mathrm{nFf}},-\infty)}, \\
			E = \operatorname{Op}(w^{1/2} a_0^{1/2} \varrho_{\mathrm{Sf}} e) &\in \Psi_{\mathrm{de,sc}}^{m,(-\infty,-\infty,-\infty,s_{\mathrm{nFf}},-\infty)}, \\
			F = \operatorname{Op}(w^{1/2} a_0^{1/2} f) &\in  \Psi_{\mathrm{de,sc}}^{m,(-\infty,-\infty,s_{\mathrm{Sf}},s_{\mathrm{nFf}},-\infty)}, 
		\end{split}
	\end{align}
	$h = \operatorname{Op}(wa_0 h) \in \Psi_{\mathrm{de,sc}}^{2m,(-\infty,-\infty,2s_{\mathrm{Sf}},2s_{\mathrm{nFf}},-\infty)}$, 
	and $R\in \Psi_{\mathrm{de,sc}}^{-m_0,(-\infty,-\infty,-2-s_0,-2-\ell_0,-\infty)}$ such that 
	\begin{equation}
		-i [ P, A] + i (P-P^*) A = \delta A \Lambda^2 A + B^*B -E^*E + F^*F +H  + R, 
		\label{eq:misc_k32}
	\end{equation}
	where $\Lambda$ is as in the previous propositions. If $\digamma$ is sufficiently large, then, for $\epsilon_2$ sufficiently small, 
	\begin{equation}
		\operatorname{WF}'_{\mathrm{de,sc}}(E) \subseteq \{ \rho^2 +\hat{\eta}^2+(\lambda+3)^2 + \varrho_{\mathrm{nf}}^2< \epsilon_1, \epsilon_2 < \varrho_{\mathrm{Sf}} < \epsilon_1 \} .
	\end{equation}

	The proof now proceeds as usual, where the key observation is that the $F^* F$ term has the same sign as the $B^*B$ term and therefore can be ultimately discarded from the estimates. Thus, 
	for some 
	\begin{equation}
		G,\tilde{B},\tilde{E} \in \Psi_{\mathrm{de,sc}}^{0,\mathsf{0}} 
	\end{equation}
	with $\tilde{B}$ elliptic at $\calK^+_+$ and $\operatorname{WF}_{\mathrm{de,sc}}'(\tilde{E})= \operatorname{WF}_{\mathrm{de,sc}}'(E)$,
	we have, for all $u\in \calS'$, the estimate  
	\begin{multline}
	\lVert \tilde{B} u \rVert_{H_{\mathrm{de,sc}}^{m,(N,N,s_{\mathrm{Sf}},s_{\mathrm{nFf}},N)} }^2 \lesssim  \lVert \tilde{E} u \rVert_{H_{\mathrm{de,sc}}^{m,(-N,-N,-N,s_{\mathrm{nFf}},-N)} }^2 + \lVert GP u \rVert_{H_{\mathrm{de,sc}}^{m-1,(-N,-N, s_{\mathrm{Sf}}+1,s_{\mathrm{nFf}}+1,-N )} }^2  \\ 
	+ \lVert u \rVert_{H_{\mathrm{de,sc}}^{-N,-N}}^2, 
	\end{multline}
	holding in the strong sense, where the essential support of $G$ can be made to be in an arbitrarily small neighborhood of $\calK^+_+$ by making $\digamma$ larger. This estimate completes the proof. 
\end{proof}

Similarly: 
\begin{propositionp}
	Fix signs $\varsigma,\sigma \in \{-,+\}$. Suppose that $m\in \bbR$ and $\mathsf{s} = (s_{\mathrm{Pf}},s_{\mathrm{nPf}},s_{\mathrm{Sf}},s_{\mathrm{nFf}},s_{\mathrm{Ff}})\in \bbR^5$ satisfying $m+s_{\mathrm{nf}}-2s_{\mathrm{Sf}}< 1$, where $s_{\mathrm{nf}}\in \{s_{\mathrm{nPf}},s_{\mathrm{nFf}}\}$, depending on $\sigma$. 
	Then, if $u\in \calS'$ satisfies 
	\begin{itemize}
		\item $\operatorname{WF}_{\mathrm{de,sc}}^{m-1,\mathsf{s}+1}(P u) \cap \calK^\varsigma_\sigma = \varnothing$, 
		\item $\operatorname{WF}_{\mathrm{de,sc}}^{m,\mathsf{s}}(u) \cap \{ \epsilon_2 < \rho^2 +\hat{\eta}^2+(\lambda+3)^2 + \varrho_{\mathrm{nf}}^2< \epsilon_1,  \varrho_{\mathrm{Sf}} < \epsilon_1 \}  = \varnothing$ for some $\epsilon_1>0$ and $\epsilon_2 \in (0,\epsilon_1)$ sufficiently small, 
	\end{itemize}
	it is the case that $\operatorname{WF}_{\mathrm{de,sc}}^{m,\mathsf{s}}(u)\cap \calK^\varsigma_\sigma = \varnothing$. 
\end{propositionp}

The proof follows the proof of \Cref{prop:K_first_propagation}, except the sign of the $B^*B$ term (along with the sign of the $\delta A \Lambda^2 A$ term) in \cref{eq:misc_k32} has to be switched, which results in having to keep the $F$ term in estimates rather than the $E$ term.

\subsection{Propagation Through $\calC$}
\label{subsec:calC}

\begin{proposition}
	Fix signs $\varsigma,\sigma \in \{-,+\}$. 
	Suppose that $m\in \bbR$ and $\mathsf{s} = (s_{\mathrm{Pf}},s_{\mathrm{nPf}},s_{\mathrm{Sf}},s_{\mathrm{nf}},s_{\mathrm{Tf}})\in \bbR^5$ satisfying $m+s_{\mathrm{nf}} - 2s_{\mathrm{Tf}} < 1$, where $s_{\mathrm{Tf}}\in \{s_{\mathrm{Pf}},s_{\mathrm{Ff}}\}$, depending on $\sigma$. 
	Then, if $u\in \calS'$ satisfies 
	\begin{itemize}
		\item $\operatorname{WF}_{\mathrm{de,sc}}^{m-1,\mathsf{s}+1}(P u) \cap \calC^\varsigma_\sigma = \varnothing$, 
		\item $\operatorname{WF}_{\mathrm{de,sc}}^{m,\mathsf{s}}(u) \cap  \{ \epsilon_2< \rho^2 +\hat{\eta}^2+(s-2)^2 + \varrho_{\mathrm{nf}}^2< \epsilon_1, \varrho_{\mathrm{Tf}} < \epsilon_1 \}  = \varnothing$ for some $\epsilon_1>0$ and $\epsilon_2 \in (0,\epsilon_1)$ sufficiently small, 
	\end{itemize}
	it is the case that $\operatorname{WF}_{\mathrm{de,sc}}^{m,\mathsf{s}}(u)\cap \calC^\varsigma_\sigma = \varnothing$. 
	\label{prop:C_first_propagation}
\end{proposition}

The fact that the threshold condition involves the linear combination $m+s_{\mathrm{nf}}-2s_{\mathrm{Tf}}$ of orders can be read off \cref{eq:misc_3fx}.
\begin{proof}
	We handle the case $\varsigma,\sigma = +$, the other three being analogous. Let $a_0 = \varrho_{\mathrm{df}}^{m_0}\varrho_{\mathrm{nFf}}^{s_0}\varrho_{\mathrm{Tf}}^{\ell_0}$, where $m_0 =1-2m$, $s_0 = -1-2s_{\mathrm{nFf}}$, $\ell_0 = -1-2s_{\mathrm{Ff}}$. Then, $\mathsf{H}_{p[g]} a_0 = \alpha a_0$ for $\alpha \in C^\infty({}^{\mathrm{de,sc}}\overline{T}^* \bbO )$ given by 
	\begin{equation}
		2^{-1}\alpha = (\hat{\eta}^2+(s-1)^2) m_0 + (s-1)s_0-s \ell_0 
	\end{equation}
	over null infinity. Exactly at $\calC^+_+$, $s=2$ and $\hat{\eta}=0$, so $\alpha = m_0 + s_0 -2\ell_0 = -2(m + s_{\mathrm{nFf}} -2 s_{\mathrm{Ff}}-1)>0$ there. 
	
	Define $a\in C^\infty({}^{\mathrm{de,sc}}\overline{T}^* \bbO)$ by $a = \chi_{\digamma}(\tilde{p}[g])^2 \chi_\digamma(\varrho_{\mathrm{Sf}} )^2 \chi_\digamma(\daleth)^2 a_0$ near $\calC^+_+$, where $\daleth$ is in \Cref{prop:daleth}, with $a$ supported near $\calC^+_+$. We calculate 
	\begin{multline}
		\mathsf{H}_{p[g]} a = \alpha a + 4(s+ \varrho_{\mathrm{nf}}c)\chi_{\digamma}(\tilde{p}[g])^2\chi_{0,\digamma}(\varrho_{\mathrm{Tf}})^2 \chi_\digamma(\daleth)^2 \varrho_{\mathrm{Tf}} a_0   \\- 2\chi_{\digamma}(\tilde{p}[g])^2\chi_\digamma(\varrho_{\mathrm{Tf}})^2\chi_{0,\digamma}(\daleth)^2 ( 4\daleth +E_4 + F_4 ) a_0  +2 \chi'_{\digamma}(\tilde{p}[g])\chi_{\digamma}(\tilde{p}[g])\chi_\digamma(\varrho_{\mathrm{Sf}})^2 \chi_\digamma(\daleth)^2  \tilde{q} \tilde{p}[g] a_0,
	\end{multline}
	where 
	\begin{itemize}
		\item  $E_4,F_4$ are as in \Cref{prop:daleth},
		\item $\tilde{q}$ is as in the proofs of the previous propositions, 
		\item $c\in C^\infty({}^{\mathrm{de,sc}}\overline{T}^* \bbO)$ arises from hitting $\varrho_{\mathrm{Tf}}$ with $\mathsf{H}_{p[g]}-\mathsf{H}_p$. 
	\end{itemize}

	For all $\digamma>0$ sufficiently large, for $\delta$ sufficiently small, we can define symbols $b,e,f,h\in S_{\mathrm{de,sc}}^{0,\mathsf{0}}$ such that 
		\begin{align}
			\mathsf{H}_{p[g]} a p +w^{-1} p_1 a &= (\delta a_0^{-2}a^2+b^2+e^2 \varrho_{\mathrm{Tf}} - f^2 + h)a_0, \\
			H_{p[g]} a +p_1 a &= (\delta a_0^{-2}a^2+b^2+e^2 \varrho_{\mathrm{Tf}} - f^2 + h)wa_0 
		\end{align}
		everywhere, with $b =  \chi_{\digamma'}(\tilde{p}[g]) \chi_\digamma(\varrho_{\mathrm{Tf}}) \chi_\digamma(\daleth) (\alpha-\delta a_0^{-1} a + w^{-1} p_1 )^{1/2}$,
		\begin{equation} 
			e = 2 \sqrt{s+ \varrho_{\mathrm{nf}}c } \chi_{\digamma}(\tilde{p}[g]) \chi_{0,\digamma}(\varrho_{\mathrm{Tf}}) \chi_\digamma(\daleth) \varrho_{\mathrm{Tf}}^{1/2},
		\end{equation} 
		$f = \sqrt{2} \chi_{\digamma}(\tilde{p}[g]) \chi_\digamma(\varrho_{\mathrm{Tf}}) \chi_{0,\digamma}(\daleth) (4\daleth+E_3+F_4)^{1/2}$,
		and $h = 2 \chi'_{\digamma}(\tilde{p}[g])\chi_\digamma(\tilde{p}[g])\chi_\digamma(\varrho_{\mathrm{Sf}})^2 \chi_\digamma(\daleth)^2 \tilde{q} \tilde{p}[g]$
		near $\calC^+_+$. 
		Since $s=2$ on the radial set in question, $e$ is well-defined. Similarly, as long as $\digamma$ is large enough, then $f$ is well-defined.

		Quantizing, we get $A = (1/2)(\operatorname{Op}(a)+\operatorname{Op}(a)^*) \in  \Psi_{\mathrm{de,sc}}^{-m_0,(-\infty,-\infty,-\infty,-s_0,-\ell_0)},$
		\begin{align}
			\begin{split} 
				B = \operatorname{Op}(w^{1/2} a_0^{1/2} b) &\in  \Psi_{\mathrm{de,sc}}^{m,(-\infty,-\infty,-\infty,s_{\mathrm{nFf}},s_{\mathrm{Ff}})}, \\
				E = \operatorname{Op}(w^{1/2} a_0^{1/2} \varrho_{\mathrm{Tf}} e) &\in \Psi_{\mathrm{de,sc}}^{m,(-\infty,-\infty,-\infty,s_{\mathrm{nFf}},-\infty)}, \\
				F = \operatorname{Op}(w^{1/2} a_0^{1/2} f) &\in  \Psi_{\mathrm{de,sc}}^{m,(-\infty,-\infty,-\infty,s_{s_{\mathrm{nFf}}},s_{\mathrm{Ff}})}, \\
				h = \operatorname{Op}(wa_0 h) &\in \Psi_{\mathrm{de,sc}}^{2m,(-\infty,-\infty,-\infty,2s_{\mathrm{nFf}},2s_{\mathrm{Ff}})},
			\end{split}
		\end{align}
		and $R\in \Psi_{\mathrm{de,sc}}^{-m_0,(-\infty,-\infty,-\infty,2s_{\mathrm{nFf}}-1,2s_{\mathrm{Ff}}-1)}$ such that 
		\begin{equation}
			-i [ P, A] + i (P-P^*) A = \delta A \Lambda^2 A + B^*B +E^*E - F^*F +H  + R
			\label{eq:misc_k3h}
		\end{equation}
		for $\Lambda = (1/2)( \operatorname{Op}(w^{1/2} a_0^{-1/2})+\operatorname{Op}(w^{1/2} a_0^{-1/2})^*) \in \Psi_{\mathrm{de,sc}}^{1-m,(-1/2,-1/2,-1/2,-1-s_{\mathrm{nFf}},-1-s_{\mathrm{Ff}})}$. 
		If $\digamma$ is sufficiently large, then
		\begin{equation}
			\operatorname{WF}'_{\mathrm{de,sc}}(F) \subseteq \{ \epsilon_2< \rho^2 +\hat{\eta}^2+(s-2)^2 + \varrho_{\mathrm{nf}}^2< \epsilon_1, \varrho_{\mathrm{Tf}} < \epsilon_1 \},
		\end{equation}
		as long as $\epsilon_2$ is sufficiently small.

		The proof now proceeds as usual, where the key observation is that the $E$ term has the same sign as the $B$ term and therefore can be ultimately discarded from the estimates. Thus, 
		for some 
		\begin{equation}
			G,\tilde{B},\tilde{F} \in \Psi_{\mathrm{de,sc}}^{0,\mathsf{0}} 
		\end{equation}
		with $\tilde{B}$ elliptic at $\calC^+_+$ and $\operatorname{WF}_{\mathrm{de,sc}}'(\tilde{F})= \operatorname{WF}_{\mathrm{de,sc}}'(F)$,
		we have, for all $u\in \calS'$, the estimate  
		\begin{multline}
			\lVert \tilde{B} u \rVert_{H_{\mathrm{de,sc}}^{m,(N,N,N,s_{\mathrm{nFf}},s_{\mathrm{Ff}})} }^2 \lesssim  \lVert \tilde{F} u \rVert_{H_{\mathrm{de,sc}}^{m,(-N,-N,-N,s_{\mathrm{nFf}},s_{\mathrm{Ff}})} }^2 + \lVert GP u \rVert_{H_{\mathrm{de,sc}}^{m-1,(-N,-N, -N,s_{\mathrm{nFf}}+1,s_{\mathrm{Sf}}+1 )} }^2  \\ 
			+ \lVert u \rVert_{H_{\mathrm{de,sc}}^{-N,-N}}^2, 
		\end{multline}
		holding in the strong sense, where the essential support of $G$ can be made to be in an arbitrarily small neighborhood of $\calC^+_+$ by making $\digamma$ larger. This estimate completes the proof. 
\end{proof}

Similarly: 
\begin{propositionp}
	Fix signs $\varsigma,\sigma \in \{-,+\}$. 
	Suppose that $m\in \bbR$ and $\mathsf{s} = (s_{\mathrm{Pf}},s_{\mathrm{nPf}},s_{\mathrm{Sf}},s_{\mathrm{nFf}},s_{\mathrm{Ff}})\in \bbR^5$ satisfying $m+s_{\mathrm{nf}} - 2s_{\mathrm{Tf}}>1$, where $s_{\mathrm{Tf}}\in \{s_{\mathrm{Pf}},s_{\mathrm{Ff}}\}$, depending on $\sigma$. 
	Then, if $u\in \calS'$ satisfies 
	\begin{itemize}
		\item $\operatorname{WF}_{\mathrm{de,sc}}^{m-1,\mathsf{s}+1}(P u) \cap \calC^\varsigma_\sigma = \varnothing$, 
		\item $\operatorname{WF}_{\mathrm{de,sc}}^{m,\mathsf{s}}(u) \cap  \{ \rho^2 +\hat{\eta}^2+(s-2)^2 + \varrho_{\mathrm{nf}}^2< \epsilon_1, \epsilon_2< \varrho_{\mathrm{Tf}} < \epsilon_1 \}  = \varnothing$ for some $\epsilon_1>0$ and $\epsilon_2 \in (0,\epsilon_1)$ sufficiently small, 
	\end{itemize}
	it is the case that $\operatorname{WF}_{\mathrm{de,sc}}^{m,\mathsf{s}}(u)\cap \calC^\varsigma_\sigma = \varnothing$. 
	\label{prop:C_second_propagation}
\end{propositionp}

\section{The radial set $\calR$}
\label{sec:radialpoint}

Let $P$ be as in the previous section. We encode the radial point estimate at $\calR$ in the qualitative statements:

\begin{theorem}[Propagation out of $\calR$, with module regularity]
	Suppose that $m\in \bbR$ and $\mathsf{s} \in (s_{\mathrm{Pf}},s_{\mathrm{nPf}},s_{\mathrm{Sf}},s_{\mathrm{nFf}},s_{\mathrm{Ff}}) \in \bbR^5$ and $s_{\mathrm{Pf},0},s_{\mathrm{Ff},0}\in \bbR$ satisfy 
	\begin{equation}
		s_{\mathrm{Pf}} > s_{\mathrm{Pf},0} > -1/2 \quad \text{ and } s_{\mathrm{Ff}} > s_{\mathrm{Ff},0} > -1/2.
	\end{equation} 
	Let $\mathsf{s}_0 = (s_{\mathrm{Pf},0},s_{\mathrm{nPf},0},s_{\mathrm{Sf},0},s_{\mathrm{nFf},0},s_{\mathrm{Ff},0}) \in \bbR^5$.
	Fix signs $\varsigma,\sigma \in \{-,+\}$ and $k,\kappa \in \bbN$. 
	Then, if $u\in \calS'$ is a solution to $P u=f$ such that 
	\begin{equation} 
		\operatorname{WF}_{\mathrm{de,sc}}^{-N,\mathsf{s}_0}(Au)\cap \calR^\varsigma_\sigma=\varnothing
		\label{eq:misc_nm1}
	\end{equation} 
	and if $\operatorname{WF}_{\mathrm{de,sc}}^{-N,\mathsf{s}+1}(Af)\cap \calR^\varsigma_\sigma=\varnothing$ for all $A\in \frakM_{\varsigma,\sigma}^{\kappa,k}$ for some $N>0$, then $\operatorname{WF}_{\mathrm{de,sc}}^{m,\mathsf{s}}(Au)\cap \calR^\varsigma_\sigma=\varnothing$
	for all $A\in \frakM_{\varsigma,\sigma}^{\kappa,k}$ as well.
	\label{thm:R1}
\end{theorem}
Recall that $\frakM^{\kappa,k}_{\varsigma,\sigma}$ was defined in \S\ref{subsec:test_modules}.
\begin{theorem}[Propagation into $\calR$, with module regularity]
		Suppose that $m\in \bbR$ and $\mathsf{s} \in (s_{\mathrm{Pf}},s_{\mathrm{nPf}},s_{\mathrm{Sf}},s_{\mathrm{nFf}},s_{\mathrm{Ff}}) \in \bbR^5$ and $s_{\mathrm{Pf},0},s_{\mathrm{Ff},0}\in \bbR$ satisfy $\max\{s_{\mathrm{Pf}},s_{\mathrm{Ff}}\}<-1/2$. Fix signs $\varsigma,\sigma \in \{-,+\}$ and $k,\kappa \in \bbN$.  Let $u\in \calS'$ denote a solution to $Pu=f$. Then, if there exists a neighborhood $U\subseteq \smash{{}^{\mathrm{de,sc}}\overline{T}^* \bbO}$ of $\calR$ such that 
		\begin{itemize}
			\item $\operatorname{WF}_{\mathrm{de,sc}}^{-N,\mathsf{s}}(Au) \cap U \subseteq \calR^\varsigma_\sigma$, 
			\item $\operatorname{WF}_{\mathrm{de,sc}}^{-N,\mathsf{s}+1}(Af) \cap \calR^\varsigma_\sigma=\varnothing$
		\end{itemize}
		for all $A\in \frakM_{\varsigma,\sigma}^{\kappa,k}$, then $\operatorname{WF}_{\mathrm{de,sc}}^{m,\mathsf{s}}(Au) \cap U =\varnothing$
		for all $A\in \frakM_{\varsigma,\sigma}^{\kappa,k}$ as well. 
		\label{thm:R2}
\end{theorem}

We only consider the case of $\calR^+_+$ explicitly. The case of $\calR^-_-$ is essentially identical, and the cases of $\calR^+_-,\calR^-_+$ have overall signs switched in the computations but are otherwise identical. 

We will prove the result in three parts: in \S\ref{subsec:basecase}, we handle the $k=0,\kappa= 0$ case (which is the de,sc-analogue of the standard radial point result described in \cite{VasyGrenoble,VasyKG}), in \S\ref{subsec:firstinduction} we handle $k>0$ via induction (this being done via a somewhat involved secondary positive commutator argument), and in \S\ref{subsec:secondinduction} we handle $\kappa>0$ via another, more straightforward induction. 
The argument is a modification of that in \cite[\S6]{HassellMelroseVasy}\cite[Appendix A]{HassellMelroseVasy2}\cite[\S3]{HassellETAL}. Because $\calR$ does not hit fiber infinity, the test modules $\frakM,\frakN$ used in this section are generated by differential operators; this is what allows the use of standard arguments. 

\subsection{Base case}
\label{subsec:basecase}
Let $\kappa,k=0$. 
We now use $\rho_0$ to denote a quadratic defining function of $\calR^+_+$ in $\Sigma_{\mathsf{m},+} \cap \mathrm {}^{\mathrm{de,sc}}\pi^{-1}(\mathrm{Ff})$, such that 
\begin{equation} 
	\calR^+_+ = \smash{\rho_0^{-1}(\{0\})} \cap \Sigma_{\mathsf{m},+}\cap \mathrm {}^{\mathrm{de,sc}}\pi^{-1}(\mathrm{Ff}).
\end{equation}  
Over $\Omega_{\mathrm{nfTf},+,0}$, we can take this to be of the form $\rho_0 = \hat{\eta}^2+(s-1)^2$ with respect to the coordinate system \cref{eq:misc_co2}. 

We first observe that the symbol $F \in S_{\mathrm{de,sc}}^{0,\mathsf{0}}$ defined by 
\begin{equation}
	\mathsf{H}_p \rho_0 = -4 \rho_0 + F
\end{equation}
vanishes cubically at $\calR^+_+$. In order to show this, it suffices to check the claim in local coordinate patches. Away from null infinity, this is familiar \cite{VasyGrenoble} from the radial point estimate for Klein--Gordon in the sc-calculus, so we only need to check the situation near null infinity. Near null infinity (recall that we are taking $\varrho_{\mathrm{df}}=\rho$ as usual over $\mathrm{nFf}\cap \mathrm{Tf}$), \cref{eq:H_nfTf_df} yields 
\begin{equation}
	\mathsf{H}_p \rho_0 = - 4(2-s) (\hat{\eta}^2+s(s-1))(s-1) + 4 (\hat{\eta}^2 + s^2-s-1) \hat{\eta}^2 , 
\end{equation}
from which it can be seen that $F$ vanishes cubically at $\calR^+_+$. Also, \begin{equation} 
	\mathsf{H}_{p[g]} \rho_0 = - 4 \rho_0 + F + \tilde{F}\varrho_{\mathrm{Ff}}
\end{equation} 
for some $\tilde{F} \in C^\infty({}^{\mathrm{de,sc}}\overline{T}^* \bbO)$.

Let $a_0 = \varrho_{\mathrm{nFf}}^{s_0}\varrho_{\mathrm{Ff}}^{\ell_0}$, where  $s_0 = -1 - 2s_{\mathrm{nFf}}$, $\ell_0 = -1-2s_{\mathrm{Ff}}$. Then, we can write 
\begin{equation} 
	\mathsf{H}_{p[g]} a_0 = \alpha a_0
\end{equation} 
for $\alpha \in C^\infty({}^{\mathrm{de,sc}}\overline{T}^* \bbO)$ given by 
\begin{equation}
	\alpha= -2(s_0(1-s)  +\ell_0s  ) = 2(1+2s_{\mathrm{nFf}})(1-s) + 2(1+2s_{\mathrm{Ff}})s
\end{equation} 
at ${}^{\mathrm{de,sc}} \pi^{-1}(\mathrm{nFf}\cap \mathrm{Ff})$, assuming without loss of generality that $\varrho_{\mathrm{Ff}}= \varrho_{\mathrm{Tf}}$ in local coordinates near $\mathrm{nFf}\cap \mathrm{Ff}$. 
Exactly at $\calR^+_+$, this is $2(1+2s_{\mathrm{Ff}})$, which has a definite sign as long as $s_{\mathrm{Ff}}\neq -1/2$. The sign found here is the same as over the whole of $\mathrm{Ff}$ in the standard sc-analysis (as it had to be, since our computation had to match the usual one slightly away from null infinity). 

Now consider, as usual, a symbol $a = \chi_\digamma(\tilde{p}[g])^2 \chi_{\digamma'}(\varrho_{\mathrm{Ff}})^2 \chi_\digamma(\rho_0)^2 a_0$ near $\calR^+_+$ and supported away from $\mathrm{df},\mathrm{Sf},\mathrm{nPf},\mathrm{Pf}$. Then 
\begin{multline}
	\mathsf{H}_{p[g]} a = \alpha a +2\chi_\digamma(\tilde{p}[g])^2 \chi_{\digamma'}(\varrho_{\mathrm{Ff}})^2 \chi_{0,\digamma}(\rho_0)^2 a_0 (4 \rho_0 - F - \tilde{F} \varrho_{\mathrm{Ff}} ) \\ + 2 \chi_\digamma(\tilde{p}[g])^2 \chi_{0,\digamma'}(\varrho_{\mathrm{Ff}})^2\chi_\digamma(\rho_0)^2 a_0 \varrho_{\mathrm{Ff}}(2s   - \varrho_{\mathrm{Ff}} c ) \\ + 2 \chi'_{\digamma}(\tilde{p}[g]) \chi_\digamma(\tilde{p}[g]) \chi_{\digamma'}(\varrho_{\mathrm{Ff}})^2 \chi_\digamma(\rho_0)^2 \tilde{q} \tilde{p}[g] a_0,
\end{multline}
where 
\begin{itemize}
	\item $\tilde{q}$ is as in the previous section, 
	\item  $c\in C^\infty({}^{\mathrm{de,sc}}\overline{T}^* \bbO)$ comes from applying $\varrho_{\mathrm{Tf}}^{-1} (\mathsf{H}_{p[g]} - \mathsf{H}_p) \in \calV_{\mathrm{b}}({}^{\mathrm{de,sc}}\overline{T}^* \bbO)$ to $\varrho_{\mathrm{Ff}}$. 
\end{itemize}
The proof now splits into two cases, depending on the sign of $s_{\mathrm{Ff}}+1/2=s_{\mathrm{Tf}}+1/2$. 
\begin{itemize}
	\item First suppose that $s_{\mathrm{Tf}} > -1/2$, so that $\alpha>0$ near $\calR^+_+$.  For all $\digamma$ sufficiently large, for all $\digamma'$ sufficiently large relative to $\digamma$, for all $\delta>0$ sufficiently small, we can define symbols $b,e,\tilde{g},h \in S_{\mathrm{de,sc}}^{0,\mathsf{0}}$ such that 
	\begin{align}
		\mathsf{H}_{p[g]} a &= (\delta a_0^{-2}a^2 + b^2  +e^2  +\varrho_{\mathrm{Ff}}\tilde{g}^2 + h ) a_0 
		\label{eq:misc_380}\\
		H_{p[g]} a &= (\delta a_0^{-2}a^2 + b^2 + e^2 +\varrho_{\mathrm{Ff}}\tilde{g}^2 + h ) wa_0
	\end{align}
	everywhere, with $b =\chi_\digamma(\tilde{p}[g])\chi_{\digamma'}(\varrho_{\mathrm{Ff}})\chi_\digamma(\rho_0)(\alpha-\delta a_0^{-1} a)^{1/2}$, 
	\begin{equation} 
		e = \sqrt{2} \chi_\digamma(\tilde{p}[g]) \chi_{\digamma'}(\varrho_{\mathrm{Ff}})\chi_{0,\digamma}(\rho_0) (4\rho_0-F - \tilde{F}\varrho_{\mathrm{Ff}} )^{1/2}
	\end{equation} 
	(for each fixed value of $\digamma$, the function $\chi_{0,\digamma}(\rho_0)$ is supported away from $\rho_0=0$, so, as long as $\digamma'$ is chosen sufficiently large, the function $4\rho_0-F - \tilde{F}\varrho_{\mathrm{Ff}}$ under the square root will be bounded away from zero on the support of the prefactor), 
	\begin{equation}
		\tilde{g} = \sqrt{4s - 2 \varrho_{\mathrm{Ff}} c} \chi_\digamma(\tilde{p}[g])\chi_{0,\digamma'}(\varrho_{\mathrm{Ff}}) \chi_\digamma(\rho_0), 
	\end{equation}
	and $h= 2 \chi_\digamma'(\tilde{p}[g])\chi_\digamma(\tilde{p}[g]) \chi_{\digamma'}(\varrho_{\mathrm{Ff}})\chi_\digamma(\rho_0)\tilde{q}\tilde{p}[g]$ near $\calR^+_+$. 
	
	Quantizing, we get $A = (1/2)(\operatorname{Op}(a)+\operatorname{Op}(a)^*) \in  \Psi_{\mathrm{de,sc}}^{-\infty,(-\infty,-\infty,-\infty,-s_0,-\ell_0)},$
	\begin{align}
		\begin{split} 
			B = \operatorname{Op}(w^{1/2} a_0^{1/2} b) &\in  \Psi_{\mathrm{de,sc}}^{-\infty,(-\infty,-\infty,-\infty,s_{\mathrm{nFf}},s_{\mathrm{Ff}})}, \\
			E = \operatorname{Op}(w^{1/2} a_0^{1/2} e) &\in  \Psi_{\mathrm{de,sc}}^{-\infty,(-\infty,-\infty,-\infty,s_{\mathrm{nFf}},s_{\mathrm{Ff}})}, \\
			\tilde{G} = \operatorname{Op}(w^{1/2} a_0^{1/2} \varrho_{\mathrm{Ff}}^{1/2} \tilde{g}) &\in  \Psi_{\mathrm{de,sc}}^{-\infty,(-\infty,-\infty,-\infty,s_{\mathrm{nFf}},-\infty )}, \\
			H = \operatorname{Op}(wa_0 h) &\in \Psi_{\mathrm{de,sc}}^{-\infty,(-\infty,-\infty,-\infty,2s_{\mathrm{nFf}},2s_{\mathrm{Ff}})},
		\end{split}
	\end{align}
	and $R\in \Psi_{\mathrm{de,sc}}^{-\infty,(-\infty,-\infty,-\infty,2s_{\mathrm{nFf}}-1,2s_{\mathrm{Ff}}-1)}$ such that 
	\begin{equation}
		-i [ P, A] + i (P-P^*) A = \delta A \Lambda^2 A + B^*B +E^*E + \tilde{G}^* \tilde{G}  +H  + R
	\end{equation}
	for 
	\begin{equation} 
		\Lambda = (1/2)( \operatorname{Op}(w^{1/2} a_0^{-1/2})+\operatorname{Op}(w^{1/2} a_0^{-1/2})^*) \in \Psi_{\mathrm{de,sc}}^{-1/2,(-1/2,-1/2,-1/2,-1-s_{\mathrm{nFf}},-1-s_{\mathrm{Ff}})}.
	\end{equation} 
	(Unlike in the estimates in the previous section, the $i(P-P^*) A$ term has the same order as $R$, so we do not need to take it into account in the principal symbolic construction, hence the absence of what we called $p_1$ in the previous section in the discussion above.)
	So, given sufficiently nice $u$, 
	\begin{equation}
		-2 \Im \langle Au , Pu \rangle_{L^2} =  \delta \lVert \Lambda A u \rVert_{L^2}^2 + \lVert Bu \rVert_{L^2}^2 + \lVert Eu  \rVert_{L^2}^2 + \lVert \tilde{G} u \rVert_{L^2}^2   + \langle Hu,u\rangle_{L^2} + \langle Ru,u \rangle_{L^2}. 
	\end{equation}
	
	The rest of the argument proceeds as in the other propagation estimates, except the regularization argument is more delicate, but in a standard way. Indeed, it is only possible to regularize by a finite amount. Consider, for each $\varepsilon,K,K'>0$, the regularizer
	\begin{equation}
		\varphi_{\varepsilon,K,K'} = \Big(1 + \frac{\varepsilon}{\varrho_{\mathrm{Ff}} \varrho_{\mathrm{nFf}}^{K'} }\Big)^{-K}.
	\end{equation}
	Then, using $\varrho_{\mathrm{nFf}}=\varrho_{\mathrm{nf}}$, 
	\begin{equation}
		\mathsf{H}_{p[g]} \varphi_{\varepsilon,K,K'} = - \frac{2K\varepsilon}{\varepsilon + \varrho_{\mathrm{Ff}}\varrho_{\mathrm{nFf}}^{K'} } \varphi_{\varepsilon,K,K'} (s+K'(1-s))
	\end{equation}
	over $\partial \bbO$, in some neighborhood of $\mathrm{nFf}\cap \mathrm{Ff}$. Notice that, at $\calR^+_+$, over $\Omega_{\mathrm{nfTf},+,0}$, we have $s+K'(1-s)=1$. 
	Combining this calculation with the one done over $\mathrm{cl}_\bbO\{r=0\}$ as part of the standard sc-analysis, we can conclude that 
	\begin{equation}
		-4K \leq \frac{2}{\varphi_{\varepsilon,K,K'}} \mathsf{H}_{p[g]} \varphi_{\varepsilon,K,K'} \Big|_{\calR^+_+}. 
	 \label{eq:misc_390}
	\end{equation}
	Define 
	\begin{equation} 
		a_{\varepsilon} = \varphi_{\varepsilon,K,K'}^2 a,
	\end{equation} 
	and likewise for the other symbols above with the exception of $b$, and define
	\begin{equation}
		b_{\varepsilon} = \varphi_{\varepsilon,K,K'} \chi_\digamma(\tilde{p}[g]) \chi_{\digamma'}(\varrho_{\mathrm{Ff}}) \chi_{\digamma}(\rho_0) \sqrt{ \alpha - \delta a_0^{-1} a_\varepsilon  + \frac{2}{\varphi_{\varepsilon,K,K'}} \mathsf{H}_{p[g]} \varphi_{\varepsilon,K,K'} },
	\end{equation}
	assuming that the symbol under the square root is positive on the support of the prefactor, so that this is a well-defined symbol. Since $\alpha = 4s_{\mathrm{Ff}}+2$ at $\calR$, as long as 
	\begin{equation} 
		K < s_{\mathrm{Ff}} + \frac{1}{2}, 
		\label{eq:misc_39b}
	\end{equation} 
	then \cref{eq:misc_390} guarantees that the symbol $b_\varepsilon$ is well-defined \emph{for all} $\varepsilon>0$,  as long as $\digamma$ is sufficiently large and $\delta$ is sufficiently small \emph{relative to} $K$. As long as these conditions are met, instead of \cref{eq:misc_380}, we have 
	\begin{equation}
		\mathsf{H}_{p[g]} a_\varepsilon = (\delta a_0^{-2}a^2_\varepsilon + b^2_\varepsilon  +e^2_\varepsilon  +\varrho_{\mathrm{Ff}}\tilde{g}^2_\varepsilon + h_\varepsilon ) a_0. 
	\end{equation}
	Thus, quantizing, we get operators, all of which are uniform families of de,sc-$\Psi$DOs of the same orders as their non-regularized counterparts, such that 
	\begin{equation}
		-i [ P, A_\varepsilon] + i (P-P^*) A_\varepsilon = \delta A_\varepsilon \Lambda^2 A_\varepsilon + B^*_\varepsilon B_\varepsilon +E^*_\varepsilon E_\varepsilon + \tilde{G}^*_\varepsilon \tilde{G}_\varepsilon  +H_\varepsilon  + R_\varepsilon.
		\label{eq:misc_392}
	\end{equation}
	Given $u$ with $\operatorname{WF}_{\mathrm{de,sc}}^{-N,\mathsf{s}_0}(u)\cap \calR^+_+=\varnothing$, we can take $K,K'$ large enough such that, as long as $\digamma$ is sufficiently large, then we can deduce from \cref{eq:misc_392} that
	\begin{multline}
		\qquad -2 \Im \langle A_\varepsilon u,P u \rangle_{L^2} = \delta \lVert \Lambda A_\varepsilon u\rVert_{L^2}^2 + \lVert B_\varepsilon u \rVert_{L^2}^2 + \lVert E_\varepsilon u \rVert_{L^2}^2 + \lVert \tilde{G}_\varepsilon u \rVert_{L^2}^2  \\ + \langle H_\varepsilon u,u \rangle_{L^2} + \langle R_\varepsilon u ,u \rangle_{L^2}, 
	\end{multline}
	from which the estimate 
	\begin{equation}
		\lVert B_\varepsilon u \rVert_{L^2}^2 + \delta \lVert \Lambda A_\varepsilon u \rVert_{L^2}^2 \leq |\langle H_\varepsilon u,u\rangle_{L^2}| + |\langle R_\varepsilon u,u \rangle_{L^2}|  +2 |\langle A_\varepsilon u,Pu \rangle_{L^2}|
		\label{eq:misc_396}
	\end{equation}
	follows. Indeed, for $\digamma,\digamma'$ sufficiently large: 
	\begin{enumerate}
		\item We have 
		\begin{equation} 
			\Lambda A_{\varepsilon} u,B_{\varepsilon} u,E_{\varepsilon}u,\tilde{G}_{\varepsilon} u \in H_{\mathrm{de,sc}}^{N,(N,N,N,KK'+s_{\mathrm{nFf},0} - s_{\mathrm{nFf}},K+s_{\mathrm{Ff},0} - s_{\mathrm{Ff}})} \subseteq L^2
			\label{eq:misc_397} 
		\end{equation} 
		as long as $K> s_{\mathrm{Ff}} - s_{\mathrm{Ff},0}$ and $KK'$ is sufficiently large, the latter of which we arrange by taking $K'$ large relative to $K$. Since $s_{\mathrm{Ff},0} \in (-1/2,s_{\mathrm{Ff}})$, $s_{\mathrm{Ff}}-s_{\mathrm{Ff},0} < s_{\mathrm{Ff}}+1/2$, so the interval $(s_{\mathrm{Ff}}-s_{\mathrm{Ff},0},s_{\mathrm{Ff}}+1/2 )$ is nonempty, which means there exists $K$ which is large enough but still satisfies \cref{eq:misc_39b}. So, choosing $K,K'$, we can assume that  \cref{eq:misc_397} holds.
		\item $H_\varepsilon u \in H_{\mathrm{de,sc}}^{N,(N,N,N,2KK'+s_{\mathrm{nFf},0} -2 s_{\mathrm{nFf}} ,2K+s_{\mathrm{Ff},0} - 2s_{\mathrm{Ff}} )} $, and, for $N$ sufficiently large, 
		\begin{equation}
			H_{\mathrm{de,sc}}^{N,(N,N,N,KK'+s_{\mathrm{nFf},0} -2 s_{\mathrm{nFf}} ,K+s_{\mathrm{Ff},0} - 2s_{\mathrm{Ff}} )} \subseteq \big(H_{\mathrm{de,sc}}^{-N,\mathsf{s}_0}\big)^*
		\end{equation}
		as long as $K> s_{\mathrm{Ff}} - s_{\mathrm{Ff},0}$ and $KK'$ is sufficiently large, which means that the $\langle H_\varepsilon u,u\rangle_{L^2}$ term is well-defined in the sense of H\"ormander, and likewise for $\langle R_\varepsilon u,u\rangle_{L^2}$.
		\item $\operatorname{WF}_{\mathrm{de,sc}}^{-N-2,\mathsf{s}_0+1}(Pu)\cap \calR^+_+=\varnothing$, and, for $N$ sufficiently large, 
		\begin{equation}
			A_\varepsilon u \in H_{\mathrm{de,sc}}^{N+2,(N,N,N,-1+KK'+s_{\mathrm{nFf},0} - 2s_{\mathrm{nFf}} ,-1+K+s_{\mathrm{Ff},0}-2s_{\mathrm{Ff}}) } \subseteq (\operatorname{WF}_{\mathrm{de,sc}}^{-N-2,\mathsf{s}_0+1})^*,
		\end{equation}
		assuming the conditions above are satisfied. This implies that $\langle A_\varepsilon u,Pu\rangle_{L^2}$ is well-defined in the sense of H\"ormander. 
	\end{enumerate}
	Having \cref{eq:misc_396}, we get, after taking $\varepsilon \to 0^+$, an estimate of the form 
	\begin{multline}
		\lVert \tilde{B} u \rVert_{H_{\mathrm{de,sc}}^{N,(N,N,N,s_{\mathrm{nFf}},s_{\mathrm{Ff}})} }^2 \lesssim  \lVert GP u \rVert_{H_{\mathrm{de,sc}}^{-N,(-N,-N, -N,s_{\mathrm{nFf}}+1,s_{\mathrm{Ff}}+1 )} }^2   \\ + \lVert G u \rVert_{H_{\mathrm{de,sc}}^{-N,(-N,-N,-N,s_{\mathrm{nFf}}-1/2 ,s_{\mathrm{Ff}}-1/2)} }^2
		+ \lVert u \rVert_{H_{\mathrm{de,sc}}^{-N,-N}}^2,
	\end{multline}
	where $\tilde{B}\in\smash{ \Psi_{\mathrm{de,sc}}^{0,\mathsf{0}}}$ is elliptic on $\smash{\calR^+_+}$ and $G \in \smash{\Psi_{\mathrm{de,sc}}^{0,\mathsf{0}}}$ whose essential support can be taken to be arbitrarily close to $\calR^+_+$ by making $\digamma,\digamma'$ larger. In order to make this an estimate in the strong sense, we can simply add a term to the right-hand side:
	\begin{multline}
	\lVert \tilde{B} u \rVert_{H_{\mathrm{de,sc}}^{N,(N,N,N,s_{\mathrm{nFf}},s_{\mathrm{Ff}})} }^2 \lesssim  \lVert GP u \rVert_{H_{\mathrm{de,sc}}^{-N,(-N,-N, -N,s_{\mathrm{nFf}}+1,s_{\mathrm{Sf}}+1 )} }^2   \\ + \lVert G u \rVert_{H_{\mathrm{de,sc}}^{-N,\mathsf{s}_0} }^2 + \lVert G u \rVert_{H_{\mathrm{de,sc}}^{-N,(-N,-N,-N,s_{\mathrm{nFf}}-1/2 ,s_{\mathrm{Ff}}-1/2)} }^2
	+ \lVert u \rVert_{H_{\mathrm{de,sc}}^{-N,-N}}^2. 
	\end{multline}
	The usual inductive argument then allows the removal of the penultimate term, and so we end up with the strong estimate 
	\begin{equation}
		\lVert \tilde{B} u \rVert_{H_{\mathrm{de,sc}}^{N,(N,N,N,s_{\mathrm{nFf}},s_{\mathrm{Ff}})} }^2 \lesssim  \lVert GP u \rVert_{H_{\mathrm{de,sc}}^{-N,(-N,-N, -N,s_{\mathrm{nFf}}+1,s_{\mathrm{Sf}}+1 )} }^2  + \lVert G u \rVert_{H_{\mathrm{de,sc}}^{-N,\mathsf{s}_0} }^2 
		+ \lVert u \rVert_{H_{\mathrm{de,sc}}^{-N,-N}}^2, 
	\end{equation}
	which completes the proof.

	\item On the other hand, if $s_{\mathrm{Tf}}<-1/2$, then $\alpha<0$ near $\calR^+_+$, then quantization yields operators as above, modulo some sign switches in their definitions, such that 
	\begin{equation}
		-i [ P, A] + i (P-P^*) A = -\delta A \Lambda^2 A - B^*B +E^*E +\tilde{G}^* \tilde{G}  +H  + R.
	\end{equation}
	From this, the strong estimate of the form 
	\begin{multline}
		\lVert \tilde{B} u \rVert_{H_{\mathrm{de,sc}}^{N,(N,N,N,s_{\mathrm{nFf}},s_{\mathrm{Ff}})} }^2 \lesssim  \lVert GP u \rVert_{H_{\mathrm{de,sc}}^{-N,(-N,-N, -N,s_{\mathrm{nFf}}+1,s_{\mathrm{Sf}}+1 )} }^2  + 	\lVert \tilde{E} u \rVert_{H_{\mathrm{de,sc}}^{-N,(-N,-N,-N,s_{\mathrm{nFf}},s_{\mathrm{Ff}})} }^2  \\ + 	\lVert \tilde{\tilde{G}} u \rVert_{H_{\mathrm{de,sc}}^{-N,(-N,-N,-N,s_{\mathrm{nFf}},-N)} }^2 
		+ \lVert u \rVert_{H_{\mathrm{de,sc}}^{-N,-N}}^2
	\end{multline}
	follows. The necessary regularization argument is simpler than the previous, as an arbitrarily large amount of regularization can be done. 
\end{itemize}
\subsection{First induction}
\label{subsec:firstinduction}
We now discuss the induction on $k$.

It will be convenient to reduce to the case where $P$ is $L^2$-symmetric. Indeed, $P$ differs $\square$ by an element of $\operatorname{Diff}^{2,-\mathsf{2}}_{\mathrm{de,sc}}$. Consequently, if $G\in \Psi_{\mathrm{de,sc}}^{-\infty,\mathsf{0}}$ has essential support disjoint from $\mathrm{df}$, then, in estimates, 
\begin{equation}
	\lVert G(P-P^*) Au \rVert_{H_{\mathrm{de,sc}}^{m,\mathsf{s}+1} } \lesssim \lVert G_0 Au \rVert_{H_{\mathrm{de,sc}}^{-N,\mathsf{s}-1}} +\lVert u \rVert_{H_{\mathrm{de,sc}}^{-N,-N}}
\end{equation}
for any $u\in \calS'$, $A\in \Psi_{\mathrm{de,sc}}$, and $G_0\in \Psi_{\mathrm{de,sc}}^{-\infty,0}$ elliptic on the essential support of $G$. For each $\varepsilon>0$, as long as we are restricting attention to a sufficiently small neighborhood of $\mathrm{Ff}$, 
we can bound 
\begin{equation}
	\lVert G_0 Au \rVert_{H_{\mathrm{de,sc}}^{-N,\mathsf{s}-1}} \leq \varepsilon \lVert G_0 Au \rVert_{H_{\mathrm{de,sc}}^{-N,\mathsf{s}}}+ C\lVert u \rVert_{H_{\mathrm{de,sc}}^{-N,-N}}
\end{equation}
for some $C\gg 1$, 
as follows from $L^2\to L^2$ bounds on multiplication by boundary-defining-functions. Consequently, the error terms that arise by replacing $P$ with $2^{-1}(P+P^*)$ can be absorbed into the left-hand sides of the desired estimates.
Conceptually, the reason this works is that, at $\calR$, the difference \begin{equation} 
	P-\square \in \operatorname{Diff}^{2,-\mathsf{2}}_{\mathrm{de,sc}}
\end{equation} 
is a full two orders lower than $P$ itself, which means that it does not affect positive commutator arguments. 
(Indeed, we have already remarked upon essentially this fact in \S\ref{subsec:basecase}.)
We could not do this in \S\ref{sec:propagation}, because there the radial sets were all at fiber infinity, and $P-\square $ has the same order as $P$ there.
(Of course, we could have instead worked with the $L^2(\bbR^{1,d},g)$-based inner product, and then a similar reduction to the symmetric case would apply under a symmetry assumption on the last term in \cref{eq:misc_273}.)

So, henceforth, assume that $P=P^*$.

Let $A_0,\ldots,A_N \in \frakN_{+}$ denote a spanning set over $\smash{\Psi_{\mathrm{de,sc}}^{0,\mathsf{0}}}$, with $A_0=1$. We laid out such a set of generators in \S\ref{subsec:test_modules}. 
For each multi-index $\alpha \in \bbN^N$ with $|k|=\alpha$ and tuple $\mathsf{s}\in \bbR^5$, let 
\begin{equation}
	A_{\alpha,\mathsf{s}} =\varrho^{-\mathsf{s}} A_1^{\alpha_1}\cdots A_N^{\alpha_N},
	\label{eq:Aas}
\end{equation}
where 
\begin{equation} 
	\varrho^{\mathsf{s}} = \varrho_{\mathrm{Pf}}^{s_{\mathrm{Pf}}}\varrho_{\mathrm{nPf}}^{s_{\mathrm{nPf}}}\varrho_{\mathrm{Sf}}^{s_{\mathrm{Sf}}}\varrho_{\mathrm{nFf}}^{s_{\mathrm{nFf}}}\varrho_{\mathrm{Ff}}^{s_{\mathrm{Ff}}}.
\end{equation}
Let 
\begin{equation} 
	D_{\mathsf{s}}=i\varrho^{-1} [P, \varrho^{-\mathsf{s}} ] \varrho^{\mathsf{s}} \in\mathrm{Diff}_{\mathrm{de,sc}}^{1,\mathsf{0}}.
\end{equation} 
Then, $\sigma_{\mathrm{de,sc}}^{1,\mathsf{0}} (D_{\mathsf{s}})$, which is proportional to $\mathsf{H}_{p} \varrho^{-\mathsf{s}}$ at $\calR$, satisfies
\begin{equation} 
	\begin{cases}
		\sigma_{\mathrm{de,sc}}^{1,\mathsf{0}} (D_{\mathsf{s}})|_{\calR^+_+} > 0 & (s_{\mathrm{Ff}}>0 ) \\
		\sigma_{\mathrm{de,sc}}^{1,\mathsf{0}} (D_{\mathsf{s}})|_{\calR^+_+} < 0 & (s_{\mathrm{Ff}}<0),
	\end{cases}
	\label{eq:misc_ggg}
\end{equation} 
as follows from \cref{eq:H_nfTf_df}. (The actual sign depends on our sign conventions, but the sign relative to other symbols appearing in the positive commutator estimates is independent of any conventions.)

In order to carry out the construction of the commutant, we recall the following algebraic computation, which is essentially \cite[eq. 6.16]{HassellMelroseVasy}\cite[eq. 3.23]{HassellETAL}. 
\begin{lemma}
	Let $P\in \Psi_{\mathrm{de,sc}}^{2,\mathsf{0}}$ be $L^2$-symmetric, $Q\in \Psi_{\mathrm{de,sc}}^{-\infty,\mathsf{0}}$.
	There exist $E_{\alpha,\mathsf{s}-1/2}=E_{\alpha,\mathsf{s}-1/2}(Q) \in \varrho^{-\mathsf{s}+1/2}\frakN^{k-1}_+$,
	with $\operatorname{WF}_{\mathrm{de,sc}}'(E_{\alpha,\mathsf{s}-1/2}) \subseteq \operatorname{WF}_{\mathrm{de,sc}}'(Q)$ such that 
	\begin{multline}
		i [P, A_{\alpha,\mathsf{s}+1/2}^* Q^* Q A_{\alpha,\mathsf{s}+1/2} ] =2 A_{\alpha,\mathsf{s}+1/2}^* Q^* \varrho^{1/2} \Re \Big[ D_{\mathsf{s}+1/2} + \sum_{j=1}^N \alpha_j C_{j,j} \Big] \varrho^{1/2} Q A_{\alpha,\mathsf{s}+1/2}  \\ + 2 \sum_{|\beta|= k, \beta\neq \alpha} A_{\alpha,\mathsf{s}+1/2}^* Q^* \varrho^{1/2} \Re [ C_{\alpha,\beta} ] \varrho^{1/2} Q A_{\beta,\mathsf{s}+1/2}   \\  + A^*_{\alpha,\mathsf{s}+1/2} Q^*  E_{\alpha,\mathsf{s}-1/2} + E_{\alpha,\mathsf{s}-1/2}^* Q A_{\alpha,\mathsf{s}+1/2}  + A_{\alpha,\mathsf{s}+1/2}^* i[P,Q^*Q]A_{\alpha,\mathsf{s}+1/2}
		\label{eq:misc_jh3}
	\end{multline}
	holds for some $C_{\alpha,\beta} \in \Psi_{\mathrm{de,sc}}^{1,\mathsf{0}}$ satisfying $i \varrho^{-1}[P,A_\alpha] = \sum_{|\beta|\leq k} C_{\alpha,\beta} A_\beta$ and
	\begin{equation} 
		\sigma_{\mathrm{de,sc}}^{0,\mathsf{0}}(C_{\alpha,\beta})|_{\calR_+^+} = 0
		\label{eq:misc_smg}
	\end{equation}
	for all multi-indices $\alpha,\beta\in \bbN^N$ with $|\alpha|=|\beta|=k$. 
\end{lemma}
Here, for $j,k\in \{1,\ldots,N\}$, $C_{j,k}$ is as in \Cref{prop:criticality}. When $\alpha,\beta$ are multi-indices that are zero except in the $j$th and $k$th slots respectively, where they are one, then we can take $C_{\alpha,\beta} = C_{j,k}$.

The proof of this lemma will be the most technical part of this paper. The precise form of the right-hand side in \cref{eq:misc_jh3} is convenient, but most important is its structure:
\begin{multline}
	2 A_{\alpha,\mathsf{s}+1/2}^* Q^* \varrho^{1/2} \Re \Big[ \underbrace{D_{\mathsf{s}+1/2}}_{\text{main term}} + \underbrace{\sum_{j=1}^N \alpha_j C_{j,j}}_{\text{subprinciple at $\calR$, by Prop.\ \ref{prop:criticality}}} \Big] \varrho^{1/2} Q A_{\alpha,\mathsf{s}+1/2}  \\ + \underbrace{2 \sum_{|\beta|= k, \beta\neq \alpha} A_{\alpha,\mathsf{s}+1/2}^* Q^* \varrho^{1/2} \Re [ C_{\alpha,\beta} ] \varrho^{1/2} Q A_{\beta,\mathsf{s}+1/2}}_{{\text{subprinciple at $\calR$, by \cref{eq:misc_smg}}}}   \\  + \underbrace{A^*_{\alpha,\mathsf{s}+1/2} Q^*  E_{\alpha,\mathsf{s}-1/2} + E_{\alpha,\mathsf{s}-1/2}^* Q A_{\alpha,\mathsf{s}+1/2}}_{\text{lower order error}}  + \underbrace{A_{\alpha,\mathsf{s}+1/2}^* i[P,Q^*Q]A_{\alpha,\mathsf{s}+1/2}}_{\text{microsupported away from $\calR$}},
\end{multline}
where the ``lower order error'' is lower order because $E_{\alpha,\mathsf{s}-1/2}\in \varrho^{-\mathsf{s}+1/2} \frakN_+^{k-1}$, whereas 
\begin{equation} 
	A_{\alpha,\mathsf{s}+1/2}\in \varrho^{-\mathsf{s}-1/2} \frakN_+^k,
\end{equation} 
so these terms, when applied to $u$, will be controlled by our inductive hypothesis.
The ``main term'' $D_{\mathsf{s}+1/2}$ is the one that comes from differentiating the weight $\varrho$, which is what one expects when proving radial-point estimates.

\begin{proof}
	We show that given any collection of $C_{\alpha,\beta}\in \Psi_{\mathrm{de,sc}}^{1,\mathsf{0}}$ satisfying
	\begin{equation} 
		i \varrho^{-1}[P,A_\alpha] = \sum_{|\beta|\leq |\alpha|} C_{\alpha,\beta} A_\beta
		\label{eq:misc_hj4}
	\end{equation}  
	(the existence of at least one such collection follows from \Cref{prop:module_commutator_lemma}),
	there exists a collection of $E_{\alpha,\mathsf{s}-1/2}^{(0)} \in \varrho^{-\mathsf{s}+1/2} \frakN^{k-1}$ with $\operatorname{WF}_{\mathrm{de,sc}}'(E_{\alpha,\mathsf{s}-1/2}^{(0)}) \subseteq \operatorname{WF}_{\mathrm{de,sc}}'(Q)$ such that
	\begin{multline}
		i [P, A_{\alpha,\mathsf{s}+1/2}^* Q^* Q A_{\alpha,\mathsf{s}+1/2} ] = 2 \sum_{|\beta| \leq k} A_{\alpha,\mathsf{s}+1/2}^* Q^* \varrho^{1/2}  \Re \Big[ \delta_{\alpha,\beta} D_{\mathsf{s}+1/2} + C_{\alpha,\beta} \Big] \varrho^{1/2} Q A_{\beta,\mathsf{s}+1/2}  \\ + A^*_{\alpha,\mathsf{s}+1/2} Q^*  E_{\alpha,\mathsf{s}-1/2}^{(0)} + E_{\alpha,\mathsf{s}-1/2}^{(0)*} Q A_{\alpha,\mathsf{s}+1/2}  + A_{\alpha,\mathsf{s}+1/2}^* i[P,Q^*Q]A_{\alpha,\mathsf{s}+1/2}
		\label{eq:misc_jh1}
	\end{multline}
	holds for each $\alpha\in \bbN^N$. We then show that we can choose $C_{\alpha,\beta}$ such that, if $|\beta| = |\alpha|$,
	\begin{equation}
		C_{\alpha,\beta} = 
		\begin{cases}
			\sum_{j=1}^N \alpha_j C_{j,j}  & (\alpha=\beta), \\
			\alpha_\delta C_{j,\nu} & (|\alpha-\beta|=2), \\ 
			 0 & (\text{otherwise}),
		\end{cases}
	\label{eq:misc_cab}
	\end{equation} 
	where in the second case $j,\nu$ are the indices in which $\alpha,\beta$ differ, with $\alpha_j = \beta_j+1$ and $\beta_\nu = \alpha_\nu +1$. 
	Defining 
	\begin{equation}
		E_{\alpha,\mathsf{s}-1/2} = E_{\alpha,\mathsf{s}-1/2}^{(0)}  +   \varrho^{1/2} \sum_{|\beta|<k}  \Re[C_{\alpha,\beta}] \varrho^{1/2} Q A_{\beta,\mathsf{s}+1/2} \in \varrho^{-\mathsf{s}+1/2} \frakN^{k-1}_+, 
	\end{equation} 
	\cref{eq:misc_jh3} holds, and by \Cref{prop:criticality}, the $C_{\alpha,\beta}$ defined by \cref{eq:misc_cab} satisfy \cref{eq:misc_smg}.
	\begin{itemize}
		\item Suppose that we are given $C_{\alpha,\beta}$ satisfying \cref{eq:misc_hj4}. The left-hand side of \cref{eq:misc_jh1} is given by 
		\begin{equation} 
			i[P,A_{\alpha,\mathsf{s}+1/2}^* Q^* Q A_{\alpha,\mathsf{s}+1/2} ] =  A_{\alpha,\mathsf{s}+1/2}^* Q^* i [P, QA_{\alpha,\mathsf{s}+1/2}] + [P,A_{\alpha,\mathsf{s}+1/2}^* Q^*] Q A. 
		\end{equation}

		Consider $i[P, A_{\alpha,\mathsf{s}+1/2}] = i\varrho^{-\mathsf{s}-1/2} [P,A_\alpha] + i[P, \varrho^{-\mathsf{s}-1/2}] A_\alpha$. Using \cref{eq:misc_hj4}, this becomes
		\begin{align}
			\begin{split} 
				i[P, A_{\alpha,\mathsf{s}+1/2}] &=   \sum_{|\beta|\leq |\alpha|}\varrho^{-\mathsf{s}+1/2}C_{\alpha,\beta} A_{\beta} + D_{\mathsf{s}+1/2}  A_{\alpha,\mathsf{s}-1/2} \\ 
			&=  \sum_{|\beta|\leq|\alpha|}C_{\alpha,\beta} A_{\beta,\mathsf{s}-1/2} + \sum_{|\beta|\leq|\alpha|}[\varrho^{-\mathsf{s}+1/2},C_{\alpha,\beta} ]A_\beta + D_{\mathsf{s}+1/2}  A_{\alpha,\mathsf{s}-1/2}.
			\end{split}
		\end{align} 
		Therefore, $i[P, QA_{\alpha,\mathsf{s}+1/2}]  = 	Qi[P, A_{\alpha,\mathsf{s}+1/2}]  + 	i[P, Q]A_{\alpha,\mathsf{s}+1/2} $ can be written, after rearrangement of the terms in $Qi[P, A_{\alpha,\mathsf{s}+1/2}]$, as 
		\begin{multline}
			i[P, QA_{\alpha,\mathsf{s}+1/2}] =   \varrho D_{\mathsf{s}+1/2} Q A_{\alpha,\mathsf{s}+1/2}+ 
			\varrho \sum_{|\beta|\leq |\alpha|}C_{\alpha,\beta} Q A_{\beta,\mathsf{s}+1/2}
			+ [D_{\mathsf{s}+1/2}Q, \varrho] A_{\alpha,\mathsf{s}+1/2}  
			\\ + \sum_{|\beta|\leq |\alpha|} [C_{\alpha,\beta}Q,\varrho]A_{\beta,\mathsf{s}+1/2} + \sum_{|\beta|\leq |\alpha|}\Big( [Q,C_{\alpha,\beta}] + Q [\varrho^{-\mathsf{s}+1/2},C_{\alpha,\beta} ] \varrho^{\mathsf{s}-1/2} \Big) A_{\beta,\mathsf{s}-1/2}  \\ + [Q,D_{\mathsf{s}+1/2}] A_{\alpha,\mathsf{s}-1/2}  + i [P,Q] A_{\alpha,\mathsf{s}+1/2}, 
		\end{multline}
		and similarly for $[P,A_{\alpha,\mathsf{s}+1/2}^*] = -[P,A_{\alpha,\mathsf{s}+1/2}]^*$.
		So, the operator 
		\begin{equation} 
			i[P,A_{\alpha,\mathsf{s}+1/2}^* Q^* Q A_{\alpha,\mathsf{s}+1/2} ] =  A_{\alpha,\mathsf{s}+1/2}^* Q^* i [P, QA_{\alpha,\mathsf{s}+1/2}] + [P,A_{\alpha,\mathsf{s}+1/2}^* Q^*] Q A
			\label{eq:misc_514}
		\end{equation} 
		is given by 
		\begin{multline}
			i[P,A_{\alpha,\mathsf{s}+1/2}^* Q^* Q A_{\alpha,\mathsf{s}+1/2} ] = A_{\alpha,\mathsf{s}+1/2}^* Q^*  \Big( \varrho D_{\mathsf{s}+1/2} + D_{\mathsf{s}+1/2}^* \varrho + \sum_{|\beta|\leq|\alpha|} (\varrho C_{\alpha,\beta} + C_{\alpha,\beta}^* \varrho) \Big) Q A_{\beta,\mathsf{s}+1/2}
			\\ + A^*_{\alpha,\mathsf{s}+1/2} Q^*  E_{\alpha,\mathsf{s}-1/2}^{(1)} + E_{\alpha,\mathsf{s}-1/2}^{(1)*} Q A_{\alpha,\mathsf{s}+1/2}  + A_{\alpha,\mathsf{s}+1/2}^* i[P,Q^*Q]A_{\alpha,\mathsf{s}+1/2}
			\label{eq:misc_7k1}
		\end{multline}
		for 
		\begin{multline}
			E_{\alpha,\mathsf{s}-1/2}^{(1)} =  [D_{\mathsf{s}+1/2}Q, \varrho] A_{\alpha,\mathsf{s}+1/2}  
			+ \sum_{|\beta|\leq |\alpha|}[C_{\alpha,\beta}Q,\varrho]  A_{\beta,\mathsf{s}+1/2}  \\ +   \sum_{|\beta|\leq|\alpha|}\Big( [Q,C_{\alpha,\beta}] + Q [\varrho^{-\mathsf{s}+1/2},C_{\alpha,\beta} ] \varrho^{\mathsf{s}-1/2} \Big) A_{\beta,\mathsf{s}-1/2}    + [Q,D_{\mathsf{s}+1/2}] A_{\alpha,\mathsf{s}-1/2}.
			\label{eq:misc_geq}
		\end{multline}
		In passing from \cref{eq:misc_514} to \cref{eq:misc_7k1}, we have recombined 
		\begin{equation}
			A^*_{\alpha,\mathsf{s}+1/2} Q^* i[P,Q] A_{\alpha,\mathsf{s}+1/2}+A^*_{\alpha,\mathsf{s}+1/2} i[P,Q^*]Q A_{\alpha,\mathsf{s}+1/2} = A_{\alpha,\mathsf{s}+1/2}^* i[P,Q^*Q]A_{\alpha,\mathsf{s}+1/2}. 
		\end{equation}
	
		Term by term, we see that $E_{\alpha,\mathsf{s}-1/2}^{(1)} \in \Psi^{-\infty,\mathsf{s}-1/2}_{\mathrm{de,sc}} \frakN_+^{k-1}$. 
		For example, $[D_{\mathsf{s}+1/2}Q,\varrho] \in \Psi_{\mathrm{de,sc}}^{-\infty,-\mathsf{2}}$, so 
		\begin{equation} 
			[D_{\mathsf{s}+1/2}Q, \varrho] A_{\alpha,\mathsf{s}+1/2} = [D_{\mathsf{s}+1/2}Q, \varrho] \varrho^{-\mathsf{s}-1/2} A_\alpha
		\end{equation} 
		is the product of an element of $\Psi_{\mathrm{de,sc}}^{-\infty,-3/2+\mathsf{s}}$ and an element of $\frakN_+^{k}$. Since $\Psi_{\mathrm{de,sc}}^{-\infty,-1} \frakN_+^k \subseteq \frakN_+^{k-1}$, we get 
		\begin{equation} 
			[D_{\mathsf{s}+1/2}Q, \varrho] A_{\alpha,\mathsf{s}-1/2} \in \Psi^{-\infty,\mathsf{s}-1/2}_{\mathrm{de,sc}} \frakN_+^{k-1} .
		\end{equation}   
		The other terms in \cref{eq:misc_geq} are analyzed similarly. 
		
		\Cref{eq:misc_7k1} looks very similar to the desired \cref{eq:misc_jh1}, except we want to commute a factor of $\varrho^{1/2}$ through each of the $D_{\mathsf{s}+1/2}$'s and $C_{\alpha,\beta}$'s. Set
		\begin{equation}
			E_{\alpha,\mathsf{s}-1/2}^{(0)} = E_{\alpha,\mathsf{s}-1/2}^{(1)}  + \varrho^{1/2}\Big( [\varrho^{1/2}, D_{\mathsf{s}+1/2}] + \sum_{|\beta|\leq |\alpha|} [\varrho^{1/2},C_{\alpha,\beta}] \Big)Q A_{\alpha,\mathsf{s}+1/2}. 
		\end{equation}
		This lies in $\Psi^{-\infty,\mathsf{s}-1/2}_{\mathrm{de,sc}} \frakN_+^{k-1}$, and \cref{eq:misc_7k1} becomes \cref{eq:misc_jh1}. Observe that 
		\begin{equation} 
			\operatorname{WF}_{\mathrm{de,sc}}'(E_{\alpha,\mathsf{s}-1/2}^{(0)}) \subseteq \operatorname{WF}_{\mathrm{de,sc}}'(Q),
		\end{equation} 
		so we have accomplished our first task. 
		\item Our second task, arranging \cref{eq:misc_cab}, is mostly a computation of the left-hand side of \cref{eq:misc_hj4} (and the reason why ``arrange'' is required is that there is some redundancy in the right-hand side, so we have some freedom to choose the $C_{\alpha,\beta}$'s). 
		We compute
		\begin{align}
			\begin{split} 
			i\varrho^{-1} [P, A_\alpha] &= \varrho^{-1}\sum_{j=1}^N \Big[ \Big( \prod_{\ell=1}^{j-1} A_\ell^{\alpha_\ell} \Big) \Big(\sum_{\varkappa=1}^{\alpha_j} A_j^{\varkappa-1}i[P,A_j] A_j^{\alpha_j-\varkappa} \Big) \Big( \prod_{\ell=j+1}^{N} A_\ell^{\alpha_\ell} \Big) \Big] \\ 
			&= \varrho^{-1}\sum_{j=1}^N \Big[ \Big( \prod_{\ell=1}^{j-1} A_\ell^{\alpha_\ell} \Big) \Big(\sum_{\varkappa=1}^{\alpha_j} A_j^{\varkappa-1} \varrho \sum_{\nu=0}^N C_{j,\nu} A_\nu A_j^{\alpha_j-\varkappa} \Big) \Big( \prod_{\ell=j+1}^{N} A_\ell^{\alpha_\ell} \Big) \Big] \\
			&= \sum_{j=1}^N \sum_{\varkappa=1}^{\alpha_j} \sum_{\nu=0}^N  \Big[ \varrho^{-1}\Big( \prod_{\ell=1}^{j-1} A_\ell^{\alpha_\ell} \Big) \Big( A_j^{\varkappa-1} \varrho C_{j,\nu} A_\nu A_j^{\alpha_j-\varkappa} \Big) \Big( \prod_{\ell=j+1}^{N} A_\ell^{\alpha_\ell} \Big) \Big].
			\end{split}
		\end{align} 
		Consider the summand, $\varrho^{-1}( \prod_{\ell=1}^{j-1} A_\ell^{\alpha_\ell} ) ( A_j^{\varkappa-1} \varrho C_{j,\nu} A_\nu A_j^{\alpha_j-\varkappa} )(\prod_{\ell=j+1}^{N} A_\ell^{\alpha_\ell} )$. 
		In the case $\nu=j$,  the commutator resulting from commuting the term $\varrho C_{j,\nu}=\varrho C_{j,j}$ all the way to the left yields 
		\begin{equation}
			\varrho^{-1} \Big( \prod_{\ell=1}^{j-1} A_\ell^{\alpha_\ell} \Big) \Big( A_j^{\varkappa-1} \varrho C_{j,j}  A_j^{\alpha_j-\varkappa+1} \Big) \Big( \prod_{\ell=j+1}^{N} A_\ell^{\alpha_\ell} \Big) - C_{j,j} A_\alpha  \in  \Psi^{1,\mathsf{0}}_{\mathrm{de,sc}} \frakN_+^{k-1}. 
		\end{equation}
		The reason why this commutator lies in $\Psi^{1,\mathsf{0}}_{\mathrm{de,sc}} \frakN_+^{k-1}$ is that, each time we commute 
		\begin{equation} 
			\varrho C_{j,j} \in \Psi_{\mathrm{de,sc}}^{1,-\mathsf{1}}
		\end{equation} 
		through an $A_\bullet\in \Psi_{\mathrm{de,sc}}^{1,\mathsf{1}}$, we  get an element of $\Psi_{\mathrm{de,sc}}^{1,\mathsf{0}}$ sandwiched between a total of $k-1$ of the various $A_\bullet$'s. \Cref{prop:module_commutator_lemma} then allows us to bring the $\Psi_{\mathrm{de,sc}}^{1,\mathsf{0}}$ to the front.
		
		A similar computation applies when $\nu\notin \{j,0\}$, with the result 
		\begin{equation}
			\varrho^{-1} \Big( \prod_{\ell=1}^{j-1} A_\ell^{\alpha_\ell} \Big) \Big( A_j^{\varkappa-1} \varrho C_{j,\nu}  A_\nu A_j^{\alpha_j-\varkappa} \Big) \Big( \prod_{\ell=j+1}^{N} A_\ell^{\alpha_\ell} \Big) - C_{j,\nu} A_\beta  \in  \Psi^{1,\mathsf{0}}_{\mathrm{de,sc}} \frakN_+^{k-1},
		\end{equation}
		where $\beta$ differs from $\alpha$ by decrementing the $j$th entry and incrementing the $\nu$th. Indeed, we can first commute the $\varrho C_{j,\nu}$ factor to the left, picking up an error in $\Psi^{1,\mathsf{0}}_{\mathrm{de,sc}} \frakN_+^{k-1}$ as before. Then, we can commute the $A_\nu$ through the $A_j$'s and the other $A_\ell$'s until it is together with the other factors of $A_\nu$. Since $\frakN_+$ is closed under commutators, in this way we end up with another $\Psi^{1,\mathsf{0}}_{\mathrm{de,sc}} \frakN_+^{k-1}$ error.
		
		For the remaining case, $\nu=0$, we simply use that $A_\nu=A_0=1$, so
		\begin{equation}
				\varrho^{-1} \Big( \prod_{\ell=1}^{j-1} A_\ell^{\alpha_\ell} \Big) \Big( A_j^{\varkappa-1} \varrho C_{j,0}   A_j^{\alpha_j-\varkappa} \Big) \Big( \prod_{\ell=j+1}^{N} A_\ell^{\alpha_\ell} \Big) \in  \Psi^{1,\mathsf{0}}_{\mathrm{de,sc}} \frakN_+^{k-1}.
		\end{equation}
		So, 
		\begin{equation}
			i\varrho^{-1}[P, A_\alpha] = \sum_{\beta\in \bbN^N, |\beta|=k} C_{\alpha,\beta} A_\beta \bmod \Psi_{\mathrm{de,sc}}^{1,\mathsf{0}} \frakN^{k-1}_\sigma, 
			\label{eq:misc_54h}
		\end{equation}
		where $C_{\alpha,\beta}$ are defined for $|\alpha|=|\beta|=k$ by \cref{eq:misc_cab}. 
		
		Since the $A_\alpha$'s for $|\alpha|<k$ span $\frakN_+^{k-1}$ over $\Psi_{\mathrm{de,sc}}^{0,\mathsf{0}}$, the $\Psi_{\mathrm{de,sc}}^{1,\mathsf{0}} \frakN^{k-1}_+$ error term in \cref{eq:misc_54h} can be written as 
		\begin{equation} 
			\sum_{\beta\in \bbN^N, |\beta|<k} C_{\alpha,\beta} A_\beta
		\end{equation} 
		for some $\{C_{\alpha,\beta}\}_{|\beta|<k} \subseteq \Psi_{\mathrm{de,sc}}^{1,\mathsf{0}}$. So, the $\{C_{\alpha,\beta}\}_{|\alpha|,|\beta|\leq k}$ defined here satisfy \cref{eq:misc_hj4} on the nose, and they satisfy  \cref{eq:misc_cab} when $|\alpha|=|\beta|=k$. This completes the second part of the proof. 
		\end{itemize} 
\end{proof}

We now return to the main line of argument. Let $k\in \bbN^+$, still taking $\kappa=0$.

We consider 
\begin{equation}
	\{C'_{\alpha,\beta} \}_{|\alpha|,|\beta| = k}=  \{\delta_{\alpha,\beta}(D_{\mathsf{s}+1/2}+D_{\mathsf{s}+1/2}^*) + C_{\alpha,\beta}+C_{\alpha,\beta}^* \}_{|\alpha|,|\beta| = k}
\end{equation}
as a matrix-valued de,sc-$\Psi$DO $C'$ whose matrix elements are indexed by multiindices $\alpha,\beta \in \bbN^N$ with $|\alpha|=k$ and $|\beta|=k$. The matrix-valued principal symbol of $C'$ is the matrix 
\begin{equation}
	\sigma_{\mathrm{de,sc}}^{1,\mathsf{0}}(C') = \{ \delta_{\alpha,\beta} 2\Re \sigma_{\mathrm{de,sc}}^{1,\mathsf{0}} (D_{\mathsf{s}+1/2}) + 2 \Re \sigma_{\mathrm{de,sc}}^{1,\mathsf{0}}(C_{\alpha,\beta})  \}_{|\alpha|,|\beta|=k}. 
\end{equation}
Choosing representatives of $\smash{\sigma_{\mathrm{de,sc}}^{1,\mathsf{0}}(D_{\mathsf{s}+1/2})}$ and $\smash{\sigma_{\mathrm{de,sc}}^{1,\mathsf{0}}(C_{\alpha,\beta})}$, for each $\alpha,\beta$, we get a representative $c'$ of $\smash{\sigma_{\mathrm{de,sc}}^{1,\mathsf{0}}(C')}$, which assigns to each point of $\smash{{}^{\mathrm{de,sc}}T^* \bbO}$ an ordinary matrix.
As long as $s_{\mathrm{Ff}}\neq -1/2$, this matrix is -- owing to \cref{eq:misc_ggg}, \cref{eq:misc_smg} -- either positive definite or negative definite near $\calR^+_+$, so 
\begin{equation}
	b = |c'|^{1/2} 
	\label{eq:misc_bh4}
\end{equation}
is, near $\smash{\calR^+_+}$, a well-defined symmetric matrix whose entries are elements of $\smash{S_{\mathrm{de,sc}}^{1/2,\mathsf{0}}}$ defined near the radial set.
Let $\smash{\{b_{\alpha,\beta}\}_{|\alpha|,|\beta|=k }}$ denote the entries of $b$. Squaring \cref{eq:misc_bh4}, we see that, for each $\alpha,\beta$, 
\begin{equation}
	c_{\alpha,\beta}' =\pm  \sum_{|\gamma|=k} b_{\alpha,\gamma}b_{\gamma,\beta}, 
\end{equation}
where the sign is positive if $s_{\mathrm{Ff}}>-1/2$ and negative otherwise.

Quantizing (and remembering that the discussion above is only valid near the radial set), there exist 
\begin{equation} 
	B_{\alpha,\beta} =B_{\beta,\alpha} \in \Psi_{\mathrm{de,sc}}^{-\infty,\mathsf{0}},
\end{equation} 
$R_{\alpha,\beta}\in \Psi_{\mathrm{de,sc}}^{-\infty,-\mathsf{1}}$, and $E\in \Psi_{\mathrm{de,sc}}^{-\infty,-\infty}$ such that 
\begin{equation}
	Q^* \varrho^{1/2} C'_{\alpha,\beta} \varrho^{1/2} Q =  Q^* \varrho^{1/2}\Big[ \pm \sum_{|\gamma|=k} B^*_{\alpha,\gamma} B_{\gamma,\beta} + R_{\alpha,\beta} \Big] \varrho^{1/2} Q + E,
\end{equation}
with
\begin{equation} 
	\sigma_{\mathrm{de,sc}}^{0,\mathsf{0}}(B_{\alpha,\beta})=b_{\alpha,\beta}
\end{equation} 
near $\calR^+_+$, at least if $Q$ has essential support in a sufficiently small neighborhood of the radial set (so that the discussion above is valid within it). Moreover, since $b$ is invertible (as it is strictly definite and not just semidefinite) near $\calR^+_+$, there exist $\Upsilon_{\alpha,\beta} \in \smash{\Psi_{\mathrm{de,sc}}^{-\infty,\mathsf{0}}}$ such that 
\begin{align}
	Q^* \varrho^{1/2} \Big(\sum_{|\gamma| = k} \Upsilon_{\alpha,\gamma} B_{\gamma,\beta} - \delta_{\alpha,\beta} \Big) \varrho^{1/2} Q &\in \Psi_{\mathrm{de,sc}}^{-\infty,-\infty} \\ 
		Q^* \varrho^{1/2} \Big(\sum_{|\gamma| = k} B_{\alpha,\gamma} \Upsilon_{\gamma,\beta} - \delta_{\alpha,\beta}\Big) \varrho^{1/2} Q &\in \Psi^{-\infty,-\infty}_{\mathrm{de,sc}} 
\end{align} 
for each $\alpha,\beta$, where $\delta_{\bullet,\bullet}$ denotes the Kronecker $\delta$ (once again, as long as the essential support of $Q$ is sufficiently close to $\calR^+_+$). (That we can arrange for the errors above to be residual rather than merely one uniform order better is an instance of the iterative parametrix construction.)

Now consider $u\in \calS'$ as in the setup of the proposition. Assuming we can justify the algebraic manipulations:
\begin{multline}
	\sum_{|\alpha| = k} \langle u, i [P, A_{\alpha,\mathsf{s}+1/2}^* Q^* Q A_{\alpha,\mathsf{s}+1/2} ] u \rangle   \\ = \Big[ \sum_{|\alpha|,|\beta|=k} \Big\langle u, A_{\alpha,\mathsf{s}+1/2}^* Q^* \varrho^{1/2} \Big(\pm \sum_{|\gamma|=k}B^*_{\alpha,\gamma} B_{\gamma,\beta} + R_{\alpha,\beta}\Big) \varrho^{1/2} Q A_{\beta,\mathsf{s}+1/2} u\Big\rangle  \\ 
	+ \langle \varrho^{1/2} Q^* A_{\alpha,\mathsf{s}+1/2}^* u, E_{\alpha,\mathsf{s}} u \rangle + \langle E_{\alpha,\mathsf{s}}u, \varrho^{1/2} Q A_{\alpha,\mathsf{s}+1/2} u \rangle +\langle u,  A_{\alpha,\mathsf{s}+1/2}^* i[P,Q^*Q]A_{\alpha,\mathsf{s}+1/2} u \rangle \Big]. 
	\label{eq:misc_nn3}
\end{multline} 
The main term is 
\begin{equation}
	\pm \sum_{|\alpha|,|\beta|,|\gamma|=k} \langle B_{\gamma,\alpha}\varrho^{1/2} Q A_{\alpha,\mathsf{s}+1/2} u , B_{\gamma,\beta} \varrho^{1/2} Q A_{\alpha,\mathsf{s}+1/2} u \rangle  = \pm \sum_{|\gamma|=k} \Big\lVert \sum_{|\alpha|=k} B_{\gamma,\alpha}\varrho^{1/2} Q A_{\alpha,\mathsf{s}+1/2} u \Big\rVert^2_{L^2}. 
	\label{eq:misc_j75}
\end{equation}
Abbreviate this as $\lVert B \varrho^{1/2} Q A_{\alpha,\mathsf{s}+1/2} u \rVert^2_{L^2}$. Thus, \cref{eq:misc_nn3} yields 
\begin{multline}
	\lVert B \varrho^{1/2} Q A_{\alpha,\mathsf{s}+1/2} u \rVert^2_{L^2} \leq  \sum_{|\alpha|=k} \Big[|\langle u,  A_{\alpha,\mathsf{s}+1/2}^* i[P,Q^*Q]A_{\alpha,\mathsf{s}+1/2} u \rangle|  +2| \langle E_{\alpha,\mathsf{s}} u, \varrho^{1/2} Q A_{\alpha,\mathsf{s}+1/2} u\rangle|  \\ 
	+ 2|\langle  Q A_{\alpha,\mathsf{s}+1/2}Pu , Q A_{\alpha,\mathsf{s}+1/2} u \rangle|  \Big] 
	+ \sum_{|\alpha|,|\beta|=k} |\langle \varrho^{1/2} Q A_{\alpha,\mathsf{s}+1/2} u, R_{\alpha,\beta}\varrho^{1/2} Q A_{\beta,\mathsf{s}+1/2} u \rangle |. \label{eq:misc_j76}
\end{multline} 

If the right-hand side of \cref{eq:misc_j75} is finite, i.e.\ if 
\begin{equation}
	\sum_{|\alpha|=k} B_{\gamma,\alpha}\varrho^{1/2} Q A_{\alpha,\mathsf{s}+1/2} u \in L^2, 
	\label{eq:misc_g4z}
\end{equation}
then, applying $\Upsilon$, we conclude that  
\begin{equation}
	\varrho^{1/2} Q A_{\alpha,\mathsf{s}+1/2} u=\varrho^{1/2} Q \varrho^{-\mathsf{s}-1/2} A_{\alpha} u \in H_{\mathrm{de,sc}}^{\infty,\mathsf{0}}.
\end{equation}
As long as $Q$ is elliptic at $\calR^+_+$, we conclude that $\operatorname{WF}_{\mathrm{de,sc}}^{*,\mathsf{s}}(A_\alpha u) \cap \calR^+_+=\varnothing$. Quantitatively, this means that the $\smash{H_{\mathrm{de,sc}}^{-N,\mathsf{s}}}$ norm of $A_\alpha u$ near $\calR^+_+$ is controlled by the inequality 
\begin{equation}
	\lVert Q_1 A_\alpha u  \rVert_{H_{\mathrm{de,sc}}^{-N,\mathsf{s}}} \lesssim
	\lVert B \varrho^{1/2} Q A_{\alpha,\mathsf{s}+1/2} u \rVert_{L^2} + \lVert u \rVert_{H_{\mathrm{de,sc}}^{-N,-N}}
\end{equation}
for some $Q_1 \in \Psi_{\mathrm{de,sc}}^{0,\mathsf{0}}$ that is elliptic at $\calR^+_+$, this holding for each $N$ and $\mathsf{s}\in \bbR^5$, and for every $u\in \calS'$.

Each of the terms on the right-hand side of \cref{eq:misc_j76} can be controlled using the hypotheses of the propositions and the inductive hypothesis:
\begin{itemize}
	\item First of all assuming that $1-Q$ has essential support away from $\calR^+_+$, $[P,Q^* Q] \in \smash{\Psi^{-\infty,-\mathsf{1}}_{\mathrm{de,sc}}}$ has essential support which is disjoint from $\smash{\calR^+_+}$ as well. As long as $Q$ has essential support sufficiently close to $\calR^+_+$, the hypotheses of either \Cref{thm:R1}, \Cref{thm:R2} imply that 
	\begin{equation}
		\operatorname{WF}_{\mathrm{de,sc}}^{*,\mathsf{s}}(A_{\alpha} u) \cap \operatorname{WF}'_{\mathrm{de,sc}}([P,Q^* Q]) = \varnothing.
	\end{equation}
	Thus, the $\smash{|\langle u,  A_{\alpha,\mathsf{s}+1/2}^* i[P,Q^*Q]A_{\alpha,\mathsf{s}+1/2} u \rangle}$ term in \cref{eq:misc_j76} is finite and can be quantitatively controlled by the $H_{\mathrm{de,sc}}^{-N,\mathsf{s}}$ norms of $A_\alpha u$ in an annular region around $\smash{\calR^+_+}$. 
	\item Now, letting $\tilde{Q}\in \Psi_{\mathrm{de,sc}}^{-\infty,\mathsf{0}}$ be such that $\operatorname{WF}'_{\mathrm{de,sc}}(1-\tilde{Q})\cap \operatorname{WF}'_{\mathrm{de,sc}}(Q)=\varnothing$,
	\begin{equation}
			| \langle E_{\alpha,\mathsf{s}} u, \varrho^{1/2} Q A_{\alpha,\mathsf{s}+1/2} u\rangle| \leq | \langle (1-\tilde{Q}) E_{\alpha,\mathsf{s}} u, \varrho^{1/2} Q A_{\alpha,\mathsf{s}+1/2} u\rangle| + | \langle \tilde{Q} E_{\alpha,\mathsf{s}} u, \varrho^{1/2} Q A_{\alpha,\mathsf{s}+1/2} u\rangle|.
	\end{equation}
	The first term on the right-hand side is straightforward to estimate, as \begin{equation} 
		E_{\alpha,\mathsf{s}}(1-\tilde{Q})^* \varrho^{1/2} Q A_{\alpha,\mathsf{s}+1/2} \in \Psi_{\mathrm{de,sc}}^{-\infty,-\infty}.
	\end{equation} 
	On the other hand,  by Cauchy--Schwarz and AM-GM, 
	\begin{align}
		2| \langle \tilde{Q} E_{\alpha,\mathsf{s}} u, \varrho^{1/2} Q A_{\alpha,\mathsf{s}+1/2} u\rangle| \leq  \epsilon^{-1} \lVert \tilde{Q} E_{\alpha,\mathsf{s}} u  \rVert_{L^2}^2 + \epsilon  \lVert\varrho^{1/2} Q A^*_{\alpha,\mathsf{s}+1/2} u  \rVert_{L^2}^2
	\end{align}
	for any $\epsilon>0$. The term $ \lVert \tilde{Q} E_{\alpha,\mathsf{s}} u  \rVert_{L^2}^2$ can be controlled by the inductive hypothesis, since $E_{\alpha,\mathsf{s}} \in \smash{\Psi_{\mathrm{de,sc}}^{-\infty,-\mathsf{s}} \frakN_+^{k-1}}$. 
	
	The other term, which is controlled in terms of the $H^{-N,\mathsf{s}}_{\mathrm{de,sc}}$ norms of $A_\alpha u$ near $\calR^+_+$ , but this is suppressed a factor of $\epsilon$ and therefore can be absorbed into the left-hand side of the ultimate estimate. 
	\item Because $R_{\alpha,\beta} \in \Psi_{\mathrm{de,sc}}^{-\infty,-\mathsf{1}}$, we have $R_{\alpha,\beta}\varrho^{1/2} Q A_{\alpha,\mathsf{s}+1/2} \in \Psi_{\mathrm{de,sc}}^{-\infty,-\mathsf{s}} \frakN_+^{k-1}$. Thus, the 
	\begin{equation} 
		|\langle \varrho^{1/2} Q A_{\alpha,\mathsf{s}+1/2} u, R_{\alpha,\beta}\varrho^{1/2} Q A_{\beta,\mathsf{s}+1/2} u \rangle |
	\end{equation} 
	terms in \cref{eq:misc_j76} can be estimated like the previous class of terms. 
	\item Finally, consider the $|\langle  Q A_{\alpha,\mathsf{s}+1/2}Pu , Q A_{\alpha,\mathsf{s}+1/2} u \rangle|$ term in \cref{eq:misc_j76}. We can write 
	\begin{equation} 
		Q \varrho^{-1/2} = \varrho^{-1/2} Q + \varrho^{-1/2} F
	\end{equation} 
	for some $F\in \Psi_{\mathrm{de,sc}}^{-\infty,-1}$. Then, 
	\begin{multline}
		|\langle  Q A_{\alpha,\mathsf{s}+1/2}Pu , Q A_{\alpha,\mathsf{s}+1/2} u \rangle| \leq |\langle Q A_{\alpha,\mathsf{s}+1/2}Pu , \varrho^{-1/2} Q  A_{\alpha,\mathsf{s}} u \rangle| \\ + |\langle \varrho^{-1/2}  Q A_{\alpha,\mathsf{s+1/2}}Pu , F  A_{\alpha,\mathsf{s}} u \rangle|. 
		\label{eq:misc_391}
	\end{multline}
	By Cauchy--Schwarz and AM-GM, the second term on the right-hand side is bounded above as follows: 
	\begin{align}
		|\langle \varrho^{-1/2}  Q A_{\alpha,\mathsf{s}}Pu , F  A_{\alpha,\mathsf{s}+1/2} u \rangle| \lesssim \lVert \varrho^{-1/2} Q A_{\alpha,\mathsf{s}+1/2} Pu \rVert_{L^2}^2 + \lVert F Q A_{\alpha,\mathsf{s}} u \rVert_{L^2}^2. 
		\label{eq:misc_393}
	\end{align}
	The first term on the right-hand side of \cref{eq:misc_393} can be controlled using the hypotheses of the propositions to be proven. On the other hand, 
	\begin{equation} 
		FQ A_{\alpha,\mathsf{s}} \in \Psi_{\mathrm{de,sc}}^{-\infty,\mathsf{s}} \frakN_+^{k-1},
	\end{equation} 
	so the second term on the right-hand side of \cref{eq:misc_393} is controlled using the inductive hypothesis. The first term on the right-hand side of \cref{eq:misc_391} can be bounded above by 
	\begin{equation}
		2| \leq |\langle Q A_{\alpha,\mathsf{s}+1/2}Pu , \varrho^{-1/2} Q  A_{\alpha,\mathsf{s}} u \rangle| \leq \epsilon^{-1} \lVert  \varrho^{-1/2} Q A_{\alpha,\mathsf{s}+1/2}Pu \rVert_{L^2}^2 + \epsilon \lVert  Q  A_{\alpha,\mathsf{s}} u \rVert_{L^2}^2. 
	\end{equation}
	As above, the $\epsilon \lVert  Q  A_{\alpha,\mathsf{s}} u \rVert_{L^2}^2$ will be able to be absorbed into the left-hand side of the ultimate estimates. The remaining term is controllable, for each $\epsilon>0$, in terms of the $\smash{H_{\mathrm{de,sc}}^{-N,\mathsf{s}+1}}$ norms of $A_\alpha Pu$ near the radial set, which are finite by the hypotheses of the propositions to be proven.
\end{itemize}
The upshot is that, assuming the algebraic manipulations above are justified, then the $\smash{H_{\mathrm{de,sc}}^{-N,\mathsf{s}}}$ norms of the $A_\alpha u$ near $\smash{\calR^+_+}$ in terms of quantities already under control by the inductive hypothesis or assumptions of the propositions. 
Regularizing, in a manner completely analogous to that in \cite{HassellMelroseVasy,HassellMelroseVasy2}\cite{HassellETAL}, suffices to show that this estimate holds in the strong sense that if the terms on the right-hand side are all finite, then the left-hand side is finite as well. This then yields the next step in the induction on $k$. 

\subsection{Second induction}
\label{subsec:secondinduction}

We now induct on $\kappa$. We first prove the following ``parlaying'' lemma:
\begin{lemma}[Cf.\ \cite{HassellETAL}, eq.\ 3.31]
	Let $m\in \bbR$, $\mathsf{s}\in \bbR^5$ be arbitrary, and let $k\in \bbN^+$ and $\kappa \in \bbN$. 
	Suppose that $u\in \calS'$ satisfies 
	\begin{itemize}
		\item $\operatorname{WF}_{\mathrm{de,sc}}^{*,\mathsf{s}}(Au) \cap \calR^+_+=\varnothing$  and
		\item  $\operatorname{WF}_{\mathrm{de,sc}}^{*,\mathsf{s}+1}(AP u) \cap \calR^+_+=\varnothing$ 
	\end{itemize}
	for all $A\in \frakM_{+,+}^{\kappa,k}$. Then, 
	\begin{equation} 
		\operatorname{WF}_{\mathrm{de,sc}}^{*,\mathsf{s}}(Au) \cap \calR^+_+=\varnothing
		\label{eq:misc_b34}
	\end{equation} 
	for all $A\in \frakM_{+,+}^{\kappa+1,k-1}$.
	\label{lem:R_ind}
\end{lemma}
The way we will use this is, in the proof of \Cref{thm:R1}, \Cref{thm:R2}, as follows: use the $\kappa=0$ case (from the previous subsection) of these theorems to conclude that the hypothesis of \Cref{lem:R_ind} holds with $0$ in place of $\kappa$ and $k+\kappa$ in place of $k$. We can then repeatedly use \Cref{lem:R_ind}, parlaying one order of $\frakN$-regularity into one order of $\frakM$-regularity each time, until eventually the conclusion of the relevant one of \Cref{thm:R1}, \Cref{thm:R2} is reached. 
\begin{proof}
	We will prove that, under the hypotheses of the lemma, \cref{eq:misc_b34} holds for all $A\in \smash{\frakM_{+,+}^{\kappa+1,j}}$ for $j\in \{0,\ldots,k-1\}$, proceeding inductively on $j$, with $j=0$ as the base case.

	For each $j$, that it suffices to check \cref{eq:misc_b34} for a set of $\Psi$DOs spanning $\smash{\frakM_{+,+}^{\kappa+1,j}}$ as a left $\Psi_{\mathrm{de,sc}}^{0,\mathsf{0}}$-module. 
	Since $\smash{\frakM_{+,+}^{\kappa+1,j}}$ is generated as a left $\Psi_{\mathrm{de,sc}}^{0,\mathsf{0}}$-module by products of the form $V_+ A_0$ for $\smash{A_0\in \frakM_{+,+}^{\kappa,j}}$ together with elements of $\smash{\frakM_{+,+}^{\kappa,j+1}}$, in order to show that \cref{eq:misc_b34} holds for all $A\in\frakM_{+,+}^{\kappa+1,j}$ it suffices to prove that 
	\begin{equation} 
		\operatorname{WF}_{\mathrm{de,sc}}^{*,\mathsf{s}}(V_+ A_0 u) \cap \calR^+_+ = \varnothing
	\end{equation}  
	for $A_0 \in \frakM_{+,+}^{\kappa,j}$. (Indeed, since $j\leq k-1$, if $A\in \frakM_{+,+}^{\kappa,j+1}$ then $A\in \frakM_{+,+}^{\kappa,k}$, so \cref{eq:misc_b34} holds for such $A$ by hypothesis. It is therefore only those $A$ of the form $A=V_+A_0$ that need to be considered.) In particular, in order to prove the result for $\kappa,j=0$, we only need to prove that 
	\begin{equation} 
		\operatorname{WF}_{\mathrm{de,sc}}^{*,\mathsf{s}}(V_+ u) \cap \calR^+_+ = \varnothing;
	\end{equation}  
	that is, we only need to consider $A_0=1$.

	Applying \Cref{prop:square_module_form}, we write $P = \chi\cdot \tau^{-2} (V_- V_+ +  (d-1) V_- - \tau i \mathsf{m} (d-2))+\varrho^2 R$ for $R=R_-$, $\chi$ as in that proposition. Consequently, for any $A_0\in \Psi_{\mathrm{de,sc}}$, 
	\begin{multline}
		\tau^{-2} V_-  V_+ A_0  u = A_0 Pu   - \chi\tau^{-2} (d-1) V_- A_0 u + \chi\tau^{-1} i \mathsf{m} (d-2) A_0 u - \varrho^2  R A_0u \\ +[\chi\tau^{-2} V_- V_+,A_0]  u + (d-1)[\chi\tau^{-2}  V_- ,A_0] u  - i\mathsf{m}(d-2)[\chi\tau^{-1},A_0]u+ [\varrho^2R,A_0] u. 
		\label{eq:misc_nnz}
	\end{multline}
		Since $\chi\tau^{-2} V_- \in \smash{\operatorname{Diff}_{\mathrm{de,sc}}^{1,( -1, -1,-\infty,-1,-1)}}$ is elliptic at $\smash{\calR^+_+}$ (this being what distinguishes $V_-$ from $V_+$; $V_+$ is instead elliptic at $\calR^-$), it suffices to prove that the sets 
	\begin{multline}
			\operatorname{WF}_{\mathrm{de,sc}}^{*,\mathsf{s}+1}(A_0 Pu),\operatorname{WF}_{\mathrm{de,sc}}^{*,\mathsf{s}+1}(\tau^{-2}  V_- A_0 u),\operatorname{WF}_{\mathrm{de,sc}}^{*,\mathsf{s}+1}(\chi \tau^{-1}   A_0 u), \operatorname{WF}_{\mathrm{de,sc}}^{*,\mathsf{s}+1}(\varrho^2 R A_0u), \\
			\operatorname{WF}_{\mathrm{de,sc}}^{*,\mathsf{s}+1}([\chi \tau^{-2} V_- V_+,A_0]u),\operatorname{WF}_{\mathrm{de,sc}}^{*,\mathsf{s}+1}([\chi\tau^{-2}  V_- ,A_0] u), \operatorname{WF}_{\mathrm{de,sc}}^{*,\mathsf{s}+1}([\chi\tau^{-1}, A_0]u) , \\
			 \operatorname{WF}_{\mathrm{de,sc}}^{*,\mathsf{s}+1}([\chi\varrho^2R,A_0] u)
	\end{multline}	
	are all disjoint from $\smash{\calR^+_+}$. 
	That this is true for $A_0 Pu$ is a hypothesis. 
	That this is true for $\tau^{-2}  V_- A_0  u, \chi \tau^{-1} A_0 u$ follows from the other hypothesis, which says that $\operatorname{WF}_{\mathrm{de,sc}}^{*,\mathsf{s}}(A_0 u) \cap \calR^+_+=\varnothing$, and from 
	\begin{equation} 
		\tau^{-2}  V_-, \chi\tau^{-1} \in \operatorname{Diff}_{\mathrm{de,sc}}^{1,-\mathsf{1}}.
	\end{equation} 
	On the other hand, to control the $\varrho^2 R A_0 u$ term: because $\frakN^2 \subset \smash{\Psi_{\mathrm{de,sc}}^{1,\mathsf{1}}} \frakN_+$, we have 
	\begin{equation} 
		\varrho^2 R  \in \Psi_{\mathrm{de,sc}}^{1,-\mathsf{1}} \frakN_+.
	\end{equation} 
	So, 
	\begin{equation} 
		\varrho^2 R A_0 u \in \Psi_{\mathrm{de,sc}}^{1,-\mathsf{1}} \frakN \frakM^{\kappa,j}_{+,+}\subseteq \Psi_{\mathrm{de,sc}}^{1,-\mathsf{1}}\frakM^{\kappa,j+1}_{+,+} \subseteq  \Psi_{\mathrm{de,sc}}^{1,-\mathsf{1}}\frakM^{\kappa,k}_{+,+}.
	\end{equation}
	Since $j\leq k-1$, our hypothesis on $u$ (which gives control of $\frakM^{\kappa,k}_{+,+}u$ at $\calR^+_+$) gives 
	\begin{equation}
		\operatorname{WF}_{\mathrm{de,sc}}^{*,\mathsf{s}+1}(\varrho^2 R A_0u) \cap \calR^+_+  \subseteq \operatorname{WF}_{\mathrm{de,sc}}^{*,\mathsf{s}}(\frakM_{+,+}^{\kappa,k}u) \cap \calR^+_+ = \varnothing,
	\end{equation}
	where the equality used the full hypothesis on $u$. 
	It only remains to check the terms in the second line of \cref{eq:misc_nnz}. 
	
	If $\kappa=0$ and $j=0$, then, since we are only considering $A_0=1$, all of the terms in the second line of \cref{eq:misc_nnz} are just zero, so we are done. 
	
	Otherwise: 
	\begin{itemize}
		\item 
		From $\varrho^2 R \in \Psi_{\mathrm{de,sc}}^{1,-\mathsf{1}} \frakN_+$, we get 
		\begin{equation} 
			[\varrho^2 R,A_0] \in \Psi_{\mathrm{de,sc}}^{1,-\mathsf{1}} (1_{\kappa>0} \frakM_{+,+}^{\kappa-1,j+1}+\frakM_{+,+}^{\kappa,j})
		\end{equation} 
		via \Cref{cor:module_commutator_lemma}, so the proposition's hypothesis implies that the final term in \cref{eq:misc_nnz} has $\operatorname{WF}_{\mathrm{de,sc}}^{*,\mathsf{s}+1}$ disjoint from $\smash{\calR^+_+}$.
		\item 
		Similarly, since $\chi \tau^{-2} V_-, \chi \tau^{-1} \in \Psi_{\mathrm{de,sc}}^{1,-\mathsf{1}}$, the same reasoning (with one less $\frakN$) also applies to the second and third-to-last terms in \cref{eq:misc_nnz}.
		\item 
		On the other hand, 
		\begin{align}
			[\chi \tau^{-2} V_- V_+,A_0]   &= \chi \tau^{-2} V_- [V_+,A_0]  + [\chi \tau^{-2} V_-, A_0] V_+; \\
			\tau^{-2} V_- [V_+,A_0] &\in \Psi_{\mathrm{de,sc}}^{1,-\mathsf{1}}( \frakM_{+,+}^{\kappa,j} + 1_{j>0} \frakM_{+,+}^{\kappa+1,j-1} ), \label{eq:misc_l0l}\\
			[\chi \tau^{-2} V_- , A_0]V_+ &\in \Psi_{\mathrm{de,sc}}^{1,-1}(1_{\kappa>0}\frakM_{+,+}^{\kappa,j} + 1_{j>0}\frakM_{+,+}^{\kappa+1,j-1}  + \frakM_{+,+}^{1,0} ) \label{eq:misc_l02}
		\end{align}
		via \Cref{cor:module_commutator_lemma}. The hypothesis of the proposition implies that 
		\begin{equation}
			\operatorname{WF}_{\mathrm{de,sc}}^{*,\mathsf{s}+1} (B u) \cap \calR^+_+ = \varnothing 
			\label{eq:misc_568}
		\end{equation} 
		for all $B\in \Psi_{\mathrm{de,sc}}^{1,-\mathsf{1}} \frakM_{+,+}^{\kappa,j}$. So, the only contributions left to control are the final term in \cref{eq:misc_l0l} and the last two terms in \cref{eq:misc_l02}
		These are all controlled by some variant of the inductive hypothesis.
		For example, since we already know the $\kappa=0$ and $j=0$ case of the result, \cref{eq:misc_568} holds also for 
		\begin{equation} 
			B\in \Psi_{\mathrm{de,sc}}^{1,-\mathsf{1}} \frakM_{+,+}^{1,0}.
		\end{equation}
		So, the final term in \cref{eq:misc_l02} is handled.
		If $j=0$, then we can immediately conclude
		\begin{align} 
			\begin{split} 
			\operatorname{WF}_{\mathrm{de,sc}}^{m,\mathsf{s}+1}(\tau^{-2} V_- [V_+,A] u ) \cap \calR^+_+ &= \varnothing,\\
			\operatorname{WF}_{\mathrm{de,sc}}^{m,\mathsf{s}+1}([\tau^{-2} V_- ,A]V_+u)\cap \calR^+_+&=\varnothing.
			\end{split} 
		\label{eq:misc_407}
		\end{align}  
		This completes the proof in the $j=0$ case. 
		If $j\geq 1$, then the inductive hypothesis says that \cref{eq:misc_568} holds
		for all 
		\begin{equation} 
			B\in \Psi_{\mathrm{de,sc}}^{1,-\mathsf{1}} \frakM_{+,+}^{\kappa+1,j-1},
		\end{equation} 
		so the remaining terms in \cref{eq:misc_l0l}, \cref{eq:misc_l02} are under control and
		we can still conclude \cref{eq:misc_407}. 
	\end{itemize}
\end{proof} 

Consequently: 
\begin{propositionp}
	If 
	\begin{equation} 
		\operatorname{WF}_{\mathrm{de,sc}}^{*,\mathsf{s}}(Au)\cap \calR^+_+ = \varnothing
		\label{eq:misc_h1z}
	\end{equation} 
	for all $A \in \frakN_+^{\kappa+k}$ and $\operatorname{WF}_{\mathrm{de,sc}}^{*,\mathsf{s}+1}(APu)\cap \calR^+_+ = \varnothing$ for all $A\in \frakM_{+,+}^{\kappa,k}$, then \cref{eq:misc_h1z} holds for all $A\in \frakM_{+,+}^{\kappa,k}$. 
\end{propositionp}

\Cref{thm:R1} and \Cref{thm:R2} follow.

\section{Proofs of main theorems}
\label{sec:proofs}

We now spell out the precise hypotheses under which the main theorems are proven. We do not aim to be maximally general here; we call a Lorentzian metric $g$ on $\bbR^{1,d}$ \emph{admissible} if the following conditions are satisfied: 
\begin{itemize}
	\item $g$ satisfies $g-g_{\bbM} \in \varrho_{\mathrm{Pf}}^2 \varrho_{\mathrm{nPf}}^2\varrho_{\mathrm{Sf}}^2\varrho_{\mathrm{nFf}}^2\varrho_{\mathrm{Ff}}^2 C^\infty(\bbO; {}^{\mathrm{de,sc}}\!\operatorname{Sym}^2 T^* \bbO)$, where $g_{\bbM}$ is the exact Minkowski metric, 
	\item $(\bbR^{1,d},g)$ is globally hyperbolic and $t$ serves as a time function, so that $\mathrm{d}t$ timelike, 
	\item $\Sigma_T = \{(t,\bfx)\in\bbR^{1,d}: t=T\}$ is a Cauchy hypersurface for each $T\in \bbR$, 
	\item any null geodesic, when projected down to $\bbM^\circ$, tends to null infinity in both directions.
\end{itemize}
The first condition specifies the precise sense in which $g$ is asymptotically flat. 

\begin{proposition}
	If $g \in g_{\bbM} + (1+t^2+r^2)^{-1} C^\infty(\bbM;{}^{\mathrm{sc}}\!\operatorname{Sym}T^* \bbM)$, then the first condition above is satisfied.
	\label{prop:allowed_decay}
\end{proposition}
\begin{proof}
A frame for ${}^{\mathrm{sc}}\!\operatorname{Sym}T^* \bbM$ is given by the sections $\mathrm{d}x_i\odot \mathrm{d} x_j$ for $i\in \{0,\ldots,d\}$, and the computations in \S\ref{sec:geometry} show that 
\begin{equation}
	\mathrm{d}x_i\odot \mathrm{d} x_j \in \varrho_{\mathrm{nPf}}^{-2} \varrho_{\mathrm{nFf}}^{-2} C^\infty(\bbO; {}^{\mathrm{de,sc}}\!\operatorname{Sym}^2 T^* \bbO).
	\label{eq:misc_567}
\end{equation}
Indeed, in order to prove \cref{eq:misc_567}, we only consider the situation near $\mathrm{nFf}\cap\mathrm{Ff}$, the other corners being similar. 
Taking the exterior derivative of \cref{eq:tr_inv_def}, we get that 
\begin{equation}
	\dd t,\dd r \in \varrho_{\mathrm{nPf}}^{-1}\varrho_{\mathrm{nFf}}^{-1}  C^\infty(\bbO;{}^{\mathrm{de,sc}} T^* \bbO ),
\end{equation}
locally.
Likewise, $r\dd \theta_j \in C^\infty(\bbO;{}^{\mathrm{de,sc}} T^* \bbO )$ locally. So,  $C^\infty(\bbM;{}^{\mathrm{sc}} T^* \bbM )\subseteq \varrho_{\mathrm{nPf}}^{-1}\varrho_{\mathrm{nFf}}^{-1}  C^\infty(\bbO;{}^{\mathrm{de,sc}} T^* \bbO )$. Taking the symmetric product yields \cref{eq:misc_567}.

Since  $(1+t^2+r^2)^{-1} \in \varrho_{\mathrm{Pf}}^2 \varrho_{\mathrm{nPf}}^4\varrho_{\mathrm{Sf}}^2\varrho_{\mathrm{nFf}}^4\varrho_{\mathrm{Ff}}^2 C^\infty( \bbO)$, this implies that 
\begin{equation} 
	g-g_{\bbM} \in \varrho_{\mathrm{Pf}}^2 \varrho_{\mathrm{nPf}}^2\varrho_{\mathrm{Sf}}^2\varrho_{\mathrm{nFf}}^2\varrho_{\mathrm{Ff}}^2 C^\infty(\bbO; {}^{\mathrm{de,sc}}\!\operatorname{Sym}^2 T^* \bbO),
\end{equation}
so $g$ is asymptotically flat in the sense above. 
\end{proof}

It is not difficult to construct $g \in g_{\bbM} + (1+t^2+r^2)^{-1} C^\infty(\bbM;{}^{\mathrm{sc}}\!\operatorname{Sym}T^* \bbM)$ besides $g_{\bbM}$ itself satisfying the other conditions above, so the discussion below applies to more than just exact Minkowski spacetime. 

Given the setup above, the d'Alembertian $\square_g$ satisfies $\square_g - \square \in \operatorname{Diff}_{\mathrm{de,sc}}^{2,-\mathsf{2}}(\bbO)$. Consider now an operator of the form 
\begin{equation}
	P = \square_g + Q + \mathsf{m}^2 
\end{equation}
for $Q \in \operatorname{Diff}_{\mathrm{de,sc}}^{1,-\mathsf{2}}(\bbO)$. Such an operator has all of the properties required in each of the previous sections, so we can cite the various results.

\subsection{Initial value problem}
\label{subsec:IVP}

We now prove \Cref{thm:main}. Let $\chi$ be as in that theorem. First, for solutions to the IVP that are assumed to be tempered:
\begin{proposition}
	Suppose that $u\in \calS'(\bbR^{1,d})$ is a solution to the IVP
	 \begin{equation}
	 	\begin{cases}
	 		Pu = f \\ 
	 		u|_{t=0} = u^{(0)}, \\ 
	 		\partial_t u|_{t=0} = u^{(1)}
	 	\end{cases}
	 \end{equation}
 	for some $f\in \calS(\bbR^{1,d})$, $u^{(0)},u^{(1)}\in \calS(\bbR^d)$. Then, $u$ has the form 
	\begin{equation}
		u = u_0 + \chi \varrho_{\mathrm{Pf}}^{d/2}\varrho_{\mathrm{Ff}}^{d/2} e^{-i \mathsf{m} \sqrt{t^2-r^2}} u_- + \chi \varrho_{\mathrm{Pf}}^{d/2}\varrho_{\mathrm{Ff}}^{d/2} e^{+i \mathsf{m} \sqrt{t^2-r^2}} u_+
	\end{equation}
	for some $u_0 \in \calS(\bbR^{1,d})$ and some  $u_\pm \in \varrho_{\mathrm{nPf}}^\infty\varrho_{\mathrm{Sf}}^\infty\varrho_{\mathrm{nFf}}^\infty C^\infty(\bbO) = \bigcap_{k\in \bbN} \varrho_{\mathrm{nPf}}^k\varrho_{\mathrm{Sf}}^k\varrho_{\mathrm{nFf}}^k C^\infty(\bbO)$.
	\label{prop:repeated_thm_main}
\end{proposition}
\begin{proof} 
In order to get started, we need to know that $u$ (which can be deduced to be smooth via the Duistermaat--H\"ormander theorem or propagation of singularities in physical space) is Schwartz in a neighborhood of $\mathrm{cl}_{\bbM}\{t=0\}$ in $\bbM$. 

One way to see this is to consider the advanced and retarded components $u^-(t,\bfx) = (1-\Theta(t)) u(t,\bfx)$ and $u^+(t,\bfx) = \Theta(t) u(t,\bfx)$, where $\Theta(t)=1_{t\geq 0}$ denotes a Heaviside function. We have $u^\pm \in \calS'(\bbR^{1,d})$; this is only nonobvious near $t=0$. For this, one can use the energy estimate corollary $u\in L^\infty_{\mathrm{loc}}(\bbR_t;L^2(\bbR^d))$.

These satisfy 
\begin{equation}
	P u^\pm(t,\bfx)  = \pm (\delta'(t) f_1(\bfx)+\delta(t) f_2(\bfx) )
\end{equation}
for some $f_1,f_2\in \calS(\bbR^d)$ depending on $u^{(0)}$ and $u^{(1)}$ and on $P$. The sc-wavefront sets $\operatorname{WF}_{\mathrm{sc}}(\delta'(t)f_1(\bfx))$, $\operatorname{WF}_{\mathrm{sc}}(\delta(t)f_2(\bfx))$ are disjoint from the sc-characteristic set of $P$. Indeed, it can be checked (either directly, or via an argument presented after the end of this proof) that 
\begin{equation}
	\operatorname{WF}_{\mathrm{sc}}(\delta'(t)f_1(\bfx)), \operatorname{WF}_{\mathrm{sc}}(\delta(t)f_2(\bfx)) \subseteq {}^{\mathrm{sc}}\! N^* \mathrm{cl}_\bbM \{t=0\} \cap {}^{\mathrm{sc}} \bbS^* \bbM , 
	\label{eq:misc_488}
\end{equation}
and the right-hand side is disjoint from the sc-characteristic set of $P$, which intersects $ {}^{\mathrm{sc}}\! N^* \mathrm{cl}_\bbM \{t=0\} $ only away from fiber infinity. 
Since $u^\pm$ vanish identically in one of the two temporal hemispheres $\mathrm{cl}_\bbM\{\mp t>0\}\backslash \mathrm{cl}_\bbM\{t=0\}$, $u^\pm$ has no sc-wavefront set over the corresponding hemisphere. We can therefore apply sc-propagation results \cite{VasyGrenoble} (noting that the wavefront sets $\operatorname{WF}_{\mathrm{sc}}(\delta'(t)f_1(\bfx))$, $\operatorname{WF}_{\mathrm{sc}}(\delta(t)f_2(\bfx))$ do not interrupt the propagation, since they are in the elliptic region) to conclude that the portion of $\operatorname{WF}_{\mathrm{sc}}(u^\pm)$ inside the sc-characteristic set  is a subset of the radial sets of the sc-Hamiltonian flow. The same therefore applies to $u=u^-+u^+$. But, by elliptic regularity in the sc-calculus (using that $f$ is Schwartz), $\operatorname{WF}_{\mathrm{sc}}(u)$ is a subset of the sc-characteristic set of $P$. So, $\operatorname{WF}_{\mathrm{sc}}(u)$ is a subset of the radial sets of the sc-Hamiltonian flow, which sit over $\overline{C}_\pm$. This implies that $u$ is Schwartz in a neighborhood of $\mathrm{cl}_{\bbM}\{t=0\}$ in $\bbM$.

This implies that 
\begin{equation}
	\operatorname{WF}_{\mathrm{de,sc}}(u) \subseteq \calR \cup {}^{\mathrm{de,sc}}\pi^{-1}(\mathrm{nPf}\cup \mathrm{nFf}). 
\end{equation} 
By \Cref{thm:propagation_through_scrI_1}, we can strengthen this to
\begin{equation}
	\operatorname{WF}_{\mathrm{de,sc}}(u) \subseteq \calR. 
	\label{eq:misc_788}
\end{equation}
Indeed, given any $m\in \bbR$ and $\mathsf{s}\in \bbR^5$, we can find some $m_0>m$ and $\mathsf{s}_0>\mathsf{s}$ such that the pair $(m_0,\mathsf{s}_0)$ satisfies the hypotheses of that theorem. The theorem then tells us that 
\begin{equation}
	\operatorname{WF}_{\mathrm{de,sc}}^{m_0,\mathsf{s}_0}(u) \subseteq \calR. 
\end{equation}
\Cref{eq:misc_788} then follows from the definition of $\operatorname{WF}_{\mathrm{de,sc}}$ (\cref{eq:WF_def}), since $\calR$ is closed.  

Now, we can find $Q_\pm \in \Psi_{\mathrm{de,sc}}^{0,\mathsf{0}}$ such that 
\begin{itemize}
	\item $1=Q_-+Q_+$, 
	\item $\operatorname{WF}'_{\mathrm{de,sc}}(Q_\pm) \cap \Sigma_{\mathsf{m},\mp} = \varnothing$.
\end{itemize} 
Let $\chi_0 \in C^\infty(\bbM)$ be identically equal to $0$ in some neighborhood of the past cap and identically equal to $1$ in some neighborhood of the future cap. Then, we can define 
\begin{equation}
	Q^\pm_- = (1-\chi_0) Q_\pm, \qquad Q^\pm_+ = \chi_0 Q_\pm. 
\end{equation}
For signs $\varsigma,\sigma \in \{-,+\}$, let $u^\varsigma_\sigma = Q^\varsigma_\sigma u$.
Observe that $P u^\varsigma_\sigma = Q^\varsigma_\sigma f + [P,Q^\varsigma_\sigma]u$. Since $\operatorname{WF}_{\mathrm{de,sc}}'([P,Q^\varsigma_\sigma]) \cap \calR=\varnothing$,  we have 
\begin{equation} 
	\operatorname{WF}_{\mathrm{de,sc}}'([P,Q^\varsigma_\sigma]) \cap \operatorname{WF}_{\mathrm{de,sc}}(u) =\varnothing,
\end{equation} 
which implies that $[P,Q^\varsigma_\sigma]u$ is Schwartz (by microlocality, \cref{eq:microlocality}). So, $f^\varsigma_\sigma = Q^\varsigma_\sigma f + [P,Q^\varsigma_\sigma]u$ is Schwartz. Moreover, by construction, 
\begin{equation}
	\operatorname{WF}_{\mathrm{de,sc}}(u^\varsigma_\sigma) \subseteq \calR^\varsigma_\sigma  .
	\label{eq:misc_877}
\end{equation}
For $\mathsf{s}$ with $s_{\mathrm{Pf}},s_{\mathrm{Ff}}<-1/2$, we can apply \Cref{thm:R2} (for each possible pair of signs) to the $u^\varsigma_\sigma$, the hypothesis of which is trivially satisfied as a consequence of \cref{eq:misc_788}, \cref{eq:misc_877}. The conclusion is that 
\begin{align}
	\begin{split} 
	u^\varsigma_- &\in H_{\mathrm{de,sc};\varsigma,-}^{\infty,(s_{\mathrm{Pf}},\infty,\infty,\infty,\infty );\infty,\infty}, \\ 
	u^\varsigma_+ &\in H_{\mathrm{de,sc};\varsigma,+}^{\infty,(\infty,\infty,\infty,\infty,s_{\mathrm{Ff}} );\infty,\infty}.
	\end{split} 
\end{align}
Taking $s_{\mathrm{Pf}},s_{\mathrm{Ff}} \in (-3/2,-1/2)$, we can cite \Cref{prop:asymptotic_main} to conclude that $u = u^-_-+u^+_-+u^-_++u^+_+$ has the form specified in the theorem.
\end{proof} 

If $f\in \calS(\bbR^d)$, then $f$, viewed initially as a function on $\Sigma_0 = \{(t,\bfx)\in \bbR^{1,d}:t=0\}$, can be extended to a Schwartz function $F$ on $\bbR^{1,d}$. 
This implies that, for any $m,s\in \bbR$,  
\begin{equation}
	\operatorname{WF}_{\mathrm{sc}}^{m,s}(\delta'(t) f(\bfx)) \subseteq \operatorname{WF}_{\mathrm{sc}}^{m,s_0}(\delta'(t) ) 
\end{equation}
for any $s_0\in \bbR$, because $\delta'(t) f(\bfx) = M_F \delta'(t)$ and $M_F \in \operatorname{Diff}_{\mathrm{sc}}^{0,-\infty}(\bbM)$. Similarly, $\operatorname{WF}_{\mathrm{sc}}^{m,s}(\delta(t) f(\bfx)) \subseteq \operatorname{WF}_{\mathrm{sc}}^{m,s_0}(\delta(t) )$. 
Consequently, in order to verify \cref{eq:misc_488}, it suffices to prove that 
\begin{equation}
	\operatorname{WF}_{\mathrm{sc}}^{m,s_0}(\delta'(t)), \operatorname{WF}_{\mathrm{sc}}^{m,s_0}(\delta(t)) \subseteq {}^{\mathrm{sc}}\! N^* \mathrm{cl}_\bbM \{t=0\} \cap {}^{\mathrm{sc}} \bbS^* \bbM 
\end{equation}
for some $s_0=s_0(m)\in \bbR$. 
Moreover, since $\partial_t \in \operatorname{Diff}_{\mathrm{sc}}^{1,0}(\bbM)$, we know that $\operatorname{WF}_{\mathrm{sc}}^{m,s_0}(\delta'(t)) \subseteq \operatorname{WF}_{\mathrm{sc}}^{m+1,s_0}(\delta(t))$, so it suffices to prove the above for just $\delta(t)$. In order to do this, we use that $\operatorname{WF}_{\mathrm{sc}}^{m,s}(w) = \calF_*^{-1}\circ  \operatorname{WF}_{\mathrm{sc}}^{s,m}(\calF w)$ for every $w\in \calS'(\bbR^{1,d})$, where $\calF$ is the spacetime Fourier transform and $\calF_*^{-1}$ is the involution of ${}^{\mathrm{sc}}\overline{T}^* \bbM$ switching frequency and position (choosing sign conventions appropriately). Thus, 
\begin{equation}
	\operatorname{WF}_{\mathrm{sc}}^{m,s_0}(\delta(t)) = \calF_* \operatorname{WF}_{\mathrm{sc}}^{s_0,m}(\delta(\bfx)).
\end{equation}
Recalling that the portion of $\operatorname{WF}_{\mathrm{sc}}^{s_0,m}(\delta(\bfx))$ over the interior is just $\operatorname{WF}^{s_0}(\delta(\bfx))$, if $s_0$ is sufficiently negative then $\operatorname{WF}_{\mathrm{sc}}^{s_0,m}(\delta(\bfx))$ is contained entirely over the boundary, which says that $\calF_* \operatorname{WF}_{\mathrm{sc}}^{s_0,m}(\delta(\bfx))$ is contained entirely at fiber infinity. Thus, $\operatorname{WF}_{\mathrm{sc}}^{m,s_0}(\delta(t)) \subseteq {}^{\mathrm{sc}} \bbS^* \bbM$. In order to see that $\operatorname{WF}_{\mathrm{sc}}^{m,s_0}(\delta(t)) \subseteq {}^{\mathrm{sc}}\! N^* \mathrm{cl}_\bbM \{t=0\} $, note that $t \delta = 0$ and $\triangle \delta = 0$, where $\triangle = -( \partial_{x_1}^2+\cdots+\partial_{x_d}^2)$ is the spatial Laplacian. The former implies that $\operatorname{WF}_{\mathrm{sc}}^{m,s_0}(\delta(t))$ is contained over $\mathrm{cl}_\bbM\{t=0\}$, and the latter implies that 
\begin{equation}
	\operatorname{WF}_{\mathrm{sc}}(\delta(t)) \subseteq \operatorname{Char}_{\mathrm{sc}}^{2,0}(\triangle). 
\end{equation}
As $\operatorname{Char}_{\mathrm{sc}}^{2,0}(\triangle) \cap {}^{\mathrm{sc}}\pi^{-1} \mathrm{cl}_\bbM\{t=0\} = N^* \mathrm{cl}_\bbM \{t=0\} $, this completes the verification. 

In order to see that a solution to the IVP \cref{eq:IVP} with Schwartz initial data (and indeed, much worse initial data) is automatically tempered, a basic energy estimate suffices; this is proved in \S\ref{subsec:temperedness}. Thus, the temperedness hypothesis of the previous proposition can be removed, yielding finally \Cref{thm:main}.

\subsection{Scattering problems}
\label{subsec:scattering}

Say that the forward problem for $P$ is well-posed if, for any $f\in \calS(\mathbb{R}^{1,d})$, there exists a unique $u\in C^\infty(\bbR^{1,d}) \cap \calS'(\bbR^{1,d})$ such that 
\begin{itemize}
	\item $Pu=f$, and 
	\item $\chi_0u \in \calS(\bbR^{1,d})$ whenever $\chi_0\in C^\infty(\bbM)$ is identically $0$ near the closed future timelike cap $\operatorname{cl}_\bbM C_+$.
\end{itemize}
There exist criteria in the literature that suffice for this. It should be possible to prove this for the $P$ considered above using the energy estimate in \S\ref{subsec:temperedness} in conjunction with a duality argument, but we do not present the details here, so the next proposition is stated with well-posedness of the forward problem as an assumption.

\begin{proposition}
	Let $v_\pm$ denote Schwartz functions on the past timelike cap of $\bbM$. 
	Then, assuming that the forward problem for $P$ is well-posed, there exists a unique function $u\in C^\infty(\bbR^{1,d}) \cap \calS'$ such that $Pu=0$ and 
		\begin{equation}
		u = u_0 + \chi \varrho_{\mathrm{Pf}}^{d/2}\varrho_{\mathrm{Ff}}^{d/2} e^{-i \mathsf{m} \sqrt{t^2-r^2}} u_- + \chi \varrho_{\mathrm{Pf}}^{d/2}\varrho_{\mathrm{Ff}}^{d/2} e^{+i \mathsf{m} \sqrt{t^2-r^2}} u_+
		\label{eq:misc_500}
		\end{equation}
	for some Schwartz $u_0\in \calS(\bbR^{1,d})$ and $u_\pm \in \varrho_{\mathrm{nPf}}^\infty \varrho_{\mathrm{Sf}}^\infty \varrho_{\mathrm{nFf}}^\infty  C^\infty(\bbO)$ such that, restricted to the past timelike cap, $u_\pm$ agree with $v_\pm$. 
\end{proposition}
\begin{proof}
	By \Cref{prop:borel}, there exists functions $u_{-,\mathrm{pre}},u_{+,\mathrm{pre}} \in  \varrho_{\mathrm{nPf}}^\infty \varrho_{\mathrm{Sf}}^\infty \varrho_{\mathrm{nFf}}^\infty \varrho_{\mathrm{Ff}}^\infty \in C^\infty(\bbO)$ such that $u_{\pm,\mathrm{pre}}$, when restricted to the past timelike cap, agree with $v_\pm$, and such that the function $u_{\mathrm{pre}}$ defined by 
	\begin{equation}  
		u_{\mathrm{pre}} = \chi \varrho_{\mathrm{Pf}}^{d/2}\varrho_{\mathrm{Ff}}^{d/2} e^{-i \mathsf{m} \sqrt{t^2-r^2}} u_{-,\mathrm{pre}} + \chi \varrho_{\mathrm{Pf}}^{d/2}\varrho_{\mathrm{Ff}}^{d/2} e^{+i \mathsf{m} \sqrt{t^2-r^2}} u_{+,\mathrm{pre}}
		\label{eq:misc_424}
	\end{equation} 
	satisfies $Pu_{\mathrm{pre}} \in \calS(\bbR^{1,d})$. Let $f=Pu_{\mathrm{pre}}$. By the existence clause of the well-posedness of the forward problem, there exists a function $w \in \calS'(\bbR^{1,d})$ such that $Pw = -f$ and $\chi_0 w$ is Schwartz whenever $\chi_0$ is identically $0$ near the future timelike cap. In particular, $w$ solves the IVP 
	\begin{equation}
		\begin{cases}
			Pw=-f, \\ 
			w|_{t=0} = w^{(0)}, \\ 
			\partial_t w|_{t=0} = w^{(1)} ,
		\end{cases}
	\end{equation}
	for some $w^{(0)},w^{(1)} \in \calS(\bbR^d)$. By \Cref{prop:repeated_thm_main}, 	$w$ has the form
	\begin{equation}
		w = w_0 + \chi \varrho_{\mathrm{Pf}}^{d/2}\varrho_{\mathrm{Ff}}^{d/2} e^{-i \mathsf{m} \sqrt{t^2-r^2}} w_- + \chi \varrho_{\mathrm{Pf}}^{d/2}\varrho_{\mathrm{Ff}}^{d/2} e^{+i \mathsf{m} \sqrt{t^2-r^2}} w_+
	\end{equation}
	for some Schwartz $w_0\in \calS(\bbR^{1,d})$ and $w_\pm \in C^\infty(\bbO)$. Moreover, $w_\pm$ can be chosen to be supported near $\mathrm{nFf}\cup \mathrm{Ff}$ (or even just near $\mathrm{Ff}$). Set $u=u_{\mathrm{pre}} + w$. This solves $Pu=0$ and has the form \cref{eq:misc_424} for $u_{\pm} = u_{\pm,\mathrm{pre}}+w_\pm$ and $u_0=w_0$. By the support condition on $w_\pm$, the restriction of $u_{\pm}$ to the past timelike caps are the same as the restriction of $u_{\pm,\mathrm{pre}}$. 
	
	Conversely, suppose that we are given $u$ of the form \cref{eq:misc_500} with $u_0\in \calS(\bbR^{1,d})$, $u_\pm \in \varrho_{\mathrm{nPf}}^\infty \varrho_{\mathrm{Sf}}^\infty \varrho_{\mathrm{nFf}}^\infty  C^\infty(\bbO)$ restricting to $v_\pm$ at the past timelike caps. Define $w_? = u-u_{\mathrm{pre}}$. This satisfies $Pw_?=-f$. Choosing $\chi_0 \in C^\infty(\bbM)$ to be identically $0$ near the future timelike cap and identically $1$ near the past timelike cap, we have
	\begin{equation}
		P(\chi_0 w_?) = -\chi_0 f + [P,\chi_0] w_?  \in \calS(\bbR^{1,d}). 
	\end{equation}
	Any function of the form \cref{eq:misc_500} lies in $H_{\mathrm{de,sc}}^{\infty,(-1/2-,\infty,\infty ,\infty,-1/2-)}(\bbO)$. 
	Since the leading order terms in the asymptotic expansions of $u,u_{\mathrm{pre}}$ at the past timelike cap agree, 
	\begin{equation}
		\chi_0 w_? \in H_{\mathrm{de,sc}}^{\infty,(-1/2+\varepsilon,\infty,\infty ,\infty,\infty)}(\bbO) 
	\end{equation}
	for any $\varepsilon<1$. By \Cref{thm:R1}, we can actually conclude that $\operatorname{WF}_{\mathrm{de,sc}}(\chi_0 w_?) \cap \calR_- = \varnothing$. Thus, by \Cref{thm:propagation_through_scrI_2}, $\chi_0 w_? \in \calS(\bbR^{1,d})$. This implies that $\chi_1 w_? \in \calS(\bbR^{1,d})$ whenever $\chi_1$ is identically $0$ near the future timelike cap. 
	So, $w_?$ solves the same forward problem that $w$ does. By the uniqueness clause of the well-posedness of the forward problem, $w_{?} = w$. This shows that $u$ is unique. 
\end{proof}

\subsection{Temperedness}
\label{subsec:temperedness}

Here, we give a self-contained proof that the solutions $u$ to the initial value problem are tempered.
The argument below is, unsurprisingly, of a standard sort via an energy estimate. The point is just that the specific assumptions under which the main theorem is stated suffice for the argument to go through.

The operators considered in the body of the paper, as well as their formal $L^2(\bbR^{1,d})$-based adjoints, have the form 
\begin{equation}
	L = (1+a_{00}) \frac{\partial^2}{\partial t^2} + \sum_{j=1}^d a_{0j} \frac{\partial}{\partial t} \frac{\partial}{\partial x_j} - \sum_{j,k=1}^d (1-a_{jk}) \frac{\partial}{\partial x_j} \frac{\partial }{\partial x_k} + \sum_{j=0}^d b_j \frac{\partial}{\partial x_j} + V  + \mathsf{m}^2
\end{equation}
for some $a_{jk} = a_{kj} \in C^\infty(\bbR^{1,d};\bbR)$, $b_j\in C^\infty(\bbR^{1,d})$, and $V\in C^\infty(\bbR^{1,d})$, all of which are decaying symbols on $\bbO$. In particular, on each Cauchy hypersurface $\Sigma_T = \{(t,\bfx)\in \bbR^{1,d} : t=T\}$, which stays away from null infinity, $\partial_t a_{j,k}(t,\bfx),\partial_{x_\ell} a_{j,k}(t,\bfx) \in \langle \bfx \rangle^{-2} L^\infty(\bbR^d_\bfx)$, $b_j(t,\bfx),V(t,\bfx) \in \langle \bfx \rangle^{-1} L^\infty(\bbR^d_\bfx)$, and likewise for higher derivatives. These suffice to prove the most basic estimates. Proving estimates that are uniform as $t\to \pm \infty$ will require taking into account temporal decay. 

Consider the $H^1$-energy 
\begin{equation}
	E[u](t) = \int_{\bbR^d} \Big(  \Big|\frac{\partial u}{ \partial t}\Big|^2  + \sum_{j=1}^d \Big|\frac{\partial u}{ \partial x_j}\Big|^2 + |u|^2 \Big) \dd^d x.
\end{equation}
Because $1+a_{00}>0$, owing to the assumption that $\nabla t$ is timelike, and because the matrix $\{1-a_{jk}\}_{j,k=1}^d$ is strictly positive definite, owing to the assumption that the hypersurface $\Sigma_T$ is spacelike for each $T$, $E[u](t)$ can be bounded above by some multiple of
\begin{equation}
	E_0[u](t) = \int_{\bbR^d} \Big( (1+a_{00})  \Big|\frac{\partial u}{ \partial t}\Big|^2  + \sum_{j,k=1}^d (1-a_{jk}) \frac{\partial u^*}{ \partial x_j} \frac{\partial u}{\partial x_k} + \mathsf{m}^2 |u|^2 \Big) \dd^d x. 
\end{equation}
Indeed, the assumptions imply that $\inf_{(t,\bfx)\in \bbR^{1,d}} (1+a_{00})>0$, as well as a similar uniform lower bound on the matrix $\smash{\{1-a_{jk}\}_{j,k=1}^d}$. Thus, $E[u](t) \leq CE_0[u](t)$ for some $C>0$. Conversely, $E_0[u](t) \leq C_0 E[u](t)$ for some other $C_0>0$.

If $u\in C^\infty(\bbR_t; C^\infty_{\mathrm{c}}(\bbR^{d}_{\bfx}))$, then 
\begin{multline}
	\frac{\mathrm{d} E_0[u]}{\mathrm{d}t} = \int_{\bbR^d} 2  \Re \Big[ \frac{\partial u^*}{\partial t} \Big( (1+a_{00})\frac{\partial^2 u}{\partial t^2} - \sum_{j,k=1}^d(1-a_{jk}) \frac{\partial^2 u}{\partial x_j \partial x_k} + \mathsf{m}^2 u  \Big)  \Big]  \dd^d x  \\ 
	+ \int_{\bbR^d} \Big( \frac{\partial a_{00}}{\partial t} \Big| \frac{\partial u}{\partial t} \Big|^2 + 2 \sum_{j,k=1}^d\frac{\partial a_{jk}}{\partial x_j} \Re \Big[ \frac{\partial u^*}{\partial t} \frac{\partial u}{\partial x_k}\Big]  - \sum_{j,k=1}^d \frac{\partial a_{jk}}{\partial t} \frac{\partial u^*}{\partial x_j} \frac{\partial u}{\partial x_k}  \Big) \dd^d x. 
	\label{eq:misc_499}
\end{multline}
The integral on the first line is 
\begin{equation}
	2\int_{\bbR^d} \Re \Big[ \frac{\partial u^*}{\partial t} Lu \Big]  \dd^d x - 2 \int_{\bbR^d} \Re \Big[ \frac{\partial u^*}{\partial t} 
	\Big( \sum_{j=1}^d a_{0j} \frac{\partial^2 u}{\partial t \partial x_j} +\sum_{j=0}^d  b_j \frac{\partial u}{\partial x_j}  + V u \Big) \Big] \dd^d x . 
\end{equation}
Using Cauchy--Schwarz and AM-GM, the first term here is bounded as follows: 
\begin{equation}
2	\Big| \int_{\bbR^d} \Re \Big[ \frac{\partial u^*}{\partial t}Lu \Big] \Big| \leq \frac{C}{\langle t\rangle^2} E_0[u](t) + \langle t\rangle^{2} \lVert Lu(t,-) \rVert_{L^2(\bbR^d)} .
\end{equation}
Likewise, 
\begin{equation}
	2\Big| \int_{\bbR^d} \Re \Big[ \frac{\partial u^*}{\partial t} \Big( Vu+ \sum_{j=0}^d b_j \frac{\partial u}{\partial x_j}  \Big) \Big]  \Big|  \leq C \Big( \sum_{j=0}^d \sup_{\bfx\in \bbR^d} |b_j(t,\bfx)| + \sup_{\bfx\in \bbR^d} |V(t,\bfx)| \Big) E_0[u](t) .
\end{equation}
Finally, since $2 \Re [\partial_t u^* \partial_{t}\partial_{x_j} u ]= \partial_{x_j}  |\partial_t u|^2$, we can write 
\begin{equation}
	2 \int_{\bbR^d} \Re \Big[ \frac{\partial u^*}{\partial t}  \sum_{j=1}^d a_{0j} \frac{\partial^2 u}{\partial t \partial x_j}\Big] \dd^d x = -2 \int_{\bbR^d} \Big| \frac{\partial u}{\partial t} \Big|^2 \sum_{j=1}^d \frac{\partial a_{0j}}{\partial x_j} \dd^d x,
\end{equation}
integrating by parts. This satisfies $ 2|\int_{\bbR^d} |\partial_t u|^2 \operatorname{div} a_{0\bullet} \dd^d x|\leq2 C \sup_{\bfx\in \bbR^d} |\operatorname{div} a_{0\bullet}(t,\bfx)|  E_0[u](t)$. Turning to the second line of \cref{eq:misc_499}, 
\begin{equation}
	\Big| \int_{\bbR^d} \frac{\partial a_{00}}{\partial t} \Big| \frac{\partial u}{\partial t}\Big|^2  \dd^d x \Big| \leq C\sup_{\bfx\in \bbR^d} \Big| \frac{a_{00}}{\partial t}\Big| E_0[u](t),
\end{equation}
and 
\begin{equation}
	\Big| \int_{\bbR^d} 2 \sum_{j,k=1}^d \frac{\partial a_{jk}}{\partial x_k} \Re \Big[ \frac{\partial u^*}{\partial t} \frac{\partial u}{\partial x_j}\Big]  - \sum_{j,k=1}^d \frac{\partial a_{jk}}{\partial t} \frac{\partial u^*}{\partial x_j} \frac{\partial u}{\partial x_k}  \dd^d x \Big| \leq \frac{3 C}{2}  \sup_{j,k} \sup_{\bfx\in \bbR^d}  \lVert \nabla a_{j,k} \rVert  E_0[u](t). 
\end{equation}
Let 
\begin{multline} 
	c(t) = C \Big( \langle t \rangle^{-2} + \sum_{j=0}^d \sup_{\bfx\in \bbR^d} |b_j(t,\bfx)| + \sup_{\bfx\in \bbR^d} |V(t,\bfx)|  + 2\sup_{\bfx\in \bbR^d} |\operatorname{div} a_{0\bullet}(t,\bfx)|  \\ + \frac{3}{2} \sup_{\bfx\in \bbR^d} \sup_{j,k} \lVert \nabla a_{j,k} \rVert \Big).
\end{multline} 
The above shows that 
\begin{equation}
	\frac{\mathrm{d} E_0[u]}{\mathrm{d} t} \leq c(t) E_0[u] + \langle t \rangle^2 \lVert Lu(t,-) \rVert_{L^2(\bbR^d)}. 
\end{equation}
Because each of the $b_j,V,\operatorname{div} a_{0\bullet},\nabla a_{j,k}$ is a decaying symbol on $\bbM$, their supremums over $\Sigma_T$ depend continuously on $T$. Thus, $c\in C^0(\bbR;\bbR^+)$. 
Gr\"onwall's inequality then says that $E_0[u](t) \leq \exp ( \int_0^t c(s) \dd s)( E_0[u](0) + \int_0^t \langle s \rangle^2 \lVert Lu(s,-) \rVert_{L^2(\bbR^d)} \dd s   )$, which implies 
\begin{equation}
	E[u](t) \leq C \exp \Big( \int_0^t c(s) \dd s \Big) \Big( C_0 E[u](0) + \int_0^t \langle s \rangle^2 \lVert L u(s,-) \rVert_{L^2(\bbR^d)} \dd s \Big).
\end{equation}
This was proven under the assumption that $u(t,-)$ be compactly supported, but using e.g. finite speed of propagation this assumption can be removed. Consequently, if $u\in C^\infty(\bbR^{1,d})$ solves $Lu\in \calS$, then 
\begin{equation} 
	E[u](t) \leq C  \exp \Big( \int_0^t c(s) \dd s \Big) \Big( C_0E[u](0) + \int_0^\infty  \langle s \rangle^2 \lVert L u(s,-) \rVert_{L^2(\bbR^d)} \dd s \Big) ,
\end{equation} 
where part of the conclusion is that, if $\partial_t u|_{t=0}\in L^2(\bbR^d_\bfx)$ and $u|_{t=0} \in H^1(\bbR^d_\bfx)$, then $u(t,\bfx)\in H^1(\bbR^d_\bfx)$ for each $t\in\bbR$. 

Now we use that being a decaying symbol on $\bbO$ implies improved decay as $t\to\infty$. Indeed, we are assuming that $b_j,V$ are symbols of order $-2$ on $\bbO$, so that $b_j ,V \in (1+t^2+x^2)^{-1/2} L^\infty(\bbR^{1,d})$. We are also assuming this of the $a_{j,k}$, and \cref{eq:misc_062}, \cref{eq:misc_064} imply that $\partial_t a_{j,k},\partial_{x_\ell} a_{j,k}$ are then also symbols of the same order (actually one better order at $\mathrm{Pf},\mathrm{Sf},\mathrm{Ff}$), so
\begin{equation}
	\partial_t a_{j,k},\partial_{x_\ell} a_{j,k} \in \frac{1}{(1+t^2+x^2)^{1/2}} L^\infty(\bbR^{1,d})
\end{equation}
as well. 
This all implies that $c \in \langle t \rangle^{-1} L^\infty(\bbR_t)$. Thus, $\int_0^t c(s) \dd s$ diverges at worst logarithmically, and so $E[u](t) \leq C_1 \langle t\rangle^{C_2} E[u](0)$ for some $C_1,C_2>0$. 
So, $u\in \langle t\rangle^{C_2} L^\infty(\bbR_t;H^1(\bbR^d_{\bfx})) \subseteq \mathcal{S}'(\bbR^{1,d})$.

\section*{Acknowledgements}
I am grateful to my PhD. advisor Peter Hintz, whose direction this work was done under. Special thanks are due to Mikhail Molodyk for his careful reading and corrections, which included patching a hole in the original proof of \Cref{prop:beth_propagation_one}. In addition, thanks are due to the two referees for their corrections and suggestions.

\printbibliography
\end{document}